\documentclass[hidelinks, 10pt]{amsart}

\usepackage[margin=4cm]{geometry}

\usepackage{amssymb}

\usepackage{cite}
\usepackage{hyperref}
\usepackage[all]{xy}
\usepackage{enumerate}
\usepackage{enumitem}
\usepackage{bbm}
\usepackage{tikz-cd}
\usepackage{extpfeil}
\usepackage{relsize}
\usepackage[bbgreekl]{mathbbol}
\usepackage{slashed}
\usepackage{amsfonts}
\usepackage{bm}
\usepackage{etoolbox}

\DeclareMathOperator{\Cris}{{Cris}}

\newcommand{\bA}{{\mathbb A}}

\newcommand{\bF}{{\mathbb F}}
\newcommand{\bG}{{\mathbb G}}

\newcommand{\bN}{{\mathbb N}}

\newcommand{\bP}{{\mathbb P}}
\newcommand{\bQ}{{\mathbb Q}}

\newcommand{\bZ}{{\mathbb Z}}

\newcommand{\cA}{{\mathcal A}}

\newcommand{\cC}{{\mathcal C}}
\newcommand{\cD}{{\mathcal D}}
\newcommand{\cE}{{\mathcal E}}
\newcommand{\cF}{{\mathcal F}}

\newcommand{\cL}{{\mathcal L}}
\newcommand{\cM}{{\mathcal M}}

\newcommand{\cO}{{\mathcal O}}
\newcommand{\cP}{{\mathcal P}}

\newcommand{\cR}{{\mathcal R}}

\newcommand{\cY}{{\mathcal Y}}

\newcommand{\fa}{{\mathfrak a}}

\newcommand{\fp}{{\mathfrak p}}

\newcommand{\oM}{\overline{M}}

\newcommand{\ou}{\overline{u}}

\newcommand{\ta}{\widetilde{a}}

\newcommand{\tm}{\widetilde{m}}

\newcommand{\tx}{\widetilde{x}}
\newcommand{\ty}{\widetilde{y}}

\newcommand{\tchi}{\widetilde{\chi}}

\newcommand{\tB}{\widetilde{B}}

\newcommand{\tE}{\widetilde{E}}
\newcommand{\tF}{\widetilde{F}}

\newcommand{\tV}{\widetilde{V}}

\newcommand{\et}{\mathrm{\acute{e}t}}

\DeclareMathOperator*{\colim}{colim}

\DeclareMathOperator{\gr}{{gr}}

\DeclareMathOperator{\Isom}{{Isom}}
\DeclareMathOperator{\Lie}{{Lie}}

\DeclareMathOperator{\Spf}{{Spf}}

\DeclareMathOperator{\Mod}{{Mod}}

\newcommand{\RGamma}{\mathrm{R}\Gamma}

\DeclareMathOperator{\Ext}{{Ext}}
\DeclareMathOperator{\Hom}{{Hom}}

\DeclareMathOperator{\id}{{id}}
\DeclareMathOperator{\im}{{Im}} 

\DeclareMathOperator{\Mor}{{Mor}}

\DeclareMathOperator{\op}{{op}}

\DeclareMathOperator{\Perf}{{Perf}}
\DeclareMathOperator{\Spec}{{Spec}}

\DeclareMathOperator{\coker}{coker}
\DeclareMathOperator{\cofib}{cofib}
\DeclareMathOperator{\QCoh}{QCoh}
\DeclareMathOperator{\Map}{Map}
\DeclareMathOperator{\Rep}{Rep}

\DeclareMathOperator{\Cat}{Cat}

\DeclareMathOperator{\diag}{diag}
\DeclareMathOperator{\WCart}{WCart}
\DeclareMathOperator{\DAlg}{DAlg}
\DeclareMathOperator{\DK}{DK}
\DeclareMathOperator{\Sch}{Sch}

\DeclareMathOperator{\Proj}{Proj}

\DeclareMathOperator{\PreStk}{PreStk}
\DeclareMathOperator{\Ring}{Ring}

\DeclareMathOperator{\ob}{ob}

\newcommand{\hD}{\widehat{D}}

\DeclareMathOperator{\Gal}{{Gal}}

\newcommand{\alg}{\mathrm{alg}}

\newcommand{\Fr}{\operatorname{Fr}}

\newcommand{\dR}{\mathrm{dR}}
\newcommand{\HT}{\mathrm{HT}}
\newcommand{\cris}{\mathrm{cris}}

\newcommand{\grp}{\mathrm{grp}}
\DeclareMathOperator{\QC}{\mathcal{QC}}
\DeclareMathOperator{\Func}{Func}
\DeclareMathOperator{\QSyn}{QSyn}
\DeclareMathOperator{\AffQSyn}{AffQSyn}
\DeclareMathOperator{\Alg}{Alg}
\DeclareMathOperator{\CAlg}{CAlg}

\DeclareMathOperator{\Stab}{Stab}

\newcommand{\Id}{\mathrm{Id}}
\newcommand{\Fil}{\mathrm{Fil}}
\newcommand{\Tr}{\mathrm{Tr}}

\newcommand{\oF}{\overline{F}}

\newcommand{\sgn}{\mathrm{sgn}}
\newcommand{\RG}{\mathrm{R}\Gamma}
\newcommand{\Bock}{\mathrm{Bock}}

\newcommand{\fib}{\mathrm{fib}}
\newcommand{\can}{\mathrm{can}}
\newcommand{\RHom}{\mathrm{RHom}}
\newcommand{\conj}{\mathrm{conj}}

\newcommand{\qrsprfd}{\mathrm{QRSPrfd}}
\newcommand{\obj}{\mathrm{obj}}
\newcommand{\naive}{\mathrm{naive}}
\newcommand{\ks}{\mathrm{ks}}
\newcommand{\cosimp}{\mathrm{cosimp}}

\newcommand{\tors}{\mathrm{tors}}
\newcommand{\perf}{\mathrm{perf}}

\newcommand{\sing}{\mathrm{sing}}

\DeclareSymbolFontAlphabet{\mathbb}{AMSb} 
\DeclareSymbolFontAlphabet{\mathbbl}{bbold}

\newcommand{\comment}[1]{}

\makeatletter

\newcommand{\Rmnum}[1]{\expandafter\@slowromancap\romannumeral #1@}
\makeatother

\newtheorem{pr}{Proposition}[section]

\newtheorem{thm}[pr]{Theorem}

\newtheorem{lm}[pr]{Lemma}

\newtheorem{cor}[pr]{Corollary}

\theoremstyle{definition}
\newtheorem{rem}[pr]{Remark}

\theoremstyle{definition}

\newtheorem{definition}[pr]{Definition}
\newtheorem{construction}[pr]{Construction}
\newtheorem{notation}[pr]{Notation}

\newtheorem{example}[pr]{Example}
\newtheorem{convention}[pr]{Convention}

\numberwithin{equation}{section}

\DeclareMathOperator{\DHod}{\slashed{D}}

\newcommand{\com}[1]{\ifstrequal{0}{1}{\textcolor{blue}{#1}}{}}

\begin{document}

\newcount\shown
\shown=1
\title[Non-decomposability of the de Rham complex]{Non-decomposability of the de Rham complex and non-semisimplicity of the Sen operator}

\author[]{Alexander Petrov}
\address{Massachusetts Institute of Technology, Cambridge, USA}
\email{alexander.petrov.57@gmail.com}

\begin{abstract}
We describe the obstruction to decomposing in degrees $\leq p$ the de Rham complex of a smooth variety over a perfect field $k$ of characteristic $p$ that lifts over $W_2(k)$, and show that there exist liftable smooth projective varieties of dimension $p+1$ whose Hodge-to-de Rham spectral sequence does not degenerate at the first page. We also describe the action of the Sen operator on the de Rham complex in degrees $\leq p$ and give examples of varieties with a non-semisimple Sen operator. 

Our methods rely on the commutative algebra structure on de Rham and Hodge-Tate cohomology, and are inspired by the properties of Steenrod operations on cohomology of cosimplicial commutative algebras. The example of a non-degenerate Hodge-to-de Rham spectral sequence relies on a non-vanishing result on cohomology of groups of Lie type. We give applications to other situations such as describing extensions in the canonical filtration on de Rham, Hodge, and \'etale cohomology of an abelian variety equipped with a group action. We also show that the de Rham complex of a smooth variety over $k$ is formal as an $E_{\infty}$-algebra if and only if the variety lifts to $W_2(k)$ together with its Frobenius endomorphism.
\end{abstract}
\maketitle
\setcounter{tocdepth}{1}
\tableofcontents
\section{Introduction}
\subsection{The main results.} A celebrated theorem of Deligne and Illusie provides an analog of Hodge decomposition on de Rham cohomology in degrees $<p$ of varieties over a perfect field $k$ of positive characteristic $p$ that lift to $W_2(k)$. In \cite{deligne-illusie} they proved that for a smooth variety $X_0$ over $k$ that admits a lift to a smooth scheme $X_1$ over the ring $W_2(k)$ of length $2$ Witt vectors, the canonical truncation of the de Rham complex $\dR_{X_0/k}:=(F_{X_0/k*}\Omega^{\bullet}_{X_0/k},d)$ in degrees $\leq p-1$ decomposes in the derived category $D(X_0^{(1)})$ of $\cO_{X^{(1)}_0}$-modules:
\begin{equation}\label{intro: de rham complex decomposition}
\tau^{\leq p-1}\dR_{X_0/k}\simeq\bigoplus\limits_{i=0}^{p-1} \Omega^i_{X^{(1)}_0/k}[-i] .
\end{equation}
Passing to cohomology, this gives decompositions for $n<p$:
\begin{equation}\label{intro: cohomology decomposition}H^n_{\dR}(X_0/k)\simeq \bigoplus\limits_{i+j=n}H^j(X_0,\Omega^i_{X_0/k})^{(1)} .\end{equation}
In this paper we investigate whether the de Rham complex of a variety liftable to $W_2(k)$ decomposes in further degrees. We prove that in general it does not decompose, and decomposition (\ref{intro: cohomology decomposition}) might not exist, answering a question of Deligne and Illusie \cite[Remarque 2.6(iii)]{deligne-illusie}:

\begin{thm}[Corollary \ref{nondeg example: main corollary}]\label{intro: nondeg example}
There exists a smooth projective variety $X_0$ over $k$ of dimension $p+1$ that lifts to $W(k)$ such that $\dim_k H^p_{\dR}(X_0/k)<\displaystyle\sum\limits_{i+j=p}\dim_k H^j(X_0,\Omega^i_{X_0/k})$. In particular, the Hodge-to-de Rham spectral sequence of $X_0$ does not degenerate at the first page, and $\dR_{X_0/k}$ is not quasi-isomorphic to $\bigoplus\limits_{i\geq 0}\Omega^i_{X^{(1)}_0/k}[-i]$.
\end{thm}

Moreover, we compute the obstruction to decomposing the truncation $\tau^{\leq p}\dR_{X_0/k}$ of the de Rham complex in degrees $\leq p$ in terms of other invariants of the variety $X_0$ and its lift $X_1$. Given that $\tau^{\leq p-1}\dR_{X_0/k}$ decomposes, the next truncation fits into an exact triangle in the derived category $D(X^{(1)}_0)$:
\begin{equation}\label{intro: deg p triangle}
\bigoplus\limits_{i=0}^{p-1}\Omega^i_{X^{(1)}_0/k}[-i]\to\tau^{\leq p}\dR_{X_0/k}\to\Omega^p_{X^{(1)}_0/k}[-p]
\end{equation}
and hence gives an extension class $$e_{X_1,p}\in H^{p+1}(X^{(1)}_0,\Lambda^p T_{X^{(1)}_0/k})\oplus H^p(X^{(1)}_0,\Lambda^pT_{X^{(1)}_0/k}\otimes\Omega^1_{X^{(1)}_0/k})\oplus\ldots$$ that vanishes if and only if (\ref{intro: deg p triangle}) splits in $D(X^{(1)}_0)$. We compute this class:

\begin{thm}[{Theorem \ref{cosimp applications: the best part de Rham} if $X_1$ lifts to $W(k)$, Corollary \ref{semiperf: smooth cor} in general}]\label{intro: main extension class}
All components of $e_{X_1,p}$ except for the one in $H^{p+1}(X^{(1)}_0,\Lambda^p T_{X^{(1)}_0/k})$ vanish, and that component is equal to
\begin{equation}\label{intro: main extension class formula}
\Bock_{X^{(1)}_1}(\ob_{F,X_1}\cdot\hspace{2pt} \alpha(\Omega^1_{X^{(1)}_0}))\in H^{p+1}(X^{(1)}_0,\Lambda^p T_{X^{(1)}_0/k})
\end{equation}
where 
\begin{itemize}
\item $\Bock_{X^{(1)}_1}:H^p(X^{(1)}_0,\Lambda^p T_{X^{(1)}_0/k})\to H^{p+1}(X^{(1)}_0,\Lambda^p T_{X^{(1)}_0/k})$ is the Bockstein homomorphism, that is the connecting morphism arising from the short exact sequence of sheaves on $X_1$: $0\to \Lambda^p T_{X^{(1)}_0/k}\to\Lambda^p T_{X^{(1)}_1/W_2(k)}\to \Lambda^p T_{X^{(1)}_0/k}\to 0$.

\item $\ob_{F,X_1}\in H^1(X^{(1)}_0,F_{X_0^{(1)}}^*T_{X^{(1)}_0})$ is the obstruction to lifting the Frobenius morphism $F_{X_0}:X_0\to X_0$ to an endomorphism of  $X_1$,

\item $\alpha(\Omega^1_{X^{(1)}_0})\in H^{p-1}(X^{(1)}_0,F_{X^{(1)}_0}^*\Omega^1_{X^{(1)}_0}\otimes(\Lambda^p\Omega^1_{X^{(1)}_0})^{\vee})$ is a certain `characteristic class' of the cotangent bundle, defined in Definition \ref{free cosimplicial: alpha definition}.\end{itemize}
\end{thm}

The class $\alpha(E)\in H^{p-1}(X_0,F_{X_0}^*E\otimes (\Lambda^p E)^{\vee})$ is defined for any vector bundle $E$ on $X_0$, and it can be made especially explicit when $p=2$, in which case it is the class of the extension
\begin{equation}
0\to F_{X_0}^*E\to S^2E\xrightarrow{e_1\cdot e_2\mapsto e_1\wedge e_2}\Lambda^2E\to 0.
\end{equation}

The fact that $e_{X_1,p}$ has only one potentially non-zero component is a shadow of an additional $\bZ/p$-grading on the de Rham complex $\dR_{X_0/k}$ in the presence of a lift $X_1$, discovered by Drinfeld through his work on prismatization \cite{drinfeld}. More generally, Bhatt-Lurie \cite{apc} showed that a lift of $X_0$ to a smooth $W_2(k)$-scheme $X_1$ induces an endomorphism $\Theta_{X_1}$ of $\dR_{X_0/k}$ in the derived category $D(X^{(1)}_0)$, called the {\it Sen operator}. 

The action of $\Theta_{X_1}$ on the $i$th cohomology sheaf $H^i(\dR_{X_0/k})\simeq\Omega^i_{X^{(1)}_0/k}$ is given by multiplication by $(-i)$. Therefore generalized eigenspaces for this operator form a decomposition
\begin{equation}\label{intro: drinfeld decomposition}
\dR_{X_0/k}\simeq\bigoplus\limits_{i=0}^{p-1}\dR_{X_0/k,i}
\end{equation}
such that the cohomology sheaves of the object $\dR_{X_0/k,i}$ are non-zero only in degrees congruent to $i$ modulo $p$. In particular, $\tau^{\leq p}\dR_{X_0/k,0}$ is an extension of $\Omega^p_{X^{(1)}_0/k}[-p]$ by $\cO_{X^{(1)}_0}$ that splits off as a direct summand from the extension (\ref{intro: deg p triangle}), corroborating the fact that the components of $e_{X_1,p}$ in $H^{p+1-i}(X^{(1)}_0,\Lambda^pT_{X^{(1)}_0}\otimes \Omega^i_{X^{(1)}_0})$ vanish for $i>0$. It was observed earlier by Achinger and Suh \cite{achinger-suh} that these components vanish for $i>1$, for a purely homological-algebraic reason which is related to our method.

The Sen operator $\Theta_{X_1}$ not only induces the decomposition (\ref{intro: drinfeld decomposition}) but also equips each direct summand $\dR_{X_0/k,i}$ with an additional structure of the nilpotent endomorphism $\Theta_{X_1}+i$. Our methods allow us to describe the action of $\Theta_{X_1}$ on $\tau^{\leq p}\dR_{X_0/k,0}$. Since $\tau^{\leq p}\dR_{X_0/k,0}$ has only two non-zero cohomology sheaves and $\Theta_{X_1}$ acts on them by zero, it naturally defines a map $\Omega^p_{X^{(1)}_0}[-p]\to\cO_{X^{(1)}_0}$ in the derived category $D(X^{(1)}_0)$ whose cohomology class we denote by $c_{X^{(1)}_1,p}\in H^p(X^{(1)}_0,\Lambda^p T_{X^{(1)}_0})$. This class can be described as:

\begin{thm}[{Theorem \ref{semiperf: main sen operator}}]\label{intro: sen class} For a smooth $W_2(k)$-scheme $X_1$ with the special fiber $X_0=X_1\times_{W_2(k)}k$ we have
\begin{equation}\label{intro: sen class equation}c_{X_1,p}=\ob_{F,X_1}\cdot\alpha(\Omega^1_{X_0})\end{equation}where the right hand side is the product of classes $\ob_{F,X_1}\in H^1(X_0,F_{X_0}^{*}T_{X_0/k})$ and $\alpha(\Omega^1_{X_0/k})\in H^{p-1}(X_0,F_{X_0}^*\Omega^1_{X_0/k}\otimes \Lambda^p T_{X_0/k})$
\end{thm}

It is not a coincidence that $e_{X_1,p}$ is the result of applying the Bockstein homomorphism to $c_{X_1,p}$ -- this follows from the basic properties of the action of the Sen operator on the diffracted Hodge cohomology of $X_1$, as we prove in Lemma \ref{sen operator: classes over zp2}. In particular, if $\tau^{\leq p}\dR_{X_0/k}$ is not decomposable then the Sen operator on $\tau^{\leq p}\dR_{X_0/k}$ must be non-semisimple. Thus Theorem \ref{intro: nondeg example} also provides an example of a liftable variety with a non-semisimple Sen operator on its de Rham cohomology, answering a question of Bhatt.

In fact, the non-vanishing of $c_{X_1,p}$ is a more frequent phenomenon than that of $e_{X_1,p}$ as we demonstrate by the following:

\begin{cor}[Proposition \ref{nonsemisimp: p dim example}]
There exists a smooth projective variety $X_0$ of dimension $p$ over $k$ equipped with a lift $X_1$ over $W_2(k)$ such that the Sen operator $\Theta_{X_1}$ on $\dR_{X_0/k}$ is not semisimple.
\end{cor}

The examples we construct admit a smooth proper morphism to a curve, and for such varieties the class $\alpha(\Omega^1_{X_0})$ can be expressed in terms of the Kodaira-Spencer class of this fibration, giving a more tangible reformulation (Theorem \ref{nonsemisimp: sen class fibration}) of Theorem \ref{intro: sen class}.

We give two different, but similar in spirit, proofs of Theorem \ref{intro: main extension class}, both of which use crucially the multiplicative structure on the de Rham and Hodge-Tate complexes.

\subsection{Proof of Theorem \ref{intro: main extension class}.} The first proof, for which we additionally need to assume that $X_0$ lifts to a smooth (formal) scheme $X$ over $W(k)$, takes place in homotopical algebra rather than algebraic geometry. To prove Theorem \ref{intro: main extension class} in full generality, in a situation when only a lift over $W_2(k)$ exists, we appeal to Theorem \ref{intro: sen class}, which implies it by Corollary \ref{semiperf: smooth cor}. 

To describe the idea, recall the structure of the proof of \cite{deligne-illusie}. Given the lift $X_1$ over $W_2(k)$, they produce a map $s:\Omega^1_{X^{(1)}_0}[-1]\to \dR_{X_0/k}$ in the derived category $D(X^{(1)}_0)$ inducing an isomorphism on cohomology sheaves in degree $1$. Using that $\dR_{X_0/k}$ can be represented by a cosimplicial commutative algebra in quasi-coherent sheaves on $X_0$ via the \v{C}ech resolution, $s$ gives rise to maps 
\begin{equation}\label{intro: deg i splitting map}s_i:S^i(\Omega^1_{X^{(1)}_0/k}[-1])\to \dR_{X_0/k}\end{equation} 
for all $i\geq 0$ where $S^i$ denotes the derived symmetric power functor in the sense of \cite{illusie-cotangent1}. When $i<p$, the derived symmetric power $S^i(\Omega^1_{X^{(1)}_0/k}[-1])$ is identified with $\Omega^i_{X^{(1)}_0/k}[-i]$, and maps $s_i$ produce a quasi-isomorphism $\bigoplus s_i:\bigoplus\limits_{i=0}^{p-1}\Omega^i_{X^{(1)}_0/k}[-i]\simeq\tau^{\leq p-1}\dR_{X_0/k}$.

This method of decomposing $\dR_{X_0/k}$ stops working for $i=p$, because $S^p(\Omega^1_{X^{(1)}_0/k}[-1])$ is now different from $\Omega^p_{X^{(1)}_0/k}[-p]$. It has non-zero cohomology sheaves in several degrees, and we analyze this object in detail in Section \ref{free cosimplicial: section}. Specifically, there is a natural map $N_p:S^p(\Omega^1_{X^{(1)}_0/k}[-1])\to\Omega^p_{X^{(1)}_0/k}[-p]$ but the map $s_p$ need not factor through it. In Lemma \ref{free cosimplicial: algebra Bockstein} we compute the map $s_p$ on the fiber of $N_p$, and thus relate the obstruction to splitting $\tau^{\leq p}\dR_{X_0/k}$ to the obstruction to the existence of a section of the map $N_p$, the latter being precisely the class $\alpha(\Omega^1_{X_0})$. 

This argument is not specific to the de Rham complex and applies more generally to derived commutative algebras whose cohomology algebra is freely generated by $H^1$. In the body of the paper we work with derived commutative algebras in the sense of Mathew, but here we state the main algebraic result for the more classical notion of cosimplicial commutative algebras:

\begin{thm}[Theorem \ref{cosimp: main theorem}]\label{intro: general cosimplicial}
Let $A$ be a cosimplicial commutative algebra in quasi-coherent sheaves on a flat $\bZ_{(p)}$-scheme $X$. Suppose that $H^0(A)=\cO_X$, the sheaf $H^1(A)$ is a locally free sheaf of $\cO_X$-modules, and the multiplication on the cohomology of $A$ induces an isomorphism $\Lambda^{\bullet}H^1(A)\simeq H^{\bullet}(A)$. If there exists a map $s:H^1(A)[-1]\to A$ in $D(X)$ inducing an isomorphism on cohomology in degree $1$, then

\begin{enumerate}
    \item $\tau^{\leq p-1}A$ is quasi-isomorphic to $\bigoplus\limits_{i=0}^{p-1}H^i(A)[-i]$.
    \item The connecting map $H^p(A)\to (\tau^{\leq p-1}A)[p+1]$ corresponding to the fiber sequence $\tau^{\leq p-1}A\to\tau^{\leq p}A\to H^p(A)[-p]$ can be described as the composition
    \begin{multline}\label{intro: general algebra formula}
    H^p(A)=\Lambda^p H^1(A)\to \Lambda^pH^1(A)/p\xrightarrow{\alpha(H^1(A)/p)}F_{X_0}^*(H^1(A)/p)[p-1]\xrightarrow{F_{X_0}^*s[p]}(\tau^{\leq 1}F^*_{X_0}(A/p))[p]\\ \xrightarrow{\varphi_{A/p}}(\tau^{\leq 1}A/p)[p]\xrightarrow{\tau^{\leq 1}\Bock_A}(\tau^{\leq 1}A)[p+1]\xrightarrow{}(\tau^{\leq p-1}A)[p+1]
    \end{multline}
    Here $\varphi_{A/p}$ is the Frobenius endomorphism of the cosimplicial commutative algebra $A/p$, and $\Bock_A:A/p\to A[1]$ is the morphism corresponding to the fiber sequence $A\xrightarrow{p}A\to A/p$ on $X$.
\end{enumerate}
\end{thm}

The fact that the map $s_p:S^p(\Omega^1_{X^{(1)}_0}[-1])\to\dR_{X_0/k}$ need not factor through $\Omega^p_{X^{(1)}_0}[-p]$ is an incarnation of the classical phenomenon responsible for the existence of non-trivial Steenrod operations on mod $p$ cohomology of topological spaces. If the map $s:\Omega^1_{X^{(1)}_0/k}[-1]\to \dR_{X_0/k}$ in the derived category $D(X^{(1)}_0)$ was represented by a genuine map between these complexes, all the maps $s_i$ would have to factor through $\Omega^i_{X^{(1)}_0/k}[-i]$ because of the commutative DG algebra structure on the de Rham complex. However, $s$ can generally only be represented by a chain-level map into a model of $\dR_{X_0/k}$ where multiplication is commutative up to a set of coherent homotopies, rather than commutative on the nose. 

Theorem \ref{intro: general cosimplicial} draws inspiration from Steenrod's construction \cite{steenrod-symmetric-homology,may-steenrod} of cohomology operations from homology classes of symmetric groups, and our method of proof implies the following statement about operations on cohomology of cosimplicial commutative $\bF_p$-algebras, defined by Priddy \cite{priddy}. The appearance of Frobenius and Bockstein homomorphisms in the description of the degree $1$ Steenrod operation is directly related (Lemma \ref{free cosimplicial: algebra Bockstein} being the common reason for both) to their appearance in (\ref{intro: general algebra formula}).

\begin{pr}[{Proposition \ref{free cosimplicial: steenrod operations description prop}}]
For a cosimplicial commutative $\bF_p$-algebra $B$,

\begin{enumerate}
    \item the degree zero Steenrod operation $P^0:H^i(B)\to H^i(B)$ is equal to the Frobenius endomorphism $\varphi_B$ of $B$,
    \item the degree $1$ Steenrod operation $P^1:H^i(B)\to H^{i+1}(B)$ is equal to Serre's Witt vector Bockstein homomorphism induced by the exact sequence $B\xrightarrow{V}W_2(B)\to B$. In particular, for any cosimplicial commutative flat $\bZ/p^2$-algebra $\tB$ lifting $B$, $P^1$ can be described as the composition $H^i(B)\xrightarrow{\varphi_B}H^i(B)\xrightarrow{\Bock_{\tB}}H^{i+1}(B)$.
\end{enumerate}
\end{pr}

We apply Theorem \ref{intro: general cosimplicial} to the diffracted Hodge complex $A:=\Omega^{\DHod}_X$ of \cite{apc} (it coincides with the Hodge-Tate cohomology of \cite{prisms} over an appropriately chosen prism), and deduce Theorem \ref{intro: main extension class} by reducing modulo $p$, using that $\dR_{X_0/k}$ is the reduction of $\Omega^{\DHod}_{X^{(1)}}$. We thus along the way obtain in Theorem \ref{cosimp applications: HT extension} a similar description of the extension class corresponding to $\tau^{\leq p}\Omega^{\DHod}_X$ as well. Theorem \ref{intro: general cosimplicial} also applies to a number of cohomology theories evaluated on abelian varieties: in Section \ref{AV coherent: section} we use it to study extensions in the canonical filtration on coherent (Proposition \ref{AV coherent: coherent main}), de Rham (Proposition \ref{AV coherent: AV de rham}), and \'etale cohomology (Proposition \ref{AV coherent: etale}) of an abelian scheme being acted on by a group. The results on coherent and de Rham equivariant cohomology of abelian schemes play the key role in our example of a liftable variety with a non-degenerate conjugate spectral sequence.

A related approach to decomposability of the de Rham complex has been previously used by Beuchot \cite{beuchot}, who used a description of cohomology sheaves of $(\Omega^1_{X^{(1)}_0}[-1]^{\otimes p})_{hS_p}$ to prove decomposability of the truncation $\tau^{\leq p}\dR_{X_0/k}$ for a liftable variety $X_0$ under the assumption that certain cohomology groups of $X_0$ vanish. The consequence of Theorem \ref{intro: general cosimplicial} for the de Rham complex was also obtained independently by Robert Burklund, by a similar method.

At the moment it is unclear how to generalize Theorem \ref{intro: general cosimplicial}, and therefore Theorem \ref{intro: main extension class}, to describe extensions in degrees $>p$. We make some preliminary observations on this in Section \ref{higher ext: section}.

Computations of the map $S^p(\Omega^1_{X^{(1)}_0}[-1])\to\dR_{X_0/k}$ that we perform for the proof of Theorem \ref{intro: main extension class} also allow us to determine when the de Rham complex can be decomposed compatibly with the algebra structure. The following result is to be contrasted with the formality result of \cite{deligne-griffiths-morgan-sullivan} for smooth projective varieties in characteristic zero:

\begin{pr}[{Proposition \ref{cosimp applications: de rham formality}}]
For a smooth variety $X_0$ over $k$ the de Rham complex $\dR_{X_0/k}$ is quasi-isomorphic to $\bigoplus\limits_{i\geq 0}\Omega^i_{X^{(1)}_0/k}[-i]$ as an $E_{\infty}$-algebra in $D(X^{(1)}_0)$ if and only if $X_0$ admits a lift over $W_2(k)$ together with its Frobenius endomorphism.
\end{pr}

The `if' direction was proven in \cite{deligne-illusie} and is the starting point of their proof of the decomposition (\ref{intro: de rham complex decomposition}).

\subsection{Liftable variety with a non-degenerate conjugate spectral sequence} It is rather non-obvious that the formula (\ref{intro: main extension class formula}) ever takes non-zero values. Our example for Theorem \ref{intro: nondeg example} is an approximation (in the sense of \cite{antieau-bhatt-mathew}) of the classifying stack of a finite flat non-commutative group scheme $G$ over $W(k)$ defined as follows. Let $E$ be an elliptic curve over $W(k)$ whose reduction is supersingular. The group scheme $G$ is 
\begin{equation}
G:=SL_p(\bF_{p^2})\ltimes (E[p]\otimes_{\bF_p}\bF_{p^2}^{\oplus p}).   
\end{equation}

Here $E[p]\otimes_{\bF_p}\bF_{p^2}^{\oplus p}$ is the product of $2p$ copies of the $p$-torsion group scheme $E[p]$, and the discrete group $SL_p(\bF_{p^2})$ acts on it via the tautological representation on $\bF_{p^2}^{\oplus p}$.

We do not apply formula (\ref{intro: main extension class formula}) to the stack $BG$ directly, but rather study the conjugate spectral sequence of an auxiliary quotient stack by applying Theorem \ref{intro: general cosimplicial} to equivariant Hodge and de Rham cohomology of the abelian scheme $E^{\times d}$ for an appropriate $d$, with respect to an action of a certain discrete group.

First, we show that the class $\alpha(-)$ does not vanish in the universal example:

\begin{pr}[{Proposition \ref{rational group cohomology: main non-vanishing}}]
Denote by $V$ the tautological $p$-dimensional representation of the algebraic group $GL_{p,k}$. The class $\alpha(V)\in \Ext^{p-1}_{GL_{p,k}}(\Lambda^p V,V^{(1)})$ is non-zero.
\end{pr}

For the proof, we find an explicit representation for $\alpha(V)$ as a Yoneda extension and apply Kempf's vanishing theorem to compare spectral sequences associated to b\^{e}te and canonical filtrations on this explicit complex. We also prove an enhancement of this non-vanishing: the class $\alpha(V)$ remains non-zero after being restricted to the discrete group $GL_p(\bF_{p^r})$ whenever $r>1$ (Proposition \ref{group cohomology: from algebraic to discrete}), and that the Bockstein homomorphism applied to $\alpha(V)$ gives a non-zero class in the cohomology of the special linear group $SL_p(\cO_F)$ of the ring of integers in an appropriately chosen number field $F$ (Proposition \ref{group cohomology: ring of integers main}). These enhancements use the technique of Cline-Parshall-Scott-van der Kallen \cite{cpsk} for comparing the cohomology a reductive group with that of its group of $\bF_q$-points.

With these non-vanishing results in hand, we apply Theorem \ref{intro: general cosimplicial} to de Rham and coherent cohomology of an abelian scheme being acted on by a group. Our example in Theorem \ref{intro: nondeg example} stems from the following potential discrepancy between de Rham and coherent cohomology:

\begin{pr}[{Proposition \ref{nondeg example: AV quotient}}]\label{intro: equivariant de rham vs coherent}
There exists an abelian scheme $A$ over $W(k)$ with an action of a discrete group $\Gamma$ such that the truncation $\tau^{\leq p}\RGamma(A_0,\cO)$ of the cohomology of the structure sheaf of the special fiber $A_0=A\times_{W(k)}k$ is $\Gamma$-equivariantly decomposable, while its de Rham cohomology $\tau^{\leq p}\RGamma_{\dR}(A_0/k)$ is not.
\end{pr}

The ultimate reason for different behavior of de Rham and coherent cohomology in Proposition \ref{intro: equivariant de rham vs coherent} is that the Frobenius morphism appearing in (\ref{intro: general algebra formula}) is zero on $H^1(A_0,\cO)$, but is non-zero on $H^1_{\dR}(A_0/k)$. We moreover arrange that the stack $[A_0/\Gamma]$ has a non-degenerate conjugate spectral sequence. We then construct a map $[A_0/\Gamma]\to BG$ and deduce that the conjugate spectral sequence of $BG$ does not degenerate either, as desired. 

The dependence of the extension class (\ref{intro: general algebra formula}) on the Frobenius action yields the following criterion for supersingularity of an elliptic curve:

\begin{pr}[{Corollary \ref{AV coherent: supersingularity criterion}}]
\label{intro: oridnary vs supseringular}
Suppose that $n\geq 2p$. For any elliptic curve $E_0$ over $k$, the discrete group $GL_n(\bZ)$ naturally acts on $E_0^{\times n}$ and therefore on $\RGamma(E_0^{\times n},\cO)$. The complex $\tau^{\leq p}\RGamma(E_0^{\times n},\cO)$ is $GL_n(\bZ)$-equivariantly quasi-isomorphic to the direct sum $\bigoplus\limits_{i=0}^pH^i(E_0^{\times n},\cO)[-i]$ of its cohomology if and only if $E_0$ is supersingular.
\end{pr}

\subsection{Computation of the Sen operator.} Finally, let us describe the idea of our proof of Theorem \ref{intro: sen class}.  It might be possible to upgrade Theorem \ref{intro: general cosimplicial} to work in the category of sheaves on $X$ equipped with an endomorphism, but we instead evaluate the Sen operator on $\tau^{\leq p}\dR_{X_0/k}$ using the technique of descent to quasiregular semiperfectoid rings developed by \cite{bms2}. We set up the foundation for our computation in Section \ref{sen operator: section}, defining in particular a version of diffracted Hodge cohomology equipped with a Sen operator relative to the base $\bZ/p^n$ for $n\geq 2$ (Theorem \ref{sen operator: zpn main}). We then prove formula (\ref{intro: sen class equation}) in Theorem \ref{semiperf: main sen operator}.

For a quasiregular semiperfectoid $W(k)$-algebra $S$ with mod $p$ reduction $S_0=S/p$ the derived de Rham complex $\dR_{S_0/k}$ is concentrated in degree zero so the Sen operator in question is an endomorphism of a classical module, rather than an object of the derived world. To compute it, we consider the natural map \begin{equation}\label{intro: semiperf sym to dr map}\bigoplus\limits_{i\geq 0} S^i(L\Omega^1_S[-1])\to \Omega^{\DHod}_S\end{equation} from the symmetric algebra on the shifted cotangent complex of $S$, comprised of the analogs of the maps (\ref{intro: deg i splitting map}). 

Since $S$ is quasiregular semiperfectoid, the map (\ref{intro: semiperf sym to dr map}) is an injection of flat $W(k)$-modules with torsion cokernel. As $\Theta_S$ acts on the source of this map via multiplication by $(-i)$ on the summand $S^i(L\Omega^1_S[-1])$, this allows us to pin down exactly how $\Theta_S$ acts on $\Omega^{\DHod}_S$. A slight variation of this argument works for a semiperfectoid $W_2(k)$-algebra $S_1$ lifting $S_0$, even if there is no lift over $W(k)$, as we prove in Lemma \ref{semiperf: sen for qrsprfd}; though, as a result, we only compute the Sen operator on $\dR_{S_0}$ rather than $\Omega^{\DHod}_{S_1/\bZ/p^2}$. 

Note that, as in the proof of Theorem \ref{intro: main extension class}, we do not directly use here the construction of the Sen operator via the Cartier-Witt stack, but rather rely crucially on its basic properties: that it is a derivation with respect to the multiplication on $\Omega^{\DHod}_X$ and that it acts by scalar multiplication on graded quotients of the conjugate filtration.

\subsection*{Notation.} We use the machinery of $\infty$-categories in the sense of \cite{lurie-htt}. When referring to an ordinary category we always view it as an $\infty$-category via the simplicial nerve construction \cite[1.1.2]{lurie-htt}. 

For an object $\cF\in D(X)$ of the derived category of quasi-coherent sheaves on a prestack $X$ we denote by $H^i(\cF)\in \QCoh(X)$ the cohomology sheaf of $\cF$ in degree $i$, and by $H^i(X,\cF)$ the abelian group of the degree $i$ cohomology of the derived global sections complex $\RGamma(X,\cF)$. By $\tau^{\leq i}\cF$ we denote the canonical truncation: it is an object with a map $\tau^{\leq i}\cF\to \cF$ such that $H^j(\tau^{\leq i}\cF)\simeq H^j(\cF)$ for $j\leq i$, and $H^j(\tau^{\leq i}\cF)=0$ for $j>i$.

When working in a stable $\infty$-category $\cC$ linear over a commutative ring $R$ we denote by $\RHom_{\cC}(M,N)\in D(R)$ the complex of morphisms between objects $M,N\in \cC$. The mapping space $\Map_{\cC}(M,N)$ is equivalent to the simplicial set underlying the truncation $\tau^{\leq 0}\RHom_{\cC}(M,N)$.

For an algebraic stack $X$ over a ring $R$ we denote by $L\Omega^i_{X/R}$ the $i$th exterior power of its cotangent complex, viewed as an object of the derived category $D(X)$ of quasicoherent sheaves on $X$. If $pR=0$ we denote by $\dR_{X/R}$ the (derived) de Rham complex of $X$ relative to $R$, viewed as an object of $D(X\times_{R,F_R} R)$, cf. \cite[\S 3]{bhatt-derived}. We sometimes drop the base $R$ from the notation if it is clear from the context, especially if $R$ is perfect.

\subsection*{Acknowledgements.}I am very grateful to Vadim Vologodsky for many conversations on the surrounding topics throughout the years, I learned most of the techniques used in this paper from him. I am extremely thankful to Luc Illusie for numerous stimulating conversations, encouragement, and many very helpful comments and corrections on the previous versions of this text. 

I am also grateful to Bhargav Bhatt, Robert Burklund, Sanath Devalapurkar, Dmitry Kubrak, Adrian Langer, Shizhang Li, Shubhodip Mondal, Arthur Ogus, Alexei Skorobogatov, and Bogdan Zavyalov for helpful conversations, suggestions, and explanations. Special thanks to Dmitry and Shizhang for pointing out errors in the previous versions of some of the arguments. I would also like to thank the anonymous referees for their numerous corrections and helpful suggestions.

This work was completed during my stay at the Max Planck Institute for Mathematics in Bonn, I am grateful to MPIM for hospitality. This research was partially conducted during the period the author served as a Clay Research Fellow.

\section{Preliminaries on homotopical algebra}

The key to our approach to the study of the de Rham and Hodge-Tate cohomology is the fact that the de Rham complex and the diffracted Hodge complex admit a structure of a commutative algebra, in the appropriate derived sense. In this expository section we summarize the necessary facts about non-abelian derived functors, derived commutative algebras in the sense of Mathew, and cosimplicial commutative algebras. The specific piece of structure that will be used in Section \ref{cosimp: section} is a map $S^p A\to A$ where $S^p$ is the derived symmetric power functor in the sense of \cite{illusie-cotangent1}. Such a map is naturally induced both by the structure of a derived commutative algebra and a cosimplicial commutative algebra structure on $A$. We choose to work with the former, but comment in Subsection \ref{dalg: cosimp subsection} on how to translate the results to the language of cosimplicial commutative algebras.

Our exposition here definitely does not introduce any original mathematics but, on the contrary, is aiming to convince the reader that the part of this machinery relevant for us is powered by a fairly classical piece of homolog(top)ical algebra.

\subsection{Non-abelian derived functors.} \label{doldpuppe subsection}
Let $X$ be an arbitrary prestack, that is a functor $X:\Ring^{\op}\to \mathrm{Spc}$ from the category of ordinary commutative rings to the $\infty$-category of spaces. We always assume that for our prestack $X$ its derived $\infty$-category of quasi-coherent sheaves $D(X)$ is a presentable category in the sense of \cite[5.5]{lurie-htt}. Here $D(X)$ is defined as the limit $\lim\limits_{\Spec R\to X}D(R)$ of derived categories of $R$ modules, cf. \cite[Chapter I.3]{gaitsgory-rozenblum}. We denote by $\PreStk$ the $\infty$-category of prestacks $X$ such that $D(X)$ is presentable. We work in this large generality simply because all actual proofs in this section take place over affine schemes and the language of prestacks provides a convenient framework for generalizing the result to more general geometric objects via descent. 

All particular prestacks involved involved in our main results will actually be classical, in the sense that they will take values in $1$-groupoids. By an algebraic stack we will always mean a 1-Artin stack. In all of our applications, $X$ will either be a scheme, a global quotient of a scheme by a group scheme, or a $p$-adic formal scheme. If $X$ is a quasi-compact scheme with affine diagonal then $D(X)$ is equivalent to the derived category (in the sense of \cite[Definition 1.3.5.8]{lurie-ha}) of the usual abelian category $\QCoh(X)$ of quasi-coherent sheaves on $X$. If such an $X$ is being acted on by an affine flat group scheme $G$ then $D([X/G])$ is equivalent to the left completion of the derived category of the abelian category $\QCoh_G(X)$ of $G$-equivariant quasicoherent sheaves on $X$, cf. \cite[Remark 1.2.10]{drinfeld-gaitsgory}.

For each $n\geq 0$ we have the derived functors $S^n,\Gamma^n,\Lambda^n:D(X)\to D(X)$ of symmetric, divided, and exterior power functors, respectively. These are the $\infty$-categorical enhancements for the non-abelian derived functors defined by Illusie \cite{illusie-cotangent1}. We refer to \cite[Section 3]{brantner-mathew} and \cite[Section 2.3]{barwick-glasman-mathew-nikolaus} for a treatment in the case $X=\Spec R$ is an affine scheme, and \cite[Section A.2]{kubrak-prikhodko-p-adic} for the general case. We will now briefly recall the construction of these functors and their basic properties. Denote by $\Mod_R$ the ordinary abelian category of $R$-modules. These functors are uniquely characterized by the following properties:

\begin{enumerate}
\item If $X=\Spec R$ is affine then for a flat module $M\in \Mod_R\subset D(R)$, the values 
\begin{multline*}S^n(M)=(M^{\otimes n})_{S_n}, \Gamma^n(M)=(M^{\otimes n})^{S_n},\\ \Lambda^n(M)=(M^{\otimes n})/\langle m_1\otimes\ldots\otimes m_n|m_i=m_j\text{ for some }i\neq j\rangle\end{multline*} are the usual symmetric, divided, and exterior powers.

\item If $X$ is an affine scheme, the functors $S^n,\Lambda^n,\Gamma^n$ preserve sifted colimits, and are polynomial functors of degree $\leq n$ in the sense of \cite[Definition 2.11]{barwick-glasman-mathew-nikolaus}.

\item These functors are natural in morphisms of the underlying stacks. That is, the functors $S^n,\Gamma^n,\Lambda^n$ can be enhanced to endomorphisms of the functor $D(-):\PreStk^{\op}\to \Cat_{\infty}$ that sends $X$ to $D(X)$ and a morphism $f:X\to Y$ to the pullback functor $f^*:D(Y)\to D(X)$.
\end{enumerate}

These derived functors are first constructed on affine schemes using the following:

\begin{lm}[{\hspace{1sp}\cite{barwick-glasman-mathew-nikolaus}}]\label{dalg: derived functors for rings}
For a commutative ring $R$, the restriction functor \begin{equation}\label{dalg: derived functors for rings formula}\Func_{\leq n}(D(R),D(R))\to \Func_{\leq n}(\Proj_R^{\mathrm{f.g}},D(R))\end{equation} is an equivalence. Here the source of the functor is the category of polynomial of degree $\leq n$ endofunctors of the stable $\infty$-category $D(R)$ that preserve sifted colimits. The target category is the category of polynomial functors of degree $\leq n$ from the abelian category of finitely generated projective $R$-modules to the $\infty$-category $D(R)$.
\end{lm}

\begin{proof}
Statement (2) is a combination of \cite[Theorem 2.19]{barwick-glasman-mathew-nikolaus} applied to $\cA=\Proj(R)^{\mathrm{f.g}}$ with the fact that $D(R)$ is the ind-completion of $\mathrm{Stab}(\cA)=\Perf(R)$.
\end{proof}

Denote by $\Mod:\Ring\to \Cat$ the functor from the ordinary category of commutative rings to the $2$-category of ordinary categories that sends a commutative ring $R$ to the ordinary category $\Mod_R$ or $R$ modules, and a morphism of rings $f:R\to R'$ to the functor $R'\otimes_R -:\Mod_R\to \Mod_{R'}$. Similarly, denote by $\Proj^{\mathrm{f.g.}}$ the subfunctor of $\Mod$ that sends a ring $R$ to the category $\Proj^{\mathrm{f.g.}}_R$ of finitely generated projective $R$-modules.

We denote by $\Mor_{\leq n}(\Proj^{\mathrm{f.g.}},\Mod)\subset \Mor(\Proj^{\mathrm{f.g.}},\Mod)$ the full subcategory of the ordinary $1$-category of natural transformations $\Proj^{\mathrm{f.g.}}\to\Mod$ spanned by transformations that are given by polynomial functors $\Proj^{\mathrm{f.g.}}_R\to\Mod_R$ of degree $\leq n$, for every ring $R$. We have the following construction principle for polynomial functors on arbitrary prestacks

\begin{lm}\label{dalg: derived functors for stacks}
For every prestack $X$ there is a functor \begin{equation}\Sigma_X:\Mor_{\leq n}(\Proj^{\mathrm{f.g.}},\Mod)\to \Func(D(X),D(X))\end{equation} satisfying the following property:

For any morphism $f:\Spec R\to X$ from an affine scheme, and a polynomial functor $T\in \Mor_{\leq n}(\Proj^{\mathrm{f.g.}},\Mod)$ there is an equivalence of functors $f^*\circ \Sigma_X(T)\simeq T^{\mathrm{derived}}_R\circ f^*$ from $D(X)$ to $D(R)$, where $T^{\mathrm{derived}}_R$ is the result of applying the inverse of the equivalence (\ref{dalg: derived functors for rings formula}) to $T_R:\Proj_{R}^{\mathrm{f.g.}}\to\Mod_R$.

Moreover, if $X$ is an algebraic stack flat over a ring $A$, then $\Sigma_X$ factors through the category $\Mor_{\leq n}(\Proj^{\mathrm{f.g.}}|_{A-flat},\Mod|_{A-flat})$ of morphisms between these functors restricted to the category of flat $A$-algebras.
\end{lm}

\begin{proof}
For a general $X$ the category $D(X)$ is, by definition, equivalent to $\lim\limits_{\Spec R\to X}D(R)$, where the limit is taken over all affine schemes mapping to $X$. Therefore $\Func(D(X),D(X))\simeq \lim\limits_{\Spec R\to X}\Func(D(X),D(R))$ and we define $\Sigma_X(T)$ as the object of this limit given by $T_R^{\mathrm{derived}}\circ f^*\in \Func(D(X),D(R))$ for every map $f:\Spec R\to X$.

The last assertion follows from the fact that if $X$ is a flat algebraic stack over $A$, then in the above limit we may restrict to morphisms $\Spec R\to X$ for which $R$ is flat over $A$, by \cite[Proposition 3.1.4.2]{gaitsgory-rozenblum}.
\end{proof}

\begin{definition}
For a prestack $X$, we define functors $S^n,\Gamma^n,\Lambda^n:D(X)\to D(X)$ as the images under $\Sigma_X$ of the corresponding polynomial functors on projective modules over rings. We will denote these functors by $S^n_X,\Gamma^n_X,\Lambda^n_X$ if the base is not clear from the context.
\end{definition}

\begin{rem}We can give an explicit recipe for computing the derived functors $S^n,\Gamma^n,\Lambda^n$ that we state in a special case, for simplicity. Suppose that $X$ is a quasi-compact separated scheme and $M\in D(X)$ is an object that can be represented by a complex $M^0\to M^1\to\ldots$ concentrated in degrees $\geq 0$ of flat quasicoherent sheaves. Then $S^n(M)$ (and likewise for the other functors $\Lambda^n,\Gamma^n$) is the totalization of the cosimplicial diagram
\begin{equation}
\begin{tikzcd}
S^n(\DK(M)^0) \arrow[r, shift left=0.65ex] \arrow[r, shift right=0.65ex] & S^n(\DK(M)^1) \arrow[r, shift left=1.3ex] \arrow[r, shift right=1.3ex] \arrow[r] &\ldots
\end{tikzcd}
\end{equation}
where $\DK(M)$ is the cosimplicial sheaf associated to the complex $M$ by the Dold-Kan equivalence. The miracle is that the resulting object $S^n(M)$ does not depend on the choice of a resolution, up to a quasi-isomorphism. \end{rem}

We will use Lemma \ref{dalg: derived functors for stacks} frequently for homotopy-coherent computations with polynomial functors. Note that specifying an object of $\Mor_{\leq n}(\Proj^{\mathrm{f.g.}},\Mod)$ amounts to a manageable collection of data: we need to give a functor $T_R:\Proj^{\mathrm{f.g.}}_R\to \Mod_R$ for every ring $R$, construct equivalences $R'\otimes_R T_R(-)\simeq T_{R'}(R'\otimes_R -)$ for all maps of rings $R\to R'$, and then check (rather than provide any additional constructions) that these equivalences are associative with respect to the composition of maps of rings. Let us illustrate this technique by constructing certain maps between the symmetric and divided power functors.

For a projective module $M$ over a ring $R$ there are natural maps between $S^nM=(M^{\otimes n})_{S_n}$ and $\Gamma^nM=(M^{\otimes n})^{S_n}$:

\begin{equation}
r_n:\Gamma^n M\hookrightarrow M^{\otimes n}\twoheadrightarrow S^nM \qquad N_n=\sum\limits_{\sigma\in S_n}\sigma: S^nM\to \Gamma^nM.
\end{equation}

The compositions $N_n\circ r_n$ and $r_n\circ N_n$ are equal to $n!\cdot \Id_{\Gamma^n M}$ and $n!\cdot \Id_{S^n M}$, respectively. Since the maps $r_n$ and $N_n$ are compatible with base change along arbitrary maps of rings $R\to R'$, they define morphism between functors $S^n$ and $\Gamma^n$ in the category $\Mor_{\leq n}(\Proj^{\mathrm{f.g.}},\Mod)$. Therefore Lemma \ref{dalg: derived functors for stacks} provides us with maps of endofunctors
\begin{equation}\label{dalg: norm prestacks}
r_n:\Gamma^n\to S^n\qquad N_n:S^n\to \Gamma^n
\end{equation}
of the category $D(X)$ for every prestack $X$.  

\begin{lm}\label{dalg: polynomial frobenius lemma}
For an $\bF_p$-algebra $R_0$ and a finite projective $R_0$-module $M$ the map $r_p:\Gamma^pM\to S^pM$ naturally factors as
\begin{equation}
\Gamma^p M\xrightarrow{\psi_M} F^*_{R_0}M\xrightarrow{\Delta_M} S^p M
\end{equation}
where the first arrow is a surjection and the second one is an injection.
\end{lm}

\begin{proof}
We start by constructing the map $\Delta_M:M\otimes_{R_0,F_{R_0}}R_0=F^*_{R_0}M\to S^p M$. Define $\Delta_M(m\otimes r)=r\cdot m^p$, this is a well-defined additive $R_0$-linear map. To check that $r_p$ factors through $\Delta_M$ we may work Zariski locally on $\Spec R_0$ and thus assume that $M$ is a free $R_0$-module. Denote a basis for $M$ by $e_1,\ldots,e_n$. Then a basis for the module $\Gamma^p M\subset M^{\otimes p}$ is given by the elements $e_I:=\sum\limits_{\sigma\in S_p/\Stab_I}e_{i_{\sigma(1)}}\otimes \ldots\otimes e_{i_{\sigma(p)}}$ where $I=(i_1,\ldots,i_p)$ runs through $S_p$-orbits on $\{1,\ldots,n\}^{\times p}$. The map $r_p$ sends $e_I$ to $[S_p:\Stab_I]\cdot e_{i_1}\cdot\ldots\cdot e_{i_p}$. If the stabilizer of $I$ inside $S_p$ has index coprime to $p$, then $\Stab_I$ contains a long cycle of length $p$ which forces all $i_1,\ldots,i_p$ to be equal. Hence only basis elements of the form $e_i^{\otimes p}\in \Gamma^p M$ are not killed by $r_p$, which gives the desired factorization.
\end{proof}

By Lemma \ref{dalg: derived functors for stacks}, having constructed maps $\psi_M, \Delta_M$ for finite projective modules in a functorial fashion, we get maps 
\begin{equation}\label{dalg: polynomial frobenius}
\Delta_M:F_{X_0}^*M\to S^p M\qquad  \psi_M:\Gamma^p M\to F_{X_0}^*M
\end{equation}
for all objects $M\in D(X_0)$ for an arbitrary $\bF_p$-prestack $X_0$. The composition $\Delta_M\circ \psi_M$ is naturally homotopic to $r_p:\Gamma^p M\to S^p M$, by construction.

A useful computational tool for us will be the following `d\'ecalage' isomorphisms:

\begin{lm}[\hspace{1sp}{\cite[Proposition 4.3.2.1]{illusie-cotangent1},\cite[Proposition A.2.49]{kubrak-prikhodko-p-adic}}]
There are natural equivalences for $M\in D(X)$

\begin{equation}\label{dalg: decalage}
S^n(M[1])\simeq (\Lambda^n M)[n] \quad \Gamma^n(M[-1])\simeq (\Lambda^n M)[-n]
\end{equation}
\end{lm}

\subsection{Derived commutative algebras.} The precise version of the notion of a commutative algebra that we will need is derived commutative algebras, in the sense of Mathew and Mondal \cite{mathew-mondal}. In this section we collect the necessary material about this notion, largely following \cite[Section 4]{raksit}. Let us stress again that this section is purely expository and contains no original results.

Recall the following point of view on ordinary commutative algebras over a field $k$, afforded by the Barr-Beck theorem. The endofunctor $S^{\bullet}=\bigoplus\limits_{n\geq 0}S^n$ of the category of $k$-vector spaces has a structure of a monad, and the category of $k$-algebras is equivalent to the category of modules over the monad $S^{\bullet}$ in the category of $k$-vector spaces. The idea of derived commutative algebras is to replicate this definition in the derived world by using the appropriate derived version of the symmetric algebra monad.

\begin{construction}[{\hspace{1sp}\cite[Construction 4.2.19]{raksit}}]
For a ring $R$ consider the endofunctor $S^{\bullet}:=\bigoplus\limits_{n\geq 0}S^n:\Proj_R\to \Proj_R$ of the ordinary category of projective $R$-modules. It has a monad structure induced by the maps $S^i(S^jM)\twoheadrightarrow S^{i\cdot j}M$ for every projective module $M$. It naturally extends to a monad, in the sense of \cite[Definition 4.7.0.1]{lurie-ha}, on the category $D(X)$ for every prestack $X$ which we denote as

\begin{equation}
S^{\bullet}:=\bigoplus\limits_{n\geq 0} S^n:D(X)\to D(X)    
\end{equation}
and we call it the derived symmetric algebra monad.
\end{construction}

We can now define derived commutative algebras on $X$:

\begin{definition}
The category $\DAlg(X)$ of derived commutative algebras on $X$ is the category of modules over the derived symmetric algebra monad $S^{\bullet}$ in $D(X)$.
\end{definition}

By construction, there is a functor $S^{\bullet}:D(X)\to\DAlg(X)$ left adjoint to the forgetful functor $\DAlg(X)\to D(X)$. For a derived commutative algebra $A\in \DAlg(X)$ we will usually denote the underlying object in $D(X)$ by the same symbol $A$. The structure of a derived commutative algebra, in particular, gives a map 

\begin{equation}
m:A\otimes_{\cO_X}A\to S^2A\to A
\end{equation}
which induces a graded commutative product operation $H^i(A)\otimes_{\cO_X}H^j(A)\to H^{i+j}(A)$ on cohomology sheaves. 

\begin{definition}
For an object $M\in D(X)$, we call $S^{\bullet}(M)\in \DAlg(X)$ the free derived commutative algebra on $M$.
\end{definition}

\begin{rem}
For every $M\in D(X)$ there are natural maps $(M^{\otimes n})_{hS_n}\to S^n(M)$ which give rise to a map from the monad $M\mapsto \bigoplus\limits_{n\geq 0} (M^{\otimes n})_{hS_n}$ to $S^{\bullet}$. It is in general far from an equivalence, we make a precise comparison between their values in a special case in Lemma \ref{free cosimplicial: cosimp vs einf}. Since modules over the the former monad are the $E_{\infty}$-algebras on $X$, we get a functor $\DAlg(X)\to \Alg_{E_{\infty}}D(X)$.
\end{rem}

\comment{\begin{lm}
For any prestack $X$ we have an equivalence $\DAlg(X)\simeq \lim\limits_{\Spec R\to X}\DAlg(R)$.
\end{lm}

\begin{proof}
There is a natural functor $F:\DAlg(X)\to \lim\limits_{\Spec R\to X}\DAlg(R)$ that fits in the commutative diagram
\begin{equation}
\begin{tikzcd}
\DAlg(X)\arrow[r,"F"]\arrow[d]& \lim\limits_{\Spec R\to X}\DAlg(R)\arrow[d] \\
D(X)\arrow[r, "\sim"] & \lim\limits_{\Spec R\to X}D(R)
\end{tikzcd}
\end{equation}
where the vertical arrows are the forgetful functors. By definition of $\DAlg(R)$, the left adjoint to the forgetful functor $\lim\limits_{\Spec R\to X}\DAlg(R)\to \lim\limits_{\Spec R\to X}D(R)$ is the limit $\lim S^{\bullet}_R$ of free symmetric algebra functors. Moreover, under the equivalence $D(X)\simeq \lim\limits_{\Spec R\to X}D(R)$ the monad $S^{\bullet}_X$ is identified with the monad $\lim S^{\bullet}_R$. In particular, there is a functor $G:\lim\limits_{\Spec R\to X}\DAlg(R)\to \Mod_{S^{\bullet}}D(X)=\DAlg(X)$. For any $f:\Spec R\to X$ the composition of $G$ with the induced map $f^*:\DAlg(X)\to \DAlg(R)$ is equivalent to the natural functor $\lim\limits_{\Spec R\to X}\DAlg(R)\to \DAlg(R)$, hence the composition $F\circ G$ is equivalent to the identity functor.

We have $F\circ G\simeq \Id$, because, for any $\Spec R\to X$, the composition of $F\circ G$ with the functor 
\end{proof}}

Derived commutative algebras in characteristic $p$ have natural Frobenius endomorphisms:

\begin{lm}\label{dalg: frobenius}
Suppose that $X$ is a prestack over $\bF_p$. For all $A\in \DAlg(X)$ there is a natural morphism $\varphi_A:F_X^*A\to A$ in $D(X)$ that is equal to the linearization of the usual Frobenius endomorphism when $A$ is a flat ordinary commutative algebra on an affine scheme $X$. 
\end{lm}

\begin{rem}
It should be possible to enhance $\varphi_A$ to a morphism of derived commutative algebras, but we do not pursue this here.
\end{rem}

\begin{proof}
We define $\varphi_A$ as the composition $F_X^*A\xrightarrow{\Delta_A} S^pA\xrightarrow{m_A}A$ where $m_A$ is a part of the $S^{\bullet}$-module structure on $A$, and $\Delta_A$ is the morphism defined in (\ref{dalg: polynomial frobenius}). Compatibility with the usual Frobenius follows from the defining formula of $\Delta_A$ given in the proof of Lemma \ref{dalg: polynomial frobenius lemma}.
\end{proof}

One of the main motivations for introducing the notion of derived commutative algebras is that many cohomological invariants arising in geometry are naturally equipped with the structure of a derived commutative algebra.

\begin{pr}\label{dalg: pushforward of rings}
If $f:X\to Y$ is a morphism of prestacks then the functor $Rf_*:D(X)\to D(Y)$ can be naturally enhanced to a functor $Rf^{\alg}_*:\DAlg(X)\to \DAlg(Y)$.
\end{pr}

\begin{proof}
Recall that $Rf_*:D(X)\to D(Y)$ is defined as the right adjoint functor to the pullback functor $f^*:D(Y)\to D(X)$. The equivalences $f^*\circ S^n_Y\simeq S^n_X\circ f^*$ induce a colimit-preserving functor $f^*_{\alg}:\DAlg(Y)\to \DAlg(X)$ given by $f^*$ on the underlying objects of $D(-)$. Since $\DAlg(Y)$ is presentable (e.g. by \cite[Proposition 4.1.10]{raksit}), this functor admits a right adjoint $Rf_*^{\alg}:\DAlg(X)\to \DAlg(Y)$, by the adjoint functor theorem \cite[Corollary 5.5.2.9(1)]{lurie-htt}.

It remains to check that $Rf_*^{\alg}$ defined this way induces $Rf_*$ on the underlying sheaves. By construction, the compositions $D(Y)\xrightarrow{f^*}D(X)\xrightarrow{S^{\bullet}_X} \DAlg(X)$ and $D(Y)\xrightarrow{S^{\bullet}_Y}\DAlg(Y)\xrightarrow{f^*_{\alg}}\DAlg(X)$ are equivalent. Their right adjoints are $\DAlg(X)\to D(X)\xrightarrow{Rf_*}D(Y)$ and $\DAlg(X)\xrightarrow{Rf_*^{\alg}}\DAlg(Y)\xrightarrow{}D(Y)$, and their equivalence is precisely the desired compatibility between $Rf_*$ and $Rf_*^{\alg}$.
\end{proof}

In particular, $Rf_*\cO_X\in D(Y)$ is naturally equipped with a derived commutative algebra structure. The functor $Rf^{\alg}_*$ is compatible with Frobenius endomorphisms:

\begin{lm}\label{dalg: geometric frobenius}
Suppose that $Y$ is a prestack over $\bF_p$ and $f:X\to Y$ is a morphism of prestacks. For $A\in \DAlg(X)$ there is a natural equivalence between the composition $F_Y^*Rf_*A\to Rf_*F_X^*A\xrightarrow{Rf_*\varphi_A}Rf_*A$ and the map $F_Y^*Rf_*A\xrightarrow{\varphi_{Rf_*^{\alg}A}}Rf_*A$ in $D(Y)$. Here $F_Y^*Rf_*A\to Rf_*F_X^*A$ is the base change map corresponding to the equality $F_Y\circ f=f\circ F_X$.
\end{lm}

\begin{proof}
By adjunction between $f^*$ and $Rf_*$, and the definition of $\varphi_A$ and $\varphi_{Rf_*^{\alg}A}$, our task is equivalent to identifying the composition
\begin{equation}\label{dalg: geometric frobenius eq1}
f^*F_Y^*Rf_*A\to f^*Rf_*F_X^*A\xrightarrow{f^*Rf_*\Delta_A}f^*Rf_*S^p_XA\xrightarrow{f^*Rf_*m_A}f^*Rf_*A\to A
\end{equation}
with the composition
\begin{equation}\label{dalg: geometric frobenius eq2}
f^*F_Y^*Rf_*A\xrightarrow{f^*\Delta_{Rf_*A}}f^*S^p_YRf_*A\xrightarrow{f^*m_{Rf_*^{\alg}A}} f^*Rf_*A\to A
\end{equation}
The map $S^p_YRf_*A\xrightarrow{m_{Rf_*^{\alg}A}} Rf_*A$ is adjoint to the map $f^*S^p_YRf_*A\simeq S^p_Xf^*Rf_*A\to S^p_XA\xrightarrow{m_A} A$, which allows us to rewrite (\ref{dalg: geometric frobenius eq2}) as
\begin{equation}\label{dalg: geometric frobenius eq3}
f^*F_Y^*Rf_*A\xrightarrow{}f^*S^p_YRf_*A\simeq S^p_Xf^*Rf_*A\to S^p_XA\xrightarrow{m_A}A
\end{equation}
For any object $M\in D(Y)$ the map $f^*F_Y^*M\xrightarrow{f^*\Delta_M}f^*S^p_YM\simeq S^p_Xf^*M$ can be identified with the composition $f^*F_Y^*M\simeq F_X^*f^*M\xrightarrow{\Delta_{f^*M}}S^p_Xf^*M$ where the first equivalence arises from the fact that $f$ intertwines the Frobenius endomorphisms of $X$ and $Y$. Applying this to $M=Rf_*A$ allows us to identify (\ref{dalg: geometric frobenius eq1}) with (\ref{dalg: geometric frobenius eq3}), as desired.
\end{proof}

We can use Lemma \ref{dalg: pushforward of rings} to construct examples of derived commutative algebras:

\begin{definition}\label{dalg: free divided power algebra def}
For a finite locally free sheaf $M$ (concentrated in degree $0$) on a scheme $X$ we define the {\it free divided power algebra} on $M[-1]$, denoted by $\Gamma^{\bullet}(M[-1])\in \DAlg(X)$, as $R\pi_*^{\alg}\cO_{B_XM^{\vee \sharp}}$ where $\pi: B_XM^{\vee \sharp}\to X$ is the relative classifying stack of the divided power group scheme $M^{\vee\sharp}$ on $X$ associated to $M^{\vee}$.
\end{definition}

This terminology is justified by the fact that the underlying $E_{\infty}$-algebra of $\Gamma^{\bullet}(M[-1])$ is identified with $\bigoplus\limits_{i\geq 0}\Lambda^i M[-i]$, e.g. by \cite[Lemma 7.8]{bhatt-lurie-prismatization}. If $X$ is an $\bF_p$-scheme, by Lemma \ref{dalg: geometric frobenius} the Frobenius map $\varphi^*_{\Gamma^{\bullet}(M[-1])}:F_{X}^*\Gamma^{\bullet}(M[-1])\to \Gamma^{\bullet}(M[-1])$ factors as $F_X^*\Gamma^{\bullet}(M[-1])\to \cO_X\to \Gamma^{\bullet}(M[-1])$ because the Frobenius endomorphism of $M^{\vee \sharp}$ factors through the identity section. 

\begin{rem} See also \cite[3.2.1]{magidson} for another definition of the free divided power derived commutative algebra defined on all complexes. Presumably it agrees with Definition \ref{dalg: free divided power algebra def}, but we do not check this here.
\end{rem}

Cohomology of a sheaf of rings on a site can be equipped with a derived commutative algebra structure:
\begin{lm}\label{dalg: site cohomology}
Let $\cC$ be a site and $\cF$ be a sheaf of (ordinary) commutative algebras over a ring $R$ on $\cC$. Then for any object $X\in \cC$ the complex $\RGamma(X,\cF)\in D(R)$ is naturally endowed with a structure of a derived commutative $R$-algebra.
\end{lm}
\begin{proof}
By \cite[01GZ]{stacks-project} we can compute $\RGamma(X,\cF)$ as a filtered colimit over all hypercovers $U_{\bullet}\to X$ of \v{C}ech cohomology with respect to $U_{\bullet}$:
\begin{equation}\label{dalg: hypercover formula}
\RGamma(X,\cF)\simeq \colim\limits_{U_{\bullet}\to X}\lim\limits_n \cF(U_n)
\end{equation}
Each $\cF(U_n)$ is a commutative $R$-algebra, which we view as an object of $\DAlg(R)$. Since the forgetful functor $\DAlg(R)\to D(R)$ commutes with limits and filtered colimits, this endows $\RGamma(X,\cF)$ with the structure of a derived commutative $R$-algebra.
\end{proof}

Applying this construction to \'etale cohomology with coefficients in $\bF_p$ produces derived commutative algebras with the special property that the Frobenius endomorphism is homotopic to identity:
\begin{lm}\label{dalg: frobenius on etale}
If $X$ is a scheme then the Frobenius endomorphism $\varphi_{\RGamma_{\et}(X,\bF_p)}$ of the derived commutative $\bF_p$-algebra $\RGamma_{\et}(X,\bF_p)$ is naturally homotopic to the identity morphism.
\end{lm}
\begin{proof}
In the formula (\ref{dalg: hypercover formula}) each $\bF_p(U_n)=\Func(\pi_0(U_n),\bF_p)$ is the algebra of $\bF_p$-valued functions on a set, and its Frobenius endomorphism is the identity, so the lemma follows.
\end{proof}

When studying the Sen operator, we will use that it is compatible with the derived commutative algebra structure on the diffracted Hodge cohomology (to be defined in Section \ref{applications: section}). Specifically, we need the following notion:

\begin{definition}\label{dalg: derivation}
For a prestack $X$ and a derived commutative algebra $A\in \DAlg(X)$, a {\it derivation} $f:A\to A$ is a map in $D(X)$ such that the map $\Id_A+\varepsilon\cdot f:A\otimes \bZ[\varepsilon]/\varepsilon^2\to A\otimes \bZ[\varepsilon]/\varepsilon^2$ in $D(X\times \bZ[\varepsilon]/\varepsilon^2)$ is equipped with the structure of a map in $\DAlg(X\times \bZ[\varepsilon]/\varepsilon^2)$.
\end{definition}

\subsection{Cosimplicial commutative algebras.}\label{dalg: cosimp subsection} For the duration of this subsection assume that $X$ is a scheme. In all of our main applications we will in fact be presented with a cosimplicial commutative algebra in the ordinary category $\QCoh(X)$. We denote by $\CAlg^{\Delta}_X$ the ordinary category of cosimplicial commutative algebras in the abelian symmetric monoidal category $\QCoh(X)$. In this subsection we make some remarks on the relation between $\CAlg^{\Delta}_X$ and $\DAlg(X)$. These facts are not used in any of our main results, but the reader who feels more comfortable with $\CAlg^{\Delta}_X$ than with $\DAlg(X)$ is encouraged to specialize the results in Section \ref{cosimp: section} to a situation where  the algebra $A$ is a cosimplicial commutative algebra, using this subsection as a dictionary. \com{A forthcoming work of Mathew and Mondal establishes an equivalence between the full subcategory of $\DAlg(X)$ ...}

In general, a cosimplicial commutative algebra gives rise to a derived commutative algebra:

\begin{lm}
Let $X$ be a scheme. There is a natural functor $\cR_X:\CAlg^{\Delta}(X)\to \DAlg(X)$ from the ordinary category of cosimplicial commutative algebras in quasi-coherent sheaves on $X$ to the category of derived commutative algebras on $X$. This functor is compatible with the cosimplicial totalization functor on the level of underlying complexes. 
\end{lm}

\begin{proof}
First of all, there is a functor from ordinary commutative algebras in $\QCoh(X)$ to $\DAlg(X)$. Indeed, for an ordinary algebra $A$ there is a natural map $S^nA\to (A^{\otimes_{\cO_X} n})_{S_n}$ where the tensor product and coinvariants are taken in the non-derived sense. Hence the commutative multiplication on $A\in \QCoh(X)$ endows it with a structure of a module over $S^{\bullet}$ in $D(X)$.

Now, if $A=\left[{\begin{tikzcd}A^0 \arrow[r, shift left=0.65ex] \arrow[r, shift right=0.65ex] & A^1 \ldots \end{tikzcd}}\right]$ is a cosimplicial commutative algebra in $\QCoh(X)$, we define $\cR_X(A)$ as $\lim\limits_{[n]\in\Delta}\cR_X(A^n)$, where each $\cR_X(A^n)$ was defined in the previous paragraph. This indeed induces the totalization on underlying sheaves because the forgetful functor $\DAlg(X)\to D(X)$ commutes with limits.
\end{proof}

\begin{example}If $f:X\to \Spec R$ is a morphism from a separated scheme to the spectrum of a ring $R$, the object $\RGamma(X,\cO_X)=f_*(\cO_X)\in D(R)$ is endowed with a structure of a cosimplicial commutative algebra using the \v{C}ech construction associated to a cover $X=\bigcup\limits_{i} U_i$:

\begin{equation}
\RGamma(X,\cO_X)\simeq \left[{\begin{tikzcd}\prod\limits_{i} \cO(U_i) \arrow[r, shift left=0.65ex] \arrow[r, shift right=0.65ex] & \prod\limits_{i<j} \cO(U_i\cap U_j) \arrow[r, shift left=1.3ex] \arrow[r, shift right=1.3ex] \arrow[r] &\ldots \end{tikzcd}}\right].
\end{equation}
The result of applying $\cR_R$ to this cosimplicial algebra is naturally equivalent to $f_*^{\alg}\cO_X$ from Lemma \ref{dalg: pushforward of rings}.\end{example} 
If $X$ is a scheme over $\bF_p$, for every cosimplicial commutative algebra $A\in \CAlg^{\Delta}_X$ there is a natural map $\varphi^{\cosimp}_A:F_X^*A\to A$ of algebras induced by the term-wise Frobenius endomorphism. It follows from the construction of the functor $\cR_X$ that $\varphi_{A}^{\cosimp}$ coincides with the Frobenius map arising from the derived commutative algebra structure:

\begin{lm}
For an $\bF_p$-scheme $X$, and a term-wise flat cosimplicial commutative algebra $A\in \CAlg^{\Delta}_X$ there is a natural homotopy between $\cR_X(\varphi_A^{\cosimp})$ and $\varphi_A$ in $D(X)$. 
\end{lm}

\begin{proof}
If $A$ is an ordinary commutative algebra, this was established along with defining $\varphi_A$ in Lemma \ref{dalg: frobenius}. The case of an arbitrary $A$ is obtained by passing to the limit over the cosimplicial category $\Delta$.
\end{proof}

For a finite locally free sheaf $M$ on a scheme $X$ we can represent the free divided power algebra $\Gamma^{\bullet}(M[-1])\in \DAlg(X)$ of Definition \ref{dalg: free divided power algebra def} by a cosimplicial commutative algebra. Denote by $\DK(M[-1])$ the cosimplicial object in the category of locally free sheaves on $X$, obtained by applying the Dold-Kan correspondence to the complex $M[-1]$.

\begin{lm}\label{dalg: free divided power as cosimp}
The derived commutative algebra $\Gamma^{\bullet}(M[-1])$ is equivalent to $\cR_X(\Gamma^{\bullet}_{\naive}(\DK(M[-1])))$ where $\Gamma^{\bullet}_{\naive}(\DK(M[-1]))$ is the result of applying term-wise the free divided power algebra functor to the cosimplicial sheaf $\DK(M[-1])$.
\end{lm}

\begin{proof}
The derived pushforward of the structure sheaf along the map $B_XM^{\vee \sharp}\to X$ is equivalent, as a derived commutative algebra, to the totalization of the cosimplicial diagram
\begin{equation}\label{dalg: free divided power as cosimp diagram}
\begin{tikzcd}\cO_X \arrow[r, shift left=0.65ex] \arrow[r, shift right=0.65ex] & \pi_*\cO_{M^{\vee\sharp}} \arrow[r, shift left=1.3ex] \arrow[r, shift right=1.3ex] \arrow[r] &\pi_*\cO_{M^{\vee\sharp}\times_X M^{\vee\sharp}}\ldots \end{tikzcd}
\end{equation}
obtained by applying the functor $\pi_*\cO_{(-)}$ to the bar-resolution associated to the group scheme $\pi:M^{\vee\sharp}\to X$ over $X$. The $n$th term of the diagram (\ref{dalg: free divided power as cosimp diagram}) is the commutative algebra $\pi_*\cO_{(M^{\vee\sharp})^{\times_X n}}\simeq\Gamma^{\bullet}_X(M^{\oplus n})$ concentrated in degree $0$, so the cosimplicial commutative algebra defined by (\ref{dalg: free divided power as cosimp diagram}) is indeed equivalent to the result of applying $\Gamma^{\bullet}_{\naive}$ to $\DK(M[-1])$.
\end{proof}

\subsection{Bockstein morphisms.} For a stable $\bZ$-linear $\infty$-category $\cC$, any object $M\in \cC$ gives rise to a natural fiber sequence
\begin{equation}
M\xrightarrow{p}M\to M/p.
\end{equation}
We will denote the corresponding connecting map by $\Bock_M:M/p\to M[1]$ and refer to it as the {\it Bockstein morphism} corresponding to $M$. Note that $\Bock_{M[1]}$ is naturally homotopic to $(-\Bock_{M}[1])$.

Similarly, for a $\bZ/p^n$-linear stable $\infty$-category $\cC$ for any object $M\in \cC$ we have a fiber sequence
\begin{equation}
M\otimes_{\bZ/p^n}\bZ/p^{n-1}\to M\to M\otimes_{\bZ/p^n}\bZ/p
\end{equation}
inducing the connecting map $\Bock_M:M\otimes\bZ/p\to M\otimes\bZ/p^{n-1}[1]$. These constructions over $\bZ$ and $\bZ/p^n$ are compatible in the sense that $\Bock_{M/p^n}$ is the composition $M/p\xrightarrow{\Bock_M}M[1]\to M/p^{n-1}[1]$.

\section{Symmetric power \texorpdfstring{$S^p$}{}, class \texorpdfstring{$\alpha$}{}, and Steenrod operations}
\label{free cosimplicial: section}

In this section we study in detail the derived functor $S^p$ of $p$th symmetric power by comparing it with the divided power functor $\Gamma^p$. For complexes of vector spaces over a field $k$ the values of $S^n$ can be described non-canonically using the computations of $S^n(k[-i])$ done by Priddy \cite{priddy}, but we crucially need to understand $S^pM$ as a complex of sheaves, rather that its separate cohomology sheaves.

\subsection{Symmetric powers vs. divided powers.}\label{free cosimplicial: symmetric vs divided subsection} Let $X$ be an arbitrary prestack. Recall that in the previous section for an object $M\in D(X)$ we defined natural morphisms 

\begin{equation}
N_n:S^nM\to \Gamma^nM\quad r_n:\Gamma^nM\to S^nM.  
\end{equation} 

Denote by $T_n(M)$ the cofiber $\cofib(N_n:S^nM\to \Gamma^nM)$ of the norm map. The functor $T_n$, especially for $n=p$, has been extensively studied in the literature. References close to our point of view are works of Friedlander-Suslin \cite[Section 4]{friedlander-suslin} and Kaledin \cite[6.3]{kaledin-coperiodic}. We have the following classical results about $T_n$:

\comment{\begin{rem}
In (2), when $2$ is invertible on $X$, $\Lambda^n M$ is naturally isomorphic to the coinvariants $(M^{\otimes n})_{S_n,\sgn}$ with respect to the alternating action of $S_n$. Also, when $2$ is invertible on $X$, the norm map $N:(M^{\otimes_n})_{S_n,\sgn}\to (M^{\otimes n})^{S_n,\sgn}$ is an isomorphism, contrary to the case of the norm map $N:S^n(M)\to \Gamma^n(M)$.
\end{rem}
}

\begin{lm}[{\hspace{1sp}\cite[Lemma 4.12]{friedlander-suslin},\cite[Lemma 6.9]{kaledin-coperiodic}}]\label{free cosimplicial: tate p}
\begin{enumerate}
\item If $X$ is a prestack over $\bZ[\frac{1}{n!}]$ then $T_n(M)\simeq 0$ for every $M\in D(X)$.

\item If $R$ is a flat $\bZ_p$-algebra and $M$ is a flat $R$-module, then $N_p:S^pM\to \Gamma^pM$ is an injection and its cokernel $T_p(M)$ is naturally isomorphic to $F^*_{R/p}(M/p)$ as an $R$-module.

\item If $X$ is a flat algebraic stack over $\bZ_p$ then there is a natural equivalence $T_p(M)\simeq i_*F_{X_0}^*i^*M$ for all $M\in D(X)$, where $i:X_0=X\times_{\bZ_p}\bF_p\to X$ is the closed immersion of the special fiber, and $F_{X_0}:X_0\to X_0$ is the absolute Frobenius morphism.



\item If $X_0$ is a prestack over $\bF_p$ then for any $M_0\in D(X_0)$ the object $T_p(M_0)$ fits into a natural fiber sequence
\begin{equation}\label{free cosimplicial: tate mod p extension}
F_{X_0}^*M_0[1]\to T_p(M_0)\to F_{X_0}^*M_0.
\end{equation}
\end{enumerate}
\end{lm}

\begin{rem}\label{free cosimp: remark Tp}
\begin{enumerate}
\item If $X$ is a flat algebraic stack over $\bZ_p$ with special fiber $X_0=X\times_{\bZ_p}\bF_p$ and $M_0\simeq i^*M\in D(X_0)$ is the reduction of an object $M\in D(X)$ then $T_p(M_0)$ can be naturally (in $M$) described as 
\begin{equation}
T_p(M_0)\simeq i^*T_p(M)\simeq i^*i_*F_{X_0}^*M_0.  
\end{equation}
Under this identification, the fiber sequence in (\ref{free cosimplicial: tate mod p extension}) is identified with the sequence induced by the fiber sequence of functors $\Id[1]\to i^*i_*\to \Id$ from $D(X_0)$ to $D(X_0)$.

\item This will not be used in any of the proofs but let us remark that one can describe the extension (\ref{free cosimplicial: tate mod p extension}) completely, even in the absence of a lift of $X_0$ to $\bZ_p$ together with the object $M$. It is proven in \cite[Theorem 6]{petrov-vaintrob-vologodsky} that in the setting of (4), at least if $X_0$ is a scheme, $T_p(M_0)$ can be upgraded to an object of the derived category of crystals of quasi-coherent crystals of $\cO$-modules $D(\Cris(X_0/\bF_p))$ on the scheme $X_0$ such that (\ref{free cosimplicial: tate mod p extension}) is an extension of crystals where $F^*_{X_0}M$ is endowed with a crystal structure using the canonical connection. Moreover, homotopy classes of splittings of (\ref{free cosimplicial: tate mod p extension}) in $D(\Cris(X_0/\bF_p))$ are in bijection with lifts of $F_{X_0}^*M_0$ to an object in $D(\Cris(X_0/(\bZ/p^2)))$. If $X_0$ is equipped with a lift $X_1$ over $\bZ/p^2$ then splittings of (\ref{free cosimplicial: tate mod p extension}) in $D(X_0)$ are in bijection with lifts of $F_{X_0}^*M_0$ to an object of $D(X_1)$.
\end{enumerate}
\end{rem}

\begin{proof}
1) We have $r_n\circ N_n=n!\cdot\ \Id_{S^nM}$, $N_n\circ r_n=n!\cdot \Id_{\Gamma^nM}$. Therefore $N_n$ is an isomorphism if $n!$ is invertible, so $T^n$ vanishes on prestacks over $\bZ[\frac{1}{n!}]$.

2) By the previous part, the map $N_p:S^pM\to \Gamma^p M$ becomes an isomorphism after inverting $p$. Since the module $S^pM$ is $p$-torsion free, the map $N_p$ is injective.

We will now construct a natural isomorphism $\alpha$ between the module $F^*_{R/p}(M/p)=M/p\otimes_{R/p,F_{R/p}}R/p$ and the cokernel of $N_p$. To an element $m\otimes r\in M/p\otimes_{R/p,F_{R/p}}R/p$ assign the element $\alpha(m\otimes r):=\widetilde{r}\cdot \tm^{\otimes p}\in \coker N_p$ where $\widetilde{r}\in R$ and $\tm\in M$ are arbitrary lifts of $r$ and $m$, respectively. 

To see that $\alpha$ gives a well-defined map $F_{R/p}^*(M/p)\to \coker N_p$ we need to check that $\widetilde{r}\cdot\tm^{\otimes p}\in \coker N_p$ does not depend on the choices of the lifts $\widetilde{r},\tm$ and that $(\tm_1+\tm_2)^{\otimes p}=\tm_1^{\otimes p}+\tm_2^{\otimes p}\in \coker N_p$. The first claim follows from the fact that $p\cdot\Gamma_R^p(M)\subset \im N_p$, and the additivity is demonstrated by the formula
\begin{equation}
(\tm_1+\tm_2)^{\otimes p}-\tm_1^{\otimes p}-\tm_2^{\otimes p}=N_p\left(\sum\limits_{i=1}^{p-1} \frac{1}{i!(p-i)!}\tm_1^{i}\tm_2^{p-i}\right).
\end{equation}

Finally, to show that $\alpha:F_{R/p}^*(M/p)\to \coker(N_p)$ is an isomorphism, we may assume that $M$ is a free $R$-module, because both functors $M\mapsto \coker (N_p:S^pM\to\Gamma^p M)$ and $M\mapsto F_{R/p}^*(M/p)$ commute with filtered colimits and every flat module can be represented as a filtered colimit of free modules. Let $\{e_i\}_{i\in I}$ be a basis of $M$ over $R$. Then $\{N_p(e_{i_1}\ldots e_{i_p})\}_{(i_1,\ldots,i_p)\in I^p\setminus I}\cup\{e_i^{\otimes p}\}_{i\in I}$ is an $R$-basis for $\Gamma_R^p(M)$ so $\{e_i^{\otimes p}\}_{i\in I}$ is an $R/p$-basis for $\coker N_p$, as desired.

3) Part (2) produced a natural short exact sequence $S^pM\to\Gamma^p M\to F^*_{R/p}(M/p)$ of $R$-modules for every projective module $M$ over a flat $\bZ_p$-algebra $R$. The functor $M\mapsto F^*_{R/p}(M/p)$ from $\Proj^{\mathrm{f.g.}}_R$ to $\Mod_R$ is polynomial (in fact, linear), and the formation of this short exact sequence is compatible with base change along arbitrary maps $R\to R'$ of flat $\bZ_p$-algebras, so Lemma \ref{dalg: derived functors for stacks} produces a fiber sequence $S^p M\to \Gamma^p M\to i_*F^*_{X_0}i^*M$ for every $M\in D(X)$ which gives the desired identification.

4) We will establish such a fiber sequence when $M_0$ is a flat sheaf on an affine scheme $X_0=\Spec R$ and the general case will follow formally as in (3). If $M_0$ is a flat $R$-module where $R$ is an $\bF_p$-algebra, then $T_p(M_0)$ is represented by the complex $S^pM_0\xrightarrow{N_p}\Gamma^pM_0$ concentrated in degrees $[-1,0]$. In Lemma \ref{dalg: polynomial frobenius lemma} we defined maps $\psi_{M_0}:\Gamma^p M_0\to F_{R_0}^*M_0$ and $\Delta_{M_0}:F_{R_0}^*M_0\to S^pM_0$ that give rise to a sequence
\begin{equation}
0\to F^*_{R_0}M_0\xrightarrow{\Delta_{M_0}} S^p_{R_0}M_0\xrightarrow{N_p}\Gamma^pM_0\xrightarrow{\psi_{M_0}} F^*_{R_0}M_0\to 0
\end{equation}
that one checks to be exact by a direct calculation with a basis of $M_0$ in case it is free, as in part (2). This exact sequence gives rise to the desired fiber sequence (\ref{free cosimplicial: tate mod p extension}) by Lemma \ref{dalg: derived functors for stacks}.
\end{proof}

Non-decomposability of the de Rham complex will arise from the fact that the map $N_p:S^p M\to \Gamma^p M$ does not have a section in general. By definition, for an object $M\in D(X)$ on an arbitrary prestack $X$ there is a natural fiber sequence \begin{equation}\label{free cosimplicial: general fiber sequence}
T_p(M)[-1]\xrightarrow{\gamma_M} S^p M\xrightarrow{N_p}\Gamma^pM.
\end{equation}

It will be slightly more natural for us to start with an arbitrary object $E\in D(X)$, and apply (\ref{free cosimplicial: general fiber sequence}) to $M=E[-1]$ to get a fiber sequence
\begin{equation}\label{free cosimplicial: alpha' extension}
T_p(E[-1])[-1]\to S^p(E[-1])\to (\Lambda^p E)[-p]
\end{equation}
where we used the d\'ecalage identification $\Gamma^p(E[-1])\simeq (\Lambda^pE)[-p]$ from Lemma \ref{dalg: decalage} to rewrite the third term. If $X$ is an algebraic stack flat over $\bZ_p$ then, by Lemma \ref{free cosimplicial: tate p}(3) this fiber sequence takes the form
\begin{equation}\label{free cosimplicial: alpha flat over zp}F_{X_0}^*(E/p)[-2]\to S^p(E[-1])\to (\Lambda^p E)[-p].\end{equation}
Here $F_{X_0}^*(E/p)$ is an abbreviation for $i_*F_{X_0}^*i^*E$, we will prefer using this notation in what follows. Let us record the description of the cohomology sheaves of $S^p(E[-1])$ in the case $E$ is a locally free sheaf, for $p>2$, the analogous result for $p=2$ will be established in Corollary \ref{free cosimplicial: symmetric power cohomology sheaves p=2}:

\begin{lm}\label{free cosimplicial: symmetric power cohomology sheaves}Suppose that $p>2$.
\begin{enumerate}
    \item If $E$ is a locally free sheaf on a scheme $X$ flat over $\bZ_p$, then \begin{equation}H^2(S^p(E[-1]))\simeq F_{X_0}^*(E/p),\quad H^p(S^p(E[-1]))\simeq\Lambda^p E,\end{equation} and all other cohomology sheaves of $S^p(E[-1])$ are zero.
    \item If $E$ is a locally free sheaf on a scheme $X_0$ over $\bF_p$, then \begin{equation}H^1(S^p(E[-1]))\simeq H^2(S^p(E[-1]))\simeq F_{X_0}^*E,\quad H^p(S^p(E[-1]))\simeq\Lambda^p E\end{equation} and all other cohomology sheaves of $S^p(E[-1])$ are zero.
\end{enumerate}
\end{lm}

\begin{proof}
This is immediate from (\ref{free cosimplicial: alpha' extension}) and Lemma \ref{free cosimplicial: tate p}(3),(4).
\end{proof}

Let now $X_0$ be an arbitrary prestack over $\bF_p$. For an object $E\in D(X_0)$ the pushout of the fiber sequence (\ref{free cosimplicial: alpha' extension}) along the map $T_p(E[-1])[-1]\to F_{X_0}^*E[-2]$ from (\ref{free cosimplicial: tate mod p extension}) defines a fiber sequence
\begin{equation}\label{free cosimplicial: alpha extension}
F_{X_0}^*E[-2]\to S^p(E[-1])\bigsqcup\limits_{T_p(E)[-2]}F_{X_0}^*E[-2]\to (\Lambda^p E)[-p].
\end{equation}

In case where $E$ is a locally free sheaf on an $\bF_p$-scheme scheme $X_0$, the sequence (\ref{free cosimplicial: alpha extension}) is simply the truncation of (\ref{free cosimplicial: alpha' extension}) in degrees $\geq 2$.

\begin{definition}\label{free cosimplicial: alpha definition}
We will denote by $\alpha(E):\Lambda^p E\to F_{X_0}^*E[p-1]$ the shift by $[p]$ of the connecting morphism corresponding to the fiber sequence (\ref{free cosimplicial: alpha extension}).
\end{definition}

Equivalently, the map $\alpha(E)$ can be defined as the composition $\Lambda^p E\simeq \Gamma^p(E[-1])[p]\xrightarrow{\psi_{E[-1]}[p]}F_{X_0}^*E[p-1]$ where the first equivalence is the d\'ecalage isomorphism, and $\psi_{E[-1]}$ is the map constructed in (\ref{dalg: polynomial frobenius}).

\begin{rem}
A posteriori, we will see that the cohomology class $\alpha(-)$ coincides with ones previously considered in the literature. The assignment $E\mapsto\alpha(E)$ can be upgraded to an extension of degree $p-1$ between strict polynomial functors (in the sense of \cite[Definition 2.1]{friedlander-suslin}) $\Lambda^{p}(-)$ and $(-)^{(1)}$ on $\bF_p$-vector spaces. Friedlander and Suslin proved \cite[(4.5.1)]{friedlander-suslin} that the space of such extensions is $1$-dimensional. Hence, after it is established that $\alpha(-)$ is not always zero (Proposition \ref{rational group cohomology: main non-vanishing}), it necessarily spans this space. In Remark \ref{rational: carter-lusztig remark}(2) we will also relate class $\alpha$ to an extension considered by Carter and Lusztig.
\end{rem}

Suppose now that $X$ is a flat algebraic stack over $\bZ_p$. Denote by $i:X_0\simeq X\times_{\bZ_p}\bF_p\hookrightarrow X$ the closed embedding of the special fiber. Let us remark that the information contained in the extension (\ref{free cosimplicial: alpha' extension}) for an object $E\in D(X)$ is completely captured by $\alpha(i^*E)$: 

\begin{lm}\label{free cosimplicial: alpha adjunction}
For an object $E\in D(X)$ denote $i^*E$ by $E_0$. The image of the map $\alpha(E_0)$ under the adjunction identification $\RHom_{X_0}(\Lambda^p E_0,F_{X_0}^*E_0[p-1])=\RHom_X(\Lambda^p E,i_*F_{X_0}^*i^*E[p-1])$ is the connecting morphism corresponding to the fiber sequence \begin{equation}\label{free cosimplicial: alpha adjunction formula}T_p(E[-1])[-1]\to S^p(E[-1])\to (\Lambda^p E)[-p]\end{equation} where we identify $T_p(E[-1])[-1]$ with $i_*F_{X_0}^*i^*E[-2]$ via Lemma \ref{free cosimplicial: tate p}(3). 
\end{lm}

\begin{proof}
In general, for two objects $M\in D(X), N\in D(X_0)$ the adjunction identification $\RHom_{D(X)}(M,i_*N)\simeq \RHom_{D(X_0)}(i^*M,N)$ can be described as sending a map $f:M\to i_*N$ to the composition $i^*M\xrightarrow{i^*f}i^*i_*N\to N$ where the second map is the counit of the adjunction. The fiber sequence $T_p(E_0[-1])[-1]\to S^p(E_0[-1])\to (\Lambda^p E_0)[-p]$ is the result of applying $i^*$ to the sequence (\ref{free cosimplicial: alpha adjunction formula}), and the map $T_p(E_0[-1])[-1]\to F_{X_0}^*E_0[-2]$ used to form the pushout sequence (\ref{free cosimplicial: alpha extension}) is precisely the counit map $i^*i_*\to\Id$ evaluated on $F^*_{X_0}E_0[-2]$ by Remark \ref{free cosimp: remark Tp}, so the lemma follows.
\end{proof}

The extension defining $\alpha(E)$ can be described in more classical terms for $p=2$. For a projective module $M$ over an $\bF_2$-algebra $R_0$ we have a natural short exact sequence
\begin{equation}\label{free cosimplicial: alpha for p=2 discrete}
0\to F_{R_0}^*M\xrightarrow{\Delta_M} S^2M\xrightarrow{j_M} \Lambda^2 M\to 0.
\end{equation}

Here $\Delta_M$ is the map defined in (\ref{dalg: polynomial frobenius}), it sends an element $m\otimes 1\in M\otimes_{R_0,F_{R_0}}R_0$ to $m\cdot m\in S^2M$, and $j_M$ sends $m_1\cdot m_2\in S^2M$ to $m_1\wedge m_2\in \Lambda^2 M$. By Lemma \ref{dalg: derived functors for stacks} this gives rise to a fiber sequence \begin{equation}\label{free cosimplicial: alpha for p=2 derived}F^*_{X_0}E\xrightarrow{\Delta_E}S^2E\xrightarrow{j_E} \Lambda^2 E\end{equation}
for every object $E\in D(X_0)$ on any $\bF_2$-prestack $X_0$.

\begin{lm}\label{free cosimplicial: alpha for p=2}
When $p=2$, and $E$ is an object of $D(X_0)$, the fiber sequence (\ref{free cosimplicial: alpha extension}) is naturally equivalent to the shift by $[-2]$ of the sequence (\ref{free cosimplicial: alpha for p=2 derived}).
\end{lm}

\begin{proof}
Consider the shift of the fiber sequence (\ref{free cosimplicial: alpha extension}) by $[2]$:

\begin{equation}\label{free cosimplicial: alpha extension p=2 shifted}
F_{X_0}^*E\to S^2(E[-1])[2]\bigsqcup\limits_{T_2(E)}F_{X_0}^*E\to\Lambda^2 E.
\end{equation}

When $X_0=\Spec R_0$ is an affine scheme and $E$ corresponds to a projective $R_0$-module, first and third terms of (\ref{free cosimplicial: alpha extension p=2 shifted}) are concentrated in degree $0$, so this fiber sequence is an exact sequence of $R_0$-modules.  We will functorially identify it with the sequence (\ref{free cosimplicial: alpha for p=2 discrete}), which will prove the Lemma in general thanks to Lemma \ref{dalg: derived functors for stacks}.

Consider the fiber sequence (\ref{free cosimplicial: alpha for p=2 derived}) for the object $E[-1]$:
\begin{equation}
F_{X_0}^*E[-1]\xrightarrow{\Delta_{E[-1]}}S^2(E[-1])\xrightarrow{j_{E[-1]}}\Lambda^2(E[-1]).
\end{equation}
By the d\'ecalage isomorphism we have $\Lambda^2 (E[-1])\simeq (S^2 E)[-2]$, so this fiber sequence provides isomorphisms $H^1(S^2(E[-1]))\simeq F^*_{X_0}E$ and $H^2(S^2(E[-1]))\simeq S^2E$. The map $S^2(E[-1])[2]\xrightarrow{j_{E[-1]}[2]}\Lambda^2(E[-1])[2]\simeq S^2E$ then naturally factors through $ S^2(E[-1])[2]\bigsqcup\limits_{T_2(E)}F_{X_0}^*E$, necessarily identifying the latter with $S^2E$. Therefore in the exact sequence of projective $R_0$-modules (\ref{free cosimplicial: alpha extension p=2 shifted}) all the terms are isomorphic to those of (\ref{free cosimplicial: alpha for p=2 derived}) so it remains to check that under these isomorphism the maps in the sequences are intertwined.

In other words, for every $\bF_2$-algebra $R_0$ we identified the fiber sequence (\ref{free cosimplicial: alpha extension p=2 shifted}) with a short exact sequence of the form
\begin{equation}\label{free cosimplicial: alpha p=2 fake sequence}
0\to F_{X_0}^*E\xrightarrow{a_E} S^2 E\xrightarrow{b_E} \Lambda^2 E\to 0
\end{equation}
for some maps $a_E, b_E$ natural in the projective module $E$ and whose formation is compatible with base change in $R_0$. We will now check that $a_E$ and $b_E$ must be equal to $\Delta_E$ and $j_E$, respectively. We may and do assume that $R_0$ is a perfect field for this.

First, the composition $j_E\circ a_E:F_{R_0}^*E\to \Lambda^2 E$ is necessarily zero. Indeed, since $F_{R_0}^*$ is an additive functor, for a vector space $E=R_0^{\oplus n}$ with a chosen basis this composition, being natural, would have to factor through $\Lambda^2(R_0)^{\oplus n}=0$.

This means that the sequence (\ref{free cosimplicial: alpha p=2 fake sequence}) fits into a commutative diagram of the form
\begin{equation}\label{free cosimplicial: alpha p=2 comparison diagram}
\begin{tikzcd}
F_{R_0}^*E\arrow[d,"f_E"]\arrow[r, "{a_{E}}"] & S^2E\arrow[d,equal]\arrow[r, "{b_E}"] & \Lambda^2 E\arrow[d,"g_E"] \\
F_{R_0}^*E\arrow[r, "{\Delta_{E}}"] & S^2E\arrow[r, "{j_E}"] &\Lambda^2 E
\end{tikzcd}
\end{equation}
for some natural maps $f_E$ and $g_E$.

The map $f_E$ is necessarily injective and $g_E$ is surjective, so by dimension reasons both are isomorphisms. But all automorphisms of either of the functors $F^*$ and $\Lambda^2$ on categories of finite-dimensional $R_0$-vector spaces for an infinite perfect field $R_0$ are given by multiplication by a scalar from $R_0^{\times}$. Since $f_E$ and $g_E$ are defined, in particular, over $R_0=\bF_2$, these scalars have to belong to $\bF_2$, that is both be equal to $1$. Hence diagram (\ref{free cosimplicial: alpha p=2 comparison diagram}) provides the desired identification of fiber sequences.
\end{proof}
We can deduce the computation of cohomology sheaves of $S^p(E[-1])$ when $p=2$, complementing Lemma \ref{free cosimplicial: symmetric power cohomology sheaves}(2):
\begin{cor}\label{free cosimplicial: symmetric power cohomology sheaves p=2}
For a locally free sheaf $E$ on an $\bF_2$-scheme $X_0$ we have
\begin{equation}
H^1(S^2(E[-1]))\simeq F^*_{X_0}E\quad H^2(S^2(E[-1]))\simeq S^2E
\end{equation} and all other cohomology sheaves are zero.
\end{cor}

For an arbitrary $p$, we can relate $S^p(E[-1])$ to the homotopy coinvariants $E[-1]^{\otimes p}_{hS_p}$ of the symmetric group $S_p$ acting on $E[-1]^{\otimes p}$ by permutations. 

\begin{lm}\label{free cosimplicial: cosimp vs einf}
For a locally free sheaf $E$ on an $\bF_p$-scheme $X_0$ there is a natural equivalence $S^p(E[-1])\simeq\tau^{\geq 1}(E[-1]^{\otimes p}_{hS_p})$.
\end{lm}

\begin{proof}
First, let us recall the values of the cohomology sheaves of $E[-1]^{\otimes p}_{hS_p}$:

\begin{lm}\label{free cosimplicial: symmetric group cohomology} If $p>2$ then
\begin{equation}\label{free cosimplicial: symmetric group cohomology formula}
H^i(E[-1]^{\otimes p}_{hS_p})\simeq\begin{cases}\Lambda^p E, i=p \\ F_{X_0}^*E, i\equiv 1\text{ or }2 \bmod 2(p-1) \text{, and } i\leq 2\\ 0\text{ otherwise}\end{cases}
\end{equation}
If $p=2$ then \begin{equation}\label{free cosimplicial: symmetric group cohomology formula 2}
H^i(E[-1]^{\otimes 2}_{hS_2})\simeq\begin{cases}S^2 E, i=2 \\ F_{X_0}^*E, i<2\\ 0, i>2\end{cases}
\end{equation}
\end{lm}

\begin{proof}
We will give the argument in the case $p>2$ and the case $p=2$ is proven analogously. Note that the $S_p$-equivariant object $E[-1]^{\otimes p}$ is equivalent to $\sgn\otimes E^{\otimes p}[-p]$ where $S_p$ acts on $E^{\otimes p}$ via permutation of factors, and $\sgn$ is the sign character. In particular, $H^p(E[-1]^{\otimes p}_{hS_p})$ is isomorphic to $\Lambda^p E$ by definition of the exterior power.

To identify other cohomology sheaves, we will compare homology of the symmetric group with that of a cyclic group. Denote by $C_p\subset S_p$ a cyclic subgroup of order $p$. There is a natural map $E[-1]^{\otimes p}_{hC_p}\to E[-1]^{\otimes p}_{hS_p}$. The coinvariants $E[-1]^{\otimes p}_{hC_p}$ for the cyclic group can be represented by the following two-periodic complex
\begin{equation}
\ldots\xrightarrow{1-\sigma} E^{\otimes p}\xrightarrow{N_{C_p}}E^{\otimes p}\xrightarrow{1-\sigma}E^{\otimes p}
\end{equation}
where $\sigma$ is the endomorphism of $E^{\otimes p}$ given by cyclic permutation of the factors, and $N_{C_p}=\sum\limits_{i=0}^{p-1}\sigma^i$. For all $i<p$ we get a map $F^*_{X_0}E\to H^i(E[-1]^{\otimes p}_{hC_p})$ given by sending a section $e\otimes 1\in E(U)\otimes_{\cO_{X_0}(U),F_{X_0}}\cO_{X_0}(U)$ on an affine open $U\subset X_0$ to $e^{\otimes p}\in E^{\otimes p}(U)$. One checks on stalks that this map is an isomorphism for all $i<p$. We get

\begin{equation}
H^i(E[-1]^{\otimes p}_{hC_p})\simeq\begin{cases}(E^{\otimes p})_{C_p}, i=p \\ F_{X_0}^*E, i<p\\ 0\text{ otherwise.}\end{cases}
\end{equation}

Since $C_p\subset S_p$ is a $p$-Sylow subgroup, the map $E[-1]^{\otimes p}_{hC_p}\to E[-1]^{\otimes p}_{hS_p}$ establishes $E[-1]^{\otimes p}_{hS_p}$ as a direct summand of $E[-1]^{\otimes p}_{hC_p}$. Hence to prove the lemma it remains to check that the surjective map $H^i(E[-1]^{\otimes p}_{hC_p})\to H^i(E[-1]^{\otimes p}_{hS_p})$ is an isomorphism for $i\equiv 1,2\bmod 2(p-1)$ and is the zero map for all other $i<p$. This can be checked on stalks, and $F_{X_0}^*$ is an additive functor, so we may assume that $X_0=\Spec \bF_p$ and $E$ is a $1$-dimensional $\bF_p$-vector space. The statement is now a consequence of classical computations of homology of the symmetric group with coefficients in the sign character, cf. \cite[Chapter V, Lemma 6.2+Proposition 7.8]{steenrod-epstein}.
\end{proof}
There is a natural map $E[-1]^{\otimes p}\to S^p(E[-1])$ which is $S_p$-equivariant with respect to the permutation action on the source and the trivial action on the target. Hence it induces a natural map $E[-1]^{\otimes p}_{hS_p}\to S^p(E[-1])$. Moreover, this map factors through $E[-1]^{\otimes p}_{hS_p}\to \tau^{\geq 1}(E[-1]^{\otimes p}_{hS_p})$ because $S^p(E[-1])$ is concentrated in degrees $\geq 1$. We thus obtain a map \begin{equation}\label{free cosimplicial: cosimp vs einf map}\xi_E:\tau^{\geq 1}(E[-1]^{\otimes p}_{hS_p})\to S^p(E[-1])\end{equation} between objects of $D(X_0)$ whose cohomology sheaves are isomorphic, by (\ref{free cosimplicial: symmetric group cohomology formula}) and Lemma \ref{free cosimplicial: symmetric power cohomology sheaves}(2). It remains to check that this particular map induces isomorphisms on cohomology sheaves.

For cohomology in degree $p$, this is a consequence of the fact that the composition $E[-1]^{\otimes p}\to S^p(E[-1])$ induces (up to a sign) the surjection $E^{\otimes p}\twoheadrightarrow \Lambda^p E$ on $H^p$ when $p>2$, and the surjection $E^{\otimes 2}\twoheadrightarrow S^2E$ when $p=2$. To show that the map $\xi_E$ induces isomorphisms on $H^1$ and $H^2$ it is enough to consider the case $X_0=\Spec \bF_p$. Since $H^1$ and $H^2$ of both functors $E\mapsto S^p(E[-1])$ and $\tau^{\geq 1}(E[-1]^{\otimes p}_{hS_p})$ are additive functors of $E$, we can moreover assume that $E$ a $1$-dimensional vector space. The fact that $\xi_E$ is an isomorphism on $H^1$ and $H^2$ in this case now follows from the fact that cohomology classes of $\bF_p[-1]^{\otimes p}_{hS_p}$ are responsible for Steenrod operations in the cohomology of $E_{\infty}$-algebras, as proven in Lemma \ref{free cosimplicial: free einf free cosimp deg p} below.
\end{proof}

\subsection{$T_p$ and commutative algebras.} We will now compute how the defect between $S^p$ and $\Gamma^p$, given by $T_p$, interacts with the structure of a derived commutative algebra. This computation plays a central role in our proof of Theorem \ref{cosimp: main theorem}.

Recall that for a complex of quasicoherent sheaves $M\in D(X)$ on a stack or a formal scheme flat over $\bZ_p$ we denote by $\gamma_M: F_{X_0}^*(M/p)[-1]\to S^p M$ the connecting map obtained from the identification $\cofib(S^pM\xrightarrow{N_p}\Gamma^p M)\simeq F_{X_0}^*(M/p)$.
 
\begin{lm}\label{free cosimplicial: algebra Bockstein}
Let $X$ be an algebraic stack or a formal scheme flat over $\bZ_p$ with special fiber $X_0:=X\times_{\bZ_p}\bF_p$. For a derived commutative algebra $A\in\DAlg(X)$ the composition $F_{X_0}^*(A/p)[-1]\xrightarrow{\gamma_A} S^pA\xrightarrow{m_A}A$ is naturally homotopic in $D(X)$ to 
\begin{equation}\label{free cosimplicial: algebra Bockstein equation}
F_{X_0}^*(A/p)[-1]\xrightarrow{\varphi_{A/p}}A/p[-1]\xrightarrow{\Bock_A[-1]}A.
\end{equation}
\end{lm}

\begin{proof}
For a projective module $M$ over a flat $\bZ_p$-algebra $R$ there is a diagram of $R$-modules
\begin{equation}\label{free cosimplicial: algebra Bockstein ordinary diagram}
\begin{tikzcd}
S^p M\arrow[r, "N_p"]\arrow[d,"\id_{S^p M}"] & \Gamma^p M\arrow[r]\arrow[d, "r_p"] & F^*_{R/p}(M/p)\arrow[d,"\Delta_{M/p}"] \\
S^p M\arrow[r, "p!"] & S^p M\arrow[r] & S^p_{R/p}(M/p)
\end{tikzcd}
\end{equation}
where both rows are short exact sequences of $R$-modules. It evidently gives rise to a diagram of polynomial functors that is moreover natural in arbitrary maps $R\to R'$. We can construct from this the following diagram in $D(X)$ where rows are fiber sequences and vertical maps organize into morphisms of fiber sequences:

\begin{equation}\label{free cosimplicial: algebra Bockstein derived diagram}
\begin{tikzcd}
S^pA\arrow[r, "N_p"]\arrow[d, equal] & \Gamma^pA\arrow[d]\arrow[r] & F^*_{X_0}(A/p)\arrow[d, "\Delta_{A/p}"]\\
S^pA\arrow[r, "p!"]\arrow[d, "m_A"] & S^p A\arrow[d, "m"]\arrow[r] & S^p_{X_0}(A/p)\arrow[d, "m_A"]\\
A\arrow[r, "p!"] & A\arrow[r] & A/p .
\end{tikzcd}
\end{equation}

The maps between the first two rows are obtained by applying the functor $\Sigma_X$ of Lemma \ref{dalg: derived functors for stacks} to the diagram (\ref{free cosimplicial: algebra Bockstein ordinary diagram}), and evaluating the resulting maps on the object $A\in D(X)$. The maps between the second and third rows are obtained by applying the functor $\cofib(p!)$ to the map $m_A:S^pA\to A$. 

The connecting morphism $F^*_{X_0}(A/p)\to (S^pA)[1]$ corresponding to the top row of (\ref{free cosimplicial: algebra Bockstein derived diagram}) is equivalent to $-\gamma_A[1]$, by definition of $\gamma_A$ given in (\ref{free cosimplicial: general fiber sequence}). Commutativity of the diagram implies the lemma because the connecting morphism $A/p\to A[1]$ of the bottom row is the negative $(-\Bock_A)$ of the Bockstein morphism by the identity $(p-1)!=-1$ in $\bF_p$, and the composition of the right vertical column is the map $\varphi_{A/p}:F_{X_0}^*(A/p)\to A/p$, by the definition of Frobenius endomorphisms of derived commutative algebras.
\end{proof}

\begin{rem}
If $A$ happens to be represented by a term-wise flat cosimplicial commutative algebra $A^{\bullet}$ in $\QCoh(X)$, for a flat $\bZ_p$-scheme $X$, then the diagram (\ref{free cosimplicial: algebra Bockstein derived diagram}) can be obtained from the strictly commutative diagram in the ordinary category of cosimplicial sheaves on $X$, where the rows are term-wise exact:

\begin{equation}\label{free cosimplicial: cosimplicial Bockstein diagram}
\begin{tikzcd}
S^p_{\naive}A^{\bullet}\arrow[r, "N"]\arrow[d, "m"] & \Gamma_{\naive}^pA^{\bullet}\arrow[d, "m"]\arrow[r] & F^*(A/p)\arrow[d, "\varphi^{\cosimp}_{A/p}"]\\
A\arrow[r, "p!"] & A\arrow[r] & A/p 
\end{tikzcd}
\end{equation}
where $S^p_{\naive}$ and $\Gamma^p_{\naive}$ denote the endofunctors of the ordinary category of cosimplicial objects in $\QCoh(X)$ induced by the (non-derived) functors $S^p$ and $\Gamma^p$.
\end{rem}

\subsection{Steenrod operations on cohomology of cosimplicial algebras.}\label{free cosimplicial: steenrod subsection} This subsection is not used in the rest of the paper and contains classical material, if only presented somewhat differently, but we include it, as Proposition \ref{free cosimplicial: steenrod operations description prop} was the original motivation for our approach and Theorem \ref{cosimp: main theorem} should be viewed as its generalization. Let us also mention the result of Scavia \cite[Theorem 1.1(iv),(v)]{scavia-steenrod} which is related to and implied by Proposition \ref{free cosimplicial: steenrod operations description prop}.

Recall that cohomology of derived symmetric powers is related to natural operations on cohomology of cosimplicial commutative algebras: \com{introduce notation calg}

\begin{lm}\label{free cosimplicial: cohomology operations}
Fix a commutative base ring $R$. The $R$-module of natural transformations between functors $H^i,H^j:\mathrm{CAlg}_R^{\Delta}\to\Mod_R$ can be described as
\begin{equation}\label{free cosimplicial: cohomology operations formula}
\Hom(H^i,H^j)\simeq H^j(S^{\bullet}(R[-i]))
\end{equation} 
where the isomorphism takes a natural transformation $\alpha:H^i\to H^j$ to the image of the class $1\in R=H^i(S^1(R[-i]))\subset H^i(S^{\bullet}(R[-i]))$ under $\alpha$ evaluated on the free algebra $S^{\bullet}(R[-i])$.
\end{lm}

\begin{rem}
The resulting module of natural transformations can be computed completely, cf. \cite{priddy} for the case $R=\bF_p$.
\end{rem}

\begin{proof}
We will describe the inverse map. Given a class $c\in H^j(S^{\bullet}(R[-i]))$, for every cosimplicial commutative algebra $A$ we define a morphism $H^i(A)\to H^j(A)$ as follows. Let $x:R[-i]\to A$ be a map representing a cohomology class $[x]\in H^i(A)$. Applying the functor $S^{\bullet}$ to this map of cosimplicial $R$-modules and composing it with the multiplication on $A$ gives a map $S^{\bullet}(R[-i])\xrightarrow{S^{\bullet} x} S^{\bullet}A\xrightarrow{m_A} A$. We declare the image of $[x]$ under the natural transformation to be the image of the class $c$ under $m_A\circ S^{\bullet}x$. This produces a map $H^j(S^{\bullet}(R[-i]))\to \Func(H^i,H^j)$ inverse to the map (\ref{free cosimplicial: cohomology operations formula}).
\end{proof}

We will now deduce from Lemma \ref{free cosimplicial: algebra Bockstein} that certain cohomology operations of degree $0$ and $1$ are related to Frobenius endomorphism, and Witt vectors Bockstein homomorphism introduced by Serre in \cite[\S 3]{serre}. Recall that for any object $M\in D(R)$ of the derived category of modules over a ring $R$ we defined a map $T_p(M)[-1]\xrightarrow{\gamma_M} S^p_R M$ in (\ref{free cosimplicial: general fiber sequence}). 

In the case $M=\bZ_p[-i]$ over $R=\bZ_p$ it takes the form $\bF_p[-i-1]\to S^p(\bZ_p[-i])$, and we denote by $P^1_{\bZ_p}\in H^{i+1}(S^p(\bZ_p[-i]))$ the resulting $p$-torsion class. For $M=\bF_p[-i]$ over $R=\bF_p$ the map $\gamma_{\bF_p}$ has the form $\bF_p[-i]\oplus\bF_p[-i-1]\to S^p_{\bF_p}(\bF_p[-i])$, and we denote by $P^0_{\bF_p}\in H^i(S^p(\bF_p[-i])), P^1_{\bF_p}\in H^{i+1}(S^p_{\bF_p}(\bF_p[-i]))$ the resulting classes. Note that $P^1_{\bF_p}$ is the image of $P^1_{\bZ_p}$ under the reduction map.

By Lemma \ref{free cosimplicial: cohomology operations} the class $P^1_{\bZ_p}$ gives rise to a natural homomorphism \begin{equation}P^1_{\bZ_p}:H^i(A)\to H^{i+1}(A)\end{equation} for every cosimplicial commutative $\bZ_p$-algebra $A$. Similarly, the classes $P^0_{\bF_p},P^1_{\bF_p}$ define operations \begin{equation}
P^0_{\bF_p}:H^i(A_0)\to H^{i}(A_0)\quad P^1_{\bF_p}:H^i(A_0)\to H^{i+1}(A_0)
\end{equation}for every cosimplicial commutative $\bF_p$-algebra $A_0$. In the following result by $W_2(A_0)$ we mean the cosimplicial commutative $\bZ/p^2$-algebra obtained by applying term-wise the length $2$ Witt vectors functor to $A_0$.  

\begin{pr}\label{free cosimplicial: steenrod operations description prop}
\begin{enumerate}
\item For a cosimplicial commutative algebra $A$ over $\bZ_p$ the operation $P^1_{\bZ_p}:H^i(A)\to H^{i+1}(A)$ is equal to the composition $$H^i(A)\to H^i(A/p)\xrightarrow{\varphi_{A/p}}H^{i}(A/p)\xrightarrow{\Bock^i_{A}}H^{i+1}(A).$$

\item For a cosimplicial commutative algebra $A_0$ over $\bF_p$ the operation $P^0_{\bF_p}:H^i(A_0)\to H^i(A_0)$ is equal to the Frobenius endomorphism of $A_0$.

\item For a cosimplicial commutative algebra $A_0$ over $\bF_p$ the operation $P^1_{\bF_p}:H^i(A_0)\to H^{i+1}(A_0)$ is equal to the connecting homomorphism induced by the exact sequence $A_0\xrightarrow{V} W_2(A_0)\to A_0$.
\end{enumerate}
\end{pr}

\begin{proof}
Given a class $[x]\in H^i(A)$ represented by a map $x:\bZ_p[-i]\to A$ we get a commutative diagram
\begin{equation}\label{free cosimplicial: steenrod diagram}
\begin{tikzcd}
T_p(\bZ_p[-i])[-1]\arrow[d, "{T_p(x)[-1]}"]\arrow[r, "\gamma_{\bZ_p[-i]}"] & S^p\bZ_p[-i]\arrow[d, "S^px"] \\
T_p(A)[-1]\arrow[r, "\gamma_{A}"] & S^pA\arrow[r, "m_A"] & A.
\end{tikzcd}
\end{equation}

By definition, the value $P^1_{\bZ_p}([x])\in H^{i+1}(A)$ is equal to the image of the class $1\in \bF_p\simeq H^{i+1}(T_p(\bZ_p[-i])[-1])$ under the clockwise composition in the diagram (\ref{free cosimplicial: steenrod diagram}). On the other hand, the counter-clockwise composition is homotopic to $\bF_p[-i-1]\xrightarrow{x[-1]\bmod p }A/p[-1]\xrightarrow{\varphi_{A/p}[-1]}A/p[-1]\xrightarrow{\Bock_{A}}A$ by Lemma \ref{free cosimplicial: algebra Bockstein}, and this implies part (1).

For a class $[x]\in H^i(A_0)$ represented by a map $x:\bF_p[-i]\to A_0$ the value $P^0_{\bF_p}([x])\in H^i(A_0)$ is defined as the image of the unit in $\bF_p=H^i(\bF_p[-i])$ under the composition $\bF_p[-i]\xrightarrow{\Delta_{\bF_p[-i]}}S^p\bF_p[-i]\xrightarrow{S^px}S^pA_0\xrightarrow{m_{A_0}} A_0$. This composition is homotopic to $\bF_p[-i]\xrightarrow{x}A_0\xrightarrow{\Delta_{A_0}}S^p A_0\xrightarrow{m_{A_0}} A_0$ which implies statement (2) by the definition of Frobenius on $A_0$.

By Lemma \ref{free cosimplicial: cohomology operations}, in (3) it is enough to check the claimed formula for $P^1_{\bF_p}$ on the universal class $1\in H^i(S^{\bullet}\bF_p[-i])$ for $A_0=S^{\bullet}\bF_p[-i]$. This cosimplicial $\bF_p$-algebra lifts to $A=S^{\bullet}\bZ_p[-i]$ and the universal class lifts along the map $H^i(A)\to H^i(A_0)$. Therefore we can apply part (1) to get that $P^1_{\bF_p}(1)$ is equal to the image of $1$ under the composition $H^i(A_0)\xrightarrow{\varphi_{A_0}}H^i(A_0)\xrightarrow{\Bock_{A/p^2}}H^{i+1}(A_0)$. 

To see that this composition is equal to the Witt vector Bockstein homomorphism, recall that there is a map of cosimplicial rings $\phi:W_2(A_0)\to A/p^2$ term-wise given by $[a_0]+V[a_1]\mapsto \ta_0^p+p\ta_1$ where $\ta_0,\ta_1\in A^i/p^2$ are arbitrary lifts of $a_0,a_1\in A^i_0$. This map fits into a commutative diagram where rows are term-wise exact sequences:
\begin{equation}
\begin{tikzcd}
0\arrow[r] & A_0\arrow[r,"V"]\arrow[d,equal] & W_2(A_0)\arrow[r]\arrow[d,"\phi"] & A_0\arrow[r]\arrow[d, "\varphi_{A_0}"] & 0 \\
0\arrow[r] & A_0\arrow[r,"p"] & A/p^2\arrow[r] & A_0\arrow[r] & 0.
\end{tikzcd}
\end{equation}
The induced map between the associated long exact sequence of cohomology proves that the connecting homomorphism induced by the top row is equal to $\Bock_{A/p^2}\circ\varphi_{A_0}$ which finishes the proof of (3).
\end{proof}

\subsection{Steenrod operations on cohomology of $E_{\infty}$-algebras.} In this expository subsection we recall how the power operations on cohomology of an $E_{\infty}$-algebra are defined, and relate them to the operations on the cohomology of cosimplicial commutative algebras. All of this material is contained in \cite{steenrod-epstein}, \cite{may-steenrod}, \cite{priddy}, and \cite{lurie-dag13}.

Let $A$ be an $E_{\infty}$-algebra over a (ordinary) commutative ring $R$. This structure, in particular, gives a multiplication map $A^{\otimes p}\to A$ in $D(R)$ which is $S_p$-equivariant where $S_p$ acts via permutations of the factors in $A^{\otimes p}$ and acts trivially on $A$. The multiplication map therefore factors through a map
\begin{equation}
m_A:A^{\otimes p}_{hS_p}\to A.
\end{equation}
This map is used to define operations on the cohomology of $A$. Given a cohomology class $[x]\in H^i(A)$ we can represent it by a map $x:R[-i]\to A$ in $D(R)$ and form the composition
\begin{equation}\label{free cosimplicial: coh class symmetric multiplication}
R[-i]^{\otimes p}_{hS_p}\xrightarrow{x^{\otimes p}}A^{\otimes p}_{hS_p}\xrightarrow{m_A}A.
\end{equation}

We now specialize to the case $R=\bF_p$. The cohomology groups of $\bF_p[-i]^{\otimes p}_{hS_p}$ are given by
\begin{equation}
H^j(\bF_p[-i]^{\otimes p}_{hS_p})=\begin{cases}
\bF_p,\text{ for }j=pi-(2k+1)(p-1)\text{ or }pi-(2k+1)(p-1)+1\text{ with }k\geq 0\\
0, \text{ otherwise}
\end{cases}
\end{equation}
if $i$ is odd, and by
\begin{equation}
H^j(\bF_p[-i]^{\otimes p}_{hS_p})=\begin{cases}
\bF_p,\text{ for }j=pi-2k(p-1)\text{ or }pi-2k(p-1)+1\text{ with }k\geq 0\\
0, \text{ otherwise}
\end{cases}
\end{equation}
if $i$ is even, see \cite[Chapter V, Lemmas 6.1, 6.2 + Proposition 7.8]{steenrod-epstein}. Thus for every $m$ of the form $2k(p-1)$ or $2k(p-1)+1$ we can define $P^{m}([x])\in H^{i+m}(A)$ as the image of a fixed generator of $H^{i+m}(\bF_p[-i]^{\otimes p}_{hS_p})$ under the map (\ref{free cosimplicial: coh class symmetric multiplication}). Applying this construction to the singular cohomology of a topological space gives rise to natural cohomology operations which satisfy special properties that are false for general $E_{\infty}$-algebras:

\begin{thm}\label{free cosimplicial: top space operations}
For a topological space $X$ and an integer $i\geq 0$ the operations $P^0:H^i_{\sing}(X,\bF_p)\to H^i_{\sing}(X,\bF_p),P^1:H^i_{\sing}(X,\bF_p)\to H^{i+1}_{\sing}(X,\bF_p)$ arising from the $E_{\infty}$-algebra structure on $C^{\bullet}_{\sing}(X,\bF_p)$ are described as
\begin{enumerate}
    \item $P^0=\Id$
    \item $P^1$ is the Bockstein homomorphism corresponding to the term-wise exact sequence of complexes $C^{\bullet}_{\sing}(X,\bF_p)\to C^{\bullet}_{\sing}(X,\bZ/p^2)\to C^{\bullet}_{\sing}(X,\bF_p)$.
\end{enumerate}
Moreover, all operations $P^m$ with $m<0$ are zero.
\end{thm}

We can reconcile these operations with the operations on cohomology of cosimplicial commutative algebras, which will justify denoting these by the same symbols $P^i$. Let $A\in\CAlg^{\Delta}_{\bF_p}$ be a cosimplicial commutative $\bF_p$-algebra. We can view it as an $E_{\infty}$-algebra with the symmetric multiplication factoring through the cosimplicial symmetric multiplication
\begin{equation}
A^{\otimes p}_{hS_p}\to S^pA\xrightarrow{m_A}A.
\end{equation}
In particular, for a class $x\in H^i(A)$ the map (\ref{free cosimplicial: coh class symmetric multiplication}) factors through the natural map $\bF_p[-i]^{\otimes p}_{hS_p}\to S^p(\bF_p[-i])$. 

\begin{lm}\label{free cosimplicial: free einf free cosimp deg p}
For every $i$ the map $\bF_p[-i]^{\otimes p}_{hS_p}\to S^p(\bF_p[-i])$ factors through an equivalence $\tau^{\geq i}(\bF_p[-i]^{\otimes p}_{hS_p})\simeq S^p(\bF_p[-i])$.
\end{lm}

\begin{proof}
The cohomology groups of the object $S^p(\bF_p[-i])$ are abstractly isomorphic to those of $\tau^{\geq i}(\bF_p[-i]^{\otimes p}_{hS_p})$ by \cite[Theorems 4.1.1,4.2.1]{priddy} so it is enough to show that the map $\tau^{\geq i}(\bF_p[-i]^{\otimes p}_{hS_p})\to S^p\bF_p[-i]$ induces an injection on cohomology in degrees $\geq i$.

For this it suffices to provide, for every $m\geq 0$, a topological space $X$ such that the operation $P^m:H^{i}_{\sing}(X,\bF_p)\to H^{i+m}_{\sing}(X,\bF_p)$ is non-zero. The Eilenberg-MacLane space $X=K(\bF_p,i)$ satisfies this condition, because the operations on cohomology of topological spaces defined from cohomology classes of Eilenberg-MacLane spaces coincide with those defined using the $E_{\infty}$-algebra structure, e.g. by the uniqueness of functorial operations \cite[\S VIII.3]{steenrod-epstein}.
\end{proof}

This discussion can be applied to the cosimplicial commutative algebra $A=C^{\bullet}_{\sing}(X,\bF_p)$ of singular cochains on a topological space $X$ whose underlying $E_{\infty}$-algebra was referenced in Theorem \ref{free cosimplicial: top space operations}. The fact that operations $P^m$ vanish for $m<0$ can thus be explained by the fact that $S^p(\bF_p[-i])$ is concentrated in degrees $\geq i$. The algebras of the form $C^{\bullet}_{\sing}(X,\bF_p)$ are special among general cosimplicial commutative $\bF_p$-algebras in that their Frobenius endomorphism is equal to the identity, because $C^{\bullet}_{\sing}(X,\bF_p)$ is defined as the algebra of $\bF_p$-valued functions on the simplicial singular set of $X$, cf. \cite[6.1]{priddy}. In particular, the relations $P^0=\Id$ and $P^1=\Bock$ in Theorem \ref{free cosimplicial: top space operations} can be deduced from our Proposition \ref{free cosimplicial: steenrod operations description prop}.

\section{Extensions in complexes underlying derived commutative algebras}
\label{cosimp: section}

In this section we prove our main algebraic result on extensions in the canonical filtration on the complexes underlying certain derived commutative algebras. In this section $X$ will be an algebraic stack flat over $\bZ_p$, or a formal scheme flat over $\Spf\bZ_p$. In both regimes we denote by $X_0=X\times_{\bZ_p}\bF_p$ the special fiber of $X$.

\begin{thm}\label{cosimp: main theorem}
Let $A\in \DAlg^{\geq 0}(X)$ be a derived commutative algebra on $X$ concentrated in degrees $\geq 0$ such that $H^0(A)=\cO_X$, $H^1(A)$ a locally free $\cO_X$-module, and multiplication on cohomology induces an isomorphism $H^{\bullet}(A)\simeq \Lambda^{\bullet}H^1(A)$. Assume further that there is a morphism $s:H^1(A)[-1]\to A$ in $D(X)$ splitting the canonical filtration on $\tau^{\leq 1}A$.

\begin{enumerate}
\item There exists a natural equivalence $\bigoplus\limits_{i=0}^{p-1} H^i(A)[-i]\simeq \tau^{\leq p-1}A$ in $D(X)$.

\item There is a natural homotopy between the map $H^p(A)\simeq \cofib(\tau^{\leq p-1}A\to\tau^{\leq p}A)[p]\to(\tau^{\leq p-1}A)[p+1]$ corresponding to the extension $\tau^{\leq p-1}A\to \tau^{\leq p}A\to H^p(A)[-p]$ and the composition
\begin{multline}\label{cosimp: main formula}
H^p(A)=\Lambda^p H^1(A)\xrightarrow{\alpha(H^1(A)/p)}F_{X_0}^*H^1(A/p)[p-1]\xrightarrow{F_{X_0}^*s[p]}(\tau^{\leq 1}F_{X_0}^*(A/p))[p]\xrightarrow{\varphi_{A/p}}\\ (\tau^{\leq 1}A/p)[p]\xrightarrow{\Bock_A}(\tau^{\leq 1}A)[p+1]\to (\tau^{\leq p-1}A)[p+1].
\end{multline}
where $\alpha(H^1(A)/p)$ is the class introduced in Definition \ref{free cosimplicial: alpha definition}.
\end{enumerate}

\end{thm}

\begin{proof}
For each $i\geq 0$, consider the map
\begin{equation}
s_i:S^i(H^1(A)[-1])\xrightarrow{S^i s}S^iA\xrightarrow{m_A}A
\end{equation} 
The composition $H^1(A)[-1]^{\otimes i}\to S^i(H^1(A)[-1])\xrightarrow{s_i}A$ induces the multiplication map $m:H^1(A)^{\otimes i}\to H^i(A)$ on $i$-th cohomology and $0$ on all other cohomology groups. By the assumption that $H^{\bullet}(A)$ is freely generated by $H^1(A)$, the map $m:H^1(A)^{\otimes i}\to H^i(A)$ identifies $H^i(A)$ with $\Lambda^iH^1(A)$.

For $i\leq p-1$ we have $S^i(H^1(A)[-1])\simeq \Gamma^i(H^1(A)[-1])\simeq (\Lambda^i H^1(A))[-i]$, and the map $H^1(A)[-1]^{\otimes i}\to S^i(H^1(A)[-1])$ is the shift by $[-i]$ of the natural surjection $H^1(A)^{\otimes i}\to \Lambda^i H^1(A)$. Therefore $s_i$ induces an isomorphism on $H^i$, which proves the first part of the theorem.

For $i=p$ the natural map $S^p(H^1(A)[-1])\xrightarrow{N_p}\Gamma^p(H^1(A)[-1])\simeq (\Lambda^p H^1(A))[-p]$ is not an equivalence anymore. Using the map $s_p$, we will relate the extension $\tau^{\leq p-1}A\to\tau^{\leq p}A\to H^p(A)[-p]$ to the extension \begin{equation}F_{X_0}^*H^1(A)/p[-2]\to S^p(H^1(A)[-1])\to \Lambda^p H^1(A)[-p]\end{equation} constructed in (\ref{free cosimplicial: alpha flat over zp}). To begin, let us compute the map induced by $s_p$ on the cohomology in degree $p$.

Since $S^p(H^1(A)[-1])$ is concentrated in degrees $\leq p$, the map $s_p$ naturally factors through $\tau^{\leq p}A$. We claim that the composition \begin{equation}\label{cosimp: deg p coh diagram}S^p(H^1(A)[-1])\xrightarrow{s_p}\tau^{\leq p}A\to H^p(A)[-p]\end{equation} is naturally homotopic to the norm map $N_p:S^p(H^1(A)[-1])\to(\Lambda^pH^1(A))[-p]$. This composition factors uniquely through the norm map, because the mapping space $\Map_{D(X)}(\fib (N_p),H^p(A)[-p])$ is contractible as $H^p(A)[-p]$ is a locally free sheaf placed in degree $p$, and $\fib (N_p)\simeq F^*_{X_0}H^1(A)/p[-2]$ is a $p$-torsion object concentrated in degrees $\leq 2\leq p$. Therefore the composition (\ref{cosimp: deg p coh diagram}) has the form $S^p(H^1(A)[-1])\xrightarrow{N_p}\Lambda^p H^1(A)[-p]\xrightarrow{\psi}H^p(A)[-p]$ for some map $\psi$. To check that $\psi$ is equal to the cup-product map, we may precompose this composition with the map $H^1(A)[-1]^{\otimes p}\to S^p H^1(A)[-1]$, and the claim follows from:

\begin{lm}\label{cosimp: tensor power vs symmetric power}
Let $E$ be a locally free sheaf on $X$. Under the d\'ecalage identification $\Gamma^p (E[-1])\simeq \Lambda^p E[-p]$ the composition $E[-1]^{\otimes p}\to S^p(E[-1])\xrightarrow{N_p}\Gamma^p(E[-1])$ is identified with the shift by $[-p]$ of the map $E^{\otimes p}\to\Lambda^p E$. 
\end{lm}

\begin{proof}
\cite[Proposition I.4.3.2.1]{illusie-cotangent1} shows that d\'ecalage equivalences are compatible with graded algebra structures on symmetric, divided power, and exterior algebras. Our assertion is a special case of this.
\end{proof}

This allows us to fit $s_p$ into the following map of fiber sequences for some map $\beta_p^{\leq p-1}$:

\begin{equation}
\begin{tikzcd}[row sep=large, column sep = large]
F^*_{X_0}(H^1(A)/p)[-2]\arrow[r,"{\gamma_{H^1(A)[-1]}}"]\arrow[d,"\beta_p^{\leq p-1}"] & S^p(H^1(A)[-1])\arrow[r,"N_p"]\arrow[d, "s_p"] & (\Lambda^p H^1(A))[-p]\arrow[d, "\sim"] \\
\tau^{\leq p-1}A\arrow[r] & \tau^{\leq p}A\arrow[r] & H^p(A)[-p].
\end{tikzcd}
\end{equation}

This diagram implies that the extension class $H^p(A)\to (\tau^{\leq p-1}A)[p+1]$ corresponding to the bottom row can be described as the composition
\begin{equation}\label{cosimp: diagram yet unknown beta}
H^p(A)\simeq \Lambda^p H^1(A)\xrightarrow{\alpha(H^1(A)/p)}F_{X_0}^*(H^1(A)/p)[p-1]\xrightarrow{\beta_p^{\leq p-1}[p+1]}(\tau^{\leq p-1}A)[p+1].
\end{equation}

Here $\alpha(H^1(A)/p)$ is the natural class attached to the vector bundle $H^1(A)/p$ on $X_0$ by Definition \ref{free cosimplicial: alpha definition}, we view it as a map from $\Lambda^p H^1(A)$ to $F^*_{X_0}(H^1(A)/p)[p-1]$ via adjunction as in Lemma \ref{free cosimplicial: alpha adjunction}.

To finish the proof of Theorem \ref{cosimp: main theorem}, it remains to compute $\beta_p^{\leq p-1}$. Note that $\beta_p^{\leq p-1}$ can be naturally recovered from the composition $\beta_p:F^*_{X_0}H^1(A)/p[-2]\xrightarrow{\beta^{\leq p-1}_p} \tau^{\leq p-1}A\to A$ as the truncation $\tau^{\leq p-1}\beta_p$, because the mapping space $\Map_{D(X)}(F^*_{X_0}H^1(A)/p[-2],\tau^{\geq p}A)$ is contractible. To identify $\beta_p$, consider the following commutative diagram
\begin{equation}
\begin{tikzcd}[column sep = large]
F^*_{X_0}(H^1(A)/p)[-2]\arrow[r,"{\gamma_{H^1(A)[-1]}}"]\arrow[d, "{F_{X_0}^*s[-1]}"]\arrow[drr, bend left=58,"\beta_p"] & S^p (H^1(A)[-1])\arrow[d, "S^ps"]\arrow[rd, "s_p"]\\
F^*_{X_0}(A/p)[-1]\arrow[r, "{\gamma_A}"] & S^pA \arrow[r, "m"] & A.
\end{tikzcd}
\end{equation}

The composition of the bottom row was proven in Lemma \ref{free cosimplicial: algebra Bockstein} to be homotopic to $F^*_{X_0}(A/p)[-1]\xrightarrow{\varphi_{A/p}}A/p[-1]\xrightarrow{\Bock_A}A$, hence $\beta_p:F_{X_0}^*(H^1(A)/p)[-2]\to A$ is homotopic to the composition \begin{equation}
F^*_{X_0}(H^1(A)/p)[-2]\xrightarrow{F_{X_0}^*s[-1]}F_{X_0}^*(A/p)[-1]\xrightarrow{\varphi_{A/p}}A/p[-1]\xrightarrow{\Bock_A[-1]}A
\end{equation} 
that can be factored as 
\begin{equation}\label{cosimp: betaleq1 formula}
F^*_{X_0}(H^1(A)/p)[-2]\xrightarrow{F_{X_0}^*s[-1]}(\tau^{\leq 1}F_{X_0}^*(A/p))[-1]\xrightarrow{\varphi_{A/p}}\tau^{\leq 1}(A/p)[-1]\xrightarrow{\Bock_A[-1]}\tau^{\leq 1}A\to A.
\end{equation}

Therefore $\beta^{\leq p-1}_p$ is the composition of the first $3$ arrows in (\ref{cosimp: betaleq1 formula}) followed by the map $\tau^{\leq 1}A\to\tau^{\leq p-1}A$. Plugging this expression for $\beta_p^{\leq p-1}$ into (\ref{cosimp: diagram yet unknown beta}) finishes the proof of the second part of the theorem.
\end{proof}

Let us record the special form that Theorem \ref{cosimp: main theorem} takes in the case of augmented algebras.

\begin{cor}\label{cosimp: augmented}
Let $A\in \DAlg^{\geq 0}(X)$ be a derived commutative algebra such that $H^0(A)\simeq \cO_X$ and the multiplication on cohomology induces an isomorphism $H^{\bullet}(A)=\Lambda^{\bullet}H^1(A)$. Assume also that $A$ is equipped with a map $\varepsilon:A\to \cO_X$ of derived commutative algebras that induces an isomorphism on $H^0$. Then 

\begin{enumerate}
\item $\tau^{\leq p-1}A$ naturally decomposes in $D(X)$ as $\bigoplus\limits_{i=0}^{p-1}H^i(A)[-i]$.

\item The extension class $H^p(A)\to \tau^{\leq p-1}A[p+1]$ corresponding to $\tau^{\leq p}A$ can be described as the composition

\begin{multline}\label{cosimp: augmented formula}
H^p(A)\xrightarrow{\alpha(H^1(A)/p)}F^*(H^1(A)/p)[p-1]\xrightarrow{\varphi_{A/p}} (H^1(A)/p)[p-1]\xrightarrow{\Bock_{H^1(A)}[p-1]} \\ \to H^1(A)[p]\xrightarrow{} \tau^{\leq p-1}A[p+1].
\end{multline}
\end{enumerate}
\end{cor}

\begin{proof}
The augmentation map $\varepsilon$ induces a splitting of the fiber sequence $H^0(A)\to \tau^{\leq 1}A\to H^1(A)[-1]$. In particular, there exists a map $s:H^1(A)[-1]\to A$ inducing an isomorphism on $H^1$, so we are in a position to apply Theorem \ref{cosimp: main theorem}. The formula (\ref{cosimp: main formula}) specializes to (\ref{cosimp: augmented formula}) because under the decompositions $\tau^{\leq 1}A\simeq H^0(A)\oplus H^1(A)[-1]$ and $\tau^{\leq 1}(A/p)\simeq H^0(A/p)\oplus H^1(A/p)[-1]$ the Bockstein and Frobenius morphisms are diagonalized.
\end{proof}

\subsection{Equivariant situation} Suppose that $X$ is a flat $\bZ_p$-scheme equipped with an action of a discrete group $G$. Let us explicitly record that Theorem \ref{cosimp: main theorem} applied to the global quotient $[X/G]$ can be rephrased as a statement about $G$-equivariant algebras on $X$, where we take the definition of the category of $G$-equivariant derived commutative algebras on $X$ to be $\DAlg([X/G])$.  We state the result in the augmented setting because this is the version that will be used in all the applications.

\begin{thm}\label{cosimp: equivariant main theorem}
Given a $G$-equivariant augmented derived commutative algebra $A$ on $X$, such that $H^0(A)=\cO_X$, the sheaf $H^1(A)$ is a locally free sheaf of $\cO_X$-modules, and the multiplication induces isomorphisms $\Lambda^{\bullet} H^1(A)\simeq H^{\bullet}(A)$, we have

\begin{enumerate}
    \item There is an equivalence $\tau^{\leq p-1}A\simeq \bigoplus\limits_{i=0}^{p-1} H^i(A)[-i]$ in $D_G(X)$.
    
    \item The map $H^p(A)\to \tau^{\leq p-1}A[p+1]$ corresponding to the fiber sequence $\tau^{\leq p-1}A\to \tau^{\leq p}A\to H^p(A)[-p]$ in $D_G(X)$ can be described as 
    \begin{equation}\label{cosimp: equivariant main formula}
H^p(A)\xrightarrow{\alpha(H^1(A)/p)}F^*(H^1(A)/p)[p-1]\xrightarrow{\varphi_{A/p}} (H^1(A)/p)[p-1]\xrightarrow{\Bock_{H^1(A)}[p-1]}H^1(A)[p]\xrightarrow{}\tau^{\leq p-1}A[p+1].
    \end{equation}
\end{enumerate}
\end{thm}

\section{Applications to de Rham and Hodge-Tate cohomology}\label{applications: section}

In this section we apply Theorem \ref{cosimp: main theorem} to de Rham and diffracted Hodge complexes. Let $k$ be a perfect field of characteristic $p>0$. We now work in a setup where $X$ is a formally smooth formal scheme over $W(k)$, and as before denote by $X_0$ its special fiber $X\times_{W(k)}k$. We denote by $F_{X_0/k}:X_0\to X_0^{(1)}:=X_0\times_{k,\Fr_p}k$ the relative Frobenius morphism, and as before $F_{X_0}:X_0\to X_0$ denotes the absolute Frobenius morphism. We have the following cohomological invariants associated to $X$ and $X_0$, each equipped with a derived commutative algebra structure.

\begin{itemize}
\item The diffracted Hodge complex $\Omega^{\DHod}_{X}\in D(X)$ defined in \cite[Notation 4.7.12]{apc} whose cohomology algebra $H^{\bullet}(\Omega^{\DHod}_X)$ is isomorphic to the algebra $\Omega^{\bullet}_X$ of differential forms on $X$. By \cite[Theorem 7.20(2)]{bhatt-lurie-prismatization} it can be identified with the derived pushforward $R\pi^{\HT}_*\cO_{X^{\DHod}}$ of the structure sheaf along the map $\pi^{\HT}:X^{\DHod}\to X$, hence Lemma \ref{dalg: pushforward of rings} equips $\Omega^{\DHod}_X$ with a structure of a derived commutative algebra in $D(X)$. The Sen operator on $\Omega^{\DHod}_{X}$ induces a decomposition $\tau^{\leq p-1}\Omega^{\DHod}_{X}\simeq\bigoplus\limits_{i=0}^{p-1}\Omega^i_X[-i]$ in $D(X)$, and, in particular, gives rise to a map $s:\Omega^1_X[-1]\to\Omega^{\DHod}_X$ that induces an isomorphism on $H^1$. 

\item The de Rham complex $\dR_{X_0/k}=F_{X_0/k*}\cO_{X_0}\xrightarrow{d}F_{X_0/k*}\Omega^1_{X_0/k}\xrightarrow{d}\dots$ viewed as an object of $D(X^{(1)}_0)$. The Cartier isomorphism provides an identification $H^{\bullet}(\dR_{X_0/k})\simeq \Omega^{\bullet}_{X^{(1)}_0/k}$ of graded algebras. By de Rham comparison, $\dR_{X_0/k}$ is naturally identified with $\Omega^{\DHod}_{X^{(1)}}/p$, where $X^{(1)}:=X\times_{W(k),\Fr_p}W(k)$ is the Frobenius-twist of the formal $W(k)$-scheme $X$, cf. \cite[Remark 4.7.18]{apc}. We give $\dR_{X_0/k}$ the structure of an object of $\DAlg(X^{(1)}_0)$ by identifying it with the derived pushforward of the structure sheaf along the map $\pi^{\HT}:(X^{(1)}_0)^{\HT}\to X_0^{(1)}$ using Lemma \ref{dalg: pushforward of rings}.
\end{itemize}

\begin{rem}
The diffracted Hodge complex can be identified with the Hodge-Tate cohomology of $X$ relative to an appropriate prism, by \cite[Example 4.7.8]{apc}. Theorem \ref{cosimp: main theorem} also applies to Hodge-Tate cohomology of smooth formal schemes over arbitrary prisms, as well as to the decompleted version of the diffracted Hodge cohomology \cite[Construction 4.9.1]{apc}.
\end{rem}

We will apply Theorem \ref{cosimp: main theorem} to $\Omega^{\DHod}_{X^{(1)}}$ and will then compute the extension in the canonical filtration on $\dR_{X_0/k}$ by reducing modulo $p$. To begin, we will relate the Frobenius endomorphism of the de Rham complex to the obstruction to lifting Frobenius onto $X\times_{W(k)}W_2(k)$.

\subsection{Obstruction to lifting Frobenius over $W_2(k)$.} We prove the results in this subsection without the smoothness assumption, and only assuming the existence of a flat lift over $W_2(k)$, for a future application in Section \ref{semiperf: section}.  For a scheme $Y_0$ over $k$ equipped with a lift $Y_1$ over $W_2(k)$ we denote by $\ob_{F,Y_1}:F_{Y_0/k}^*L\Omega^1_{Y^{(1)}_0/k}\to \cO_{Y_0}[1]$ the obstruction to lifting $F_{Y_0/k}:Y_0\to Y_0^{(1)}$ to a morphism $Y_1\to Y_1^{(1)}$, as defined by Illusie \cite{illusie-cotangent1}. We also denote by $\dR_{Y_0/k}$ the derived de Rham complex of $Y_0$ relative to $k$, viewed as an object of $D(Y_0^{(1)})$, cf. \cite[\S 3]{bhatt-derived}. It is equipped with a filtration $\Fil^{\conj}_{\bullet}$ whose graded pieces are equivalent to the shifted exterior powers of the cotangent complex: $\gr^{\conj}_i\dR_{Y_0/k}\simeq L\Omega^i_{Y_0^{(1)}/k}[-i]$. Note also that the natural map induces an equivalence $\dR_{Y_0/\bF_p}\simeq \dR_{Y_0/k}$.

Using the relation between the cotangent complex of $Y_0$ over $W(k)$ and the de Rham complex of $Y_0$, due to \cite{prisms} and \cite{illusie-conjugate}, we will prove:

\begin{pr}\label{applications: frobenius lift obstruction prop}
Let $Y_0$ be a quasisyntomic scheme over $k$ equipped with a flat lift $Y_1$ over $W_2(k)$. Denote by $s:L\Omega^1_{Y^{(1)}_0/k}[-1]\to \Fil_1^{\conj}\dR_{Y_0/k}$ the splitting of the conjugate filtration in degree $1$ arising from $Y_1$. The composition \begin{equation}F_{Y_0/k}^*L\Omega^1_{Y^{(1)}_0/k}[-1]\xrightarrow{F^*_{Y_0/k}s}F_{Y_0/k}^*\Fil_1^{\conj}\dR_{Y_0/k}\xrightarrow{dF_{Y^{(-1)}_0/k}} \Fil^{\conj}_1\dR_{Y^{(-1)}_0/k}\end{equation} is homotopic to the composition $F_{Y_0/k}^*L\Omega^1_{Y^{(1)}_0}[-1]\xrightarrow{\ob_{F,Y_1}}\cO_{Y_0}\to\Fil_1^{\conj}\dR_{Y^{(-1)}_0}$. Here $Y_0^{(-1)}$ is the twist $Y_0\times_{k,\Fr_p^{-1}}k$ by the inverse of Frobenius and $dF_{Y^{(-1)}_0/k}$ is the map induced by the functoriality of the de Rham complex.
\end{pr}
\begin{rem}
For $Y_0$ smooth over $k$ this result also follows from the explicit model for the map $s:\Omega^1_{Y_0^{(1)}/k}[-1]\to \dR_{Y_0/k}$ using local Frobenius lifts, provided by \cite{deligne-illusie}, cf. \cite{srinivas}.
\end{rem}
\begin{rem}\label{applications: absolute vs relative obstruction}
Using the isomorphism $Y_0^{(1)}\simeq Y_0$ of $\bF_p$-schemes we may view $\ob_{F,Y_1}$ as a map $F^*_{Y_0}L\Omega^1_{Y_0}[-1]\to\cO_{Y_0}[1]$ that coincides with the obstruction to lifting the absolute Frobenius morphism $F_{Y_0}:Y_0\to Y_0$ to an endomorphism of $Y_1$, we will denote this map by the same symbol $\ob_{F,Y_1}$.
\end{rem}
\begin{convention}\label{applications: sign flip}We take $\ob_{F,Y_1}$ to mean the {\it negative} of the obstruction class defined in \cite[III.2.2]{illusie-cotangent1}, to avoid a trailing sign in all of the subsequent expressions. 
\end{convention}

We start by recalling from \cite{illusie-conjugate} how a lift of $Y_0$ over $W_2(k)$ provides a decomposition of $L\Omega^1_{Y_0/W(k)}$. In general, if $Y_0$ is a quasisyntomic scheme over $k$ we have the fundamental triangle corresponding to the morphisms $Y_0\to \Spec k\to \Spec W(k)$
\begin{equation}\label{applications: fundamental triangle formula}
\cO_{Y_0}[1]\simeq \cO_{Y_0}\otimes_k L\Omega^1_{k/W(k)}\to  L\Omega^1_{Y_0/W(k)}\to L\Omega^1_{Y_0/k}.
\end{equation}

The natural map $\cO_{Y_0}\otimes_k L\Omega^1_{k/W(k)}\simeq \fib(L\Omega^1_{Y_0/W(k)}\to L\Omega^1_{Y_0/k})\to \cO_{Y_0}\otimes_k L\Omega^1_{k/W_2(k)}\simeq \fib(L\Omega^1_{Y_0/W_2(k)}\to L\Omega^1_{Y_0/k})$ establishes the source as a direct summand of the target, because the map $L\Omega^1_{k/W(k)}\to L\Omega^1_{k/W_2(k)}$ in the derived category of $k$-vector spaces induces an equivalence $L\Omega^1_{k/W(k)}\simeq \tau^{\geq -1}L\Omega^1_{k/W_2(k)}\simeq H^{-1}(L\Omega^1_{k/W_2(k)})[-1]$. Hence a flat scheme $Y_1$ over $W_2(k)$ lifting $Y_0$ induces a splitting of the fiber sequence (\ref{applications: fundamental triangle formula}) via the map $di:L\Omega^1_{Y_0/k}\simeq i^*L\Omega^1_{Y_1/W_2(k)}\to L\Omega^1_{Y_0/W_2(k)}$, where $i:Y_0\hookrightarrow Y_1$ is the inclusion of the special fiber, cf. \cite[\S 4]{illusie-conjugate}. We denote by $s'_{Y_1}:L\Omega^1_{Y_0/k}\to L\Omega^1_{Y_0/W(k)}$ the resulting section of (\ref{applications: fundamental triangle formula}). Denote also by $r_{Y_1}:L\Omega^1_{Y_0/W_2(k)}\to \cO_{Y_0}[1]$ the composition $L\Omega^1_{Y_0/W_2(k)}\to\tau^{\geq 1}L\Omega^1_{Y_0/W_2(k)}\simeq L\Omega^1_{Y_0/W(k)}\to \cO_{Y_0}[1]$ where the last map is the complementary splitting of the first map in (\ref{applications: fundamental triangle formula}).

\begin{lm}\label{applications: frobenius obstruction definition}
For a flat scheme $Y_1$ over $W_2(k)$ the obstruction $\ob_{F,Y_1}:F_{Y_0/k}^*L\Omega^1_{Y^{(1)}_0/k}\to \cO_{Y_0}[1]$ to lifting $F_{Y_0/k}:Y_0\to Y^{(1)}_0$ to a morphism from $Y_1$ to $Y_1^{(1)}$ is homotopic to the composition

\begin{equation}
F_{Y_0/k}^*L\Omega^1_{Y^{(1)}_0/k}\xrightarrow{F^*_{Y_0/k}di^{(1)}} F_{Y_0/k}^*L\Omega^1_{Y^{(1)}_0/W_2(k)}\xrightarrow{dF_{Y_0/k}}L\Omega^1_{Y_0/W_2(k)}\xrightarrow{r_{Y_1}} \cO_{Y_0}[1].
\end{equation}
\end{lm}

\begin{proof}
Let us start by recalling the definition of this obstruction given in \cite[Proposition III.2.2.4]{illusie-cotangent1}. The scheme $Y_1$ viewed as a lift of $Y_0$ produces a map $L\Omega^1_{Y_0/W_2(k)}\to \cO_{Y_0}[1]$, and composing it with $dF_{Y_0/k}$ we get a map
\begin{equation}\gamma_{Y_0}:F_{Y_0/k}^*L\Omega^1_{Y^{(1)}_0/W_2(k)}\xrightarrow{dF_{Y_0/k}} L\Omega^1_{Y_0/W_2(k)}\to \cO_{Y_0}[1].\end{equation}
On the other hand, the lift $Y^{(1)}_1$ of $Y^{(1)}_0$ gives rise to the map $r_{Y_0^{(1)}}:L\Omega^1_{Y^{(1)}_0/W_2(k)}\to \cO_{Y^{(1)}_0}[1]$, and pulling it back along $F_{Y_0/k}$ we obtain a map
\begin{equation}
\gamma_{Y^{(1)}_0}:F_{Y_0/k}^*L\Omega^1_{Y^{(1)}_0/W_2(k)}\to F^*_{Y_0/k}\cO_{Y^{(1)}_0}[1]\simeq \cO_{Y_0}[1].
\end{equation}
By definition, $\ob_{F,Y_1}:F_{Y_0/k}^*L\Omega^1_{Y^{(1)}_0/k}\to \cO_{Y_0}[1]$ is the unique up to homotopy map such that the composition $F_{Y_0/k}^*L\Omega^1_{Y^{(1)}_0/W_2(k)}\to F_{Y_0/k}^*L\Omega^1_{Y^{(1)}_0/k}\xrightarrow{\ob_{F,Y_1}}\cO_{Y_0}[1]$ is homotopic to the difference $\gamma_{Y_0}-\gamma_{Y_0^{(1)}}$ (we have incorporated Convention \ref{applications: sign flip} at this point). Equivalently, $\ob_{F,Y_1}$ is the composition \begin{equation}F^*_{Y_0/k}L\Omega^1_{Y^{(1)}_0/k}\xrightarrow{F_{Y_0/k}^*di^{(1)}}F_{Y_0/k}^*L\Omega^1_{Y^{(1)}_0/W_2(k)}\xrightarrow{\gamma_{Y_0}-\gamma_{Y^{(1)}_0}}\cO_{Y_0}[1].\end{equation}
However, the composition $F_{Y_0/k}^*L\Omega^1_{Y^{(1)}_0/k}\xrightarrow{F^*_{Y_0/k}di^{(1)}} F_{Y_0/k}^*L\Omega^1_{Y^{(1)}_0/W_2(k)}\xrightarrow{\gamma_{Y^{(1)}_0}}\cO_{Y_0}[1]$ is zero by construction, so the lemma follows from the definition of $\gamma_{Y_0}$.
\end{proof}

\begin{proof}[Proof of Proposition \ref{applications: frobenius lift obstruction prop}]
\cite[Proposition 4.15]{prisms} in the smooth case, and \cite[Corollary 3.3]{illusie-conjugate} in general identifies $\Fil_1^{\conj}\dR_{Y_0/k}$ with the shifted cotangent complex $L\Omega^1_{Y^{(1)}_0/W(k)}[-1]$. Moreover, the fiber sequence $\cO_{Y^{(1)}_0}\to \Fil_1^{\conj}\dR_{Y_0/k}\to L\Omega^1_{Y^{(1)}_0/k}[-1]$ induced by the conjugate filtration is identified with the shift of the fundamental triangle \begin{equation}\label{applications: fundamental w2 triangle}L\Omega^1_{k/W(k)}\otimes_k\cO_{Y^{(1)}_0}\to L\Omega^1_{Y^{(1)}_0/W(k)}\to L\Omega^1_{Y^{(1)}_0/k}\end{equation} corresponding to the sequence of morphisms $Y^{(1)}_0\to\Spec k\to \Spec W(k)$. Denote by $i:Y^{(1)}_0\hookrightarrow Y^{(1)}_1$ the usual inclusion. We have a map $s'_{Y_1^{(1)}}:L\Omega^1_{Y^{(1)}_0/k}\to L\Omega^1_{Y^{(1)}_0/W(k)}$ that splits this fundamental triangle, hence defining a map $s'_{Y^{(1)}_1}:L\Omega^1_{Y^{(1)}_0/k}[-1]\to\Fil_1^{\conj}\dR_{Y_0/k}$. By the uniqueness of functorial decompositions of the de Rham complex \cite[Theorem 5.10]{li-mondal}, $s'_{Y^{(1)}_1}$ is naturally equivalent to the section $s_{Y^{(1)}_1}$ constructed using the Sen operator.

Since the identification $\Fil_1^{\conj}\dR_{Y_0/k}\simeq L\Omega^1_{Y^{(1)}_0/W(k)}[-1]$ is functorial in $Y_0$, the map $F_{Y_0/k}^*\Fil_1^{\conj}\dR_{Y_0/k}\to \Fil_1^{\conj}\dR_{Y_0^{(-1)}/k}$ induced by $F_{Y_0/k}$ is identified with the shift of $dF_{Y_0/k}:F_{Y_0/k}^*L\Omega^1_{Y^{(1)}_0/W(k)}\xrightarrow{} L\Omega^1_{Y_0/W(k)}$. The latter map factors as \begin{equation}\label{applications: frob on cotangent complex factorization}
F^*_{Y_0/k}L\Omega^1_{Y^{(1)}_0/W(k)}\to \cO_{Y_0}[1]\to L\Omega^1_{Y_0/W(k)}
\end{equation}
because $dF_{Y_0/k}:F_{Y_0/k}^*L\Omega^1_{Y^{(1)}_0/k}\to L\Omega^1_{Y_0/k}$ is zero. Given factorization (\ref{applications: frob on cotangent complex factorization}) we see that $dF_{Y_0/k}:F_{Y_0/k}^*L\Omega^1_{Y^{(1)}_0/W(k)}\xrightarrow{} L\Omega^1_{Y_0/W(k)}$ can be identified with the composition
\begin{equation}\label{applications: frob on cotangent complex idempotent}
F_{Y_0/k}^*L\Omega^1_{Y^{(1)}_0/W(k)}\xrightarrow{dF_{Y_0/k}} L\Omega^1_{Y_0/W(k)}\to L\Omega^1_{Y_0/W_2(k)}\xrightarrow{r_{Y_1}}\cO_{Y_0}[1]\to L\Omega^1_{Y_0/W(k)}
\end{equation}
because the composition of the middle two maps in (\ref{applications: frob on cotangent complex idempotent}) becomes the identity when composed $\cO_{Y_0}[1]\to L\Omega^1_{Y_0/W(k)}$, regardless of the lift $Y_1$ we chose. Our goal is to calculate the composition of (\ref{applications: frob on cotangent complex idempotent}) with the map $F^*_{Y_0/k}L\Omega^1_{Y^{(1)}_0}\xrightarrow{F^*_{Y_0/k}s'_{Y_1^{(1)}}} F^*_{Y_0/k}L\Omega^1_{Y^{(1)}_0/W(k)}$, and in view of the above factorization Proposition \ref{applications: frobenius lift obstruction prop} follows from the formula for the obstruction to lifting Frobenius given by Lemma \ref{applications: frobenius obstruction definition}.
\end{proof}

The Frobenius arising from the derived commutative structure on the de Rham complex coincides with the map induced by the geometric Frobenius morphism:

\begin{lm}\label{applications: cosimplicial frob is frob}
For a smooth scheme $X_0$ over $k$ the Frobenius map $\varphi_{\dR_{X_0/k}}:F^*_{X^{(1)}_0}\dR_{X_0/k}\to \dR_{X_0/k}$ of the derived commutative algebra $\dR_{X_0/k}\in \DAlg(X_0^{(1)})$ is naturally identified with the map $dF_{X^{(1)}_0}:F_{X^{(1)}_0}^*\dR_{X_0/k}\to \dR_{X_0/k}$ induced by the functoriality of the de Rham complex under the absolute Frobenius endomorphism. 

Explicitly, this morphism is the composition \begin{equation}F_{X_0^{(1)}}^*\dR_{X_0/k}\to F^*_{X^{(1)}_0}F_{X_0/k*}\cO_{X_0}\xrightarrow{\varphi_{{X^{(1)}_0}}} \cO_{X^{(1)}_0}\to\dR_{X_0/k}\end{equation} where the first map is induced by the map from the de Rham complex to its $0$th term.
\end{lm}

\begin{proof}
We endowed $\dR_{X_0/k}$ with the structure of a derived commutative algebra by identifying it with the pushforward $R\pi_{*}^{\HT}\cO_{(X^{(1)}_0)^{\HT}}$ of the structure sheaf along the morphism of stacks $\pi^{\HT}:(X^{(1)}_0)^{\HT}\to X^{(1)}_0$. By Lemma \ref{dalg: geometric frobenius}, the map $\varphi_{\dR_{X_0/k}}:F_{X_0^{(1)}}^*\dR_{X_0/k}\to \dR_{X_0/k}$ is the composition $F^*_{X^{(1)}_0}\dR_{X_0/k}=F_{X^{(1)}_0}^*R\pi^{\HT}_{*}\cO_{(X^{(1)}_0)^{\HT}}\to R\pi^{\HT}_*F_{(X^{(1)}_0)^{\HT}}^*\cO_{(X^{(1)}_0)^{\HT}}\simeq R\pi^{\HT}_*\cO_{(X^{(1)}_0)^{\HT}}=\dR_{X_0/k}$. The Frobenius endomorphism of the stack $X_0^{\HT}$ coincides with the endomorphism induced from $F_{X^{(1)}_0}$ by functoriality, cf. \cite[Remark 3.6]{bhatt-lurie-prismatization}. Hence $\varphi_{\dR_{X_0/k}}$ is equivalent to the morphism induced by $F_{X_0}$ by functoriality of the de Rham complex. The last assertion follows from the fact that the maps $dF_{X^{(1)}_0}:F_{X^{(1)}_0}^*F_{X_0/k*}\Omega^i_{X_0}\to\Omega^i_{X^{(1)}_0}$ are zero for all $i\geq 1$.
\end{proof}
Let us record the compatibility between two Frobenius morphisms on the level of cohomology  arising from Lemma \ref{applications: cosimplicial frob is frob}:

\begin{lm}\label{applications: coh cosimplicial frob is frob}
For a smooth $k$-scheme $X_0$, the Frobenius morphism $\varphi_{\RGamma_{\dR}(X_0/k)}$ of the derived commutative $k$-algebra $\RGamma_{\dR}(X_0/k)$ is naturally homotopic to the morphism induced by the relative Frobenius morphism $F_{X_0/k}:X_0\to X^{(1)}_0$.
\end{lm}

\subsection{Computing the extension in diffracted Hodge complex of a \texorpdfstring{$W(k)$}{}-scheme.} We can now prove the main result of this section:

\begin{thm}\label{cosimp applications: HT extension}
For a smooth formal scheme $X$ over $\Spf W(k)$ there is a natural decomposition $\tau^{\leq p-1}\Omega^{\DHod}_X\simeq \bigoplus\limits_{i=0}^{p-1}\Omega^i_X[-i]$ in $D(X)$. The class of the extension $\tau^{\leq p-1}\Omega^{\DHod}_X\to\tau^{\leq p}\Omega^{\DHod}_X\to \Omega^p_X[-p]$ in $D(X)$ is naturally equivalent to the composition

\begin{equation}\label{cosimp applications: main formula}
\Omega^p_X\xrightarrow{\alpha(\Omega^1_{X_0})}F_{X_0}^*\Omega^1_{X_0}[p-1]\xrightarrow{\ob_{F,X\times_{W(k)}{W_2(k)}}}\cO_{X_0}[p]\xrightarrow{\Bock_{\cO_X}}\cO_X[p+1]\to \Omega^{\DHod}_X[p+1].
\end{equation}
\end{thm}

\begin{proof}
Applying Theorem \ref{cosimp: main theorem} to the derived commutative algebra $\Omega^{\DHod}_X$ on $X$ equipped with the splitting $s:\Omega^1_X[-1]\to \Omega^{\DHod}_X$ gives the first assertion (which by now we knew anyway thanks to the existence of the Sen operator), as well as the following formula for the extension class of $\tau^{\leq p}\Omega^{\DHod}_X$:
\begin{multline}
\Omega^p_X\xrightarrow{\alpha(\Omega^1_{X_0})}F_{X_0}^*\Omega^1_{X_0}[p-1]\xrightarrow{F_{X_0}^*s[p]}(\tau^{\leq 1}F^*_{X_0}\dR_{X^{(-1)}_0/k})[p]\xrightarrow{\varphi_{\dR_{X^{(-1)}_0/k}}}(\tau^{\leq 1}\dR_{X_0^{(-1)}/k})[p]\\ \xrightarrow{\Bock_{\Omega^{\DHod}_X}}(\tau^{\leq 1}\Omega^{\DHod}_X)[p+1]\to (\tau^{\leq p-1})\Omega^{\DHod}_X[p+1].
\end{multline}

Using Proposition \ref{applications: frobenius lift obstruction prop} together with Remark \ref{applications: absolute vs relative obstruction} and Lemma \ref{applications: cosimplicial frob is frob} we can rewrite this composition as
\begin{multline}
\Omega^p_X\xrightarrow{\alpha(\Omega^1_{X_0})}F_{X_0}^*\Omega^1_{X_0}[p-1]\xrightarrow{\ob_{F,X\times_{W(k)}W_2(k)}}\cO_{X_0}[p]\to(\tau^{\leq 1}\dR_{X_0^{(-1)}/k})[p]\\ \xrightarrow{\Bock_{\Omega^{\DHod}_X}}(\tau^{\leq 1}\Omega^{\DHod}_X)[p+1]\to (\tau^{\leq p-1}\Omega^{\DHod}_X)[p+1].
\end{multline}
Converting this into (\ref{cosimp applications: main formula}) amounts to the observation that the Bockstein maps for $\Omega^{\DHod}_X$ and $\cO_X$ are related by the commutative square
\begin{equation}
\begin{tikzcd}
\cO_{X_0}\arrow[r]\arrow[d, "\Bock_{\cO_{X}}"] & \dR_{X^{(-1)}_0/k}\arrow[d, "\Bock_{\Omega^{\DHod}_{X}}"] \\
\cO_X[1]\arrow[r] & \Omega^{\DHod}_X[1].
\end{tikzcd}
\end{equation}
\end{proof}

\subsection{Cohomological consequences.} We can rewrite the answer provided by (\ref{cosimp applications: main formula}) as a formula for the extension class in $H^{p+1}(X,\Lambda^p T_X)$. Getting to this statement from Theorem \ref{cosimp applications: HT extension} amounts to a general piece of bookkeeping concerning the Bockstein homomorphisms, given by Lemma \ref{cosimp applications: Bockstein} below.

\begin{thm}\label{cosimp applications: the best part}
The class of the extension $\bigoplus\limits_{i=0}^{p-1}\Omega^i_{X}[-i]\to\tau^{\leq p}\Omega^{\DHod}_X\to\Omega^p_{X}[-p]$ in $H^{p+1}(X,\Lambda^pT_{X})$ is equal to \begin{equation}\Bock_X(\ob_{F,X\times_{W(k)}W_2(k)}\cup \alpha(\Omega^1_{X_0})),\end{equation} that is the result of applying the Bockstein homomorphism $\Bock_{X}:H^p(X_0,\Lambda^pT_{X_0})\to H^{p+1}(X,\Lambda^p T_{X})$ to the product of classes $\alpha(\Omega^1_{X_0})\in\Ext^{p-1}_{X_0}(\Omega^p_{X_0},F_{X_0}^*\Omega^1_{X_0})=H^{p-1}(X_0,\Lambda^p T_{X_0}\otimes F_{X_0}^*\Omega^1_{X_0})$ and $\ob_{F,X\times_{W(k)}W_2(k)}\in H^1(X_0,F^*T_{X_0})$.
\end{thm}

\begin{lm}\label{cosimp applications: Bockstein}
For two objects $M,N\in D(X)$ denote by $\underline{\RHom}_{\cO_X}(M,N)\in D(X)$ their internal $\Hom$ object. The Bockstein morphism \begin{equation}\RHom_{X_0}(i^*M,i^*N)=\RGamma(X_0,i^*\underline{\RHom}_{\cO_X}(M,N))\to \RGamma(X,\underline{\RHom}_{\cO_X}(M,N)[1])=\RHom_X(M,N[1])\end{equation}
can be described as sending $f:i^*M\to i^*N$ to the composition
\begin{equation}
M\to i_*i^*M\xrightarrow{i_*f}i_*i^*N\to N[1].
\end{equation}
\end{lm}

\begin{proof}
For an object $K\in D(X)$ there is a natural fiber sequence $K\to i_*i^*K\to K[1]$ coming from the identification $i_*i^*K\simeq\cofib(K\xrightarrow{p} K)$. Applying $\RGamma(X,-)$ to the second map $i^*i_*K\to K[1]$ induces the Bockstein homomorphism $\RGamma(X_0,i^*K)\to \RGamma(X,K[1])$ on cohomology. The lemma aims to compute this map in the case $K=\underline{\RHom}_{\cO_X}(M,N)$. By adjunction, we have $i_*i^*\underline{\RHom}_{\cO_X}(M,N)\simeq i_*\underline{\RHom}_{\cO_{X_0}}(i^*M,i^*N)\simeq \underline{\RHom}_{\cO_X}(M,i_*i^*N)$, and the Bockstein map $i_*i^*\underline{\RHom}_{\cO_X}(M,N)\to \underline{\RHom}_{\cO_X}(M,N)[1]$ is induced by composing with the map $i_*i^*N\to N[1]$. This proves the lemma by the observation that the adjunction equivalence $\RHom_{{X_0}}(i^*M,i^*N)\simeq \RHom_{X}(M,i_*i^*N)$ takes a map $f:i^*M\to i^*N$ to the composition $M\to i_*i^*M\xrightarrow{i_*f}i_*i^*N$.
\end{proof}

Reducing modulo $p$ we also get a description of the extension class of $\tau^{\leq p}\dR_{X_0/k}$:

\begin{thm}\label{cosimp applications: the best part de Rham}
If $X_0$ is a smooth scheme over $k$ equipped with a smooth lift $X$ over $W(k)$ then the class of the extension $\bigoplus\limits_{i=0}^{p-1}\Omega^i_{X^{(1)}_0}[-i]\to\tau^{\leq p}\dR_{X_0/k}\to\Omega^p_{X^{(1)}_0}[-p]$ in $H^{p+1}(X^{(1)}_0,\Lambda^pT_{X^{(1)}_0})$ is equal to \begin{equation}\Bock_{X^{(1)}_{W_2(k)}}(\ob_{F,X^{(1)}_{W_2(k)}}\cup \alpha(\Omega^1_{X^{(1)}_0})),\end{equation} the result of applying the Bockstein homomorphism $\Bock_{X^{(1)}_{W_2(k)}}:H^p(X^{(1)}_0,\Lambda^pT_{X^{(1)}_0})\to H^{p+1}(X^{(1)}_0,\Lambda^p T_{X^{(1)}_0})$ to the product of classes $\alpha(\Omega^1_{X^{(1)}_0})\in\Ext^{p-1}_{X^{(1)}_0}(\Omega^p_{X^{(1)}_0},F_{X_0^{(1)}}^*\Omega^1_{X^{(1)}_0})=H^{p-1}(X^{(1)}_0,\Lambda^p T_{X^{(1)}_0}\otimes F_{X^{(1)}_0}^*\Omega^1_{X^{(1)}_0})$ and $\ob_{F,X^{(1)}_{W_2(k)}}\in H^1(X^{(1)}_0,F_{X^{(1)}_0}^*T_{X^{(1)}_0})$.
\end{thm}

For the convenience of applications, let us state explicitly what Theorem \ref{cosimp applications: the best part} tells us about the differentials in the Hodge-Tate spectral sequence

\begin{cor}
For a smooth $W(k)$-scheme $X$ the Hodge-Tate spectral sequence $E_2^{ij}=H^i(X,\Omega^j_{X/W(k)})\Rightarrow H^{i+j}_{\DHod}(X)$ has no non-zero differentials on pages $E_2,\ldots,E_p$ and for every $i\geq 0$ the differential $d_{p+1}^{i,p}:H^i(X,\Omega^p_{X/W(k)})\to H^{i+p+1}(X,\cO_X)$ on page $E_{p+1}$ can be described as the composition

\begin{multline}
H^{i}(X,\Omega^p_{X/W(k)})\to H^i(X_0,\Omega^p_{X_0/k})\xrightarrow{\alpha(\Omega^1_{X_0})}H^{i+p-1}(X_0,F_{X_0}^*\Omega^1_{X_0/k})\xrightarrow{\ob_{F,X}}\\ H^{i+p}(X_0,\cO_{X_0})\xrightarrow{\Bock_{X}}H^{i+p+1}(X,\cO_X).
\end{multline}
\end{cor}

\subsection{Decomposing de Rham complex compatibly with the algebra structure} Another application of the results of Section \ref{free cosimplicial: section} is an obstruction to formality of the de Rham complex as a commutative algebra. In what follows, we define the formal derived commutative algebra $\bigoplus\limits_{i\geq 0}\Omega^i_{X^{(1)}_0/k}[-i]$ as the free divided power algebra on the object $\Omega^1_{X^{(1)}_0}[-1]\in D(X^{(1)}_0)$, see Definition \ref{dalg: free divided power algebra def}.

\begin{pr}\label{cosimp applications: de rham formality}
Let $X_0$ be an arbitrary smooth scheme over $k$. The following are equivalent

\begin{enumerate}
\item $\dR_{X_0/k}$ is equivalent to $\bigoplus\limits_{i\geq 0}\Omega^i_{X^{(1)}_0/k}[-i]$ as a derived commutative algebra in $D(X^{(1)}_0)$

\item $\dR_{X_0/k}$ is equivalent to $\bigoplus\limits_{i\geq 0}\Omega^i_{X^{(1)}_0/k}[-i]$ as an $E_{\infty}$-algebra in $D(X^{(1)}_0)$

\item there exists a map $\dR_{X_0/k}\to\cO_{X_0^{(1)}}$ of $E_{\infty}$-algebras in $D(X_0^{(1)})$ that induces an isomorphism $H^0(\dR_{X_0/k})\simeq\cO_{X_0^{(1)}}$

\item $X_0$ together with its Frobenius endomorphism admits a lift over $W_2(k)$.
\end{enumerate}
\end{pr}

\begin{proof}
Formally, (1) implies (2), and (2) implies (3). Let us first show that (4) implies (1). It is well-known that (4) implies the formality of $\dR_{X_0/k}$ as a commutative DG algebra \cite[Remarque 2.2(ii)]{deligne-illusie}, or as an $E_{\infty}$-algebra \cite[Proposition 3.17]{bhatt-derived}. We have not discussed the relation between commutative DG algebras and derived commutative algebras, so let us give an independent argument for the equivalence $\dR_{X_0/k}\simeq\bigoplus\limits_{i\geq 0}\Omega^i_{X^{(1)}_0/k}[-i]$ in $\DAlg(X_0)$, in the presence of a Frobenius lift over $W_2(k)$.

For this equivalence, we will represent $\dR_{X_0/k}$ by the \v{C}ech complex associated to the flat cover $X_0^{\perf}\to X_0$ where $X_0^{\perf}:=\lim\limits_{\leftarrow}X^{(-n)}_0$ is the perfection of $X_0$. By fpqc descent for derived de Rham cohomology we have that $\dR_{X_0/k}$ is equivalent to the cosimplicial totalization of the following diagram of derived commutative algebras in $D(X_0^{(1)})$, cf. \cite[Remark 8.15]{bms2}:

\begin{equation}\label{cosimp applications: de rham cech diagram}
\begin{tikzcd}\dR_{X_0^{\perf}/k} \arrow[r, shift left=0.65ex] \arrow[r, shift right=0.65ex] & \dR_{X_0^{\perf}\times_{X_0}X_0^{\perf}/k} \arrow[r, shift left=1.3ex] \arrow[r, shift right=1.3ex] \arrow[r] &\ldots \end{tikzcd}
\end{equation}
Since each $(X_0^{\perf})^{\times_{X_0}n}$ is a quasiregular semiperfect scheme, all terms of this diagram are classical commutative algebras in $\QCoh(X_0^{(1)})$ placed in degree $0$. The lift $X_1$ of $X_0$ together with a Frobenius endomorphism induces lifts of the scheme $X_0^{\perf}$, the morphism $X_0^{\perf}\to X_0$, and the Frobenius endomorphism of $X_0^{\perf}$. Therefore each algebra $\dR_{(X_0^{\perf})^{\times_{X_0}n}/k}$ decomposes as the Frobenius-twist of $\Gamma^{\bullet}(L\Omega^1_{(X_0^{\perf})^{\times_{X_0}n}/k}[-1])$, by \cite[Proposition 3.17]{bhatt-derived}. Moreover all the maps in (\ref{cosimp applications: de rham cech diagram}) are compatible with these decompositions, so the cosimplicial commutative algebra defined by this diagram is quasi-isomorphic to $\Gamma^{\bullet}_{\naive}(\DK(\Omega^1_{X_0^{(1)}/k}[-1]))$. This cosimplicial commutative algebra is a model for the derived commutative algebra $\bigoplus\limits_{i\geq 0}\Omega^i_{X^{(1)}_0/k}[-i]$ by Lemma \ref{dalg: free divided power as cosimp}, so the implication (4)$\Rightarrow$(1) is proven. 

This is not logically necessary, but we will first give the proof of the implication (1)$\Rightarrow$(4) to illustrate the idea of the implication (3)$\Rightarrow$(4) in a simpler case. Suppose that such an equivalence $\alpha:\bigoplus\limits_{i\geq 0}\Omega^i_{X^{(1)}_0/k}[-i]\simeq \dR_{X_0/k} $ of derived commutative algebras exists. In particular, $\tau^{\leq 1}\dR_{X_0/k}$ is decomposed as $\cO_{X^{(1)}_0}\oplus\Omega^1_{X^{(1)}_0/k}[-1]$, as an object of $D(X_0)$. By \cite[Th\'eor\`eme 3.5]{deligne-illusie} any such decomposition is induced, up to a homotopy, by a lift of $X_0$ over $W_2(k)$. Denote the lift corresponding to $\tau^{\leq 1}\alpha$ by $X_1$. 
 
The Frobenius endomorphism $\varphi_{\bigoplus\limits \Omega^i_{X^{(1)}_0/k}[-i]}:F_{X^{(1)}_0}^*\bigoplus\limits_{i\geq 0}\Omega^i_{X^{(1)}_0/k}[-i]\to \bigoplus\limits_{i\geq 0}\Omega^i_{X^{(1)}_0/k}[-i]$ is zero on all components with $i\geq 1$ and is the usual adjunction map $F_{X^{(1)}_0}^*\cO_{X^{(1)}_0}\xrightarrow{\sim}\cO_{X^{(1)}_0}$ on the structure sheaf. Hence composing the section $\Omega^1_{X^{(1)}_0}[-1]\to \dR_{X_0/k}$ induced by $\alpha$ with the Frobenius map $\varphi_{\dR_{X_0/k}}:F_{X^{(1)}_0}^*\dR_{X_0/k}\to\dR_{X_0/k}$ is homotopic to zero. By Proposition \ref{applications: frobenius lift obstruction prop} and Lemma \ref{applications: cosimplicial frob is frob} this implies that the composition $F^*_{X_0^{(1)}}\Omega^1_{X^{(1)}_0}[-1]\xrightarrow{\ob_{F,X_1}}\cO_{X^{(1)}_0}\xrightarrow{}\dR_{X_0/k}$ is homotopic to zero. Since the second map admits a section, it follows that $\ob_{F,X_1}=0$ and $X_1$ is a lift of $X_0$ that admits a lift of Frobenius.

Finally, we prove that (3) implies (4). The augmentation map $\varepsilon:\dR_{X_0/k}\to\cO_{X_0^{(1)}}$ induces, in particular, a map $s:\Omega^1_{X_0^{(1)}}[-1]\to \dR_{X_0/k}$ in $D(X_0^{(1)})$ inducing an isomorphism on $H^1$ and such that the composition $\varepsilon\circ s$ is zero. Denote by $X_1$ the lift of $X_0$ corresponding to this section $s$. The map $s:\Omega^1_{X^{(1)}_0}[-1]\to \dR_{X_0/k}$ also induces a map $s_p:S^p(\Omega^1_{X^{(1)}_0}[-1])\to \dR_{X_0/k}$. The assumption that $\varepsilon$ is a map of $E_{\infty}$-algebras implies that the composition \begin{equation}\Omega^1_{X^{(1)}_0}[-1]^{\otimes p}_{hS_p}\to S^p(\Omega^1_{X^{(1)}_0}[-1])\xrightarrow{s_p} \dR_{X_0/k}\xrightarrow{\varepsilon}\cO_{X_0^{(1)}}\end{equation} is homotopic to zero. We will deduce from this that the composition $T_p(\Omega^1_{X^{(1)}_0/k}[-1])[-1]\to S^p\Omega^1_{X^{(1)}_0}[-1]\to \dR_{X_0/k}\xrightarrow{\varepsilon}\cO_{X_0^{(1)}}$ is homotopic to zero.

Denote by $K$ the fiber of the map $\Omega^1_{X^{(1)}_0}[-1]^{\otimes p}_{hS_p}\to \Omega^p_{X^{(1)}_0}[-p]$, it is equipped with a map $K\to T_p(\Omega^1_{X^{(1)}_0}[-1])[-1]$, and the object $\fib(K\to T_p(\Omega^1_{X^{(1)}_0}[-1])[-1])\simeq \fib(\Omega^1_{X^{(1)}_0}[-1]^{\otimes p}_{hS_p}\to S^p\Omega^1_{X^{(1)}_0}[-1])$ is concentrated in degrees $\leq 0$ by Lemma \ref{free cosimplicial: cosimp vs einf}. Our assumption implies that the map $T_p(\Omega^1_{X^{(1)}_0}[-1])[-1]\to \dR_{X_0/k}\xrightarrow{\varepsilon}\cO_{X_0^{(1)}}$ factors through the map $T_p(\Omega^1_{X^{(1)}_0}[-1])[-1]\to \fib(K[1]\to T_p(\Omega^1_{X^{(1)}_0}[-1]))$ which forces it to be homotopic to zero because this fiber object is concentrated in degrees $\geq -1$.

In Lemma \ref{free cosimplicial: tate p}(4) we constructed a map $F_{X_0}^*M\to T_p(M)[-1]$ for any $M\in D(X_0)$ such that the composition $F_{X_0}^*M\to T_p(M)[-1]\xrightarrow{\gamma_M}S^p M$ is the map $\Delta_M: F_{X_0}^*M\to S^p M$. By definition, the Frobenius map is the composition $F_{X^{(1)}_0}^*\dR_{X_0/k}\xrightarrow{\Delta_{\dR_{X_0/k}}}S^p\dR_{X_0/k}\xrightarrow{m}\dR_{X_0/k}$ which allows us to conclude that the composition \begin{equation}F_{X_0^{(1)}}^*\Omega^1_{X^{(1)}_0}[-1]\to T_p(\Omega^1_{X^{(1)}_0}[-1])[-1]\to S^p\Omega^1_{X^{(1)}_0}[-1]\xrightarrow{s_p} \dR_{X_0/k}\end{equation}
is homotopic to the composition \begin{equation}\label{applications: augmentation formula}F_{X^{(1)}_0}^*\Omega^1_{X_0^{(1)}}[-1]\xrightarrow{F^*_{X_0^{(1)}}s}F_{X_0^{(1)}}^*\dR_{X_0/k}\xrightarrow{\varphi_{\dR_{X_0/k}}}\dR_{X_0/k}.\end{equation} We are given that the composition of (\ref{applications: augmentation formula}) with $\varepsilon:\dR_{X_0/k}\to\cO_{X_0^{(1)}}$ is homotopic to zero, which is equivalent to the vanishing of the obstruction to lifting $F_{X_0}$ on $X_1$, by Proposition \ref{applications: frobenius lift obstruction prop}.
\end{proof}
\begin{rem}
Bhargav Bhatt pointed out that the condition (3) is equivalent to the existence of a section of the gerbe $X_0^{\DHod}\to X_0$, by a version of \cite[7.3]{bhatt-lurie-prismatization}, hence the implication (3)$\Rightarrow$(4) also follows from \cite[Conjecture 5.14]{bhatt-lurie-prismatization}.
\end{rem}

\section{Preliminaries on the Sen operator}\label{sen operator: section}

In this section we collect preliminary material on the Sen operator, and more generally on the derived category of sheaves equipped with an endomorphism. We first discuss the Sen operator on diffracted Hodge cohomology of a scheme over $\bZ_p$ and relate its non-semisimplicity to non-decomposability of the de Rham complex, and then give a parallel discussion over $\bZ/p^n,n\geq 2$ for which we, in particular, describe the Hodge-Tate locus of the Cartier-Witt stack of $\bZ/p^n$, generalizing \cite[Example 5.15]{bhatt-lurie-prismatization}. For the purposes of the application to the de Rham complex in Section \ref{semiperf: section}, the case of $\bZ/p^n$ subsumes that of $\bZ_p$, but we invite the reader to consider the case of schemes over $\bZ_p$ first.

\subsection{Category of objects with an endomorphism.} Let $X$ be a flat formal scheme over $W(k)$ and as before $D(X)$ is the $\infty$-category of quasicoherent sheaves on $X$. In this section we will work in the $\infty$-category $D_{\bN}(X):=\Func(B\bN,D(X))$ of objects of $D(X)$ equipped with an additional endomorphism. Objects of this category are pairs $(M,f_M)$ where $M$ is an object of $D(X)$ and $f_M:M\to M$ is an endomorphism of $M$ in $D(X)$. The morphisms between objects $(M,f_M)$ and $(N, f_N)$ are given by 

\begin{equation}\label{sen operator: morphisms in DN formula}
\RHom_{D_{\bN}(X)}((M,f_M),(N,f_N))=\fib(\RHom_{D(X)}(M,N)\xrightarrow{f\mapsto f\circ f_M-f_N\circ f}\RHom_{D(X)}(M,N)).
\end{equation}

For a scalar $\lambda\in \bZ_p$ we have the functor $D(X)\to D_{\bN}(X)$ sending an object $M\in D(X)$ to itself equipped with the endomorphism $\lambda\cdot \Id_M$. We will often use the following special cases of the formula (\ref{sen operator: morphisms in DN formula}):

\begin{lm}\label{sen operator: morphisms in DN}
For an arbitrary object $(M, f_M)\in D_{\bN}(X)$ the morphisms between it and an object of the form $(N, \lambda\cdot \Id_N)$ can be described as

\begin{equation}\label{sen operator: morphism to scalar}
\RHom_{D_{\bN}(X)}((M, f_M),(N,\lambda\cdot\Id_N))= \RHom_{D(X)}(\cofib(f_M-\lambda\cdot \Id_M:M\to M),N)
\end{equation}
\begin{equation}\label{sen operator: morphism from scalar}
\RHom_{D_{\bN}(X)}((N,\lambda\cdot\Id_N),(M,f_M))=\RHom_{D(X)}(N,M^{f_M=\lambda}).
\end{equation}
\end{lm}

\subsection{Diffracted Hodge cohomology and Sen operator over $\bZ_p$.} The works of Drinfeld \cite{drinfeld} and Bhatt-Lurie imply that the diffracted Hodge cohomology is equipped with a natural endomorphism, referred to by \cite{apc} as `Sen operator'. We work with arbitrary (not necessarily smooth over $W(k)$) flat formal schemes because our proof will proceed through a computation with diffracted Hodge cohomology of quasiregular semiperfectoid rings.

\begin{thm}[\hspace{1sp}{\cite[Construction 4.7.1]{apc}}]\label{sen operator: general theorem}
For a bounded $p$-adic formal scheme $X$ there is a natural object $(\Omega^{\DHod}_X,\Theta_X)\in D_{\bN}(X)$ whose underlying object is the diffracted Hodge complex, equipped with a filtration $\Fil^{\conj}_{\bullet}$ such that the graded quotients $\gr_i^{\conj}$ of this filtration are equivalent to $(L\Omega^i_X[-i],-i)$.
\end{thm}

From now on we assume that $X$ is a flat formal $\bZ_p$-scheme. When $X$ is formally smooth over $W(k)$ for some perfect field $k$, the conjugate filtration on $\Omega^{\DHod}_X$ coincides with the canonical filtration, and $\Theta_X$ acts on $H^i(\Omega^{\DHod}_X)\simeq\Omega^i_X$ by $(-i)\cdot\Id_{\Omega^i_X}$. The interesting information contained in this new cohomological invariant of $X$ is the data of extensions between the graded quotients of the conjugate filtration. One of our main goals, achieved in Theorem \ref{semiperf: main sen operator}, is to explicate what this information is for the smallest potentially non-split step $(\Fil^{\conj}_p\Omega^{\DHod}_X,\Theta_X)$ of the conjugate filtration.

Given that $\Theta_X$ acts via multiplication by $-i$ on $\gr_i^{\conj}\Omega^{\DHod}_X$, the product $\Theta_X(\Theta_X+1)\ldots (\Theta_X+i)$ is naturally homotopic to zero as an endomorphism of $\Fil_{i}^{\conj}\Omega^{\DHod}_X$. Therefore for each $i$ the object $\Fil_i^{\conj}\Omega^{\DHod}_X\in D_{\bN}(X)$ is naturally equipped with the structure of a $\bZ_p[t]/t(t+1)\ldots(t+i)$-module, and the decomposition of the spectrum of this ring into a union of connected components induces a decomposition
\begin{equation}
(\Omega^{\DHod}_X,\Theta_X)\simeq \bigoplus\limits_{i=0}^{p-1}(\Omega^{\DHod}_{X,i},\Theta_X)
\end{equation}
such that for every $n\in \bZ_p$ the endomorphism $\Theta_X+n$ of $\Fil_i\Omega^{\DHod}_{X,n\bmod p}$ is topologically nilpotent for every $i$, cf. \cite[Remark 4.7.20]{apc}. Restricting to the special fiber $X_0:=X\times_{\bZ_p}\bF_p\xhookrightarrow{i} X$ we similarly get an endomorphism $\Theta_X$ of $\dR_{X_0}$, and a decomposition
\begin{equation}
(\dR_{X_0},\Theta_X)\simeq \bigoplus\limits_{i=0}^{p-1}(\dR_{X_0,i},\Theta_X).
\end{equation}

We will now set up a way to package the information about the extensions in the conjugate filtration on $(\Fil_p^{\conj}\Omega^{\DHod}_X,\Theta_X)$. There is a decomposition $(\Fil^{\conj}_{p-1}\Omega^{\DHod}_X,\Theta_X)\simeq \bigoplus\limits_{i=0}^{p-1}(L\Omega^i_X[-i],-i)$ and the fiber sequence $(\Fil^{\conj}_{p-1}\Omega^{\DHod}_X,\Theta_X)\to (\Fil_p^{\conj}\Omega^{\DHod}_X,\Theta_X)\to (L\Omega^p_X[-p],-p)$ gives rise to the map in $D_{\bN}(X)$:
\begin{equation}\label{sen operator: extension map}(L\Omega^p_X[-p],-p)\to \bigoplus\limits_{i=0}^{p-1}(L\Omega^i_X[-i],-i)[1].\end{equation}
In the remainder of this section by $L\Omega^i_{X_0}$ we will mean the $i$th exterior power of the cotangent complex of $X_0$ relative to $\bF_p$, so that $L\Omega^i_{X_0}\simeq L\Omega^i_X|_{X_0}$.

\begin{lm}\label{sen operator: ext group description}
There is a natural equivalence
\begin{equation}\label{sen operator: ext group description formula}
\Map_{D_{\bN}(X)}((L\Omega^p_X[-p],-p), \bigoplus\limits_{i=0}^{p-1}(L\Omega^i_X[-i],-i)[1])\simeq \Map_{D(X_0)}(L\Omega^p_{X_0}, \cO_{X_0}[p]).
\end{equation}

Under this identification the map from the LHS to $\Map_{D(X)}(L\Omega^p_X[-p],\bigoplus\limits_{i=0}^{p-1}L\Omega^i_X[-i])$ induced by the forgetful functor $D_{\bN}(X)\to D(X)$ sends $f\in\Map_{D(X_0)}(L\Omega^p_{X_0},\cO_{X_0}[p])$ to the composition $L\Omega^p_X\to L\Omega^p_{X_0}\xrightarrow{f}\cO_{X_0}[p]\xrightarrow{\Bock_{\cO_X}[p]}\cO_X[p+1]$.
\end{lm}

\begin{proof}
By Lemma \ref{sen operator: morphisms in DN} the left hand side of (\ref{sen operator: ext group description formula}) is equivalent to $\Map_{D(X)}(L\Omega^p_X,\cO_X^{p=0}[p+1])$, because the summands $(L\Omega^i_X[-i],-i)$ with $0<i\leq p-1$ do not contribute to this mapping space. The fiber of multiplication by $p$ on $\cO_X$ is identified with $\cO_X/p[-1]=i_*i^*\cO_X[-1]$ hence this mapping space can be further described as $\Map_{D(X)}(L\Omega^p_X,i_*i^*\cO_X[p])\simeq \Map_{D(X_0)}(L\Omega^p_{X_0}, \cO_{X_0}[p])$. 

For the second assertion note that, firstly, under the identification $\Map_{D_{\bN}(X)}((L\Omega^p_X,-p),(\cO_X[p],0))\simeq \Map_{D(X)}(L\Omega^p_X,\cO_X^{p=0}[p+1])$ the map induced by the forgetful functor $D_{\bN}(X)\to D(X)$ is given by composing with the natural map $\cO_X^{p=0}\to\cO_X$ which is nothing but the Bockstein map $\Bock_{\cO_X}[-1]:i_*\cO_{X_0}[-1]\to\cO_X$. Secondly, the adjunction identification $\Map_{D(X_0)}(L\Omega^p_{X_0}, \cO_{X_0}[p])\simeq \Map_{D(X)}(L\Omega^p_X,i_*i^*\cO_X[p])$ sends $f\in \Map_{D(X_0)}(L\Omega^p_{X_0}, \cO_{X_0}[p])$ to $L\Omega^p_X\to i_*L\Omega^p_{X_0/\bF_p}\xrightarrow{i_*f}i_*\cO_{X_0}[p]$, which implies the claim.
\end{proof}

\begin{notation}\label{sen operator: zp classes notation}
We denote by $c_{X,p}\in \Map_{D(X_0)}(L\Omega^p_{X_0}, \cO_{X_0}[p])$ the result of transporting the map (\ref{sen operator: extension map}) along the equivalence (\ref{sen operator: ext group description formula}), and by $e_{X,p}\in\Map_{D(X)}(L\Omega^p_X,\cO_X[p+1])$ we denote the extension class in the conjugate filtration on $\Fil_p^{\conj}\Omega^{\DHod}_{X,0}$.
\end{notation}

 Let us explicitly record how $e_{X,p}$ can be recovered from $c_{X,p}$, which is immediate from Lemma \ref{sen operator: ext group description}:

\begin{lm}\label{sen operator: classes bockstein relation w}
The element $e_{X,p}\in\Map_{D(X)}(L\Omega^p_X,\cO_X[p+1])$ corresponding to the extension $\cO_X\to\Fil^{\conj}_p\Omega^{\DHod}_{X,0}\to L\Omega^p_X[-p]$ is naturally homotopic to the composition \begin{equation}
    L\Omega^p_X\to L\Omega^p_{X_0}\xrightarrow{c_{X,p}}\cO_{X_0}[p]\xrightarrow{\Bock_{\cO_X}}\cO_X[p+1].
\end{equation}
\end{lm}

Lemma \ref{cosimp applications: Bockstein} can be used to restate Lemma \ref{sen operator: classes bockstein relation w} as follows

\begin{cor}\label{sen operator: smooth classes w lift}
The map $e_{X,p}$ is naturally homotopic to the image of $c_{X,p}$ under the Bockstein morphism

\begin{equation}
\Bock: \RHom_{X_0}(L\Omega^p_{X_0},\cO_{X_0}[p])\to \RHom_{X}(L\Omega^p_{X},\cO_X[p+1])
\end{equation}
arising from the object $\RHom_{X}(L\Omega^p_{X},\cO_X[p+1])\in D(\bZ_p)$.
\end{cor}

The map $c_{X,p}$ can also be read off from the action of the Sen operator on $\dR_{X_0}$, as we will now show. In general, if $(M_1,0)\to (M,f_M)\to (M_2,0)$ is a fiber sequence in $D_{\bN}(X_0)$ then $f_M:M\to M$ naturally factors as $M\to M_2\xrightarrow{h}M_1\to M$ for a morphism $h\in\Map_{D(X_0)}(M_2,M_1)$. The element $\delta\in \Map_{D_{\bN}(X_0)}((M_2,0), (M_1,0)[1])=\Map_{D(X_0)}(M_2,M_1[1])\oplus \Map_{D(X_0)}(M_2,M_1)$ corresponding to this fiber sequence is then given by $\underline{\delta}\oplus h$ where $\underline{\delta}$ is the image of $\delta$ under the forgetful map $D_{\bN}(X_0)\to D(X_0)$. Applying this discussion to $\Fil_p^{\conj}\dR_{X_0}$ we obtain:

\begin{lm}\label{sen operator: cXp as nilpotent operator}
The endomorphism $\Theta_X$ of $\Fil^{\conj}_p\dR_{X_0,0}$ naturally factors as 
\begin{equation}
\Fil^{\conj}_p\dR_{X_0,0}\to L\Omega^p_{X_0}[-p]\xrightarrow{c_{X,p}[-p]}\cO_{X_0}=\Fil_0^{\conj}\dR_{X_0,0}\to\Fil_p^{\conj}\dR_{X_0,0}.
\end{equation}
\end{lm}

\begin{proof}
Given the paragraph preceding the statement of the lemma, this amounts to observing that the map $\RHom_{D_{\bN}(X)}((L\Omega^p_X[-p],-p),(\cO_X,0))\to \RHom_{D_{\bN}(X_0)}((L\Omega^p_{X_0}[-p],-p),(\cO_{X_0},0))$ induced by the restriction to the special fiber $X_0$ can be described as $\RHom_{D(X_0)}(L\Omega^p_{X_0},\cO_{X_0}[p])\xrightarrow{(\Bock,\id)}\Map_{D(X_0)}(L\Omega^p_{X_0},\cO_{X_0}[p+1])\oplus \RHom_{D(X_0)}(L\Omega^p_{X_0},\cO_{X_0}[p])$ under the identification (\ref{sen operator: ext group description formula}).
\end{proof}

For our computation of the Sen operator in Theorem \ref{semiperf: main sen operator} it will be convenient to directly relate $c_{X,p}$ to the action of $\Theta_X$ on all of the object $\Fil^{\conj}_p\dR_{X_0}$, rather than its weight zero part:

\begin{lm}\label{sen operator: cXp as nilpotent operator on dR}
The endomorphism $\Theta_X-\Theta_X^p$ of $\Fil^{\conj}_p\dR_{X_0}$ naturally factors as 
\begin{equation}
    \Fil^{\conj}_p\dR_{X_0}\to L\Omega^p_{X_0}[-p]\xrightarrow{c_{X,p}[-p]}\cO_{X_0}=\Fil^{\conj}_0\dR_{X_0}\to \Fil^{\conj}_p\dR_{X_0}.
\end{equation}
\end{lm}

\begin{proof}
The endomorphism $\Id-\Theta_X^{p-1}: \Fil^{\conj}_p\dR_{X_0}\to \Fil^{\conj}_p\dR_{X_0}$ is an idempotent corresponding to the direct summand $\Fil^{\conj}_p\dR_{X_0,0}$ of $\Fil^{\conj}_p\dR_{X_0}$. Therefore $\Theta_X-\Theta_X^p=\Theta_X\cdot (\Id-\Theta_X^{p-1})$ can be factored as $\Fil^{\conj}_p\dR_{X_0}\to\Fil^{\conj}_p\dR_{X_0,0}\xrightarrow{\Theta_X}\Fil^{\conj}_p\dR_{X_0,0}\to \Fil^{\conj}_p\dR_{X_0}$ where the first and last maps establish $\Fil^{\conj}_p\dR_{X_0,0}$ as a direct summand of $\Fil^{\conj}_p\dR_{X_0}$. Hence the claim follows from Lemma \ref{sen operator: cXp as nilpotent operator}.
\end{proof}

Since $\Fil_{p}^{\conj}\dR_{X_0}$ decomposes as $\bigoplus\limits_{i=0}^{p-1}\Fil_p^{\conj}\dR_{X_0,i}\simeq \Fil_{p}^{\conj}\dR_{X_0,0}\oplus\bigoplus\limits_{i=0}^{p-1}L\Omega^i_{X_0}[-i]$ compatibly with the Sen operator, $\Theta_X$ on $\Fil_p^{\conj}\dR_{X_0}$ is semisimple if and only if $c_{X,p}\sim 0$. In particular, if $c_{X,p}$ vanishes then the conjugate filtration on the diffracted Hodge complex splits in degrees $\leq p$:

\begin{cor}\label{sen operator: semisimplicity implies decomposability}
If the Sen operator on $\Fil^{\conj}_p\dR_{X_0}$ is semisimple then there exists a decomposition
\begin{equation}
    (\Fil^{\conj}_p\Omega^{\DHod}_X,\Theta_X)\simeq\bigoplus\limits_{i=0}^p(L\Omega^i_X[-i],-i).
\end{equation}
\end{cor}

\subsection{Diffracted Hodge cohomology and Sen operator over \texorpdfstring{$\bZ/p^n,n\geq 2$}{}.}

In this subsection we give a discussion parallel to the above for a flat scheme $X_{n-1}$ over $\bZ/p^n$ for $n\geq 2$. We will define the diffracted Hodge complex $\Omega^{\DHod}_{X_{n-1}/\bZ/p^n}$ of $X_{n-1}$ relative to $\bZ/p^n$, equipped with a Sen operator with properties analogous to Theorem \ref{sen operator: general theorem}. Theorem \ref{sen operator: zpn main} extends the construction from \cite[\S 5]{bhatt-lurie-prismatization} where the Sen operator on the  mod $p$ reduction of our $\Omega^{\DHod}_{X_{n-1}/\bZ/p^n}$ is defined.

\begin{thm}\label{sen operator: zpn main}
For a scheme $X_{n-1}$ quasisyntomic over $\bZ/p^n$ there is a natural object $(\Omega^{\DHod}_{X_{n-1}/\bZ/p^n},\Theta_{X_{n-1}})\in D_{\bN}(X_{n-1})$ equipped with a filtration $\Fil^{\conj}_{\bullet}$ with graded quotients equivalent to $(L\Omega^i_{X_{n-1}/\bZ/p^n}[-i],-i)$. The object $\Omega^{\DHod}_{X_{n-1}/\bZ/p^n}\in D(X_{n-1})$ has a natural structure of a filtered derived commutative algebra such that $\Theta_{X_{n-1}}$ is a derivation.

The base change $\Omega^{\DHod}_{X_{n-1}/\bZ/p^n}\otimes_{\cO_{X_{n-1}}}\cO_{X_0}$ is identified with $\dR_{X_0}$, and the base change of the conjugate filtration matches the conjugate filtration on $\dR_{X_0}$. If $X$ is a quasisyntomic formal scheme over $\bZ_p$ then for $X_{n-1}:=X\times_{\bZ_p}\bZ_p/p^n$ we have $\Omega^{\DHod}_{X_{n-1}/\bZ/p^n}\simeq \Omega^{\DHod}_{X}\otimes_{\cO_X}\cO_{X_{n-1}}$ compatibly with the conjugate filtrations and the Sen operators.
\end{thm}

Given a quasisyntomic $\bZ/p^2$-scheme $X_1$ with special fiber $X_0=X_1\times_{\bZ/p^2}\bF_p\xhookrightarrow{i}X_1$, the de Rham complex $\dR_{X_0}\simeq i^*\Omega^{\DHod}_{X_1/\bZ/p^2}$, in particular, gets equipped with a Sen operator $\Theta_{X_1}$. Generalized eigenspaces for $\Theta_{X_1}$ give a decomposition (cf. \cite[Remark 5.16]{bhatt-lurie-prismatization})
\begin{equation}\label{sen operator: drinfeld decomposition}
\dR_{X_0}\simeq\bigoplus\limits_{i=0}^{p-1}\dR_{X_0,i}.
\end{equation}
The object $\Fil^{\conj}_p\dR_{X_0,0}$ fits into a fiber sequence \begin{equation}\label{sen operator: filp0 mod p2 formula}\cO_{X_0}\to \Fil^{\conj}_p\dR_{X_0,0}\to L\Omega^p_{X_0}[-p]\end{equation} on which $\Theta_{X_1}$ naturally acts, inducing zero on the first and third terms. 

\begin{rem}\label{sen operator: drinfeld decomposition stacks}
Let us record for a future use that the decomposition (\ref{sen operator: drinfeld decomposition}) also exists for any smooth Artin stack $X_0$ over $k$, equipped with a flat lift over $W_2(k)$. Namely, for any quasisyntomic $W_2(k)$-algebra $S_1$ with mod $p$ reduction $S_0=S_1/p$ we have a natural decomposition $\RGamma_{\dR}(S_0/k)\simeq \bigoplus\limits_{i=0}^{p-1}\RGamma_{\dR}(S_0/k)_i$ characterized by the fact that $\theta_{S_1}+i$ is nilpotent on $\RGamma_{\dR}(S_0/k)_i$ where $\theta_{S_1}$ is the Sen operator arising from $S_1$. Moreover, each $\RGamma_{\dR}(S_0/k)_i$ satisfies flat descent, so as in \cite[Construction 2.7]{antieau-bhatt-mathew} we can extend $\RGamma_{\dR}(-/k)_i$ to a functor from the category of syntomic stacks over $W_2(k)$ to $D(k)$, such that there is a natural equivalence
\begin{equation}
\RGamma_{\dR}(X_0/k)\simeq \bigoplus\limits_{i=0}^{p-1} \RGamma_{\dR}(X_0/k)_i
\end{equation}
for any syntomic stack $X_0$ over $k$ equipped with a syntomic $W_2(k)$-lift $X_1$.

In particular, if $X_0$ is a smooth Artin stack over $k$ that can be lifted over $W_2(k)$, the conjugate spectral sequence for $X_0$ has no non-zero differentials on pages $E_2,\ldots,E_p$. The fact that there are no differentials on these pages touching any of the entries $H^i(X_0,\Omega^j_{X_0})$ with $i+j<p$ was proven in \cite[Theorem 1.3.23]{kubrak-prikhodko-di}.
\end{rem}

\begin{notation}\label{sen operator: zp2 classes notation}
We denote by $e_{X_1,p}:L\Omega^p_{X_0}\to\cO_{X_0}[p+1]$ the connecting map corresponding to the fiber sequence (\ref{sen operator: filp0 mod p2 formula}), and by $c_{X_1,p}:L\Omega^p_{X_0}\to \cO_{X_0}[p]$ the map induced by the nilpotent operator $\Theta_{X_1}$ on $\Fil_p^{\conj}\dR_{X_0,0}$.
\end{notation}

This is consistent with Notation \ref{sen operator: zp classes notation} in the sense that for a quasisyntomic formal $\bZ_p$-scheme $X$ with mod $p^2$ reduction $X_1=X\times_{\bZ_p}\bZ/p^2$ the map $c_{X_1,p}$ is naturally homotopic to $c_{X,p}$, and $e_{X_1,p}$ is the mod $p$ reduction of $e_{X,p}$.

We can also consider the endomorphism $\Theta_{X_1}-\Theta_{X_1}^p$ of $\Fil_p^{\conj}\dR_{X_0}$ which acts by zero on $\Fil_{p-1}^{\conj}$ and hence naturally induces a map $\Theta_{X_1}-\Theta_{X_1}^p:L\Omega^p_{X_0}[-p]\to\Fil_{p-1}^{\conj}\dR_{X_0}$. As in the previous discussion in the presence of a lift of $X_0$ over $W(k)$, this map, $c_{X_1,p}$, and $e_{X_1,p}$ are related as follows:

\begin{lm}\label{sen operator: classes over zp2}
There is a natural homotopy \begin{equation}e_{X_1,p}\sim\Bock_{\RHom_{D(X_1)}(L\Omega^p_{X_1},\cO_{X_1}[p])}(c_{X_1,p})\end{equation} where \begin{equation}\Bock_{\RHom_{D(X_1)}(L\Omega^p_{X_1},\cO_{X_1}[p])}:\RHom_{D(X_0)}(L\Omega^p_{X_0},\cO_{X_0}[p])\to \RHom_{D(X_0)}(L\Omega^p_{X_0},\cO_{X_0}[p+1])\end{equation} is the Bockstein homomorphism induced by $\RHom_{D(X_1)}(L\Omega^p_{X_1},\cO_{X_1}[p])\in D(W_2(k))$. The map $\Theta_{X_1}-\Theta_{X_1}^p:L\Omega^p_{X_0}[-p]\to \Fil_{p-1}^{\conj}\dR_{X_0}$ is naturally homotopic to the composition $L\Omega^p_{X_0}[-p]\xrightarrow{c_{X_1,p}}\cO_{X_0}\to\Fil_{p-1}^{\conj}\dR_{X_0}$.
\end{lm}

We will now define diffracted Hodge cohomology relative to $\bZ/p^n$, proving Theorem \ref{sen operator: zpn main}, and will then prove Lemma \ref{sen operator: classes over zp2}. As in \cite{apc}, the construction of diffracted Hodge cohomology over $\bZ/p^n$ together with its Sen operator arises from the computation of the Cartier-Witt stack of the base ring $\bZ/p^n$. I learned the following fact from Sanath Devalapurkar:

\begin{lm}\label{sen operator: wcartht zpn}
There is an isomorphism of stacks $\WCart^{\mathrm{HT}}_{\bZ/p^n}\simeq \bG_{a,\bZ/p^n}^{\sharp}/\bG_{m,\bZ/p^n}^{\sharp}$ where the quotient is taken with respect to the scaling action. Moreover, the natural map $\WCart_{\bZ/p^n}^{\HT}\to\WCart_{\bZ_p}^{\HT}$ is identified with $\bG_{a,\bZ/p^n}^{\sharp}/\bG_{m,\bZ/p^n}^{\sharp}\to B\bG_{m,\bZ/p^n}^{\sharp}\to B\bG^{\sharp}_{m,\bZ_p}$.
\end{lm}

\begin{proof}
We denote by $\bG_a^{\sharp}$ the divided power envelope of $0$ in $\bG_a$, and by $\bG_m^{\sharp}$ the divided power envelope of $1$ in $\bG_m$, both viewed as group schemes over $\bZ_p$. In this proof we will use repeatedly the identifications of group schemes $\bG_m^{\sharp}\simeq W^{\times}[F]:=W^{F=1}$ and $\bG_a^{\sharp}\simeq W[F]:=W^{F=0}$ proven in \cite[Lemmma 3.4.11, Variant 3.4.12]{apc} and \cite[Lemma 3.2.6]{drinfeld}, together with the fact that the multiplication action of $W^{\times }[F]$ on $W[F]$ corresponds to the usual scaling action of $\bG_m^{\sharp}$ on $\bG_a^{\sharp}$ under these identifications.

By \cite[Construction 3.8]{bhatt-lurie-prismatization} the stack $\WCart^{\HT}_{\bZ/p^n}$ is the quotient of $(\Spec \bZ/p^n)^{\DHod}$ by $\bG_{m,\bZ/p^n}^{\sharp}$ where for a test algebra $S$ in which $p$ is nilpotent we have $(\Spec\bZ/p^n)^{\DHod}(S)=\Map(\bZ/p^n,W(S)/V(1))$ where the mapping space is taken in the category of animated commutative rings. As in \cite[Example 5.15]{bhatt-lurie-prismatization} we can rewrite this more explicitly as $(\Spec\bZ/p^n)^{\DHod}(S)\simeq \{x\in W(S)| V(1)x=p^n\}$ with the $\bG_m^{\sharp}(S)=W(S)^{F=1}$-action given by multiplication on $x$. The key computation that will allow us to identify this set with $\bG_a^{\sharp}(S)$ in a $\bG_m^{\sharp}(S)$-equivariant fashion is the following:

\begin{lm}
We have $p^n=V(p^{n-1})$ in $W(\bZ/p^n)$.
\end{lm}

\begin{proof}
Recall that for any ring $R$ there is the ghost map $W(R)\to R^{\bN}$ sending a Witt vector $[x_0]+V[x_1]+V^2[x_2]+\ldots$ to $(x_0,x_0^p+px_1,x_0^{p^2}+px_1^p+p^2x_2,\ldots)$. This is a map of rings and it is injective if $R$ is $p$-torsion-free. For an element $a\in W(R)$ with ghost coordinates $(a_0,a_1,\ldots)$ the ghost coordinates of $F(a)$ are given by $(a_1,a_2,\ldots)$, and the ghost coordinates of $V(a)$ are $(0, pa_0,pa_1,\ldots)$.

Therefore the ghost coordinates of $V(p^{n-1})\in W(\bZ_p)$ are $(0,p^n,p^n,\ldots)$ and the ghost coordinates of $p^n-V(p^{n-1})$ are equal to $(p^n,0,0,\ldots)$.  The isomorphism $\bG_a^{\sharp}\simeq W[F]$ composed with the ghost map sends $r\in \bG_a^{\sharp}$ to $(r,0,0,\ldots)$, hence $p^n-V(p^{n-1})\in W(\bZ_p)$ is the image of $p^n\in \bG_a^{\sharp}(\bZ_p)$ under the natural map $\bG_a^{\sharp}\simeq W[F]\subset W$. Since $p^n$ is annihilated by the map $\bG_a^{\sharp}(\bZ_p)\to \bG_a^{\sharp}(\bZ/p^n)$ induced by $\bZ_p\twoheadrightarrow \bZ/p^n$, the element $p^n-V(p^{n-1})$ is zero in $W(\bZ/p^n)$, as desired.
\end{proof}
We can now define the isomorphism $\bG^{\sharp}_{a,\bZ/p^n}\simeq (\Spec \bZ/p^n)^{\DHod}$ of stacks over $\bZ/p^n$. For a test $\bZ/p^n$-algebra $S$ the Verschiebung map on $W(S)$ is injective, so an element $y\in W(S)$ is annihilated by $F$ if and only if $VF(y)=V(1)\cdot y$ vanishes. As we just computed, $V(1)\cdot V(p^{n-2})=VFV(p^{n-2})=V(p^{n-1})=p^n$ in $W(S)$. Therefore we can define a bijection $\bG_a^{\sharp}(S)\simeq (\Spec \bZ/p^n)^{\DHod}$ by
\begin{multline}
\bG^{\sharp}_{a,\bZ/p^n}(S)=\{y\in W(S)|Fy=0\}=\{y\in W(S)|V(1)\cdot y=0\}\ni y\mapsto \\ \mapsto y+V(p^{n-2})\in \{x\in W(S)|V(1)\cdot x=p^n\}=(\Spec\bZ/p^n)^{\DHod}(S)
\end{multline}

The resulting isomorphism intertwines the $\bG_{m}^{\sharp}$-action on $(\Spec \bZ/p^n)^{\DHod}$ with the usual scaling action on $\bG_a^{\sharp}$ because for $a\in \bG_m^{\sharp}(S)=W(S)^{F=1}$ we have $a\cdot(y+V(p^{n-2}))=a\cdot y+V(F(a)p^{n-2})=a\cdot y+V(p^{n-2})$.
\end{proof}

We can now define diffracted Hodge cohomology relative to $\bZ/p^n$. Given a flat scheme $X_{n-1}$ over $\bZ/p^n$, the stack $\WCart_{X_{n-1}}^{\HT}$ lives over $\WCart_{\bZ/p^n}^{\HT}$ and we define $(X_{n-1}/\bZ/p^n)^{\DHod}$ as the fiber product

\begin{equation}
\begin{tikzcd}
(X_{n-1}/\bZ/p^n)^{\DHod}\arrow[r]\arrow[d] & \WCart^{\HT}_{X_{n-1}} \arrow[d] \\
\Spec(\bZ/p^n)\arrow[r,"\eta"] & \WCart^{\HT}_{\bZ/p^n} 
\end{tikzcd}
\end{equation}
where $\eta:\Spec(\bZ/p^n)\to \WCart^{\HT}_{\bZ/p^n}\simeq \bG_{a,\bZ/p^n}^{\sharp}/\bG_{m,\bZ/p^n}^{\sharp}$ is the composition $\Spec(\bZ/p^n)\xrightarrow{0}\bG_{a,\bZ/p^n}^{\sharp}\to \bG_{a,\bZ/p^n}^{\sharp}/\bG^{\sharp}_{m,\bZ/p^n}$. The stack $(X_{n-1}/\bZ/p^n)^{\DHod}$ is equipped with a map $\pi^{\DHod}_{X_{n-1}/\bZ/p^n}:(X_{n-1}/\bZ/p^n)^{\DHod}\to X_{n-1}$ obtained as the composition $(X_{n-1}/\bZ/p^n)^{\DHod}\to \WCart^{\HT}_{X_{n-1}}\to X_{n-1}$. Define the diffracted Hodge cohomology of $X_{n-1}$ relative to $\bZ/p^n$ as the pushforward of the structure sheaf along the map $\pi^{\DHod}_{X_{n-1}/\bZ/p^n}$:
\begin{equation}\Omega^{\DHod}_{X_{n-1}/\bZ/p^n}:=R\pi_{X_{n-1}/\bZ/p^n *}\cO_{(X_{n-1}/\bZ/p^n)^{\DHod}}\in D(X_{n-1})\end{equation}  We will now prove that this object shares the basic properties of diffracted Hodge cohomology of $\bZ_p$-schemes:

\begin{proof}[Proof of Theorem \ref{sen operator: zpn main}]
Since, by construction, the stack $(X_{n-1}/\bZ/p^n)^{\DHod}$ is equipped with an action of $\bG_{m,\bZ/p^n}^{\sharp}$ such that the map $\pi^{\DHod}_{X_{n-1}/\bZ/p^n}:(X_{n-1}/\bZ/p^n)^{\DHod}\to X_{n-1}$ is $\bG_{m,\bZ_p}^{\sharp}$-equivariant for the trivial action of $\bG_{m,\bZ_p}^{\sharp}$ on the target, the object $\Omega^{\DHod}_{X_{n-1}/\bZ/p^n}\in D(X_{n-1})$ is naturally equipped with a $\bG_{m,\bZ_p}^{\sharp}$-action. As in \cite[Theorem 3.5.8]{apc}, this gives rise to an endomorphism $\Theta_{X_{n-1}}$ of $\Omega^{\DHod}_{X_{n-1}/\bZ/p^n}$. We endow $\Omega^{\DHod}_{X_{n-1}/\bZ/p^n}$ with the structure of a derived commutative algebra via Lemma \ref{dalg: pushforward of rings}, and $\Theta_{X_{n-1}}$ is seen to be a derivation (Definition \ref{dalg: derivation}) by following \cite[Construction 3.5.4]{apc}.

We will now construct the conjugate filtration on $\Omega^{\DHod}_{X_{n-1}/\bZ/p^n}$ and identify its graded quotients. To do this we will compare $\Omega_{X_{n-1}/\bZ/p^n}^{\DHod}$ with the `absolute' diffracted Hodge cohomology $\Omega_{X_{n-1}}^{\DHod}$. By \cite[Construction 4.7.1]{apc} the object $\Omega^{\DHod}_{X_{n-1}}$ is equipped with a filtration $\Fil_n^{\conj}\Omega^{\DHod}_{X_{n-1}}$ with graded quotients $\gr_n^{\conj}\simeq L\Omega^n_{X_{n-1}/\bZ_p}[-n]$. In the case $X_{n-1}=\Spec(\bZ/p^n)$ we have an equivalence $\Omega^{\DHod}_{\bZ/p^n}\simeq\cO(\bG_{a,\bZ/p^n}^{\sharp})$ of derived commutative $\bZ/p^n$-algebras by Lemma \ref{sen operator: wcartht zpn}.

By base change for the diagram
\begin{equation}
\begin{tikzcd}
(X_{n-1}/\bZ/p^n)^{\DHod}\arrow[r]\arrow[d] & (X_{n-1})^{\DHod} \arrow[d] \\
X_{n-1}\arrow[r,"\eta"] & (\Spec \bZ/p^n)^{\DHod}\times_{\bZ/p^n}X_{n-1} 
\end{tikzcd}
\end{equation}
we have \begin{equation}\Omega^{\DHod}_{X_{n-1}/\bZ/p^n}\simeq \bZ/p^n\otimes_{\cO(\bG_{a,\bZ/p^n}^{\sharp})}\Omega^{\DHod}_{X_{n-1}}\end{equation} 

The ordinary commutative algebra $\cO(\bG_{a,\bZ/p^n}^{\sharp})$ is endowed with a conjugate filtration via the identification $\cO(\bG_{a,\bZ/p^n}^{\sharp})\simeq \Omega^{\DHod}_{\bZ/p^n}$, and the associated graded algebra is the free divided power algebra on the $\bZ/p^n$-module $L\Omega^1_{\bZ/p^n/\bZ_p}[-1]\simeq \bZ/p^n$. Moreover, the conjugate filtration on $\Omega^{\DHod}_{X_{n-1}}$ makes it into an object of the derived category of filtered $\cO(\bG_{a,\bZ/p^n}^{\sharp})$-modules in $D(X_{n-1})$. Equipping $\bZ/p^n$ with the trivial filtration $\Fil_0^{\conj}\bZ/p^n=\bZ/p^n,\Fil_{-1}^{\conj}\bZ/p^n=0$, we get an induced tensor product filtration $\bZ/p^n\otimes_{\cO(\bG_{a,\bZ/p^n}^{\sharp})}\Omega^{\DHod}_{X_{n-1}}$, and hence on the complex $\Omega^{\DHod}_{X_{n-1}/\bZ/p^n}$.

We will first check that the filtered complex $\Omega^{\DHod}_{X_{n-1}/\bZ/p^n}\otimes_{\cO_{X_{n-1}}}\cO_{X_0}$ is equivalent to $\dR_{X_0}$, and that there is an equivalence of filtered complexes $\Omega^{\DHod}_{X_{n-1}/\bZ/p^n}\simeq \Omega^{\DHod}_{X}\otimes_{\cO_X}\cO_{X_{n-1}}$ compatible with the Sen operators if $X$ is a flat formal $\bZ_p$-scheme lifting $X_{n-1}$. Since we assume that $X_{n-1}$ is quasisyntomic over $\bZ/p^n$, the formal scheme $X$ is quasisyntomic over $\bZ_p$, and $\Omega^{\DHod}_X$ coincides with the derived pushforward of the structure sheaf along the map of stacks $X^{\DHod}\to X$, by specializing to the Hodge-Tate locus the isomorphism of \cite[Theorem 7.20(2)]{bhatt-lurie-prismatization}. We have the following relations between the relevant Cartier-Witt stacks:
\begin{multline}
\WCart^{\HT}_{X_{n-1}}\times_{\WCart^{\HT}_{\bZ/p^n}} \WCart^{\HT}_{\bF_p}\simeq \WCart^{\HT}_{X_0}  \\ \WCart^{\HT}_{X}\times_{\WCart^{\HT}_{\bZ_p}}\WCart^{\HT}_{\bZ/p^{n}}\simeq \WCart^{\HT}_{X_{n-1}}
\end{multline}

These follows directly from the definition of $\WCart^{\HT}_{(-)}$, as in \cite[Remark 3.5]{bhatt-lurie-prismatization}. The identification $\dR_{X_0}\simeq \Omega^{\DHod}_{X_{n-1}/\bZ/p^n}\otimes_{\bZ/p^n}\bF_p$ follows from the fact that $\WCart^{\HT}_{\bF_p}\simeq \Spec\bF_p$ and the map $\WCart^{\HT}_{\bF_p}\to \WCart^{\HT}_{\bZ/p^n}$ factors through the map $\eta$. The identification $\Omega^{\DHod}_{X}/p^n\simeq \Omega^{\DHod}_{X_{n-1}/\bZ/p^n}$ is implied by the fact that the composition $\Spec\bZ/p^n\xrightarrow{\eta} \WCart^{\HT}_{\bZ/p^n}\to \WCart^{\HT}_{\bZ_p}$ is equal to the composition $\Spec\bZ/p^n\to \Spf\bZ_p\to \WCart^{\HT}_{\bZ_p}$.

Finally, we will compute the quotients of the conjugate filtration on $\Omega^{\DHod}_{X_{n-1}/\bZ/p^n}$. The associated graded object $\gr_{\conj}^{\bullet}(\Omega^{\DHod}_{X_{n-1}/\bZ/p^n})\simeq \gr^{\bullet}_{\conj}(\bZ/p^n\otimes_{\cO(\bG_{a,\bZ/p^n}^{\sharp})}\Omega^{\DHod}_{X_{n-1}})$ is equivalent to $\bZ/p^n\otimes_{\bZ/p^n\langle t\rangle}\bigoplus\limits_{i\geq 0}L\Omega^{i}_{X_{n-1}/\bZ_p}[-i]$ where we identified $\gr^{\bullet}_{\conj}\cO(\bG_{a,\bZ/p^n}^{\sharp})$ with the divided power algebra $\bZ/p^n\langle t\rangle$ in one variable, such that $\Theta(t)=-t$.

We have a natural map $\gr_{\conj}^{\bullet}(\Omega^{\DHod}_{X_{n-1}/\bZ/p^n})\to \bigoplus\limits_{i\geq 0}(L\Omega^i_{X_{n-1}/\bZ/p^n}[-i],-i)$ of graded objects of $D_{\bN}(X_{n-1})$. To check that it is an equivalence we may assume that $X_{n-1}$ is isomorphic to the spectrum of a polynomial ring over $\bZ/p^n$, and the result follows because such $X_{n-1}$ lifts over $\bZ_p$.
\end{proof}

We will now establish analogs of Corollary \ref{sen operator: smooth classes w lift} and Lemmas \ref{sen operator: cXp as nilpotent operator}, \ref{sen operator: cXp as nilpotent operator on dR} over $\bZ/p^2$, thus proving Lemma \ref{sen operator: classes over zp2}. Given a flat scheme $X_1$ over $\bZ/p^2$ and two objects $M,N\in D(X_1)$, consider the complex of morphisms $\RHom_{D_{\bN}(X_1)}((M,p),(N,0))$, viewed as an object of $D(\bZ/p^2)$. The restriction along $i:X_0\hookrightarrow X_1$ induces the map \begin{multline}\RHom_{D_{\bN}(X_1)}((M,p),(N,0))\to \\ \RHom_{D_{\bN}(X_0)}((i^*M,0),(i^*N,0))\simeq \RHom_{X_0}(i^*M,i^*N)\oplus \RHom_{X_0}(i^*M,i^*N[-1]).\end{multline} We denote by $r:\RHom_{D_{\bN}(X_1)}((M,p),(N,0))\to \RHom_{X_0}(i^*M,i^*N[-1])$ the second component of this composition. Explicitly, the data of a 1-morphism $f:(M,p)\to (N,0)$ amounts to a map $f:M\to N$ in $D(X_1)$ and a homotopy between $p\cdot f$ and $0$. Restricting to $X_0$, this gives a homotopy from the zero morphism $i^*M\xrightarrow{0}i^*N$ to itself, which is equivalent to the data of a 1-morphism $i^*M\to i^*N[-1]$. By Lemma \ref{sen operator: morphisms in DN} this map is nothing but $r(f)$.
\begin{lm}
Let $X_1$ be a flat scheme over $\bZ/p^2$. Denote by $X_0=X_1\times_{\bZ/p^2}\bF_p\xhookrightarrow{i}X_1$ its special fiber. For any two objects $M,N\in D(X_1)$ the composition 
\begin{equation}
\RHom_{D_{\bN}(X_1)}((M,p),(N,0))\to \RHom_{D(X_1)}(M,N)\to \RHom_{D(X_0)}(i^*M,i^*N)
\end{equation}
where the first map is forgetting the endomorphisms and the second map is induced by $i$, can be identified with the composition 
\begin{equation}
\RHom_{D_{\bN}(X_1)}((M,p),(N,0))\xrightarrow{r} \RHom_{D(X_0)}(i^*M,i^*N[-1])\xrightarrow{\Bock}  \RHom_{D(X_0)}(i^*M,i^*N)
\end{equation}
where the second map is Bockstein morphism associated to the object $\RHom_{D(X_1)}(M,N[-1])$.
\end{lm}

\begin{proof}
By Lemma \ref{sen operator: morphisms in DN} we can identify $$\RHom_{D_{\bN}(X_1)}((M,p),(N,0))$$ with $\RHom_{D(X_1)}(M,N^{p=0})$, and the forgetful map to $\RHom_{D(X_1)}(M,N)$ is given by composing with the natural map $N^{p=0}\to N$. Moreover, we can identify $\RHom_{D(X_1)}(M,N^{p=0})$ with $\RHom_{D(X_1)}(M,N)^{p=0}$ so that the desired identification becomes a consequence of the following Lemma \ref{sen operator: p fiber reduction}, applied to $A=\RHom_{D(X_1)}(M,N)$.
\end{proof}

\begin{lm}\label{sen operator: p fiber reduction}
For any object $A\in D(\bZ/p^2)$ the composition $A^{p=0}\to A\to i_*i^*A$ is naturally identified with the composition 
\begin{equation}A^{p=0}\xrightarrow{r_A} i_*i^*A[-1]\xrightarrow{\Bock_{A[-1]}}i_*i^*A\end{equation} where the map $r_A$ is induced by the identification $i_*i^*(A^{p=0})=i_*(i^*A)^{p=0}\simeq i_*i^*A\oplus i_*i^*A[-1]$.
\end{lm}

\begin{proof}
Recall that $i_*i^*A\simeq A\otimes_{\bZ/p^2}\bF_p$ and there is a natural fiber sequence $i_*i^*A\xrightarrow{a_1}A\xrightarrow{a_2}i_*i^*A$ obtained by taking the tensor product of the sequence $\bF_p\to \bZ/p^2\to \bF_p$ with $A$. Consider the commutative diagram

\begin{equation}
\begin{tikzcd}
A \arrow[r, "p"]\arrow[d,"a_2"] & A\arrow[d,equal]\\
i_*i^*A\arrow[r,"a_1"] & A
\end{tikzcd}
\end{equation}

Taking the fibers of the horizontal arrows induces a commutative diagram
\begin{equation}\label{sen operator: p fiber reduction diagram}
\begin{tikzcd}
\fib(A\xrightarrow{p}A)\arrow[d]\arrow[r] & A\arrow[d]\\
\fib(i_*i^*A\xrightarrow{a_1} A)\arrow[r] & i_*i^*A. \\
\end{tikzcd}
\end{equation}

The induced map $A^{p=0}=\fib(A\xrightarrow{p}A)\to \fib(i_*i^*A\xrightarrow{a_1} A)\simeq i_*i^*A[-1]$ is the map $r_A$, and the map $i_*i^*A[-1]\simeq \fib(i_*i^*A\xrightarrow{a_1}A)\to i_*i^*A$ is $\Bock_{A[-1]}$, by definition. Hence the statement of the lemma amounts to commutativity of the diagram (\ref{sen operator: p fiber reduction diagram}).
\end{proof}

The upshot is that, given a quasisyntomic $\bZ/p^n$-scheme $X_{n-1}$ for $n\geq 2$, the Sen operator on $\Omega^{\DHod}_{X_{n-1}/\bZ/p^n}$ induces an operator on $\dR_{X_0}$, and all of the information about the restriction of the latter operator to $\Fil^{\conj}_p\dR_{X_0}$ is captured by the class $c_{X_{n-1},p}$. If $X_{n-1}$ lifts to a flat formal $\bZ_p$-scheme $X$, then this $c_{X_{n-1},p}$ also remembers the data of the Sen operator on $\Fil^{\conj}_p\Omega^{\DHod}_{X_{n-1}/\bZ/p^n}$, by Lemma \ref{sen operator: ext group description}. In the absence of a lift to $W(k)$ it is at the moment unclear to me whether the Sen operator on $\Fil^{\conj}_p\Omega^{\DHod}_{X_{n-1}/\bZ/p^n}$ contains more information than its mod $p$ reduction.

\section{Sen operator via descent from semiperfectoid rings}\label{semiperf: section}

In this section we prove Theorem \ref{semiperf: main sen operator} via descent from quasiregular semiperfectoid rings, this approach was suggested by Bhargav Bhatt. As a consequence, we extend Theorem \ref{cosimp applications: the best part} to the situation where $X_0$ is only liftable to $\bZ/p^2$ rather than $\bZ_p$. 

\begin{thm}\label{semiperf: main sen operator}
Let $X_1$ be a quasisyntomic scheme over $\bZ/p^2$ with special fiber $X_0=X_1\times_{\bZ/p^2}\bF_p$. The map $c_{X_1,p}:L\Omega^p_{X_0}\to\cO_{X_0}[p]$ arising from the Sen operator on the de Rham complex of $X_0$ is naturally homotopic to the composition
\begin{equation}
L\Omega^p_{X_0}\xrightarrow{\alpha(L\Omega^1_{X_0})}F_{X_0}^*L\Omega^1_{X_0}[p-1]\xrightarrow{\ob_{F,X_1}}\cO_{X_0}[p].
\end{equation}
\end{thm}

The particular interpretation of $c_{X_1,p}$ that will be used in the proof is that the map $\Theta_{X_1}-\Theta_{X_1}^p:L\Omega^p_{X_0}[-p]\to\Fil_{p-1}^{\conj}\dR_{X_0}\simeq \bigoplus\limits_{i=0}^{p-1}L\Omega^i_{X_0}[-p]$ factors as $L\Omega^p_{X_0}[-p]\xrightarrow{c_{X_1,p}}\cO_{X_0}\xrightarrow{\oplus}\Fil_{p-1}^{\conj}\dR_{X_0}$, as we established in Lemma \ref{sen operator: classes over zp2}.

We start by recalling the description of diffracted Hodge cohomology of quasiregular semiperfectoid algebras. Let $S$ be a quasiregular semiperfectoid flat $\bZ/p^n\bZ$-algebra, as defined in \cite[Definition 4.20]{bms2}. The cotangent complex $L\Omega^1_{S/\bZ/p^n}$ is concentrated in degree $(-1)$ and $H^{-1}(L\Omega^1_{S/\bZ/p^n})$ is a flat $S$-module. For brevity, we denote $H^{-1}(L\Omega^1_{S/\bZ/p^n})$ by $M_S$. The objects $L\Omega^i_{S/\bZ/p^n}[-i]=\Lambda^i(H^{-1}(L\Omega^1_{S/\bZ/p^n})[1])[-i]\simeq \Gamma^iM_S$ are also flat $S$-modules placed in degree zero. Since $\Omega^{\DHod}_{S/\bZ/p^n}$ admits an exhaustive filtration with graded pieces given by $L\Omega^i_{S/\bZ/p^n}[-i]$, the diffracted Hodge cohomology complex $\Omega^{\DHod}_{S/\bZ/p^n}$ is a commutative flat $S$-algebra concentrated in degree $0$.

We denote by $S_0:=S/p$ the mod $p$ reduction of $S$, and by $\Omega^{\DHod}_{S_0}\simeq \Omega^{\DHod}_{S/\bZ/p^n}\otimes_S S_0$ the diffracted Hodge cohomology of $S_0$ which coincides with the derived de Rham cohomology $\dR_{S_0/\bF_p}$. Thanks to the fact that $\Omega^{\DHod}_{S/\bZ/p^n}$ is concentrated in degree $0$, the Sen operator $\Theta_S:\Omega^{\DHod}_{S/\bZ/p^n}\to \Omega^{\DHod}_{S/\bZ/p^n}$ is a genuine endomorphism of a discrete $S$-algebra, and we may use tools from ordinary algebra rather than higher algebra to study it. The reader is encouraged to view the next result as an analog of Theorem \ref{cosimp: main theorem}.

\newcommand{\ODH}{\Omega^{\DHod}}
\begin{pr}\label{semiperf: sen for qrsprfd}
Let $S$ be a quasiregular semiperfectoid flat $\bZ/p^2$-algebra. Suppose that $f:\Omega^{\DHod}_{S/\bZ/p^2}\to\Omega^{\DHod}_{S/\bZ/p^2}$ is a derivation of $S$-algebras that preserves the conjugate filtration, acting on $\gr_i^{\conj}\Omega^{\DHod}_{S/\bZ/p^2}\simeq \Gamma^i M_S$ by $-i$. Denote by $s:M_S\to \Fil_1^{\conj}\ODH_{S/\bZ/p^2}$ the splitting of the conjugate filtration given by $M_S\simeq \ker(f+1:\Fil^{\conj}_1\ODH_{S/\bZ/p^2})\subset \Fil^{\conj}_1\ODH_{S/\bZ/p^2}$. Then the map $f-f^p:\Fil^{\conj}_p\ODH_{S_0}\to \Fil^{\conj}_p\ODH_{S_0}$ on the mod $p$ reduction of $\Fil_p^{\conj}\Omega^{\DHod}_{S/\bZ/p^2}$ factors as 

\begin{equation}\label{semiperf: sen for qrsprfd equation}
\Fil^{\conj}_p\ODH_{S_0}\twoheadrightarrow{} \Gamma^p_{S_0}(M_S/p)\twoheadrightarrow{}F_{S_0}^*(M_S/p)\xrightarrow{F_{S_0}^*s}F_{S_0}^*\Fil^{\conj}_1\ODH_{S_0}\xrightarrow{\varphi_{S_0}}\Fil^{\conj}_1\ODH_{S_0}\hookrightarrow \Fil^{\conj}_p\ODH_{S_0}
\end{equation}
where the surjection $\Gamma^p_{S_0}(M_S/p)\twoheadrightarrow{}F_{S_0}^*(M_S/p)$ is the map $\psi_{M_S/p}$ defined in (\ref{dalg: polynomial frobenius}).
\end{pr}

\begin{proof}
\newcommand{\OpS}{\Omega^{\DHod}_{S/\bZ/p^2}}
\newcommand{\OpN}{\Omega^{\DHod}_{S_0}}
To prove the proposition it is enough to check that $f-f^p$ coincides with the composition (\ref{semiperf: sen for qrsprfd equation}) on elements $y\in \Fil^{\conj}_p\Omega^{\DHod}_{S_0}$ whose image in $\gr^{\conj}_p\simeq \Gamma_{S_0}^p(M_S/p)$ has the form $x^{[p]}$ for some $x\in M_S/p$. Indeed, the $S_0$-module $\Fil^{\conj}_p\Omega^{\DHod}_{S_0}$ is spanned by elements of this form and the submodule $\Fil^{\conj}_{p-1}\Omega^{\DHod}_{S_0}$, and both $f-f^p$ and the composition (\ref{semiperf: sen for qrsprfd equation}) are identically zero on $\Fil^{\conj}_{p-1}\Omega^{\DHod}_{S_0}$. The composition (\ref{semiperf: sen for qrsprfd equation}) takes the element $y$ to $s(x)^p$, and we will prove that $f-f^p$ does the same.

Consider the endomorphism of $\Fil^{\conj}_p\OpS$ given by $F:=(-1)^p\prod\limits_{i=0}^{p-1}(f+i)$. It reduces modulo $p$ to the map $f-f^p:\Fil^{\conj}_p\ODH_{S_0}\to \Fil^{\conj}_p\ODH_{S_0}$. Note also that $F$ annihilates $\Fil^{\conj}_{p-1}\ODH_{S}$. Let $\ty\in \Fil^{\conj}_p\Omega^{\DHod}_{S/\bZ/p^2}$ be a lift of $y$ such that the image of $\ty$ in $\gr^{\conj}_p\OpS\simeq\Gamma^p_S(M_S)$ has the form $\tx^{[p]}$ for some $\tx\in M_S$ lifting $x\in M_S/p$. Then the image of $p!\cdot \ty$ in $\gr^{\conj}_p\OpS$ is equal to $\tx^p$, hence $p!\cdot\ty-s(\tx)^p$ lies in the submodule $\Fil^{\conj}_{p-1}\OpS\subset \Fil^{\conj}_p\OpS$. Therefore $F(p!\cdot \ty-s(\tx)^p)=0$. 
By definition of the section $s$, we have $f(s(\tx))=-s(\tx)$. Since $f$ is a derivation, $f(s(\tx)^p)=-p\cdot s(\tx)^p$, and $F(s(\tx)^p)=(-1)^p\prod\limits_{i=0}^{p-1}(-p+i)\cdot s(\tx)^p$. Hence we get the following equation on the element $F(\ty)$:

\begin{equation}\label{semiperf: main prop equation1}
p!F(\ty)=\prod\limits_{i=0}^{p-1}(p-i)\cdot s(\tx)^p.
\end{equation}

We now use crucially that $\OpS$ is flat over $\bZ/p^2$: we can cancel out $p!$ on both sides of (\ref{semiperf: main prop equation1}) to get $F(\ty)=s(\tx)^p+p\cdot a$ for some $a\in \Fil^{\conj}_p\OpS$. Reducing modulo $p$ gives the desired equality \begin{equation}\label{semiperf: main prop equation2}
(f-f^p)(y)=s(x)^p.\end{equation}
\end{proof}

We will now deduce a computation of the Sen operator for an arbitrary quasisyntomic $\bZ/p^2$-scheme using descent.

\begin{proof}[Proof of Theorem \ref{semiperf: main sen operator}]
We will prove that $c_{X_1,p}$ is equivalent to the composition 

\begin{equation}\label{semiperf: desired sen formula}
L\Omega^p_{X_0}\xrightarrow{\alpha(L\Omega^1_{X_0})}F_{X_0}^*L\Omega^1_{X_0}[p-1]\xrightarrow{F_{X_0}^*s[p-1]}F_{X_0}^*\dR_{X_0}[p]\to F_{X_0}^*F_{X_0*}\cO_{X_0}[p]\to \cO_{X_0}[p].
\end{equation}

To see that this is equivalent to the desired statement note that the composition $F_{X_0}^*L\Omega^1_{X_0}\xrightarrow{F_{X_0}^*s[1]}F_{X_0}^*\dR_{X_0}[1]\to F_{X_0}^*F_{X_0*}\cO_{X_0}[1]\to \cO_{X_0}[1] $ is homotopic to the obstruction class $\ob_{F,X_1}:F_{X_0}^*L\Omega^1_{X_0}\to\cO_{X_0}[1]$. In Proposition \ref{applications: frobenius lift obstruction prop} we produced a homotopy between these maps composed with $\cO_{X_0}[1]\to \Fil_1^{\conj}\dR_{X_0}[1]$ but the choice of the lift $X_1$ splits the latter map. 

Denote by $d_{X_1,p}$ the composition of maps in (\ref{semiperf: desired sen formula}). We denote by $\QSyn_{\bZ/p^2}$ the category of schemes quasisyntomic over $\bZ/p^2$, by $\AffQSyn_{\bZ/p^2}\subset \QSyn_{\bZ/p^2}$ the full subcategory of affine quasisyntomic schemes, and by $\qrsprfd_{\bZ/p^2}\subset\AffQSyn_{\bZ/p^2}$ the opposite of the category of quasiregular semiperfectoid $\bZ/p^2$-algebras.

Proposition \ref{semiperf: sen for qrsprfd} established that $c_{X_1,p}=d_{X_1,p}$ when $X_1=\Spec S$ is the spectrum of a quasiregular semiperfectoid $\bZ/p^2$-algebra, because the map $\Gamma^p_{S_0}(M_S/p)\to F_{S_0}^*(M_S/p)$ is the shift by $[-p]$ of the map $L\Omega^p_{S_0}\xrightarrow{\alpha(L\Omega^1_{S_0})}F^*_{S_0}L\Omega^1_{S_0}[p-1]$. The case of an arbitrary $X_1\in \Sch_{\bZ/p^2}$ will follow by a descent argument in the spirit of \cite{bms2}, cf. \cite[Proposition 10.3.1]{blm}, \cite[Proposition 4.4]{li-mondal}, and \cite[1.3]{kubrak-prikhodko-di} for arguments of similar type. 

Let $\QC_{\obj}$ be the $\infty$-category of $\infty$-categories equipped with a distinguished object, as defined in \cite[Tag 020S]{kerodon}. Its objects are pairs $(C,\cC)$ where $\cC$ is a $\infty$-category and $C$ is an object of $\cC$, and $1$-morphisms from $(C,\cC)$ to $(D,\cD)$ are pairs $(F:\cC\to \cD,\alpha: F(C)\to D)$ where $F$ is a functor and $\alpha$ is a $1$-morphism in $\cD$. 

Assigning to a scheme $X_1\in \QSyn_{\bZ/p^2}$ the pair $(L\Omega^i_{X_0/\bF_p},D(X_0))$ defines a functor $L\Omega^i:\QSyn_{\bZ/p^2}^{\op}\to \QC_{\obj}$, where a morphism $f:X_1\to Y_1$ is sent to $f_0^*:D(Y_0)\to D(X_0)$ and $\Lambda^i df_{0}:f_0^*L\Omega^i_{Y_0}\to L\Omega^i_{X_0}$ (this functor factors through $\QSyn_{\bF_p}$). For $i=0$ we denote this functor simply by $\cO$. The formation of classes $c_{X_1,p}$ and $d_{X_1,p}$ defines morphisms from $L\Omega^p[-p]$ to $\cO$. Denote by $\Func^0(\QSyn^{\op}_{\bZ/p^2},\QC_{\obj})$ the category of functors $F$ equipped with an equivalence between the composition of $F$ with the forgetful functor $\QC_{\obj}\to \Cat_{\infty}$ and the functor $X_1\mapsto D(X_0)$.

\begin{lm}
For any $i,j$ the restriction induces an equivalence 
\begin{equation}\label{semiperf: map equation}
\Map_{\Func^0(\QSyn^{\op}_{\bZ/p^2},\QC_{\obj})}(L\Omega^i[-i],L\Omega^j[-j])\to \Map_{\Func^0(\qrsprfd^{\op}_{\bZ/p^2},\QC_{\obj})}(L\Omega^i[-i],L\Omega^j[-j]).\end{equation}
\end{lm}

\begin{proof}
Affine schemes corresponding to quasiregular semiperfectoid rings form a basis in the flat topology on the category of all quasisyntomic $\bZ/p^2$-schemes by \cite[Lemma 4.28]{bms2} which implies the result by flat descent for the exterior powers of the cotangent complex.
\end{proof}

In general, specifying a natural transformation between functors $F_1,F_2$ into an $\infty$-category requires (among further higher homotopies) specifying maps $a_{X_1}:F_1(X_1)\to F_2(X_1)$ for all objects $X_1$ together with the additional data of homotopies between $F_2(f)\circ a_{X_1}$ and $a_{Y_1}\circ F_1(f)$ for every map $f:X_1\to Y_1$. Crucially, if $X=\Spec S$ is the spectrum of a quasiregular semiperfectoid $\bZ/p^2$-algebra, then the mapping space $\Map_{D(X_0)}(L\Omega^p_{X_0}[-p],\cO_{X_0})$ is discrete because both $L\Omega^p_{X_0}[-p]$ and $\cO_{X_0}$ are concentrated in degree zero. Therefore the target of the map (\ref{semiperf: map equation}) is a discrete space and an element of it is completely determined by its values on the objects: functoriality with respect to morphisms in $\qrsprfd_{\bZ/p^2}$ amounts to checking a condition rather than specifying additional structure.

Therefore it is enough to compare the values of  $c_{X_1,p}$ and $d_{X_1,p}$ on objects of $\qrsprfd_{\bZ/p^2}$, and the corollary is proven. 
\end{proof}

We can now generalize the results of Section \ref{applications: section} on extensions in the conjugate filtration on the de Rham complex to schemes over $k$ that only admit a lift to $W_2(k)$ and are not necessarily smooth.

\begin{cor}
For a quasisyntomic scheme $X_0$ over $\bF_p$ equipped with a flat lift $X_1$ over $\bZ/p^2$ the class $e_{X_1,p}\in\RHom_{D(X_1)}(L\Omega^p_{X_0},\cO_{X_0}[p+1])$ of the extension $\bigoplus\limits_{i=0}^{p-1} L\Omega^i_{X_0}[-i]\to\Fil_p^{\conj}\dR_{X_0}\to L\Omega^p_{X_0}[-p]$ is naturally homotopic to $\Bock_{X_1}(\ob_{F,X_1}\circ \alpha(L\Omega^1_{X_0}))$.
\end{cor}

\begin{proof}
This is Lemma \ref{sen operator: classes over zp2} combined with Theorem \ref{semiperf: main sen operator}.
\end{proof}
 
For convenience of applications, let us also explicitly state this result on the level of cohomology classes in the case of smooth varieties:

\begin{cor}\label{semiperf: smooth cor}
For a smooth scheme $X_0$ over $k$ equipped with a smooth lift $X_1$ over $W_2(k)$ the class $c_{X_1,p}\in H^p(X_0,\Lambda^p T_{X_0/k})$ is equal to $\ob_{F,X_1}\cup \alpha(\Omega^1_{X_0/k})$, and $e_{X_1,p}\in H^{p+1}(X_0,\Lambda^pT_{X_0/k})$ is equal to $\Bock_{X_1}(\ob_{F,X_1}\cup \alpha(\Omega^1_{X_0}))$.
\end{cor}

Hence the first potentially non-trivial differentials in the conjugate spectral sequence of a liftable scheme are described as follows:

\begin{cor}
For a smooth scheme $X_0$ over $k$ equipped with a lift $X_1$ over $W_2(k)$ the conjugate spectral sequence $E_{ij}^2=H^i(X^{(1)}_0,\Omega^j_{X^{(1)}_0})\Rightarrow H^{i+j}_{\dR}(X_0/k)$ has no non-zero differentials on pages $E_2,\ldots, E_{p}$. The differentials $d^{i,p}_{p+1}:H^i(X^{(1)}_0,\Omega^p_{X^{(1)}_0})\to H^{i+p+1}(X^{(1)}_0,\cO)$ on page $E_{p+1}$ can be described as

\begin{equation}
\Bock_{\cO_{X_1^{(1)}}}\circ (\ob_{F,X_1}\cup\alpha(\Omega^1_{X^{(1)}_0}))-(\ob_{F,X_1}\cup\alpha(\Omega^1_{X^{(1)}_0}))\circ\Bock_{\Omega^p_{X_1^{(1)}}}
\end{equation}
where $\ob_{F,X_1}\cup\alpha(\Omega^1_{X^{(1)}_0})$ denotes the map $H^j(X_0^{(1)},\Omega^p_{X_0^{(1)}})\to H^{j+p}(X_0,^{(1)},\cO_{X_0^{(1)}})$ induced by the product with the class $c_{X^{(1)}_1,p}=\ob_{F,X_1}\cup\alpha(\Omega^1_{X^{(1)}_0})\in H^p(X_0^{(1)},\Lambda^pT_{X^{(1)}_0})$, for $j=i,i+1$.
\end{cor}

\begin{proof}
As usual, denote the inclusion of the special fiber by $i:X_0\hookrightarrow X_1$. The differential $d_{p+1}^{i,p}$ is induced by the map $i^*e_{X_1^{(1)},p}:\Omega^p_{X^{(1)}_0}\to \cO_{X^{(1)}_0}[p+1]$ which is the image of $\ob_{F,X_1}\cup\alpha(\Omega^1_{X^{(1)}_0})\in \RHom_{X^{(1)}_0}(\Omega^p_{X^{(1)}_0},\cO_{X^{(1)}_0}[p])$ under the Bockstein homomorphism associated with the $W_2(k)$-module $\RHom_{X_1^{(1)}}(\Omega^p_{X_1^{(1)}},\cO_{X_1^{(1)}}[p])$. For the purposes of computing the effect of $i^*e_{X_1^{(1)},p}$ on the cohomology groups it is enough to compute $i_*i^*e_{X_1^{(1)},p}:i_*\Omega^p_{X^{(1)}_0}\to i_*\cO_{X^{(1)}_0}[p+1]$. This is achieved by Lemma \ref{cosimp applications: hom bockstein} below.
\end{proof}

\begin{lm}\label{cosimp applications: hom bockstein}
Let $X_1$ be a flat scheme over $W_2(k)$ with special fiber $X_0:=X_1\times_{W_2(k)}k\xhookrightarrow{i}X_1$. For any two objects $M,N\in D(X_1)$ the composition 
\begin{equation}
\RHom_{D(X_0)}(i^*M,i^*N)\xrightarrow{\Bock_{\RHom_{X_1}(M,N)}} \RHom_{D(X_0)}(i^*M,i^*N)[1]\xrightarrow{i_*} \RHom_{D(X_1)}(i_*i^*M,i_*i^*N)[1]
\end{equation}
is given by sending a map $f:i^*M\to i^*N$ to $\Bock_N\circ i_*f-i_*f[1]\circ \Bock_M$.
\end{lm}

\begin{proof}
For any object $K\in D(X_1)$ there is a natural fiber sequence $i_*i^*K\to K\to i_*i^*K$ that we will view as a two-step filtration on $K$, defining therefore a functor $B:D(X_1)\to D_{\Fil}(X_1)$ to the category of filtered objects of $D(X_1)$. The complex of filtered morphisms between $M$ and $N$ fits into the fiber sequence

\begin{equation}\label{cosimp applications: hom bockstein eq1}
\RHom_{D(X_1)}(i_*i^*M,i_*i^*N)\to \RHom_{D_{\Fil}(X_1)}(M,N)\to \RHom_{D(X_1)}(i_*i^*M,i_*i^*N)^{\oplus 2}
\end{equation}

Here the first term is identified with the complex of morphisms from $M$ to $N$ that shift the filtration down by $1$, and the second map associates to a filtered morphism its effect on the graded pieces of the filtrations. The induced map $\RHom_{D(X_1)}(i_*i^*M,i_*i^*N)^{\oplus 2}\to \RHom_{D(X_1)}(i_*i^*M,i_*i^*N)[1]$ is then given by $(f_0,f_1)\mapsto \Bock_N\circ f_0-f_1[1]\circ \Bock_M$.

The map $\RHom_{D(X_1)}(M,N)\to \RHom^{\Fil}_{D(X_1)}(M,N)$ induced by the functor $B$ then induces a map of fiber sequences

\begin{equation}\label{cosimp applications: hom bockstein eq2}
\begin{tikzcd}
\RHom_{D(X_1)}(i_*i^*M,i_*i^*N)\arrow[r] & \RHom^{\Fil}_{D(X_1)}(M,N)\arrow[r] & \RHom_{D(X_1)}(i_*i^*M,i_*i^*N)^{\oplus 2}\\
\RHom_{D(X_0)}(i^*M,i^*N)\arrow[r]\arrow[u, "{i_*}"] & \RHom_{D(X_1)}(M,N)\arrow[r]\arrow[u] & \RHom_{D(X_0)}(i^*M,i^*N)\arrow[u,"{(i_*,i_*)}"]
\end{tikzcd}
\end{equation}
and this proves the lemma because $\Bock_{\RHom_{X_1}(M,N)}$ is precisely the connecting morphism induced by the bottom row of (\ref{cosimp applications: hom bockstein eq2}).
\end{proof}

\section{Sen operator of a fibration in terms of Kodaira-Spencer class}\label{nonsemisimp: section}

In this section we specialize the formula for the class $c_{X_1,p}$ from Corollary \ref{semiperf: smooth cor} to $p$-dimensional $W_2(k)$-schemes that are fibered over a curve. The fact that the cotangent bundle of such a scheme admits a line subbundle allows us (Theorem \ref{nonsemisimp: sen class fibration}) to relate the class $\alpha(\Omega^1_{X_0})$ to the Kodaira-Spencer class of the fibration. We then use this relation in Proposition \ref{nonsemisimp: p dim example} to give examples showing that the Sen operator $\Theta_X$ on $\dR_{X_0}$ might be non-semisimple. By Corollary \ref{sen operator: semisimplicity implies decomposability}, this is the case in a situation when the conjugate spectral sequence is non-degenerate, so Corollary \ref{nondeg example: main corollary} also provides such examples. But we would like to demonstrate that non-semisimplicity of $\Theta_X$ is a much more frequent phenomenon than non-degeneration of the Hodge-to-de Rham spectral sequence, appearing even for familiar classes of varieties.

We start by introducing the notation needed to state Theorem \ref{nonsemisimp: sen class fibration}. Let $f:X_1\to Y_1$ be a smooth morphism of smooth $W_2(k)$-schemes with $\dim Y_1=1$,  $\mathrm{rel.dim}(f)=p-1$. We will describe the class $c_{X_1,p}\in H^p(X_0,\Lambda^p T_{X_0})$ in terms of the Kodaira-Spencer class of $f$ and the obstruction to lifting the Frobenius morphism of $Y_0$ to $Y_1$. Recall that the Kodaira-Spencer class $\ks_{f_0}:T_{Y_0}\to R^1f_{0*}T_{X_0/Y_0}$ is obtained from the fundamental triangle 
\begin{equation}\label{nonsemisimp: fundamental triangle}
T_{X_0/Y_0}\to T_{X_0}\to f_0^*T_{Y_0}    
\end{equation}
by applying the functor of $0$th cohomology to the morphism $T_{Y_0}\to Rf_{0*}T_{X_0/Y_0}[1]$ corresponding to the class of the extension (\ref{nonsemisimp: fundamental triangle}) by adjunction. Denote by $\ks_{f_0}^{p-1}$ the composition $T_{Y_0}^{\otimes p-1}\xrightarrow{\ks_{f_0}^{\otimes p-1}} (R^1f_{0*}T_{X_0/Y_0})^{\otimes p-1}\to R^{p-1}f_{0*}\Lambda^{p-1}T_{X_0/Y_0}$ where the second map is the cup product on cohomology. 

By our assumption on the dimensions, the Leray spectral sequence for $f_0$ identifies $H^p(X_0,\Lambda^p T_{X_0})$ with $H^1(Y_0,R^{p-1}f_{0*}\Lambda^pT_{X_0})=H^{1}(Y_0, T_{Y_0}\otimes R^{p-1}f_{0*}\Lambda^{p-1}T_{X_0/Y_0})$.

\begin{thm}\label{nonsemisimp: sen class fibration}
The class $c_{X,p}\in H^p(X_0,\Lambda^p T_{X_0})$ is equal, up to multiplying by an element of $\bF_p^{\times}$, to the product of the class $\ob_{F,Y}\in H^1(Y_0,F_{Y_0}^*T_{Y_0})=H^1(Y_0,T_{Y_0}^{\otimes p})$ with $\ks^{p-1}_{f_0}\in H^0(Y_0,T_{Y_0}^{\otimes 1-p}\otimes R^{p-1}f_{0*}\Lambda^{p-1}T_{X_0/Y_0})$.
\end{thm}

We have the following description of $\alpha(E)$ for a rank $p$ vector bundle admitting a line subbundle.  It will be proven, for $p>2$, as a consequence of our computations with group cohomology in Section \ref{group cohomology: section}, under Lemma \ref{group cohomology: reducible formula}. For the proof in the case $p=2$ see Remark \ref{nonsemisimp: reducible formula p=2} below.

\begin{lm}\label{nonsemisimp: reducible formula}
Let $E$ be a vector bundle on $X_0$ of rank $p$ that fits into an extension
\begin{equation}\label{nonsemisimp: reducible formula extension}
0\to L\to E\to E'\to 0
\end{equation}
where $L$ is a line bundle, and $E'$ is a vector bundle of rank $p-1$. The class of this extension defines an element $v(E)\in \Ext^1_{X_0}(E',L)=H^1(X_0,L\otimes (E')^{\vee})$. Denote by $v(E)^{p-1}\in H^{p-1}(X_0,L^{\otimes p-1}\otimes (\det E')^{\vee})$ the image of $v(E)^{\otimes p-1}\in H^{p-1}(X_0, (L\otimes (E')^{\vee})^{\otimes p-1})$ under the map induced by $(L\otimes (E')^{\vee})^{\otimes p-1}\to \Lambda^{p-1}(L\otimes (E')^{\vee})=L^{\otimes p-1}\otimes (\det E')^{\vee}$.

The class $\alpha(E)\in \Ext^{p-1}_{X_0}(\Lambda^p E, F_{X_0}^*E)=H^{p-1}(X_0,F_{X_0}^*E\otimes L^{\vee}\otimes (\det E')^{\vee})$ is equal, up to multiplying by a scalar from $\bF_p^{\times}$, to the image of $v(E)^{p-1}$ under the map induced by $L^{\otimes p-1}\otimes(\det E')^{\vee}=F_{X_0}^*L\otimes L^{\vee}\otimes(\det E')^{\vee}\hookrightarrow F_{X_0}^*E\otimes L^{\vee}\otimes(\det E')^{\vee}$.
\end{lm}

\begin{rem}\label{nonsemisimp: reducible formula p=2}
The lemma is straightforward when $p=2$. It follows from the existence of the following map of extensions:
\begin{equation}
\begin{tikzcd}
F_{X_0}^*E\arrow[r] & S^2E\arrow[r] & \Lambda^2 E \\
L^{\otimes 2}\arrow[r]\arrow[u, hook] & L\otimes E\arrow[r]\arrow[u, hook] & L\otimes E'.\arrow[u, equal]
\end{tikzcd}
\end{equation}
Here the left vertical map is the pullback along $F_{X_0}$ of the inclusion $L\hookrightarrow E$, the top row is the extension representing $\alpha(E)$, and the bottom row is the tensor product of (\ref{nonsemisimp: reducible formula extension}) with $L$. 

Our proof of this lemma for $p>2$ is rather ad hoc: it relies on the fact that the class $\alpha(V)$ for the tautological representation $V$ of $GL_p$ is non-zero, as proven in Proposition \ref{rational group cohomology: main non-vanishing}. Given that both $\alpha(E)$ and $v(E)^{p-1}$ admit explicit representatives as Yoneda extensions ((\ref{group cohomology: de rham p-1}) for the former and a Koszul complex for the latter), it would be nicer to have a direct proof of the identity claimed in Lemma \ref{nonsemisimp: reducible formula}, in the spirit of the proof for $p=2$.
\end{rem}

\comment{\begin{lm}\label{nonsemisimp: split reducible formula}
Let $E$ be a vector bundle on $X_0$ of rank $p$ that fits into an extension \begin{equation}0\to L\to E\to \bigoplus\limits_{i=1}^{p-1}L_i\to 0\end{equation} where $L,L_1,\ldots,L_{p-1}$ are line bundles. The corresponding extension class
is an element $v_1\oplus\ldots\oplus v_{p-1}\in \bigoplus\limits_{i=1}^{p-1} H^1(X_0, L\otimes L_i^{\otimes -1})$, denote by $v_1\cup \ldots\cup v_{p-1}\in H^{p-1}(X_0,L^{\otimes p-1}\otimes \bigotimes L_i^{-1})$ the cup-product of the components of the extension class. 

Then the class $\alpha(E)\in \Ext^{p-1}_{X_0}(\Lambda^p E,F^*E)\simeq H^{p-1}(X_0,F^*E\otimes L^{-1}\otimes \bigotimes\limits_{i=1}^{p-1}L_i^{-1})$ is equal, up to multiplication by an element of $\bF_p^{\times}$, to  the image of $v_1\cup\ldots\cup v_{p-1}$ under the map induced by $L^{\otimes p-1}\otimes \bigotimes L_i^{-1}\hookrightarrow F^*E\otimes L^{-1}\otimes \bigotimes\limits_{i=1}^{p-1}L_i^{-1}$.
\end{lm}
}

\begin{proof}[Proof of Theorem \ref{nonsemisimp: sen class fibration}]
By Corollary \ref{semiperf: smooth cor} the class $c_{X_1,p}$ is the product of $\alpha(\Omega^1_{X_0})\in H^{p-1}(X_0,\Lambda^p T_{X_0}\otimes (F_{X_0}^*T_{X_0})^{\vee})$ with $\ob_{F,X_1}\in  H^1(X_0,F^*_{X_0}T_{X_0})$. The vector bundle $\Omega^1_{X_0}$ fits into the exact sequence dual to (\ref{nonsemisimp: fundamental triangle}):
\begin{equation}
0\to f^*\Omega^1_{Y_0}\to \Omega^{1}_{X_0}\to \Omega^1_{X_0/Y_0}\to 0
\end{equation}
and we can apply Lemma \ref{nonsemisimp: reducible formula} to $E=\Omega^1_{X_0}$. It gives that the class $\alpha(\Omega^1_{X_0})\in H^{p-1}(X_0, F_{X_0}^*\Omega^1_{X_0}\otimes \Lambda^p T_{X_0})$ is the image of the class $v(\Omega^1_{X_0})^{p-1}\in H^{p-1}(X_0,F_{X_0}^*f_0^*\Omega^1_{Y_0}\otimes \Lambda^p T_{X_0})$. Therefore the product $\alpha(\Omega^1_{X_0})\cdot \ob_{F,X_1}\in H^p(X_0,\Lambda^p T_{X_0})$ is equal to the product of $v(\Omega^1_{X_0})^{p-1}$ with the image of $\ob_{F,X_1}$ under the map $H^1(X_0,F_{X_0}^*T_{X_0})\to H^1(X_0, F_{X_0}^*f_0^*T_{Y_0})$.

By functoriality of obstructions, this image is equal to the image of $\ob_{F,Y_1}$ under the pullback map $H^1(Y_0,F_{Y_0}^*T_{Y_0})\to H^1(X_0, F_{X_0}^*f_0^*T_{Y_0})$. By the Leray spectral sequence the receptacle of the class $v(\Omega^1_{X_0})^{p-1}$ fits into the exact sequence
\begin{multline}
0\to H^1(Y_0, R^{p-2}f_{0*}(F_{X_0}^*f_0^*\Omega^1_{Y_0}\otimes \Lambda^p T_{X_0}))\to H^{p-1}(X_0, F_{X_0}^*f_0^*\Omega^1_{Y_0}\otimes \Lambda^p T_{X_0})\xrightarrow{\rho} \\ H^0(Y_0, R^{p-1}f_{0*}(F_{X_0}^*f_0^*\Omega^1_{Y_0}\otimes \Lambda^p T_{X_0}))\to 0.
\end{multline}

Since $H^2(Y_0,\cF)=0$ for any quasicoherent sheaf $\cF$ on the curve $Y_0$, our product $v(\Omega^1_{X_0})^{p-1}\cdot \ob_{F,X_1}$ only depends on the image of $v(\Omega^1_{X_0})^{p-1}$ in $H^0(Y_0, R^{p-1}f_{0*}(F_{X_0}^*f_0^*\Omega^1_{Y_0}\otimes \Lambda^p T_{X_0}))$ and is equal to the product of this image $\rho(v(\Omega^1_{X_0})^{p-1})$ with $\ob_{F,Y_1}$ under the identification $H^p(X_0,\Lambda^p T_{X_0})\simeq H^1(Y_0,T_{Y_0}\otimes R^{p-1}f_{0*}\Lambda^{p-1}T_{X_0/Y_0})$.

It remains to observe that the image of the class $v(\Omega^1_{X_0})\in H^1(X_0, f^*\Omega^1_{Y_0}\otimes T_{X_0/Y_0})$ in $H^0(Y_0,R^1f_{0*}(f_0^*\Omega^1_{Y_0}\otimes T_{X_0/Y_0}))=H^0(Y_0, \Omega^1_{Y_0}\otimes R^1f_{0*}T_{X_0/Y_0})$ is the Kodaira-Spencer map $\ks_{f_0}$. Therefore $\rho(v(\Omega^1_{X_0})^{p-1})\in H^0(Y_0, R^{p-1}f_{0*}(F_{X_0}^*f_0^*\Omega^1_{Y_0}\otimes \Lambda^p T_{X_0}))=H^0(Y_0, F_{Y_0}^*\Omega^1_{Y_0}\otimes T_{Y_0}\otimes R^{p-1}f_{0*}\Lambda^{p-1}T_{X_0/Y_0})$ is equal to $\ks_{f_0}^{p-1}$ and the desired formula for $c_{X_1,p}$ is proven. 
\end{proof}

We will now demonstrate that there does exist a fibration as in Theorem \ref{nonsemisimp: sen class fibration} for which $c_{X_1,p}$ is non-zero. From now until the end of this section assume that $k=\overline{\bF}_p$.

\begin{pr}\label{nonsemisimp: p dim example}
For every $p$ there exists a smooth projective scheme $X$ over $W(k)$ with $\dim_{W(k)}(X)=p$ such that the class $c_{X,p}\in H^p(X_0,\Lambda^p T_{X_0})$ is non-zero, and therefore the Sen operator on $\dR_{X_0}$ is not semisimple.
\end{pr}

\begin{rem}
Since $X$ has relative dimension $p$, the de Rham complex $\dR_{X_0}$ of $X_0$ is necessarily decomposable, by \cite[Corollaire 2.3]{deligne-illusie}.
\end{rem}

\begin{proof}
We will construct $X$ as the $(p-1)$th relative Cartesian power $S^{\times_Y(p-1)}$ of an appropriate smooth projective morphism $h: S\to Y$ of relative dimension $1$, where $Y$ is a geometrically connected smooth projective relative curve over $W(k)$. 
Assume that 

\begin{enumerate}
\item The Kodaira-Spencer map $\ks_{h_0}:T_{Y_0}\to R^1h_{0*}T_{S_0/Y_0}$ is an injection of vector bundles
\item $\coker \ks_{h_0}$ is a direct sum $\bigoplus L_j$ of line bundles such that $\deg L_j< \frac{1}{p-1}\deg \Omega^1_{Y_0}$ for all $j$.
\end{enumerate}

Injectivity of the Kodaira-Spencer map implies that the curve fibration $h$ is not isotrivial, which forces the genus of $Y$ and all of the fibers of $h$ to be larger than or equal to $2$ by \cite[Th\'eorem\`e 4]{szpiro}. In particular, the Frobenius endomorphism of $Y_0$ does not lift to $Y\times_{W(k)}{W_2(k)}$ so the class $\ob_{Y_1,F}\in H^1(Y_0,F_{Y_0}^*T_{Y_0})$ is non-zero, cf. \cite[Lemma I.5.4]{raynaud-froblifts}. We will now prove that under these conditions on $S$ the class $c_{X,p}\in H^p(X_0,\Lambda^p T_{X_0})$ for the $p$-dimensional $W(k)$-scheme $X$ is non-zero, and then will check that a family of curves $h$ satisfying (1) and (2) does indeed exist.

Denote by $\pi_1,\ldots,\pi_{p-1}:X=S^{\times_Y (p-1)}\to S$ the projection maps, and by $f:X\to Y$ the map down to $Y$ (so that $f=h\circ \pi_i$, for all $i$). We have $T_{X/Y}\simeq \bigoplus\limits_{i=1}^{p-1}\pi_i^*T_{S/Y}$. For each $i$ the pushforward $R^1f_*(\pi^{*}_i T_{S/Y})=H^1(Rh_*\circ R\pi_{i*}(\pi^*_i T_{S/Y}))=H^1(Rh_*(T_{S/Y}\otimes R\pi_{i*}\cO_{X}))$ is equal to $R^1h_*T_{S/Y}$, because $h_*T_{S/Y}=0$ by our assumption on the genus of the fibers of $h$. The Kodaira-Spencer class $\ks_f: T_Y\to R^1f_*T_{X/Y}=\bigoplus\limits_{i=1}^{p-1} R^1h_*T_{S/Y}$ is then equal to the diagonal map $\bigoplus\limits_{i=1}^{p-1} \ks_h$, because the extension $T_{X/Y}\to T_X\to f^*T_Y$ is the Baer sum of the pullbacks along $\pi_i$ of the extension $T_{S/Y}\to T_S\to h^*T_Y$.

The pushforward $R^{p-1}f_*\Lambda^{p-1}T_{X/Y}=R^{p-1}f_*(\bigotimes\limits_{i=1}^{p-1} \pi^*_iT_{S/Y})$ is likewise identified with $(R^1h_*T_{S/Y})^{\otimes p-1}$, and the $(p-1)$th power $\ks_f^{p-1}$ of the Kodaira-Spencer map is equal to $\ks_h^{\otimes p-1}:T_Y^{\otimes p-1}\to (R^1h_*T_{S/Y})^{\otimes p-1}$.

By Theorem \ref{nonsemisimp: sen class fibration}, the class $c_{X,p}\in H^p(X_0,\Lambda^p T_{X_0})=H^1(Y_0,T_{Y_0}\otimes R^{p-1}\Lambda^{p-1}T_{X_0/Y_0})$ is (up to a scalar from $\bF_p^{\times}$) the image of the obstruction to lifting Frobenius $\ob_{F,Y_1}\in H^1(Y_0,F^*T_{Y_0})=H^1(Y_0,T_{Y_0}^{\otimes p})$ under the map $\id_{T_{Y_0}}\otimes \ks_{h_0}^{\otimes p-1}:T_{Y_0}^{\otimes p}\to T_{Y_0}\otimes (R^1h_{0*}T_{S_0/Y_0})^{\otimes p-1}$. The proof of non-vanishing of $c_{X,p}$ will be completed if we can show that the induced map $H^1(Y_0, T_{Y_0}^{\otimes p})\to H^1(Y_0,T_{Y_0}\otimes (R^1h_{0*}T_{S_0/Y_0})^{\otimes p-1})$ is injective. 

For brevity, denote the vector bundle $\coker \ks_{h_0}$ by $Q$. The map $\id_{T_{Y_0}}\otimes \ks_{h_0}^{\otimes p-1}$ is an injection of vector bundles whose cokernel admits a filtration with graded pieces isomorphic to $T_{Y_0}^{\otimes i}\otimes Q^{\otimes p-i}$ for $i=1,\ldots,p-1$. By our assumption (2), each $T_{Y_0}^{\otimes i}\otimes Q^{\otimes p-i}$ is a direct sum of line bundles of degree $< \frac{1}{p-1}\deg\Omega^1_{Y_0}\cdot (p-i)-\deg\Omega^1_{Y_0}\cdot i<0$ because $\deg \Omega^1_{Y_0}=2g(Y_0)-2$ is positive. Therefore  $H^0(Y_0,\coker(\id_{T_{Y_0}}\otimes \ks_{h_0}^{\otimes p-1}))=0$ which implies that the map induced by $\id_{T_{Y_0}}\otimes \ks_{h_0}^{\otimes p-1}$ on cohomology in degree $1$ is injective.

We will now prove that there exists a family of curves $h:S\to Y$ satisfying properties (1), (2). We will construct $Y$ as a complete intersection of ample divisors in an appropriate compactification $\cM_g^*$ of the moduli space of curves of genus $g$, following the idea of Mumford (\hspace{1sp}\cite[\S 1]{oort-complete}) for constructing proper subvarieties of moduli spaces of curves. We only need to make sure that this construction goes through over $W(k)$, and check that the condition (2) is fulfilled. Denote by $M_{g,W(k)}$ the coarse moduli scheme of the stack $\cM_{g,W(k)}$, and by $\oM_{g,W(k)}$ the coarse moduli scheme of its Deligne-Mumford compactification.

By \cite[Theorem V.2.5]{faltings-chai} there exists a projective scheme $A_{g,W(k)}^*$ over $W(k)$ that contains the coarse moduli scheme $A_{g,W(k)}$ of principally polarized abelian varieties of dimension $g$ as a dense open subscheme such that $\dim_{W(k)}(A_{g,W(k)}^*\setminus A_{g,W(k)})=g$. Moreover, the Torelli map extends to a morphism $j:\oM_{g,W(k)}\to A^*_{g,W(k)}$ that induces a locally closed immersion of the locus $U\subset \oM_{g,W(k)}$ of smooth curves without nontrivial automorphisms into $A^*_{g,W(k)}$.  Denote also by $\pi:\cC\to U$ the universal curve over $U$.

Assume from now on that $g\geq \max(4, \frac{p}{3}+1)$. The inequality $g\geq 4$ ensures that the locus of curves without nontrivial automorphisms has complement of codimension $\geq 2$ in $\cM_g$, and the condition $g\geq \frac{p}{3}+1$ will be used to ensure property (2). Denote by $M_{g,W(k)}^*$ the closure of $j(M_{g,W(k)})$ inside $A_{g,W(k)}^*$. This is a flat projective scheme over $W(k)$ that contains $U$ as a dense open subscheme, such that fibers of $M_{g,W(k)}^*\setminus U$ over both points of $\Spec W(k)$ have codimension $\geq 2$ (though we do not claim that this complement is flat over $W(k)$). Denote by $\omega$ the line bundle $\det (\pi_*\Omega^1_{\cC/U})$ on $U$. By the property \cite[Theorem V.2.5(1)]{faltings-chai} of the Satake compactification, some positive power $\omega^{\otimes m}$ extends to a very ample $L$ line bundle on $M_{g,W(k)}^*$.

Take $Y\subset U\subset M_{g,W(k)}^*$ to be a smooth proper complete intersection of zero loci of sections of $L$ that has $\dim(Y/W(k))=1$. It is possible to find such a complete intersection entirely contained in $U$ because the codimension of the complement of $U$ is at least $2$. Let $h:S\to Y$ be the restriction of the universal curve $\cC$ to $Y$. We have the exact sequence corresponding to the closed embedding $\iota:Y\hookrightarrow U$:
\begin{equation}
0\to T_Y\xrightarrow{d\iota} T_{M_{g,W(k)}}|_Y\to L|_Y^{\oplus 3g-4}\to 0    
\end{equation}
where we identified the normal bundle to $Y$ with $L|_Y^{\oplus 3g-4}$. Since $U$ is the locus where the map $\cM_{g,W(k)}\to M_{g,W(k)}$ from the moduli stack of curves to the coarse moduli space is an isomorphism, the sheaf $T_{M_{g,W(k)}}|_Y$ is identified with $R^1h_*T_{S/Y}$ in a way that identifies $d\iota$ with the Kodaira-Spencer map $\ks_f$.

By \cite[p. 50, Theorem 2]{harris-mumford}, the determinant line bundle $\det(T_{M_g,W(k)}|_Y)$ is isomorphic to $\omega^{\otimes -13}|_Y$ and, in particular, has negative degree. Therefore $\deg(T_Y)+(3g-4)\cdot \deg L|_Y\leq 0$ giving that $\deg L|_Y\leq -\deg(T_Y)\cdot \frac{1}{3g-4}< -\deg(T_Y)\cdot \frac{1}{p-1}$ by our assumption on $g$. Therefore the family $h:S\to Y$ satisfies all of the desired assumptions.
\end{proof}

\begin{rem}
At the moment it is unclear to me whether the class $c_{X,p}$ is non-zero for other natural classes of varieties. For instance, when $p=2$ is $c_{X,2}\in H^2(X_0,\Lambda^2 T_{X_0})\simeq k$ non-zero for a general K3 surface $X$ over $W_2(k)$?\footnote{Forthcoming work of Ogus and Vologodsky on the relation between the Sen operator and divisibility of Frobenius on crystalline cohomology implies that $c_{X,2}$ is zero for every K3 surface over $W_2(k)$.}
\end{rem}

\section{Cohomology of abelian varieties}\label{AV coherent: section}

In this section, we apply Theorem \ref{cosimp: main theorem} to the de Rham, coherent, and \'etale cohomology of abelian varieties equipped with a group action. These results will play the key role in our example of a liftable variety with a non-degenerate conjugate spectral sequence, which ultimately relies on the existence of a supersingular abelian variety with an action of a group, whose Hodge cohomology is equivariantly decomposable, but de Rham cohomology is not.

In this section, for a commutative ring $R$ and a discrete group $G$ we denote by $D_G(R)$ the derived $\infty$-category of complexes of $R$-modules equipped with a $G$-action. 
\subsection{Coherent cohomology.} Recall that for a group $G$ acting on a finite-dimensional $k$-vector space $V$ we have a class $\alpha(V)\in \Ext^{p-1}_G(\Lambda^p V, V^{(1)})=H^{p-1}(G,V^{(1)}\otimes(\Lambda^p V)^{\vee})$ defined in Definition \ref{free cosimplicial: alpha definition}, corresponding to the extension $V^{(1)}[-2]\to\tau^{\geq 2}S^p(V[-1])\to\Lambda^p V[-p]$ in $D_G(k)$. If there exists a representation of $G$ on a free $W_2(k)$-module $\tV$ such that $\tV/p\simeq V$ then the module $\tV^{(1)}\otimes(\Lambda^p\tV)^{\vee}$ defines the Bockstein homomorphisms $\Bock^{i}:H^i(G,V^{(1)}\otimes(\Lambda^p V)^{\vee})\to H^{i+1}(G,V^{(1)}\otimes(\Lambda^p V)^{\vee})$.

\begin{pr}\label{AV coherent: coherent main}
Let $A$ be an abelian scheme over $W(k)$ equipped with an action of a discrete group $G$.
\begin{enumerate}[leftmargin=*]
\setcounter{enumi}{-1}
\item There exists a $G$-equivariant equivalence $\tau^{\leq p-1}\RGamma(A,\cO)\simeq \bigoplus\limits_{i=0}^{p-1}H^i(A,\cO)[-i]$.

\item  If $F^*_{A_0}:H^1(A_0,\cO)\to H^1(A_0,\cO)$ is zero, then there exists a $G$-equivariant equivalence 
\begin{equation}
\tau^{\leq p}\RGamma(A,\cO)\simeq\bigoplus\limits_{i=0}^{p}H^i(A,\cO)[-i].
\end{equation}
\item Suppose that $A$ admits an endomorphism $\tF_A:A\to A$ that lifts the absolute Frobenius endomorphism $F_{A_0}$ (in particular, $A_0$ is ordinary) and commutes with the action of $G$.  Then the extension class $H^p(A,\cO)\to \bigoplus\limits_{i=0}^{p-1}H^i(A,\cO)[p-i+1]$ in $D_G(W(k))$ corresponding to $\tau^{\leq p}\RGamma(A,\cO)$ lands in the direct summand $H^1(A,\cO)[p]$ and the resulting class in $\Ext^p_{G,W(k)}(H^p(A,\cO),H^1(A,\cO))=H^p(G,H^1(A,\cO)\otimes\Lambda^pH^1(A,\cO)^{\vee})$ is equal to
\begin{equation}\label{AV coherent: ext class formula}
\tF^*_A(\Bock^{p-1}(\alpha(H^1(A_0,\cO))))
\end{equation}
where $\alpha(H^1(A_0,\cO))\in H^{p-1}(G,H^1(A_0,\cO)^{(1)}\otimes\Lambda^pH^1(A_0,\cO)^{\vee})$ is the class corresponding to the representation of $G$ on the $k$-vector space $H^1(A_0,\cO)$, the map $\Bock^{p-1}$ is the Bockstein homomorphism induced by the $W(k)$-module $H^1(A,\cO)^{(1)}\otimes\Lambda^pH^1(A,\cO)^{\vee}$, and $\tF^*_A$ denotes the morphism induced by $\tF_A^*\otimes \id:H^1(A,\cO)^{(1)}\otimes \Lambda^p H^1(A,\cO)^{\vee}\to H^1(A,\cO)\otimes \Lambda^p H^1(A,\cO)^{\vee}$
\end{enumerate}
\end{pr}

\begin{proof}
The identity section $e:\Spec W(k)\to A$ induces the augmentation $e^*:\RGamma(A,\cO)\to W(k)$, and the cohomology algebra $H^*(A,\cO)$ is the exterior algebra on $H^1(A,\cO)$ so we can apply Theorem \ref{cosimp: equivariant main theorem} to the $G$-equivariant derived commutative algebra $\RGamma(A,\cO)$. Part (0) then follows. To prove statements (1) and (2), we will apply (\ref{cosimp: equivariant main formula}).  The formula for the extension $H^p(A,\cO)\to \tau^{\leq p-1}\RGamma(A,\cO)[p+1]$ in $D_G(W(k))$ reads:
\begin{equation}\label{AV coherent: main formula}
\Lambda^p H^1(A,\cO)\xrightarrow{\alpha(H^1(A_0,\cO))} H^1(A_0,\cO)^{(1)}[p-1]\xrightarrow{F^*_{A_0}} H^1(A_0,\cO)[p-1]\xrightarrow{\Bock_{H^1(A,\cO)}} H^1(A,\cO)[p].
\end{equation}

Here we identified the Frobenius endomorphism of the cosimplicial algebra $\RGamma(A_0,\cO)$ with $F^*_{A_0}$ by Lemma \ref{dalg: geometric frobenius}. In particular, if $F^*_{A_0}$ on $H^1(A_0,\cO)$ is zero, then the composition (\ref{AV coherent: main formula}) vanishes, which proves part (1).

Regarding part (2), note that the lift of Frobenius $\tF_{A}^*$ induces the following commutative diagram in $D_G(W(k))$:
\begin{equation}
\begin{tikzcd}
H^1(A_0,\cO)^{(1)}\arrow[r, "F_{A_0}^*"]\arrow[d, "\Bock_A^{(1)}"] & H^1(A_0,\cO)\arrow[d, "\Bock_A"] \\
H^1(A,\cO)^{(1)}[1]\arrow[r, "\tF_A^*"] & H^1(A, \cO)[1].
\end{tikzcd}
\end{equation}
This allows us to rewrite (\ref{AV coherent: main formula}) as 
\begin{equation}\label{AV coherent: rewritten formula}
\Lambda^p H^1(A,\cO)\xrightarrow{\alpha(H^1(A_0,\cO))} H^1(A_0,\cO)^{(1)}[p-1]\xrightarrow{\Bock_{H^1(A,\cO)^{(1)}}[p]} H^1(A,\cO)^{(1)}[p]\xrightarrow{\tF^*_A} H^1(A,\cO)[p].
\end{equation}
The desired formula (\ref{AV coherent: ext class formula}) now follows from Lemma \ref{cosimp applications: Bockstein}.
\end{proof}

\begin{cor}\label{AV coherent: supersingularity criterion}
Let $n\geq 2p$ be an integer and $E_0$ be an elliptic curve over $k$. The group $GL_n(\bZ)$ acts on the abelian variety $E_0^{n}$ and therefore acts on the complex $\RGamma(E_0^n,\cO)$. The truncation $\tau^{\leq p}\RGamma(E_0^n,\cO)$ decomposes in $D_{GL_n(\bZ)}(k)$ as a direct sum of its cohomology modules if and only if $E_0$ is supersingular. 
\end{cor}

\begin{proof}
We can choose a lift $E$ of $E_0$ to an elliptic curve over $W(k)$ and apply Proposition \ref{AV coherent: coherent main} to the abelian scheme $E^n$ that is being acted on by the group $G=GL_n(\bZ)$. If $E_0$ is supersingular then Proposition \ref{AV coherent: coherent main}(1) gives that $\tau^{\leq p}\RGamma(E_0^n,\cO)$ is decomposable, for all $n$. 

If $E_0$ is ordinary, assume moreover that our chosen lift $E$ is the canonical one, so that $F_{E_0}$ lifts to a map $\tF_E:E\to E$. Following Proposition \ref{AV coherent: coherent main}(2), we need to check that the mod $p$ reduction of the class $\tF^*_{E^n}(\Bock^{p-1}(\alpha(H^1(E_0,\cO))))\in H^p(GL_n(\bZ),H^1(E^n,\cO)\otimes\Lambda^pH^1(E^n,\cO)^{\vee})$ is non-zero. Note that the map $\tF^*_{E^n}$ is a bijection, so we only need to check that its argument is non-zero modulo $p$. 

Let $F/\bQ$ be the quadratic extension provided by Proposition \ref{group cohomology: ring of integers main}. Identifying the abelian group $\cO_F$ with $\bZ^{\oplus 2}$ we get an embedding $SL_p(\cO_F)\subset GL_{2p}(\bZ)\subset GL_n(\bZ)$. The induced action of $SL_p(\cO_F)$ on $H^1(E^n,\cO)=H^1(E,\cO)\otimes_{W(k)}W(k)^{\oplus n}\simeq W(k)^{\oplus n}$ contains the tautological representation $W(k)^{\oplus p}$ as a direct summand, hence the mod $p$ reduction of the class $\Bock^{p-1}(\alpha(H^1(E^n,\cO)))$ does not vanish in $H^p(SL_p(\cO_F), H^1(E_0^n,\cO)^{(1)}\otimes \Lambda^p H^1(E_0,\cO)^{\vee})$. Therefore the analogous class in $H^p(GL_n(\bZ), H^1(E_0^n,\cO)^{(1)}\otimes \Lambda^p H^1(E_0,\cO)^{\vee})$ does not vanish either, and the non-decomposability of $\tau^{\leq p}\RGamma(E_0^n,\cO)$ in $D_{GL_n(\bZ)}(k)$ follows.
\end{proof}

\subsection{De Rham cohomology.} We can also apply Theorem \ref{cosimp: main theorem} to the de Rham cohomology of an abelian variety equipped with a group action. 

\begin{pr}\label{AV coherent: AV de rham}
Let $A_0$ be an abelian variety over $k$ equipped with an action of a discrete group $G$. 

\begin{enumerate}
    \item $\tau^{\leq p-1}\RGamma_{\dR}(A_0/k)$ is $G$-equivariantly equivalent to $\bigoplus\limits_{i=0}^{p-1}H^i_{\dR}(A_0/k)[-i]$.
    \item The extension class $H^p_{\dR}(A_0/k)\to\bigoplus\limits_{i=0}^{p-1}H^i_{\dR}(A_0/k)[p+1-i]$ corresponding to $\tau^{\leq p}\RGamma_{\dR}(A_0/k)$ lands in the direct summand $H^1_{\dR}(A_0/k)[p]$ and the resulting class in $\Ext^p_{G,k}(H^p_{\dR}(A_0/k),H^1_{\dR}(A_0/k))=H^p(G,H^1_{\dR}(A_0/k)\otimes \Lambda^p H^1_{\dR}(A_0/k)^{\vee})$ equals to \begin{equation}F_{A_0}^*\circ\Bock^{p-1}(\alpha(H^1_{\dR}(A_0/k)))\end{equation} where we use the crystalline cohomology $H^1_{\cris}(A_0/W_2(k))$ as a lift of the $G$-representation $H^1_{\dR}(A_0/k)$ to define the Bockstein homomorphism, and $F_{A_0}^*\circ -$ denotes composing an element of $\Ext^p_{G,k}(H^p_{\dR}(A_0/k),H^1(A_0/k)^{(1)})$ with the map $F_{A_0}^*:H^1_{\dR}(A_0/k)^{(1)}\to H^1_{\dR}(A_0/k)$ induced on $1$st de Rham cohomology by the Frobenius endomorphism of $A_0$.
\end{enumerate}
\end{pr}

\begin{proof}
The derived commutative algebra $\RG_{\dR}(A_0/k)$ in the category of $G$-representations on $k$-vector spaces admits a lift to $W(k)$ given by the crystalline cohomology $\RG_{\cris}(A_0/W(k))$. This algebra has an augmentation $e^*:\RG_{\cris}(A_0/W(k))\to \RG_{\cris}(\Spec k/W(k))=W(k)$ so we can apply Theorem \ref{cosimp: equivariant main theorem} to it. This gives decomposability of $\tau^{\leq p-1}\RG_{\dR}(A_0/k)$ and describes the extension class $H^p_{\cris}(A_0/W(k))\to H^1_{\cris}(A_0/W(k))[p]$ as the composition

\begin{multline}\label{AV coherent: AV de rham formula}
H^p_{\cris}(A_0/W(k))=\Lambda^p H^1_{\cris}(A_0/W(k))\to \Lambda^p H^1_{\dR}(A_0/k)\xrightarrow{\alpha(H^1_{\dR}(A_0/k))} \\ H^1_{\dR}(A_0/k)^{(1)}[p-1]\xrightarrow{\varphi_{\RG_{\dR}(A_0/k)}}H^1_{\dR}(A_0/k)[p-1]\xrightarrow{\Bock_{H^1_{\cris}(A_0/W(k))}} H^1_{\cris}(A_0/W(k))[p].
\end{multline}

By Lemma \ref{applications: coh cosimplicial frob is frob} the morphism $\varphi_{\RG_{\dR}(A_0/k)}$ is equivalent to the pullback on de Rham cohomology along $F_{A_0}$, and in particular it lifts to the endomorphism $F_{A_0}^*$ of $\RG_{\cris}(A_0/W(k))$. Hence the composition of the last two arrows of (\ref{AV coherent: AV de rham formula}) is equivalent to 
\begin{equation}
H^1_{\dR}(A_0/k)^{(1)}[p-1]\xrightarrow{\Bock_{H^1_{\cris}(A_0/W(k))^{(1)}}}H^1_{\cris}(A_0/W(k))^{(1)}[p]\xrightarrow{F_{A_0}^*}H^1_{\cris}(A_0/W(k))[p].
\end{equation}

As in the deduction of Theorem \ref{cosimp applications: the best part} from Theorem \ref{cosimp applications: HT extension}, we now apply Lemma \ref{cosimp applications: Bockstein} to conclude that the class defined by the composition (\ref{AV coherent: AV de rham formula}) equals to $F_{A_0}^*\circ\Bock^{p-1}(\alpha(H^1_{\dR}(A_0/k)))\in \Ext^{p}_G(H^p_{\cris}(A_0/W(k)),H^1_{\cris}(A_0/W(k)))$. Reducing this class modulo $p$ gives the desired result.
\end{proof}

\subsection{Galois action on \'etale cohomology of abelian varieties.} This application was suggested by Alexei Skorobogatov who had independently conjectured the validity of Proposition \ref{AV coherent: etale} for $p=2$. 

Let $A$ be an abelian variety over an arbitrary field $F$ of characteristic not equal to $p$. Denote by $\oF$ a separable closure of $F$ and by $G_F:=\Gal(\oF/F)$ the absolute Galois group of $F$. We view $\RGamma_{\et}(A_{\oF},\bZ_p)$ as an object of the derived category $\hD_{G_F}(\bZ_p)$ of $p$-complete $\bZ_p$-modules equipped with a continuous action of $G_F$. More precisely, it is the inverse limit of categories $D_{G_F}(\bZ/p^n)$ of discrete $\bZ/p^n$-modules over the profinite group $G_F$.

\begin{pr}\label{AV coherent: etale}
\begin{enumerate}
    \item 
The truncation $\tau^{\leq p-1}\RGamma_{\et}(A_{\oF},\bZ_p)$ is equivalent to $\bigoplus\limits_{i=0}^{p-1} H^i_{\et}(A_{\oF},\bZ_p)[-i]$ in $\hD_{G_F}(\bZ_p)$.

\item The extension class $H^p_{\et}(A_{\oF},\bZ_p)\to \bigoplus\limits_{i=0}^{p-1} H^i_{\et}(A_{\oF},\bZ_p)[p+1-i]$ corresponding to the fiber sequence $\tau^{\leq p-1}\RGamma_{\et}(A_{\oF},\bZ_p)\to \tau^{\leq p}\RGamma_{\et}(A_{\oF},\bZ_p)\to H^p_{\et}(A_{\oF},\bZ_p)[-p]$ can be described as

\begin{multline}\label{AV coherent: etale formula}
H^p_{\et}(A_{\oF},\bZ_p)\xrightarrow{\alpha(H^1_{\et}(A_{\oF},\bF_p))}H^1_{\et}(A_{\oF},\bF_p)[p-1]\xrightarrow{\Bock_{H^1_{\et}(A_{\oF},\bZ_p)}}H^1_{\et}(A_{\oF},\bZ_p)[p]\\ \xrightarrow{\oplus}\bigoplus\limits_{i=0}^{p-1} H^i_{\et}(A_{\oF},\bZ_p)[p+1-i].
\end{multline}
\end{enumerate}
\end{pr}

\begin{proof}
The \'etale cohomology $\RGamma_{\et}(A_{\oF},\bZ_p)$ has a structure of a derived commutative algebra in $\hD_{G_F}(\bZ_p)$ by Lemma \ref{dalg: site cohomology}. The pullback along the identity section $e: \Spec F\to A$ induces an augmentation, and multiplication on cohomology induces isomorphisms $H^i_{\et}(A_{\oF},\bZ_p)\simeq\Lambda^i H^1_{\et}(A_{\oF},\bZ_p)$. Hence we may apply Theorem \ref{cosimp: equivariant main theorem} to $\RGamma_{\et}(A_{\oF},\bZ_p)$ which readily implies part (1).

Part (2) follows from the formula (\ref{cosimp: equivariant main formula}) by the fact that the Frobenius map $\varphi_{\RGamma_{\et}(A_{\oF},\bF_p)}:H^1_{\et}(A_{\oF},\bF_p)\to H^1_{\et}(A_{\oF},\bF_p)$ is the identity map, by Lemma \ref{dalg: frobenius on etale}. 
\end{proof}

We can apply this to describe some of the differentials in the Hochschild-Serre spectral sequence of $A$ with coefficients in $\bZ_p$ and $\bF_p$. Here $H^i(G_F,-)$ denotes continuous cohomology of the Galois group $G_F$.

\begin{cor}
\begin{enumerate}
\item In the spectral sequence $E_{2}^{i,j}=H^i(G_F,H^j_{\et}(A_{\oF},\bZ_p))\Rightarrow H_{\et}^{i+j}(A,\bZ_p)$ there are no non-zero differentials on pages $E_2,\ldots, E_{p-1}$, and the differentials $d_{p}^{i,p}:H^i(G_F,H^p_{\et}(A_{\oF},\bZ_p))\to H^{i+p}(G_F,H^1_{\et}(A_{\oF},\bZ_p))$ can be described as 
\begin{equation}
H^i(G_F,\Lambda^p H^1_{\et}(A_{\oF},\bZ_p))\xrightarrow{\alpha(H^1_{\et}(A_{\oF},\bF_p))}H^{i+p-1}(G_F,H^1_{\et}(A_{\oF},\bF_p))\xrightarrow{\Bock^{i+p-1}_{H^1_{\et}(A_{\oF},\bZ_p)}}H^{i+p}(G_F,H^1_{\et}(A_{\oF},\bZ_p)).
\end{equation}

\item Likewise, the spectral sequence $E_{2}^{i,j}=H^i(G_F,H^j_{\et}(A_{\oF},\bF_p))\Rightarrow H^{i+j}_{\et}(A,\bF_p)$ has no non-zero differentials on pages $E_2,\ldots, E_{p-1}$, and the differentials $d_{p}^{i,p}:H^i(G_F,H^p_{\et}(A_{\oF},\bF_p))\to H^{i+p}(G_F,H^1_{\et}(A_{\oF},\bF_p))$ are equal to  

\begin{equation}
\Bock^{i+p-1}_{H^1_{\et}(A_{\oF},\bZ/p^2)}\circ\alpha(H^1_{\et}(A_{\oF},\bF_p))-\alpha(H^1_{\et}(A_{\oF},\bF_p))\circ \Bock_{H^p_{\et}(A_{\oF},\bZ/p^2)}^{i}.
\end{equation}
\end{enumerate}
\end{cor}

\begin{proof}
Part (1) follows immediately from the formula (\ref{AV coherent: etale formula}) because $d_p^{i,p}$ is the result of applying the functor $H^i(G_F,-)$ to that map. Part (2) follows by applying Lemma \ref{cosimp applications: hom bockstein} to describe the effect of the mod $p$ reduction of (\ref{AV coherent: etale formula}) on cohomology.
\end{proof}

\section{Liftable variety with non-degenerate conjugate spectral sequence}

In this section we construct an example of a smooth projective variety over $k$ that lifts to a smooth projective scheme over $W(k)$ but whose conjugate spectral sequence has a non-zero differential. Let us start by describing our example.

Choose an elliptic curve $E$ over $W(k)$ such that the special fiber $E_0=E\times_{W(k)}k$ is supersingular. Denote $p^2$ by $q$, and consider the finite flat group scheme \begin{equation}H:=E[p]\otimes_{\bF_p}\bF_q^{\oplus p}.\end{equation} It is isomorphic to a product of $2p$ copies of $E[p]$, and we write it this way to define an action of the group $GL_{p}(\bF_q)$ on it via the tautological representation on $\bF_q^{\oplus p}$. Define the non-commutative finite flat group scheme
\begin{equation}\label{nondeg example: G formula}
G:=SL_{p}(\bF_q)\ltimes (E[p]\otimes_{\bF_p}\bF_q^{\oplus p}).
\end{equation}

The main result of this section is that the classifying stack $BG_0$ of the special fiber of this finite group scheme has a non-degenerate conjugate spectral sequence. In Corollary \ref{nondeg example: main corollary} we then approximate the stack $BG$ by a smooth projective scheme whose special fiber also has a non-zero differential in its conjugate spectral sequence.

\begin{thm}\label{nondeg example: maint thm stack}
The differential $d^{0,p}_{p+1}:H^0(BG^{(1)}_0,L\Omega^p_{BG^{(1)}_0/k})\to H^{p+1}(BG^{(1)}_0,\cO)$ on the $(p+1)$st page of the conjugate spectral sequence of the classifying stack $BG_0$ is non-zero.
\end{thm}

In general, if $X_0$ is a smooth Artin stack over $k$, the spectral sequence associated with the conjugate filtration on its de Rham complex starts with the second page of the form $E_2^{i,j}=H^i(X^{(1)}_0,L\Omega^j_{X^{(1)}_0/k})$. As was discussed in Remark \ref{sen operator: drinfeld decomposition stacks}, if $X_0$ admits a lift over $W_2(k)$ then there are no non-zero differentials on pages $E_2,\ldots,E_p$. In this case we denote by $d^{i,j}_{p+1}:H^i(X^{(1)}_0,L\Omega^j_{X^{(1)}_0/k})\to H^{i+p+1}(X^{(1)}_0,L\Omega^{j-p}_{X^{(1)}_0/k})$ the differentials on page $E_{p+1}$.

Although we already have a formula for the differential $d^{0,p}_{p+1}$, provided by Theorem \ref{cosimp applications: the best part}, we take a somewhat roundabout approach to proving Theorem \ref{nondeg example: maint thm stack} by proving as an intermediary that for a certain $n$ the quotient stack of the abelian variety $E_0^{\times n}$ by a well-chosen infinite group has a non-degenerate conjugate spectral sequence. This non-degeneracy will arise from the results of Section \ref{AV coherent: section} by contrasting the extensions in the canonical filtrations on de Rham and Hodge cohomology of $E_0^{\times n}$, viewed as complexes with a group action.

Let $F$ be a number field provided by Proposition \ref{group cohomology: ring of integers main}. It is a quadratic extension of $\bQ$ such that $\cO_F/p\simeq\bF_{p^2}=\bF_q$. Consider the abelian scheme $A:=E\otimes_{\bZ}\cO_F^{\oplus p}$ equipped with the natural action of the group $GL_p(\cO_F)$. The `$\otimes$' symbol refers here to Serre's tensor product, cf. \cite[1.7.4]{conrad-cm}. Explicitly, choosing a $\bZ$-basis in $\cO_F$ we get an identification $A\simeq E^{\times 2p}$ and the group $GL_p(\cO_F)$ acts through the embedding $GL_p(\cO_F)\hookrightarrow GL_{2p}(\bZ)$.

We can view the multiplication-by-$p$ map $A\xrightarrow{[p]}A$ as an $A[p]$-torsor on $A$ and therefore get a $GL_p(\cO_F)$-equivariant classifying map $A\to BA[p]$. The action of $GL_p(\cO_F)$ on the $p$-torsion group scheme $A[p]$ factors through $GL_p(\cO_F/p)=GL_p(\bF_q)$ and $A[p]$ is $GL_p(\bF_q)$-equivariantly isomorphic to $H$. Hence the classifying map can be viewed as a $GL_p(\cO_F)$-equivariant morphism $A\to BH$ and therefore induces a morphism \begin{equation}f:[A/SL_p(\cO_F)]\to [(BH)/SL_p(\bF_q)]\simeq B(SL_p(\bF_q)\ltimes H)=BG\end{equation} of quotient stacks. Pullback along $f$ induces a morphism between conjugate spectral sequences of the special fibers of these stacks.  In particular, there is a commutative square
\begin{equation}
\begin{tikzcd}
H^0(BG^{(1)}_0,L\Omega^p_{BG^{(1)}_0})\arrow[r,"f_0^*"]\arrow[d, "d^{0,p}_{p+1}"] & H^0([A^{(1)}_0/SL_p(\cO_F)],L\Omega^p_{[A^{(1)}_0/SL_p(\cO_F)]})\arrow[d, "d^{0,p}_{p+1}"] \\
H^{p+1}(BG^{(1)}_0,\cO)\arrow[r,"f_0^*"] & H^{p+1}([A^{(1)}_0/SL_p(\cO_F)],\cO).
\end{tikzcd}
\end{equation}

Therefore the following Lemma \ref{nondeg example: classifying map on hodge coh} and Proposition \ref{nondeg example: AV quotient} imply Theorem \ref{nondeg example: maint thm stack}.
 
 \begin{lm}\label{nondeg example: classifying map on hodge coh}
 The map \begin{equation}f_0^*:H^0(BG_0,L\Omega^p_{BG_0/k})\to H^0([A_0/SL_p(\cO_F)],L\Omega^p_{[A_0/SL_p(\cO_F)]})\end{equation} is an isomorphism.
 \end{lm}
 
 \begin{pr}\label{nondeg example: AV quotient}
The differential $d_{p+1}^{0,p}:H^0(Y^{(1)}_0,L\Omega^p_{Y_0/k})\to H^{p+1}(Y^{(1)}_0,\cO)$ in the conjugate spectral sequence for $Y_0=[A_0/SL_p(\cO_F)]$ is non-zero.
 \end{pr}
\begin{rem}
A trivialization of the cotangent bundle to the abelian variety $A_0$ provides a decomposition of the de Rham complex of $A_0$ \cite[Remarque 2.6 (iv)]{deligne-illusie}. Proposition \ref{nondeg example: AV quotient} demonstrates that this decomposition in general cannot be chosen to be compatible with the action of the group of automorphisms of the lift $A$.
\end{rem}
 
 \begin{proof}[Proof of Lemma \ref{nondeg example: classifying map on hodge coh}]

For a finite group scheme $\Gamma$ over $k$ denote by $\pi_{\Gamma}:\Spec k\to B\Gamma$ the natural quotient morphism, and by $e_{\Gamma}:\Spec k \to \Gamma$ the identity section. Recall that the cotangent complex $L\Omega^1_{B\Gamma/k}$ is described as the shift $e_{\Gamma}^*L\Omega^1_{\Gamma/k}[-1]$ of the co-Lie complex of $\Gamma$, where we view $L\Omega^1_{B\Gamma/k}\in D^+(B\Gamma)$ as an object of the derived  category of representations of $\Gamma$. Recall that since the scheme $\Gamma$ is a locally complete intersection over $k$, the cotangent complex $L\Omega^1_{\Gamma/k}$ has Tor-amplitude $[-1,0]$.
 
In particular, $L\Omega^{1}_{BG_0/k}$ is concentrated in degrees $\geq 0$, its $p$th exterior power $L\Omega^p_{BG_0/k}$ is concentrated in non-negative degrees as well, so $H^0(BG_0,L\Omega^p_{BG_0/k})$ is simply the invariant subspace $(H^0(\pi_{G_0}^*L\Omega^p_{BG_0/k}))^{G_0}=(H^0(\pi_{G_0}^*L\Omega^p_{BG_0/k}))^{SL_p(\bF_q)}$. In the last equality we used that $H_0$ is commutative, hence the adjoint action of $G_0$ on the cotangent complex $\pi_{G_0}^*L\Omega^p_{BG_0}\simeq \pi_{H_0}^*L\Omega^p_{BH_0}$ factors through $SL_p(\bF_q)=G_0/H_0$. Similarly, $H^0([A_0/SL_p(\cO_F)],L\Omega^p_{[A_0/SL_p(\cO_F)]})=H^0(A_0,\Omega^p_{A_0})^{SL_p(\cO_F)}=H^0(A_0,\Omega^p_{A_0})^{SL_p(\bF_q)}$.
 
By Lemma \ref{nondeg example: p torsion cotangent complex} below, $f_0$ induces an isomorphism $H^0(f_0^*L\Omega^1_{BG_0})\simeq \Omega^1_{A_0}$ which upon passing to derived $p$-th exterior powers induces an isomorphism $H^0(f_0^*L\Omega^p_{BG_0})\simeq \Omega^p_{A_0}$, because the map $\Lambda^p H^0(L\Omega^1_{BG_0})\to H^0(\Lambda^p L\Omega^1_{BG_0})$ is an isomorphism, as follows from the explicit description of the cotangent complex of $BG_0$ provided by the first part of Lemma \ref{nondeg example: p torsion cotangent complex}. Therefore it induces an isomorphism on the subspaces of $SL_p(\bF_q)$-invariants, which finishes the proof of Lemma \ref{nondeg example: classifying map on hodge coh}.  \end{proof}
 
 \begin{lm}\label{nondeg example: p torsion cotangent complex}
 Let $q_T:T\to S$ be an abelian scheme over a base $S$ such that $p\cO_S=0$. We denote the dual Lie algebra $e^*\Omega^1_{T/S}$ by $\omega_{T}$, where $e:S\to T$ is the identity section. Denote by $q:BT[p]\to S$ the structure morphism. The cotangent complex of $BT[p]$ can be described as $L\Omega^1_{BT[p]/S}\simeq q^*\omega_T\oplus q^*\omega_T[-1]$, equipped with a trivial $T[p]$-action.
 
 The classifying map $f:T\to BT[p]$ corresponding to the torsor $T\xrightarrow{[p]}T$ induces the map $df:f^*L\Omega^1_{BT[p]/S}\to \Omega^1_{T/S}$ that gives an isomorphism $H^0(f^*L\Omega^1_{BT[p]/S})\simeq \Omega^1_{T/S}$.
 \end{lm}
 
 \begin{proof}
By definition, the classifying map fits into the Cartesian square
\begin{equation}\label{nondeg example: p torsion cotangent complex diagram}
\begin{tikzcd}
T\arrow[r,"q_T"]\arrow[d, "{[}p{]}"] & S \arrow[d, "\pi"] \\
T\arrow[r, "f"] & BT[p]
\end{tikzcd}
\end{equation}
that induces an equivalence $L\Omega^1_{T\xrightarrow{[p]}T}\simeq q_T^*L\Omega^1_{S/BT[p]}$. In turn, the base change along for the Cartesian square
\begin{equation}
\begin{tikzcd}
T[p]\arrow[r]\arrow[d] & S\arrow[d, "\pi"] \\
S\arrow[r, "\pi"] & BT[p]
\end{tikzcd}
\end{equation}
implies that $L\Omega^1_{S/BT[p]}$ is equivalent to $e^*L\Omega^1_{T[p]/S}$, therefore $L\Omega^1_{T\xrightarrow{[p]}T}$ is equivalent to $q_T^*e^*L\Omega^1_{T[p]/S}$.

Transitivity triangles for the two vertical morphisms in (\ref{nondeg example: p torsion cotangent complex diagram}) fit into the following commutative diagram of sheaves on $T$
\begin{equation}\label{nondeg example: p torsion cotangent complex diagram2}
\begin{tikzcd}
q_T^*\pi^*L\Omega^1_{BT[p]/S}\arrow[r]\arrow[d] &  0\arrow[r]\arrow[d] & q_T^*e^*L\Omega^1_{T[p]/S}\arrow[d,"\sim"] \\
\left[p\right]^*\Omega^1_{T/S}\arrow[r, "{d[p]}"] & \Omega^1_{T/S}\arrow[r] & q_T^*e^*L\Omega^1_{T[p]/S}. 
\end{tikzcd}
\end{equation}

We have $d[p]=0$, and the bottom triangle gives that $e^*L\Omega^1_{T[p]/S}\simeq \omega_T[1]\oplus\omega_T$ which implies that $L\Omega^1_{BT[p]/S}\simeq q^*\omega_T\oplus q^*\omega_T[-1]$. The left vertical map in (\ref{nondeg example: p torsion cotangent complex diagram2}) is the pullback of the map $df:f^*L\Omega^1_{BT[p]/S}\to \Omega^1_{T/S}$ along $[p]:T\to T$. Since $[p]$ is a faithfully flat map, to prove that $df$ induces an isomorphism on $H^0$ it is enough to prove that the left vertical map in (\ref{nondeg example: p torsion cotangent complex diagram2}) induces an isomorphism on $H^0$. But this follows from the fact that the right vertical map in this diagram is an equivalence.
\end{proof}

To prove Proposition \ref{nondeg example: AV quotient} we observe that the conjugate spectral sequence of a quotient stack of a scheme $X_0$ by an action of a discrete group $\Gamma$ has a non-zero differential in a certain situation where the Hodge cohomology of $X_0$ is $\Gamma$-equivariantly decomposable while the de Rham cohomology is not. The precise statement is:

\begin{lm}\label{nondeg example: quotient differential}
Let $X_0$ be a smooth scheme over $k$ equipped with an action of a discrete group $\Gamma$ such that $X_0$ admits a lift $X_1$ over $W_2(k)$ to which the action of $\Gamma$ lifts. Suppose that for all $i\leq p$ the complex $\tau^{\leq p-i}\RG(X_0,\Omega^i_{X_0/k})$ is $\Gamma$-equivariantly equivalent to $\bigoplus\limits^{p-i}_{j=0}H^j(X,\Omega^i_{X/k})[-j]$ but the map $H^p_{\dR}([X_0/\Gamma]/k)\to H^p_{\dR}(X_0/k)^{\Gamma}$ is not surjective. Then the conjugate spectral sequence for the stack $Y_0=[X_0/\Gamma]$ has a non-zero differential $d^{0,p}_{p+1}:H^0(Y^{(1)}_0,L\Omega^p_{Y^{(1)}_0/k})\to H^{p+1}(Y^{(1)}_0,\cO_{Y^{(1)}_0})$ on the $(p+1)$th page.
\end{lm}

\begin{proof}
For an algebraic stack $Z$ over $k$ that satisfies $L\Omega^i_{Z/k}\in D^{\geq 0}(Z)$ for all $i$, we have $H^p_{\dR}(Z/k)=\Fil_p^{\conj}H^p_{\dR}(Z/k)$, and we denote by $a_{Z,p}:H^p_{\dR}(Z/k)\to H^0(Z^{(1)},L\Omega^p_{Z^{(1)}/k})$ the map on cohomology induced by the morphism $\Fil_p^{\conj}\dR_Z\to\gr_p^{\conj}\dR_Z\simeq L\Omega^p_{Z^{(1)}/k}[-p]$. Note that both $X_0$ and $[X_0/\Gamma]$ satisfy this condition.

The conclusion of the lemma is equivalent to saying that the map $a_{[X_0/\Gamma],p}:H^p_{\dR}([X_0/\Gamma]/k)\to H^0([X^{(1)}_0/\Gamma],L\Omega^p_{[X^{(1)}_0/\Gamma]/k})=H^0(X^{(1)}_0,\Omega^p_{X^{(1)}_0/k})^{\Gamma}$ is not surjective. Consider the commutative diagram induced by the pullback along the map $\pi:X_0\to [X_0/\Gamma]$:

\begin{equation}
\begin{tikzcd}
0\to \Fil_{p-1}^{\conj}H^p_{\dR}([X_0/\Gamma])\arrow[r]\arrow[d] & H^p_{\dR}([X_0/\Gamma]) \arrow[r, "a_{[X_0/\Gamma],p}"]\arrow[d] & H^0([X^{(1)}_0/\Gamma],L\Omega^p_{[X_0^{(1)}/\Gamma]})\arrow[d,equal]  \\
0\to (\Fil_{p-1}^{\conj}H^p_{\dR}(X_0))^{\Gamma}\arrow[r] & H^p_{\dR}(X_0)^{\Gamma}\arrow[r,"a_{X_0,p}"]& H^0(X^{(1)}_0,\Omega^p_{X^{(1)}_0})^{\Gamma}.
\end{tikzcd}
\end{equation}

Both rows are exact on the left and in the middle, and our goal is to show that the top row is not exact on the right. For the sake of a contradiction, assume that $a_{[X_0/\Gamma],p}$ is surjective. Then the second map in the bottom row is of course surjective as well, so both rows are exact sequences. However, we will now check that the left vertical map is surjective which will yield the contradiction with our assumption that the middle vertical map is not surjective.

By the assumption that $X_0$ lifts to $W_2(k)$ together with the group action, we have splittings $\Fil^{\conj}_{p-1}H^p_{\dR}([X_0/\Gamma])\simeq \bigoplus\limits_{i=0}^{p-1}H^{p-i}([X^{(1)}_0/\Gamma],L\Omega^i_{[X^{(1)}_0/\Gamma]})$ and $\Fil_{p-1}^{\conj}H^p_{\dR}(X_0)\simeq\bigoplus\limits_{i=0}^{p-1} H^{p-i}(X^{(1)}_0,\Omega^i_{X^{(1)}_0})$ (the second splitting is moreover $\Gamma$-equivariant), and the map $\pi^*$ is compatible with these decompositions. Thus to prove that the map $\Fil^{\conj}_{p-1}H^p_{\dR}([X_0/\Gamma])\to (\Fil^{\conj}_{p-1}H^p_{\dR}(X_0))^{\Gamma}$ is surjective it is enough to show that $H^{p-i}([X^{(1)}_0/\Gamma],L\Omega^i_{[X_0^{(1)}/\Gamma]})\to H^{p-i}(X_0,\Omega^i_{X_0^{(1)}})^{\Gamma}$ is surjective for all $0\leq i\leq p-1$. The group $H^{p-i}([X^{(1)}_0/\Gamma],L\Omega^i_{[X^{(1)}_0/\Gamma]})$ is identified with the group cohomology $H^{p-i}(\Gamma, \RGamma(X^{(1)}_0,\Omega^i_{X^{(1)}_0}))$ so this surjectivity is implied by the fact that each of the complexes $\tau^{\leq p-i}\RGamma(X_0,\Omega^i_{X^{(1)}_0})$ is $\Gamma$-equivariantly decomposable. 
\end{proof}

\begin{proof}[Proof of Proposition \ref{nondeg example: AV quotient}.] The results of Section \ref{AV coherent: section} put us in a position to apply Lemma \ref{nondeg example: quotient differential} to the action of $SL_p(\cO_F)$ on $A_0\simeq E_0\otimes_{\bZ}\cO_F^{\oplus p}$. On the one hand, since $A_0$ is a product of supersingular elliptic curves, $F_{A_0}^*$ is zero on the coherent cohomology group $H^1(A_0, \cO)$, so Proposition \ref{AV coherent: coherent main} asserts that $\tau^{\leq p}\RGamma(A,\cO)$ is $SL_p(\cO_F)$-equivariantly decomposable, and hence so is the mod $p$ reduction $\tau^{\leq p}\RGamma(A_0,\cO)$ of this complex. This implies that the cohomology complexes $\tau^{\leq p}\RGamma(A_0,\Omega^i_{A_0})$ are decomposable as well because $\Omega^i_{A_0}$ is trivial as a $SL_p(\cO_F)$-equivariant bundle, so $\RGamma(A_0,\Omega^i_{A_0})$ is $SL_p(\cO_F)$-equivariantly quasi-isomorphic to $\RGamma(A_0,\cO)\otimes_{k}\Lambda^i(\Lie A_0)^{\vee}$. Hence the assumptions on decomposability of the Hodge cohomology in Lemma \ref{nondeg example: quotient differential} are satisfied.

On the other hand, using Proposition \ref{AV coherent: AV de rham}, we will see that the map $H^p_{\dR}([A_0/SL_p(\cO_F)])\to H^p_{\dR}(A_0)^{SL_p(\cO_F)}$ is not surjective. This map fits into a long exact sequence 

\begin{equation}
\ldots\to H^p_{\dR}([A_0/SL_p(\cO_F)])\to H^p_{\dR}(A_0)^{SL_p(\cO_F)}\xrightarrow{\delta} H^{p+1}(SL_p(\cO_F),\bigoplus\limits_{i=0}^{p-1} H^i_{\dR}(A_0)[-i])\to\ldots
\end{equation}
Here $\delta$ lands in the direct summand $H^p(SL_p(\cO_F),H^1_{\dR}(A_0))$ and it is given by the composition
\begin{multline}
(\Lambda^p H^1_{\dR}(A_0))^{SL_p(\cO_F)}\xrightarrow{\Bock^{p-1}(\alpha(H^1_{\dR}(A_0/k)))}H^p(SL_p(\cO_F),H^1_{\dR}(A_0)^{(1)})\xrightarrow{F^*_{A_0}} \\ \xrightarrow{F^*_{A_0}} H^p(SL_p(\cO_F),H^1_{\dR}(A_0))
\end{multline}
in the notation of Proposition \ref{AV coherent: AV de rham}. We will prove that $\delta$ is non-zero, thus checking that $H^p_{\dR}([A_0/SL_p(\cO_F)])\to H^p_{\dR}(A_0)^{SL_p(\cO_F)}$ is not surjective. 

Non-vanishing of $\delta$ can be checked after replacing the base field $k$ by a finite extension, hence we may assume that $k$ contains $\bF_{q}=\bF_{p^2}$. The $k$-vector space $H^1_{\dR}(A_0)$ can be $SL_p(\cO_F)$-equivariantly identified with \begin{equation}H^1_{\dR}(E_0)\otimes_{\bZ}\cO_F^{\oplus p}=H^1_{\dR}(E_0)\otimes_k (k\otimes_{\bF_p}\cO_F^{\oplus p}/p).\end{equation} The group $SL_p(\cO_F)$ acts on the RHS via its tautological action on $\cO_F^{\oplus p}/p=\bF_q^{\oplus p}$. The $k$-vector space $k\otimes_{\bF_p}\cO_F^{\oplus p}/p$ is $SL_p(\cO_F)$-equivariantly isomorphic to $\bigoplus\limits_{\tau\in \Gal(F/\bQ)} V^{\tau}$ where $V^{\tau}$ is isomorphic to $V:=\bF_q^{\oplus p}\otimes_{\bF_q}k$ as a $k$-vector space but the $SL_p(\cO_F)$-action is modified by precomposing with the automorphism $\tau\in \Gal(F/\bQ)\simeq \bZ/2$. We get the following $SL_p(\cO_F)$-equivariant description of 1st de Rham cohomology of $A_0$:
\begin{equation}
H^1_{\dR}(A_0)\simeq H^1_{\dR}(E_0)\otimes_k \bigoplus\limits_{\tau\in \Gal(F/\bQ)} V^{\tau}.
\end{equation}

Recall that the Frobenius map $F_{E_0}^*:H^1_{\dR}(E_0/k)^{(1)}\to H^1_{\dR}(E_0/k)$ on de Rham cohomology of $E_0$ is a non-zero operator of rank $1$, cf. \cite[Corollary 5.11]{oda}. Choose any element $\xi\in H^1_{\dR}(E_0/k)$ such that $F^*_{E_0}(\xi)\neq 0$, and choose an arbitrary lift $\widetilde{\xi}\in H^1_{\cris}(E_0/W_2(k))$ of $\xi$. This defines a split $SL_p(\cO_F)$-equivariant embedding $\iota_{\xi}:V\xrightarrow{\xi\otimes \id_V}H^1_{\dR}(E_0)\otimes_k V\subset H^1_{\dR}(A_0)$ that lifts to an embedding $\tV\xrightarrow{}H^1_{\cris}(A_0/W_2(k))$. We have a commutative diagram
\begin{equation}\label{nondeg example: one element extension diagram}
\begin{tikzcd}
k=(\Lambda^p V)^{SL_p(\cO_F)}\arrow[rr,"\Bock^{p-1}_{\tV}(\alpha(V))"]\arrow[d, "\Lambda^p\iota_{\xi}"] & & H^{p}(SL_p(\cO_F),V^{(1)})\arrow[r, "\iota_{\xi}^{(1)}"] & H^{p}(SL_p(\cO_F),H^1_{\dR}(A_0)^{(1)})\arrow[d, "F_{A_0}^*"] \\
H^p_{\dR}(A_0)^{SL_p(\cO_F)}\arrow[rrr, "\delta"] & & & H^p(SL_p(\cO_F),H^1_{\dR}(A_0)).
\end{tikzcd}
\end{equation}

By Proposition \ref{group cohomology: ring of integers main} the class $\Bock_{\tV}^{p-1}(\alpha(V))$ induces a non-zero map $k=(\Lambda^p V)^{SL_p(\cO_F)}\to H^p(SL_p(\cO_F),V^{(1)})$. By our choice of $\xi$ the composition $F_{A_0}^*\circ\iota_{\xi}^{(1)}:V^{(1)}\to H^1_{\dR}(A_0)$ is a split injection. Therefore the composition of the top row and right vertical arrow in (\ref{nondeg example: one element extension diagram}) is injective, so $\delta$ is non-zero. 

We have thus checked that all the conditions of Lemma \ref{nondeg example: quotient differential} are satisfied, hence the differential $d^{0,p}_{p+1}:H^0(Y^{(1)}_0,L\Omega^p_{Y^{(1)}_0/k})\to H^{p+1}(Y^{(1)}_0,\cO)$ in the conjugate spectral sequence for the stack $Y_0=[A_0/SL_p(\cO_F)]$ is non-zero, as desired.
\end{proof}

Theorem \ref{nondeg example: maint thm stack} is therefore proven, by combining Proposition \ref{nondeg example: AV quotient} with Lemma \ref{nondeg example: classifying map on hodge coh}. We can now produce an example of a liftable smooth projective variety with a non-degenerate conjugate spectral sequence:

\begin{cor}\label{nondeg example: main corollary}
Let $k$ be any perfect field of characteristic $p$. There exists a smooth projective scheme $X$ over $W(k)$ of relative dimension $p+1$, such that the differential $H^0(X^{(1)}_0,\Omega^p_{X^{(1)}_0})\to H^{p+1}(X^{(1)}_0,\cO)$ on the $(p+1)$st page of the conjugate spectral sequence of $X_0$ is non-zero.
\end{cor}

\begin{proof}
We will use the classical technique of approximating the classifying stack of a finite flat group scheme by a projective variety, originated by Serre \cite{serre}, cf. \cite[Theorem 1.2]{antieau-bhatt-mathew}.

We choose, using Lemma \ref{nondeg example: approximation finite fields} below, a complete intersection $Z\subset\bP^N_{W(k)}$ of relative dimension $p+1$ over $W(k)$, equipped with a free action of the group scheme $G$, such that $Z/G$ is smooth over $W(k)$. Our example is $X:=Z/G$, it is a smooth projective scheme over $W(k)$ equipped with the classifying map $f:X\to BG$. The induced map $f_0^*:H^{p+1}(BG_0,\cO)\to H^{p+1}(X_0,\cO)$ is seen to be injective by considering the Hochschild-Serre spectral sequence for the morphism $Z_0\to Z_0/G_0$ using the fact that $H^i(Z_0,\cO)=0$ for $i\leq p$. Therefore the differential $H^0(X^{(1)}_0,\Omega^p_{X^{(1)}_0})\to H^{p+1}(X^{(1)}_0,\cO)$ is non-zero.
\end{proof}

\begin{lm}\label{nondeg example: approximation finite fields}
Let $k$ be any perfect field of characteristic $p$. For any finite flat group scheme $\Gamma$ over $W(k)$ and an integer $d\geq 0$ there exists a complete intersection $Z\subset\bP^N_{W(k)}$ of relative dimension $d$ equipped with a free action of $\Gamma$ such that the quotient $Z/\Gamma$ is a smooth scheme.
\end{lm}

\begin{proof}
If $k$ is infinite, this is proven in \cite[2.7-2.9]{bms}. For a finite $k$ we will check that the same construction goes through if we appeal to a Bertini theorem over finite fields due to Gabber and Poonen. \cite[Lemma 2.7]{bms} provides us with an action of $\Gamma$ on a projective space $\bP^N_{W(k)}$ such that there is an open subscheme $U\subset\bP^N_{W(k)}$ preserved by $\Gamma$ on which $\Gamma$ acts freely, the quotient $U/\Gamma$ is smooth over $W(k)$, and the dimension of every component of both fibers of $\bP^N_{W(k)}\setminus U$ is at most $N-d-1$. Consider the quotient scheme $Q:=\bP^N_{W(k)}/\Gamma$ and denote by $M\subset Q$ the image of $\bP^N_{W(k)}\setminus U$ under the quotient map. By construction, $Q\setminus M$ is smooth over $W(k)$, and $M$ has fiber-wise codimension $\geq d+1$ inside $Q$.

Following the argument below \cite[Lemma 2.9]{bms}, it is enough for us to find very ample line bundles $\cL_1,\ldots,\cL_{N-d}$ on $Q$ and sections $s_i\in H^0(Q, \cL_i)$ such that their common vanishing locus is smooth and does not intersect $M$. By induction on the number $N-d$, it is enough to produce a very ample line bundle $\cL$ on $Q$ and a section $s\in H^0(Q,\cL)$ such that $V(s)\cap M\subset V(s)$ is of fiber-wise codimension $\geq d+1$, and $V(s)\setminus V(s)\cap M$ is smooth. Note that both conditions can be checked just on the special fiber. 

Let $\cL$ be some very ample line bundle on $Q$. We choose a closed point on every irreducible component of $M_0$, and denote $B$ the set comprised of these points. Let $n_0$ be an integer such that $H^1(Q, \cL^{\otimes n})=0$ for all $n\geq n_0$. Applying \cite[Corollary 1.6]{gabber}, for some $n>n_0$ we can find a section $s\in H^0(Q_0,\cL^{\otimes n})$ of the line bundle $\cL^{\otimes n}$ on the special fiber $ Q_0:=Q\times_{W(k)}k$, such that $V(s)\setminus M_0\cap V(s)$ is smooth of dimension $N-1$, and $V(s)$ does not contain any of the points from $B$. The latter property implies that the codimension of $M_0\cap V(s)$ inside $V(s)$ is at least $d+1$. By our assumption on the vanishing of $H^1(Q,\cL^{\otimes n})$, we can lift $s$ to a section of $\cL^{\otimes n}$ on $Q$ whose vanishing locus satisfies the desired properties.

Having constructed the sections $s_1,\ldots, s_{N-d}$, we define $Z\subset \bP^N_{W(k)}$ to be the preimage of $V(s_1,\ldots,s_{N-d})\subset Q\setminus M\subset Q=\bP^N_{W(k)}/\Gamma$. The action of $\Gamma$ on $Z$ is free and the quotient $Z/\Gamma=V(s_1,\ldots,s_{N-d})$ is smooth, by the choice of sections $s_1,\ldots,s_{N-d}$.
\end{proof}

\begin{rem}
\begin{enumerate}[leftmargin=*]
\item If in the construction of the abelian scheme $A$ we used an elliptic curve $E$ with {\it ordinary} reduction, the statement of Proposition \ref{nondeg example: AV quotient} would be false. Indeed, if $E_0$ is ordinary then the stack $[A_0/SL_p(\cO_F)]$ admits a lift to $W_2(k)$ together with its Frobenius endomorphism, hence the conjugate filtration is split in all degrees by \cite[Remarque 2.2(ii)]{deligne-illusie}. The part of the proof showing that the map $H^p_{\dR}([A_0/SL_p(\cO_F)])\to H^p_{\dR}(A_0)^{SL_p(\cO_F)}$ is not surjective goes through just as well because it only used that $\varphi_{E_0}^*$ is non-zero on $H^1_{\dR}(E_0)$, but it is no longer true that $\tau^{\leq p}\RGamma(A_0,\cO)$ is $SL_p(\cO_F)$-equivariantly decomposable, by Proposition \ref{AV coherent: coherent main}. So both Hodge and de Rham cohomology of $A_0$ are not $SL_p(\cO_F)$-equivariantly decomposable, but the conjugate filtration on $\RGamma_{\dR}(A_0)$ is $SL_p(\cO_F)$-equivariantly split.

\item As in the proof of \cite[Corollaire 2.4]{deligne-illusie}, the variety $X_0$ in Corollary \ref{nondeg example: main corollary} necessarily has a non-zero differential in its Hodge-to-de Rham spectral sequence.

\item There are several examples in the literature of smooth projective varieties over $k$ with a non-degenerate Hodge-to-de Rham spectral sequence that lift to some ramified extension of $W(k)$, but to the best of my understanding none of them lift to $W_2(k)$.
\end{enumerate} 
\end{rem}

The class of examples constructed in this section can also be used to illustrate the fact that the differential $d^{p+1}_{0,p}$ in the conjugate spectral sequence might vanish even if the map $c_{X_1,p}$ constructed out of the Sen operator induces a non-zero map on cohomology. 

\begin{pr}
There exists a smooth projective scheme $Y$ over $W(k)$ such that the map $H^0(Y_0,\Omega^{p}_{Y_0})\xrightarrow{c_{Y_1,p}}H^p(Y_0,\cO)$ is non-zero yet the differential $d_{0,p}^{p+1}:H^0(Y_0,\Omega^p_{Y_0})\to H^{p+1}(Y_0,\cO)$ in the conjugate spectral sequence of $Y_0$ is zero. Moreover, we can choose $Y$ such that there exists another smooth projective $Y'$ with $Y'_0\simeq Y_0$ and the class $c_{Y'_1,p}$ vanishes.
\end{pr}
\begin{proof}
Let $E$ be an elliptic curve over $W(k)$ such that its reduction $E_0$ is {\it ordinary}. We repeat the constructions of $G$ and $A$ with this $E$ as an input: $G$ is a finite flat group scheme given by the semi-direct product $SL_p(\bF_{p^2})\ltimes E[p]^{\oplus 2p}$, and $A$ is the abelian scheme $E^{\times 2p}=E\otimes_{\bZ}\cO_F^{\oplus p}$ endowed with the action of the group $\Gamma:=SL_p(\cO_F)$ for an appropriate real quadratic field $F\supset\bQ$ provided by Proposition \ref{group cohomology: ring of integers main}.

We will show that the stack $BG$ satisfies the properties in the statement of the proposition as soon as $E\times_{W(k)}W_2(k)$ is not the canonical lift of $E_0$. The vanishing of $d_{0,p}^{p+1}$ holds because $BG_0$ admits a lift over $W_2(k)$ together with its Frobenius endomorphism (constructed out of the canonical lift of the ordinary elliptic curve $E_0$). On the other hand, $c_{BG_1,p}$ depends not only on the mod $p$ stack $BG_0$, but also on $BG_1=BG\times_{W(k)}W_2(k)$, and we will show that the map $H^0(BG_0,L\Omega^p_{BG_0})\to H^p(BG_0,\cO)$ it induces is not zero when $E_1=E\times_{W(k)}W_2(k)$ does not admit a lift of the Frobenius endomorphism of $E_0$.

By Proposition \ref{nondeg example: classifying map on hodge coh} it suffices to check that analogous non-vanishing for the stack $X:=[A/\Gamma]$. We will use a generalization of Theorem \ref{semiperf: main sen operator} to algebraic stacks, obtained by descent: the map $c_{X_1,p}$ is the composition of $\alpha(\Omega^1_{X_0})$ with $\ob_{F, X_1}$ where, by definition, $\ob_{F, X_1}$ is the class in $H^1(X_0, F_{X_0}^*T_{X_0})$ obtained by descent from obstructions to lifting Frobenius onto smooth $W_2(k)$-schemes mapping to $X_1$. We will not appeal to any interpretation of $\ob_{F, X_1}$ in terms of deformation theory of the stack $X_1$ itself, but will only use that its image under the natural map $H^1(X_0, F^*_{X_0} T_{X_0})\to H^1(A_0, F^*_{A_0}T_{A_0})$ induced by the \'etale map $A_0\to X_0=[A_0/\Gamma]$ is the obstruction to lifting Frobenius onto $A_1=E_1^{\times 2p}$. 

We would like to check that the composition of the top row of the following diagram is non-zero:
\begin{equation}\label{nondeg example: ordinary diagram}
\begin{tikzcd}
H^0([A_0/\Gamma],L\Omega^p_{[A_0/\Gamma]})\arrow[rr, "{\alpha(L\Omega^1_{[A_0/\Gamma]})}"] & & H^{p-1}([A_0/\Gamma],F^*_{[A_0/\Gamma]}L\Omega^1_{[A_0/\Gamma]})\arrow[r,"{\ob_{F,[A_1/\Gamma]}}"] & H^p([A_0/\Gamma],\cO) \\
H^0(A_0,\Omega^p_{A_0})^{\Gamma}\arrow[rr,"{\alpha(H^0(A_0,\Omega^1_{A_0}))}"]\arrow[u,"\sim"] & & H^{p-1}(\Gamma,H^0(A_0,F_{A_0}^*\Omega^1_{A_0}))\arrow[r]\arrow[u]  & H^{p}(\Gamma,\tau^{\leq 1}\RGamma(A_0,\cO))\arrow[u]
\end{tikzcd}
\end{equation}
Here the vertical maps in the left commutative square are given by pullback along the map $X_0=[A_0/\Gamma]\to B\Gamma$, using that the cotangent bundle $L\Omega^1_{[A_0/\Gamma]}$ descends along it to the representation $H^0(A_0,\Omega^1_{A_0})$, viewed as a vector bundle on the classifying stack $B\Gamma$. The right commutative square is obtained from the map $\ob_{F, [A_1/\Gamma]}:\RGamma(A_0,F_{A_0}^*\Omega^1_{A_0})\to \RGamma(A_0,\cO)[1]$ in the derived category $D(\Rep_k \Gamma)$ by applying to it the morphism of functors $\tau^{\leq 0}\to \Id$ followed by $H^{p-1}(\Gamma,-)$.

By Proposition \ref{group cohomology: ring of integers main} the bottom left horizontal map $\alpha(H^0(A_0,\Omega^1_{A_0}))=\alpha(H^0(E_0,\Omega^1_{E_0})\otimes_{\bZ_p}\cO_F^{\oplus p})$ is non-zero. The composition
\begin{equation}
H^{p-1}(\Gamma,H^0(A_0,F_{A_0}^*\Omega^1_{A_0}))\to H^{p}(\Gamma,\tau^{\leq 1}\RGamma(A_0,\cO))\to H^{p-1}(\Gamma, H^1(A_0, \cO))
\end{equation}
is simply induced by the map $\ob_{F, A_1}: H^0(A_0,F_{A_0}^*\Omega^1_{A_0})\to H^1(A_0, \cO)$. Since we chose $E_1$ to be a lift of $E_0$ other than the canonical one, this map is an isomorphism. Hence it remains to ensure that the rightmost vertical map in (\ref{nondeg example: ordinary diagram}) is injective. It fits into a long exact sequence
\begin{equation}\label{nondeg example: o cohomology ordinary exact sequence}
H^{p-1}(\Gamma,\tau^{\geq 2}\RGamma(A_0,\cO))\xrightarrow{\delta} H^p(\Gamma,\tau^{\leq 1}\RGamma(A_0,\cO))\to H^p(\Gamma, \RGamma(A_0,\cO))
\end{equation}
where the first map $\delta$ is induced by the connecting homomorphism $\tau^{\geq 2}\RGamma(A_0,\cO)\to (\tau^{\leq 1}\RGamma(A_0,\cO))[1]$. By the mod $p$ reduction of Proposition \ref{AV coherent: coherent main}(0) this connecting homomorphism becomes zero when composed with the map $\tau^{[2,p-1]}\RGamma(A_0,\cO)\to \tau^{\geq 2}\RGamma(A_0,\cO)$. The latter map induces a surjection on $H^{p-1}$, hence the connecting homomorphism $\delta$ in (\ref{nondeg example: o cohomology ordinary exact sequence}) vanishes, and therefore the right vertical map in (\ref{nondeg example: ordinary diagram}) is injective, finishing the proof of non-vanishing of the map $c_{[A_1/\Gamma],p}:H^0([A_0/\Gamma],L\Omega^p_{[A_0/\Gamma]})\to H^p(A_0,\cO)$.

By Lemma \ref{nondeg example: classifying map on hodge coh} (its proof did not use supersingularity of the initial elliptic curve) the natural map $[A_0/\Gamma]\to BG_0$ induces an isomorphism on $H^0(-,\Omega^p)$, hence the map $c_{BG_1,p}:H^0(BG_0,L\Omega^p_{BG_0})\to H^p(BG_0,\cO)$ is non-zero as well. As we mentioned above, the differential $d^{p+1}_{0,p}$ for $BG_0$ is zero, so approximating $BG$ by a quotient of a complete intersection of dimension $\geq p+1$ by a free action of $G$ gives an example of the behavior described in the first sentence of the proposition. 

Finally, let us explain how to choose $Y$ so that there exists another lift $Y'$ such that the class $c_{Y_1',p}$ vanishes. Note that we already have an example of this behavior for smooth Artin stacks: $BG_0$ admits a lift to which Frobenius lifts, hence the Sen operator corresponding to that lift is semisimple. Let $\bG=SL_p(\bF_q)\ltimes \cE[p]^{\oplus 2p}$ be the group scheme over the universal local deformation ring $R\simeq W(k)[[t]]$ of $E_0$ (cf. \cite[Theorem 2.2.1, Theorem 2.3.3]{oort-deforms}) constructed out of the universal curve $\cE$ over $R$. By the same approximation procedure as in the proof of Corollary \ref{nondeg example: main corollary} (cf. \cite[Proposition 4.2.3]{raynaud}) we may find $\cY$ smooth projective over $R$ with a map $\cY\to B\bG$ inducing an isomorphism on $H^{\leq p+1}(-,\cO)$ on the fibers over every geometric point of $R$. Let $t_{\can}\in W(k)$ be the value of the parameter $t$ such that $\cE_{t_{\can}}$ is the canonical lift of $E_0$ over $W(k)$, and take $t_0\in W(k)$ to be any value such that $\cE_{t_0}\times_{W(k)}W_2(k)$ is not the canonical lift of $E_0$ over $W_2(k)$. Then taking $Y=\cY_{t_0}$ and $Y'=\cY_{t_{\can}}$ we obtain the desired behavior.
\end{proof}

\section{Extensions in higher degrees}\label{higher ext: section}

It would be interesting to generalize the methods of Section \ref{cosimp: section} to treat extensions in the canonical filtration on the de Rham complex in degrees $>p$. In this section, which is independent of the rest of the paper, we make some preliminary remarks, roughly amounting to extending part (1) of Theorem \ref{cosimp: main theorem} to higher degrees.

Suppose that $X$ is an arbitrary (not necessarily flat) scheme over $\bZ_{(p)}$ and let $A\in \DAlg(X)$ be a derived commutative algebra such that $H^0(A)=\cO_X$, $H^1(A)$ is a locally free sheaf, and multiplication induces isomorphisms $\Lambda^i H^1(A)\simeq H^i(A)$ for all $i$. 

\begin{lm}\label{higher ext: mult2splittings}
Given maps $s_i:\Lambda^i H^1(A)[-i]\to A,s_j:\Lambda^j H^1(A)[-j]\to A$ in $D(X)$ that induce isomorphisms on $i$th and $j$th cohomology respectively, there exists a map $s_{i+j}:\Lambda^{i+j}H^1(A)[-i-j]\to A$ that induces an isomorphism on $H^{i+j}$ if the integer $\binom{i+j}{i}$ is not divisible by $p$.
\end{lm}

\begin{proof}
For any object $M\in D(X)$, there is a natural map $m_{i,j}:S^iM\otimes S^jM\to S^{i+j}M$ that, in case $M$ is a projective module over a ring, is given by the surjection $(M^{\otimes (i+j)})_{S_i\times S_j}\to (M^{\otimes (i+j)})_{S_{i+j}}$ where we view $S_{i}\times S_j$ as a subgroup of $S_{i+j}$. Choosing a set of representatives $\Sigma\subset S_{i+j}$ for left cosets of this subgroup, we can define a map (that is independent of the choice of $\Sigma$) $N_{S_{i+j}/S_i\times S_j}:S^{i+j}M\to S^iM\otimes S^jM$ given by $\sum\limits_{\sigma\in \Sigma}\sigma$. The composition $S^{i+j}M\xrightarrow{N_{S_{i+j}/S_{i}\times S_j}}S^iM\otimes S^jM\xrightarrow{m_{i,j}}S^{i+j}M$ is equal to multiplication by $\binom{i+j}{i}=[S_i\times S_j:S_{i+j}]$.

Applying this to $M=H^1(A)[1]$ and using the decalage equivalences $S^i(H^1(A)[1])\simeq (\Lambda^i H^1(A))[i]$, we obtain a map $N_{S_{i+j}/S_i\times S_j}:\Lambda^{i+j}H^1(A)\to \Lambda^i H^1(A)\otimes\Lambda^j H^1(A)$ whose composition with the wedge product map is the multiplication by $\binom{i+j}{i}$ map on $\Lambda^{i+j}H^1(A)$.

Consider the product of sections $s_i$ and $s_j$ given by
\begin{equation}
s'_{i+j}:\Lambda^i H^1(A)[-i]\otimes\Lambda^j H^1(A)[-j]\xrightarrow{s_{i}\otimes s_j}A\otimes A\xrightarrow{m}A.
\end{equation}

It induces exterior multiplication $\Lambda^i H^1(A)\otimes \Lambda^j H^1(A)\to \Lambda^{i+j}H^1(A)$ when evaluated on $H^{i+j}$, so precomposing $s'_{i+j}$ with the map $N_{S_{i+j}/S_i\times S_j}$ gives rise to a map $s_{i+j}:\Lambda^{i+j}H^1(A)[-i-j]\to A$ that induces multiplication by $\binom{i+j}{i}$ on $H^{i+j}$. Under the assumption that $p\nmid\binom{i+j}{i}$ the map $s_{i+j}$ thus gives the desired splitting in degree $i+j$. 
\end{proof}

\begin{cor}\label{higher ext: extendingppower}
Given an integer $n\geq 0$, suppose that there exist morphisms $s_{p^i}:H^{p^i}(A)[-{p^i}]\to A$ inducing an isomorphism on $H^{p^i}$, for all $i=0,1,\dots n$. Then there exists an equivalence $\tau^{\leq p^{n+1}-1}A\simeq\bigoplus\limits_{i=0}^{p^{n+1}-1}H^i(A)[-i]$.
\end{cor}

\begin{proof}
We need to construct sections $s_m:H^m(A)[-m]\to A$ for all $0\leq m\leq p^{n+1}-1$. We will construct these by induction on $m$, the base case being $m=0$, where $s_0:H^0(A)\to A$ is simply the natural map that exists for every complex concentrated in non-negative degrees. Suppose that the sections $s_{m'}$ have been constructed for all $m'<m$. Let $m=a_rp^r+\dots+a_1p+a_0$ be the base $p$ expansion of $m$. The splitting $s_{p^r}$ is given to us and $s_{m-p^r}$ has already been constructed, so $s_m$ is provided by Lemma \ref{higher ext: mult2splittings} applied to $(i,j)=(p^r,m-p^r)$.
\end{proof}

\begin{example}
If $X_0$ is a smooth scheme over $k$ that lifts to a scheme $X_1$ over $W_2(k)$, such that the class $e_{X_1,p}$ vanishes, the conditions of Corollary \ref{higher ext: extendingppower} are satisfied for $A=\dR_{X_0/k}$ on $X_0$ with $n=1$, hence the truncation $\tau^{\leq p^2-1}\dR_{X_0/k}$ decomposes as a direct sum of its cohomology sheaves.
\end{example}

\renewcommand{\alg}{{}}
\section{Non-vanishing in rational group cohomology}
\label{rational: section}

In this section we prove non-vanishing results related to the class $\alpha$ of Definition \ref{free cosimplicial: alpha definition}. Let $V$ be a finite-dimensional $k$-vector space equipped with the tautological action of the algebraic group $GL(V)$, viewed as a group scheme over $k$. We denote by $V^{(1)}$ the vector space $V\otimes_{k,\Fr_p}k$ on which $GL(V)$ acts through the relative Frobenius $F_{GL(V)/k}:GL(V)\to GL(V)^{(1)}=GL(V^{(1)})$. Under the equivalence between representations of $GL(V)$ and quasicoherent sheaves on the classifying stack $BGL(V)$ the representation $V^{(1)}$ corresponds to the pullback $F_{BGL(V)}^*V$ under absolute Frobenius, because $V$ and $GL(V)$ can be descended to $\bF_p\subset k$.

The construction of Definition \ref{free cosimplicial: alpha definition} applied to the stack $BGL(V)$ with its tautological vector bundle associates to $V$ a class $\alpha(V)\in \Ext^{p-1}_{GL(V)}(\Lambda^p V,V^{(1)})$ where $V^{(1)}$ is the Frobenius twist of the representation $V$. Recall that $\alpha(V)$ is the extension class of
\begin{equation}\label{rational group cohomology: alpha formula}
V^{(1)}[-2]\to \tau^{\geq 2}S^p(V[-1])\to\Lambda^p V[-p].
\end{equation}

We will start by proving that this extension is generally non-split when viewed as an extension of algebraic representations 

\begin{pr}\label{rational group cohomology: main non-vanishing}
If $\dim V=p$ then the class $\alpha(V)|_{SL(V)}\in \Ext^{p-1}_{SL(V)}(\Lambda^p V, V^{(1)})=H^{p-1}_{\alg}(SL(V),V^{(1)})$ is non-zero.
\end{pr}

\begin{rem}
In this section, we will work in the bounded below derived category $D^+(\Rep_{SL(V)})$ of the abelian category of $k[SL(V)]$-comodules. It is equivalent to $D^+(BSL(V))$, and we may view $\alpha(V)$ as a morphism in $D^+(\Rep_{SL(V)})$.
\end{rem}

\begin{rem}\label{group cohomology: p=2 remark}
When $p=2$, an explicit model for $\tau^{\geq 2}S^p(V[-1])$ is at our disposal by Lemma \ref{free cosimplicial: alpha for p=2}. In this case Proposition \ref{rational group cohomology: main non-vanishing} is asserting that the extension $$0\to V^{(1)}\to S^2V\to \Lambda^2V\to 0$$ does not admit an $SL(V)$-equivariant section, for a $2$-dimensional space $V$, and this is elementary to show. Indeed, if $e_1,e_2$ is a basis in $V$ then a matrix $\left(\begin{matrix}a & b\\ c & d\end{matrix}\right)$ satisfying $ad-bc=1$ sends an element $\lambda_{11}e_1^2+\lambda_{12}e_1e_2+\lambda_{22}e_2^2\in S^2V$ to $(\lambda_{11}a^2+\lambda_{12}\cdot ab+\lambda_{22}b^2)e_1^2+\lambda_{12}e_1e_2+(\lambda_{11}c^2+\lambda_{12}\cdot cd+\lambda_{22}d^2)e_2^2$, so the invariants $(S^2V)^{SL_2}$ vanish. But the action of $SL_2$ on $\Lambda^2 V$ is trivial, so the surjection $S^2V\to \Lambda^2V$ does not admit a section. 
\end{rem}

We will need the following facts about cohomology of representations of $SL(V)$:
\begin{lm}\label{rational group cohomology: basic facts}
Let $V$ be a $p$-dimensional vector space over $k$.
\begin{enumerate}
\item For any $n\geq 0$ the cohomology $H^i(SL(V),V^{\otimes n})$ vanishes for all $i>0$.

\item The antisymmetrization map $\psi:\Lambda^pV\to V^{\otimes p}$ defined by $\psi(v_1\wedge\ldots\wedge v_p)=\sum\limits_{\sigma\in S_p}\sgn(\sigma)v_{\sigma(1)}\otimes\ldots\otimes v_{\sigma(p)}$ induces an isomorphism $k\simeq \Lambda^pV\simeq (V^{\otimes p})^{SL(V)}$.

\item For any sequence of positive integers $i_1,\ldots,i_m$ the $SL(V)$-representation $S^{i_1}V\otimes\ldots\otimes S^{i_m}V$ has no cohomology in degrees $>0$. If $i_1+\ldots+i_m=p$ and at least one of the $i_1,\ldots,i_m$ is larger than or equal to $2$, this representation has no non-zero $SL(V)$-invariants.

\end{enumerate}
\end{lm}

\begin{proof}
Part (1) follows from classical results on good filtrations, see \cite[II.4]{jantzen}. We recall here relevant facts from this theory, from which the desired vanishing follows immediately. In general, consider a split reductive group $G$ over $k$ with a Borel subgroup $B\subset G$ and a maximal torus $T\subset B$. Recall that for a weight $\lambda\in X^*(T)$ there is the associated $G$-equivariant line bundle $\cO(\lambda)$ on the flag variety $G/B$, obtained by applying the equivalence \{$G$-equivariant vector bundles on $G/B$\}$\simeq$\{representations of $B$\} to the character $B\to T\xrightarrow{\lambda}\bG_m$. The global sections $H^0(\lambda):=H^0(G/B,\cO(\lambda))$ is a finite-dimensional representation of $G$ which is non-zero if and only if $\lambda$ is a dominant weight. A filtration $\ldots\subset W_i\subset W_{i-1}\subset \ldots$ on a finite-dimensional representation $W$ is called {\it good} if every quotient $W_{i-1}/W_i$ is isomorphic to the representation $H^0(G/B, \cO(\lambda_i))$ for some dominant weight $\lambda$. If a representation $W$ admits a good filtration then $H^i(G, W)=0$ for all $i>0$, by Kempf vanishing \cite[Proposition II.4.13]{jantzen}.

For $n=0$, part (1) is asserting that $H^{>0}(SL(V),k)=0$ which is the case because $k=H^0(\lambda)$ for the trivial character $\lambda$. The tautological representation $V$ of $SL(V)$ has the form $H^0(\lambda)$ for $\lambda$ the highest weight of $V$, so it tautologically has a good filtration. Tensor product of two representations with good filtrations admits a good filtration as well \cite[Proposition II.4.21]{jantzen}, hence $V^{\otimes n}$ admits a good filtration which implies that $H^{>0}(SL(V),V^{\otimes n})=0$ for $n>0$.

Part (2) is proven by a direct computation. The map $\psi$ is an $SL(V)$-equivariant injection so all we need is to prove that $(V^{\otimes p})^{SL(V)}$ has dimension at most $1$. Let $e_1,\ldots,e_p$ be a basis for $V$ and consider first the action of the maximal torus $T=\{\diag(t_1,\ldots,t_{p-1},(t_1\cdot\ldots\cdot t_{p-1})^{-1})|t_i\in\bG_m\}$ on $V^{\otimes p}$. An element $e_{i_1}\otimes\ldots\otimes e_{i_p}$ is being acted on by $T$ via the character $t_1^{d_1-d_p}\ldots t_{p-1}^{d_{p-1}-d_p}$ where $d_j$ is the number of indices $r=1,\ldots,p$ such that $i_r=j$. Therefore the subset of $T$-invariants $(V^{\otimes p})^T\subset V^{\otimes p}$ is spanned by tensors $e_{i_1}\otimes\ldots\otimes e_{i_p}$ for which $(i_1,\ldots,i_p)$ is a permutation of the sequence $(1,2\ldots ,p)$. We can embed the group $S_p$ into $SL(V)$ by sending $\tau\in S_p$ to the operator $e_i\mapsto \sgn(\tau)e_{\tau(i)}$. The subspace $(V^{\otimes p})^T$ is preserved by $S_p$ and is isomorphic to the regular representation $k[S_p]$ of $S_p$. Hence $S_p\ltimes T$-invariants in $V^{\otimes p}$ are $1$-dimensional and part (2) is proven.

(3) For any $n$ the representation $S^n V$ is a Weyl module for $SL(V)$ (cf. \cite[II.2.16]{jantzen}) so it tautologically has a good filtration, and hence any tensor product of symmetric powers of $V$ has a good filtration and therefore has no higher cohomology.

If $i_1+\ldots+i_m=p$ then either $m=1$ and the representation in question is $S^p V$, or all $i_1,\ldots,i_m$ are less than $p$. In the case $m=1$ the representation $S^p V$ has no non-zero invariants by \cite[Prop II.4.13]{jantzen}. If all $i_1,\ldots,i_m$ are less than $p$ then we have a $SL(V)$-equivariant embedding \begin{equation}\psi:S^{i_1}V\otimes\ldots\otimes S^{i_m}V\simeq \Gamma^{i_1}V\otimes\ldots\otimes\Gamma^{i_m}V
\hookrightarrow V^{\otimes(i_1+\ldots+i_m)}=V^{\otimes p}.\end{equation} 

By (2), the $1$-dimensional space of $SL(V)$-invariants in $V^{\otimes p}$ is being acted on by the symmetric group $S_p$ via the sign character, but unless all $i_1,\ldots, i_m$ are equal to $1$, the image of $\psi$ is being acted on trivially by an odd permutation (and $p$ is larger than $2$), which proves the vanishing of $SL(V)$ invariants in $S^{i_1}V\otimes\ldots \otimes S^{i_m}V$.
\end{proof}

We will prove Proposition \ref{rational group cohomology: main non-vanishing} by calculating the effect of the map $S^p(V[-1])\to (\Lambda^p V)[-p]$ on derived $SL(V)$-invariants. To calculate $SL(V)$-cohomology of $S^p(V[-1])$, we introduce the following construction. For any $i\geq 0$ and $k$-vector spaces $M,N$ there is a natural map
\begin{equation}
\varepsilon^i_{M,N}:\Lambda^i M\otimes\Lambda^i N\to S^i(M\otimes N)
\end{equation}
given by
\begin{equation}
\varepsilon_{M,N}^i(m_1\wedge\ldots\wedge m_i\otimes n_1\wedge\ldots\wedge n_i)=\displaystyle\sum\limits_{\sigma\in S_i}\sgn(\sigma)m_1\otimes n_{\sigma(1)}\cdot\ldots\cdot m_i\otimes n_{\sigma(i)}.
\end{equation}
By the procedure described in subsection \ref{doldpuppe subsection}, this gives rise to natural maps $\varepsilon^i_{M,N}:\Lambda^i M\otimes\Lambda^i N\to S^i(M\otimes N)$ for arbitrary objects $M,N\in D(k)$. 

\begin{lm}\label{rational group cohomology: symemtric power invariants}
\begin{enumerate}
\item If $V$ is a $p$-dimensional $k$-vector space, and $W$ is any vector space, then the map $\varepsilon_{V,W}^p:\Lambda^p W\simeq \Lambda^p V\otimes\Lambda^p W\to S^p(V\otimes W)$ induces an isomorphism $\Lambda^p W\simeq S^p(V\otimes W)^{SL(V)}$ onto the invariants for the action of $SL(V)$ induced by the tautological action on $V$ and trivial action on $W$.
\item If $W$ now is an arbitrary object of $D^{\geq 0}(k)$, and $V$ is a $p$-dimensional $k$-vector space as before, then the map $\varepsilon_{V,W}^p:\Lambda^p V\otimes \Lambda^p W\to S^p(V\otimes W)$ induces an equivalence $\Lambda^p W\simeq \RGamma(SL(V), S^p(V\otimes W))$.
\end{enumerate}
\end{lm}

\begin{proof}
(1) To prove that it is an isomorphism onto the space of $SL(V)$-invariants in $S^p(V\otimes W)$, choose a basis $\{e_i,i\in I\}$ for $W$. Then \begin{equation}\label{rational group cohomology: sp basis decomposition}S^p(V\otimes W)=S^p(\bigoplus\limits_{i\in I} V)\simeq\bigoplus\limits_{\sum\limits_{i\in I}d_i=p}\bigotimes\limits_{i\in I}S^{d_i}V\end{equation} and $\varepsilon_{V,W}^p$ is an isomorphism onto the sum of all direct summands with $d_i\in\{0,1\}$. All the other summands have no non-zero $SL(V)$-invariants by Lemma \ref{rational group cohomology: basic facts} (3).

To prove (2), choose a cosimplicial vector space $W^{\bullet}$ whose totalization is equivalent to $W$. Then $S^p(V\otimes W)$ as an object of $D^{\geq 0}(\Rep SL(V))$ is the totalization of the cosimplicial representation $S^p(V\otimes W^{\bullet})$. In view of the decomposition (\ref{rational group cohomology: sp basis decomposition}), vanishing of higher cohomology of products of symmetric powers of $V$ (Lemma \ref{rational group cohomology: symemtric power invariants}(3)) implies that each term $S^p(V\otimes W^n)$ has no higher $SL(V)$-cohomology.  Hence $\RGamma(SL(V), S^p(V\otimes W))$ is equivalent to the totalization of the cosimplicial vector space $S^p(V\otimes W^{\bullet})^{SL(V)}$, and assertion (2) follows from part (1). 
\end{proof}

Having calculated cohomology of $SL(V)$ with coefficients in $S^p(V\otimes W)$ using the map $\varepsilon_{V,W}^p$, let us establish the following statement about the interaction of this map with the norm  map $N_p$:

\begin{lm}\label{rational group cohomology: epsilon on frob twsit}
For all $M,N\in D(k)$ the composition
\begin{multline}\label{rational group cohomology: additive to tensor formula}
\Lambda^p M\otimes N^{(1)}\xrightarrow{\id_{\Lambda^p M}\otimes \Delta_N} \Lambda^p M\otimes S^p N\simeq \Lambda^pM \otimes \Lambda^p(N[-1])[p]\xrightarrow{\varepsilon^p_{M,N[-1]}[p]} \\ S^p(M\otimes N[-1])[p]\xrightarrow{N_p [p]} \Gamma^p(M\otimes N[-1])[p]\simeq \Lambda^p(M\otimes N)
\end{multline}
is zero.
\end{lm}

\begin{proof}
We may assume that both $M$ and $N$ are vector spaces placed in degree $0$. This vanishing can be obtained by a direct calculation via the Koszul complex to see how $\varepsilon$ interacts with the d\'ecalage isomorphism, but instead we observe that even for a fixed $M$ there are no non-zero morphisms $\Lambda^p M\otimes N^{(1)}\to \Lambda^p(M\otimes N)$ of functors of $N$. There is a natural embedding $\Lambda^p(M\otimes N)\hookrightarrow(M\otimes N)^{\otimes p}$ so it suffices to check that there are no non-zero morphisms of functors $N^{(1)}\to N^{\otimes p}$.

This follows from \cite[Theorem 2.13]{friedlander-suslin} which asserts that there are non non-zero morphisms from an additive functor to a tensor product of strict polynomial functors of positive degree.
\end{proof}

\begin{proof}[Proof of Proposition \ref{rational group cohomology: main non-vanishing}] By Lemma \ref{rational group cohomology: symemtric power invariants}(2) applied to $W=k[-1]$ the natural map $\varepsilon_{V,k[-1]}^p:\Lambda^pV\otimes\Lambda^p k[-1]\to S^p(V\otimes k[-1])$ induces a quasi-isomorphism $k[-p]\simeq \Lambda^p (k[-1])\simeq \RGamma(SL(V),S^p(V[-1]))$.

In other words, $H^p(SL(V), S^p(V[-1]))\simeq k$ and $SL(V)$-cohomology groups in all other degrees are zero. Let us now show that the map $\RGamma(SL(V), S^p(V[-1]))\to \RGamma(SL(V),\Gamma^p(V[-1]))\simeq \RGamma(SL(V), (\Lambda^p V)[-p])$ induced by $N_p$ is zero. By Lemma \ref{rational group cohomology: basic facts}(1) the trivial representation $\Lambda^p V$ has no cohomology in degrees $>0$, so the natural map $\RGamma(SL(V),\Gamma^p(V[-1]))\to \Gamma^p(V[-1])$ is an equivalence. Hence it suffices to check that the composition
\begin{equation}\label{rational group cohomology: norm on invariants map}
\Lambda^pV\otimes \Lambda^p(k[-1])\xrightarrow{\varepsilon^p_{V,k[-1]}}S^p(V[-1])\xrightarrow{N_p}\Gamma^p(V[-1])
\end{equation}
is zero. 

This follows from Lemma \ref{rational group cohomology: epsilon on frob twsit} with $M=V, N=k$, because in this case the map $N^{(1)}\to S^p N$ is an isomorphism, so the composition (\ref{rational group cohomology: additive to tensor formula}) is the same as (\ref{rational group cohomology: norm on invariants map}).

In particular, the map $k\simeq H^p(SL(V),S^p(V[-1]))\to H^p(SL(V), \Gamma^p (V[-1]))\simeq k$ is zero, and the long exact sequences of $SL(V)$-cohomology from the fiber sequence $T_p(V[-1])[-1]\to S^p(V[-1])\to \Gamma^p (V[-1])$ implies that $H^i(SL(V),T_p(V))\simeq k$ for $i=p-2,p-1$ and is zero for all other $i$. On the other hand, the fiber sequence
\begin{equation}
V^{(1)}[1]\to T_p(V)\to V^{(1)}
\end{equation}
from (\ref{free cosimplicial: tate mod p extension}) yields a long exact sequence
\begin{equation}\label{rational group cohomology: tate long exact sequence}
\ldots\to H^{i+1}(SL(V),V^{(1)})\to H^{i+1}(SL(V), T_p(V))\to H^i(SL(V), V^{(1)})\to\ldots
\end{equation}
In particular, $SL(V)$-cohomology of $V^{(1)}$ is non-zero in at least one degree, or else the cohomology of $T_p(V)$ would be zero. 

By \cite[Theorem 2.4(b)]{cpsk} the the cohomology of $SL(V)$ with coefficients in a finite-dimensional representation is non-zero only in finitely many degrees. Let $i_{-}, i_+$ be, respectively, the minimal and the maximal $i$ such that $H^i(SL(V), V^{(1)})\neq 0$. The exact sequence (\ref{rational group cohomology: tate long exact sequence}) then gives the isomorphisms $H^{i_+}(SL(V), T(V))\simeq H^{i_+}(SL(V),T(V))$ and $H^{i_{-}-1}(SL(V),T(V))\simeq H^{i_{-}}(SL(V), V^{(1)})$. Since we already established that the cohomology of $T(V)$ is non-zero only in degrees $p-2,p-1$, we necessarily have $i_{-}=i_+=p-1$ and $H^{p-1}(SL(V), V^{(1)})\simeq k$. We established:

\begin{lm}\label{rational group cohomology: frob twist cohomology}
The cohomology $H^i(SL(V), V^{(1)})$ is $k$ for $i=p-1$ and vanishes in all other degrees.
\end{lm}

We can now finish the proof of Proposition \ref{rational group cohomology: main non-vanishing}. We already know that the map $S^p(V[-1])\to (\Lambda^pV)[-p]$ does not have a section in $D(\Rep SL(V))$ because it induces a non-surjective map on $H^p(SL(V),-)$. From the fiber sequence $V^{(1)}[-1]\to S^p(V[-1])\to  \tau^{\geq 2}S^p(V[-1])$  we know that $H^p(SL(V), S^p(V[-1]))\to H^p(SL(V), \tau^{\geq 2}S^p(V[-1]))$ is surjective, because $H^p(SL(V), V^{(1)})=0$, therefore $H^p(SL(V), \tau^{\geq 2}S^p(V[-1]))\to H^p(SL(V),(\Lambda^p V)[-p])$ is not surjective, and the fiber sequence (\ref{rational group cohomology: alpha formula}) is not split, as desired.
\end{proof}

\begin{rem}
The computation of $H^{p-1}(SL(V),V^{(1)})$ for $p=2$ follows directly from the short exact sequence $V^{(1)}\to S^2 V\to \Lambda^2 V$, by the computation in Remark \ref{group cohomology: p=2 remark}, and for $p=3$ it is a result of Stewart \cite[Theorem 1]{stewart}.
\end{rem}

Having proved that $\alpha(V)$ is non-zero, we can identify it with other extensions that appeared in the literature, providing more compact Yoneda models for it. These other models are not used in any other parts of the paper. From now on $V$ is a $k$-vector space of arbitrary finite dimension.

Friedlander and Suslin \cite[(4.5.1)]{friedlander-suslin} have calculated the Ext groups between $\Lambda^p V$ and $V^{(1)}$ in the category $\cP$ of strict polynomial functors on $k$-vector spaces. The only non-zero Ext group $\Ext^i_{\cP}(\Lambda^p, (-)^{(1)})$ is in degree $i=p-1$ and it is spanned by a Yoneda extension constructed from the de Rham complex as follows. Loc. cit. calculates Ext groups from $V^{(1)}$ to $\Lambda^p V$, but this is equivalent to our assertion by dualizing.

Consider the affine space $\bA(V^{\vee}):=\Spec S^{*}(V)$ corresponding to the vector space $V^{\vee}$. It is equipped with an action of $GL(V)$ and the de Rham complex $\Omega^{\bullet}_{\bA(V^{\vee})/k}$ can be viewed as a complex of representations of $GL(V)$ on $k$-vector spaces. Explicitly, $\Omega^{\bullet}_{\bA(V^{\vee})/k}$ has the form
\begin{equation}
S^{*}(V)\xrightarrow{d}S^{*}(V)\otimes_k V\xrightarrow{d}S^{*}(V)\otimes_k\Lambda^2V\xrightarrow{d}\ldots
\end{equation}
where the de Rham differential $d:S^{*}(V)\otimes_{k}\Lambda^i V\to S^{*}(V)\otimes_k\Lambda^{i+1}V$ is given by 
\begin{equation}d(v_1\cdot\ldots\cdot v_n\otimes \omega)=\sum\limits_{i=1}^dv_1\cdot\ldots\cdot \widehat{v}_i\cdot\ldots\cdot v_d\otimes v_i\wedge \omega.
\end{equation}

The action of the center $\bG_m\subset GL(V)$ introduces a grading on the de Rham complex with the $n$-th graded piece given by 
\begin{equation}
\Omega^{\bullet}_n(V):=S^nV\xrightarrow{d} S^{n-1}V\otimes V\xrightarrow{d}\ldots\xrightarrow{d}\Lambda^n V.    
\end{equation}
The Cartier isomorphism describes the cohomology modules of this complex as follows:
\begin{lm}[{cf. \cite[Theorem 4.1]{friedlander-suslin}}]
When $p\nmid n$ the complex $\Omega^{\bullet}_n(V)$ is acyclic. If $n=pk$ for an integer $k$ then there are $GL(V)$-equivariant isomorphisms $H^i(\Omega^{\bullet}_{pk}(V))\simeq \Lambda^i(V^{(1)})\otimes S^{k-i}V^{(1)}$ for all $i$.
\end{lm}

In particular, for $n=p$ the only non-zero cohomology groups of the complex $\Omega^{\bullet}_p$ are in degrees $0$ and $1$, both isomorphic to $V^{(1)}$. Consider the complex $\Omega^{\leq p-1}_p$ obtained from $\Omega^{\bullet}_p(V)$ by removing the last term:
\begin{equation}\label{group cohomology: de rham p-1}
\Omega^{\leq p-1}_p(V):=S^p V\to S^{p-1}V\otimes V\to\ldots\to V\otimes\Lambda^{p-1}V.
\end{equation} 
This is a complex with $H^0\simeq H^1\simeq V^{(1)}$, $H^{p-1}\simeq\Lambda^p V$ and all other cohomology groups are equal to zero. 

\begin{pr}
For every finite-dimensional $k$-vector space $V$ there is a natural equivalence $S^p(V[-1])\simeq \Omega^{\leq p-1}_p(V)[-1]$ in the derived category of representations of $GL(V)$.
\end{pr}
\begin{proof}
Note that the functor $V\mapsto S^p(V[-1])$ can be upgraded to an object of the derived category of strict polynomial functors (in the sense of \cite[\S 2]{friedlander-suslin}): if we choose a cosimplicial $k$-vector space $D^{\bullet}$, with all terms finite-dimensional, such that the totalization of $D^{\bullet}$ is equivalent to $k[-1]$, then $S^p(V[-1])$ is represented by the totalization of $S^p(D^{\bullet}\otimes V)$.

First, the truncation $\tau^{\leq 1}\Omega^{\leq p-1}_p(V)$ is equivalent to $T(V)[1]$ by \cite[Lemma 4.12]{friedlander-suslin}. Hence both $S^p(V[-1])$ and $\Omega^{\leq p-1}_p(V)[-1]$ are extensions of $(\Lambda^p V)[-p]$ by $T(V)[-2]$ in the derived category of $\cP$. 

The map $\Ext^{p-1}_{\cP}(\Lambda^p V, T(V))\to \Ext^{p-1}_{\cP}(\Lambda^p V, V^{(1)})$ is an isomorphism because $\Ext^i_{\cP}(\Lambda^p V,V^{(1)})$ vanishes for all $i\neq p-1$. Moreover, $\Ext^{p-1}_{\cP}(\Lambda^p V, V^{(1)})$ is $1$-dimensional and is generated by $\tau^{\geq 1}\Omega^{\leq p-1}_p(V)$, by the proof of \cite[(4.5.1)]{friedlander-suslin}.

Hence $\Omega^{\leq p-1}_p(V)$ defines a non-zero element in the $1$-dimensional space $\Ext^{p-1}_{\cP}(\Lambda^pV,T(V))$. We already established that the fiber sequence $T(V)[-2]\to S^p(V[-1])\to(\Lambda^p V)[-p]$ is non-split in the category of $GL(V)$-representations, hence it is non-split in the derived category of strict polynomial functor as well, and thus defines a non-zero element in $\Ext^{p-1}_{\cP}(\Lambda^p V, T(V))$. It differs from the class of $\Omega^{\leq p-1}_p(V)$ by a non-zero scalar, so we may choose an equivalence $\Omega^{\leq p-1}_p(V)[-1]\simeq S^p(V[-1])$ in the derived category of strict polynomial functors.
\end{proof}

\begin{rem}\label{rational: carter-lusztig remark}
Adrian Langer pointed out that the class $\alpha(V)$ can also be represented by a Yoneda extension considered by Carter and Lusztig. The dual of the exact sequence in \cite[(61) on p. 235]{carter-lusztig} reads as
    \begin{equation}\label{rational: carter-lusztig}
    0\to V^{(1)}\to S^{(p)}(V)\to S^{(p-1,1)}(V)\to \ldots \to S^{(1,1,\ldots,1)}(V)\to 0
    \end{equation}
    where $S^{(p-k,1,\ldots,1)}(V)$ is a Weyl module for $GL(V)$ explicitly described for $k<p-1$ as
    \begin{equation}
    S^{(p-k,\overbrace{1,\ldots,1}^{k})}(V)=\Lambda^{k+1}V\otimes  S^{p-k-1}V/\langle\sum\limits_{\sigma\in C_{k+2}}\sgn(\sigma)v_{\sigma(1)}\wedge\ldots\wedge v_{\sigma(k+1)}\otimes v_{\sigma(k+2)}\cdot v_{k+3}\cdot\ldots\cdot v_p\rangle
    \end{equation}
    where we quotient out the subspace spanned by these sums for all possible collections of vectors $v_1,\ldots,v_p\in V$, and $\sigma$ runs over all cyclic permutations of elements $1,2,\ldots,k+2$. In particular, $S^{(p)}(V)=S^pV$ and $S^{(1,1,\ldots,1)}(V)=\Lambda^p V$, so (\ref{rational: carter-lusztig}) defines a degree $p-1$ Yoneda extension of $\Lambda^p V$ by $V^{(1)}$ in $GL(V)$-modules. The differentials $S^{(p-k+1,1,\ldots,1)}(V)\to S^{(p-k,1,\ldots,1)}(V)$ in the sequence (\ref{rational: carter-lusztig}) are defined by the formula
    \begin{equation}
    v_1\wedge\ldots \wedge v_{k}\otimes v_{k+1}\cdot\ldots \cdot v_p\mapsto \sum\limits_{\sigma\in C_{p-k}}v_{1}\wedge\ldots v_{k}\wedge v_{\sigma(k+1)}\otimes v_{\sigma(k+2)}\cdot \ldots \cdot v_{\sigma(p)}
    \end{equation}
    where the sum runs over all cyclic permutations of indices $k+1,\ldots,p$. The surjections $\Lambda^{k+1}V\otimes S^{p-k-1}V\to S^{(p-k,1,\ldots,1 )}$ intertwine this differential with the de Rham differential $d:\Lambda^{k}V\otimes S^{p-k}V\to \Lambda^{k+1}V\otimes S^{p-k-1}V$, and therefore define a map from the extension $V^{(1)}[-1]\to \tau^{\geq 1}\Omega^{\leq p-1}_p\to \Lambda^p V[1-p]$ to the extension (\ref{rational: carter-lusztig}) inducing the identity on $V^{(1)}$ and $\Lambda^p V$. Hence (\ref{rational: carter-lusztig}) provides another model for the class $\alpha(V)$.
\end{rem}

\section{From algebraic cohomology to cohomology of the group of \texorpdfstring{$\bF_q$}{Fq}-points}\label{group cohomology: section}
\renewcommand{\alg}{\mathrm{alg}}
Throughout this section, we fix a finite field $\bF_q$ of characteristic $p$ and always assume that our base field $k$ contains $\bF_q$. In this section we show that $\alpha(V)$ remains non-zero when restricted to the discrete group $SL_p(\bF_q)$, provided that $q>p$:

\begin{pr}\label{group cohomology: from algebraic to discrete}
If $\dim V=p$ and $q>p$ then the restriction map $H^{p-1}_{\alg}(SL(V),V^{(1)})\to H^{p-1}(SL_p(\bF_q),V^{(1)})$ is injective.
\end{pr}

In this and the next section we use the notation $H^*_{\alg}$ for the cohomology of representations of algebraic groups, and $H^*$ is reserved for cohomology of discrete groups. In general, it follows from results of Cline, Parshall, Scott, and van der Kallen that the restriction from cohomology of a split reductive group $G$ to that of its $\bF_q$-points is injective for a large enough $q$:

\begin{thm}[{\hspace{1sp}\cite[Theorem 6.6]{cpsk}+\hspace{1sp}\cite[Theorem 2.1]{cps-detecting}}]
Let $G$ be a split reductive group over $\bF_p$, and $W$ be a finite-dimensional representation of it. For large enough $q=p^r$ the restriction map \begin{equation}
   H^n_{\alg}(G,W)\to H^{n}(G(\bF_q),W) 
\end{equation} is injective for all $n$. 
\end{thm}

Our Proposition \ref{group cohomology: from algebraic to discrete} is only marginally stronger than this result in that we show that for any $q>p$ injectivity holds in the particular case $G=SL_p, W=V^{(1)}$. We give here an argument that follows closely the proof of \cite[Theorem 6.6]{cpsk} in order to introduce the techniques that will be used in Section \ref{borel: section}.

Let us briefly describe the method of \cite{cpsk}. Let $B\subset G$ be a Borel subgroup of a split reductive group $G$ over $k$. We have a commutative diagram 
\begin{equation}
\begin{tikzcd}
H^n_{\alg}(G,W)\arrow[r]\arrow[d,"\sim"] & H^n(G(\bF_q),W)\arrow[d] \\
H^n_{\alg}(B,W)\arrow[r] & H^n(B(\bF_q),W). \\
\end{tikzcd}
\end{equation}
The fact that the left vertical map is an isomorphism is a consequence of the vanishing $H^{>0}(G/B,\cO)=0$ (\hspace{1sp}\cite[\S 6 Theorem 1(a)]{kempf}) of the cohomology of the structure sheaf on the flag variety:

\begin{thm}[\hspace{1sp}{\cite[Theorem 2.1]{cpsk}}]\label{group cohomology: restriction to Borel iso}
For every algebraic $G$-module $W$ restriction induces an isomorphism $H^i_{\alg}(G,W)\simeq H^i_{\alg}(B,W)$ for all $i$.
\end{thm}

Therefore it is enough to prove that the bottom horizontal arrow is injective. This is now a more tangible question as the Borel subgroup is isomorphic to the semi-direct product $T\ltimes U$ of the maximal torus $T$ and the unipotent radical $U$ of $B$. Since the algebraic cohomology of a torus as well as the cohomology of the finite group $T(\bF_q)$ with $p$-torsion coefficients vanish in positive degrees, we have $H^i_{\alg}(B, W)=H^i_{\alg}(U,W)^T$ and $H^i(B(\bF_q), W)=H^i(U(\bF_q), W)^{T(\bF_q)}$. This reduces the problem to the study of the action of $T$ on $H^i_{\alg}(U, W)$ which, for the purposes of proving the desired eventual injectivity, can be reduced to the study of the action of $\bG_m$ on $H^{\bullet}_{\alg}(\bG_a,k)$ and $\bF_q^{\times}$ on $H^{\bullet}(\bF_q,k)$ which is performed in Proposition \ref{group cohomology: additive group cohomology}.

Choose a basis $e_1,\ldots,e_p$ of $V$, and let $B_p\subset SL(V)$ be the subgroup of matrices preserving each of the subspaces $\langle e_1,\ldots,e_i\rangle\subset V$. Denote also by $T_p\subset B_p$ the maximal torus of the diagonal matrices, and by $U_p\subset B_p$ the subgroup of strictly upper-triangular matrices.

\subsection{Proof of Proposition \ref{group cohomology: from algebraic to discrete} when \texorpdfstring{$p=2$}{}.}

We will now prove that when $p=2$ and $q>p$ the restriction map $H^1_{\alg}(B_2,V^{(1)})\to H^1(B_2(\bF_q),V^{(1)})$ is injective. We treat the case $p=2$ separately both because there are slight notational differences from the case $p>2$, arising from the fact that $H^{\bullet}_{\alg}(\bG_a,k)$ has a different-looking ring structure, and because it explains the idea of the general proof in a setting less burdened by the combinatorial difficulties.

Denote by $\chi_1:T_2\to\bG_m$ the character $\diag(a,a^{-1})\mapsto a$ identifying the maximal torus $T_2$ with $\bG_m$. The conjugation action of $T_2$ on $U_2\simeq \bG_a$ is then given by $\chi_1^2$.

When restricted to $B_2$, the representation $V^{(1)}$ fits into an exact sequence $0\to\chi_1^2\to V^{(1)}\to \chi_1^{-2}\to 0$. It induces the long exact sequences in cohomology of $B_2$ and $B_2(\bF_q)$:

\begin{equation}
\begin{tikzcd}
& H^1_{\alg}(U_2,\chi_1^2)^{T_2}\arrow[r]\arrow[d] & H^1_{\alg}(U_2,V^{(1)})^{T_2}\arrow[r]\arrow[d] & H^1_{\alg}(U_2,\chi_1^{-2})^{T_2}\arrow[d] \\
H^0(U(\bF_q),\chi_1^{-2})^{T_2(\bF_q)}\arrow[r] & H^1(U_2(\bF_q),\chi_1^2)^{T_2(\bF_q)}\arrow[r] & H^1(U_2(\bF_q),V^{(1)})^{T_2(\bF_q)}\arrow[r] & H^1(U_2(\bF_q),\chi_1^{-2})^{T_2(\bF_q)}
\end{tikzcd}
\end{equation}

We have $H^1_{\alg}(U_2,\chi^2_1)\simeq\Hom_{\grp}(\bG_a,\bG_a)$ with the action of $T_2$ described as follows: for any $k$-algebra $R$, an element $a\in T_2(R)=R^{\times}$ acts by sending a homomorphism $f:\bG_a\to\bG_a$ to $a^{2}\cdot f(a^{-2}\cdot -)$. Similarly, $H^1_{\alg}(U_2,\chi^{-2}_1)\simeq\Hom_{\grp}(\bG_a,\bG_a)$ but the torus action sends $f$ to $a^{-2}\cdot f(a^{-2}\cdot -)$. All homomorphisms $f:\bG_a\to \bG_a$ have the form $f(x)=\sum\limits_{i=0}^N a_i x^{p^i}$ for some $N\geq 0$ and $a_i\in k$, where $x$ is a coordinate on $\bG_a$. It follows that $H^1_{\alg}(U_2,\chi_1^{-2})^{T_2}=0$ and $H^1_{\alg}(U_2,\chi_1^{2})^{T_2}\subset \Hom_{\grp}(\bG_a,\bG_a)$ is the one-dimensional space of homomorphisms of the form $x\mapsto a_0\cdot x$. Therefore the map $H^1_{\alg}(U_2,\chi_1^2)^{T_2}\to H^1_{\alg}(U_2,V^{(1)})^{T_2}$ is surjective, and the restriction map $H^1_{\alg}(U_2,\chi_1^2)^{T_2}\to H^1(U_2(\bF_q),\chi_1^2)^{T(\bF_q)}$ is injective for any $q$, because a non-zero cocycle of the form $x\mapsto a_0\cdot x\in H^1_{\alg}(U_2, \chi_1^2)\simeq\Hom_{\alg}(\bG_a,\bG_a)$ stays non-zero in $H^1(U_2(\bF_q),\chi_1^2)=\Hom_{\grp}(\bF_q,k)$.

It remains to observe that $H^0(U_2(\bF_q),\chi_1^{-2})$ is a one-dimensional vector space on which $T_2(\bF_q)=\bF_q^{\times}$ acts via the character $\chi_1^{-2}$. Therefore the invariant subspace of this $0$th cohomology group is trivial as soon as $q>2$. This implies that the map $H^1(U_2(\bF_q),\chi_1^2)^{T_2(\bF_q)}\to H^1(U_2(\bF_q),V^{(1)})^{T_2(\bF_q)}$ is injective, and hence the restriction $H^1_{\alg}(U_2,V^{(1)})^{T_2}\to H^1(U_2(\bF_q),V^{(1)})^{T_2(\bF_q)}$ is injective.

\subsection{Proof of Proposition \ref{group cohomology: from algebraic to discrete} when \texorpdfstring{$p>2$}{}.}

As in the special case $p=2$ that we dealt with above, the key to the proof is to analyze the action of the maximal torus on the cohomology of the unipotent radical of a Borel subgroup of $SL_p$. For a split torus $T$ over $k$ we denote by $X^*(T)=\Hom_{\grp}(T,\bG_m)$ its lattice of characters. We use additive notation for the group operation in $X^*(T)$. If $U$ is a unipotent algebraic group equipped with an action of $T$, we denote by $\Delta_U\subset X^*(T)$ the set of characters appearing in the $T$-representation $\Lie U$. 

Here is the main computation in the case $U\simeq \bG_a^n$ using which we will study the case of an arbitrary unipotent $U$ by a d\'evissage argument. 

\begin{pr}\label{group cohomology: additive group cohomology}
Assume that $p>2$. Let $A$ be a group scheme over $\bZ$ isomorphic to $\bG_{a,\bZ}^{n}$, equipped with an action of a split torus $T$. Denote by $\fa:=(\Lie A)^{\vee}$ the dual Lie algebra of $A$. We denote by $\fa_k$ and $\fa_{W_2(k)}$ the modules $\fa\otimes_{\bZ}k$ and $\fa\otimes_{\bZ}W_2(k)$.

\begin{enumerate}
\item There is a $T_k$-equivariant identification \begin{equation}H^{\bullet}_{\alg}(A, k)=\Lambda^*(\bigoplus\limits_{i=0}^{\infty} \fa_k^{(i)}\cdot x_i)\otimes S^*(\bigoplus\limits_{i=1}^{\infty} \fa_k^{(i)}\cdot\beta(x_i))\end{equation} where $x_i$ and $\beta(x_i)$ are formal symbols, invariant under $T_k$. The cohomological degrees of $x_i$ and $\beta(x_i)$ are $1$ and $2$, respectively.

\item Suppose that the cardinality of $\bF_q$ is $q=p^r$. There is a $T(\bF_q)$-equivariant identification \begin{equation}H^{\bullet}(A(\bF_{p^r}),k)=\Lambda^*(\bigoplus\limits_{i=0}^{r-1} \fa_k^{(i)}\cdot x_i)\otimes S^*(\bigoplus\limits_{i=1}^{r} \fa_k^{(i)}\cdot \beta(x_i))\end{equation} and the map $H^{\bullet}_{\alg}(A,k)\to H^{\bullet}(A(\bF_{p^r}),k)$ sends $x_i$ to $x_{i\, \mathrm{mod}\, r}$, and $\beta(x_i)$ to $\beta(x_{((i-1)\, \mathrm{mod}\, r)+1})$, for all $i$. In particular, this map is surjective in every degree.

\comment{\item For $n>1$ we have $H^{\bullet}(A(W_n(\bF_{p^r})),k)=\Lambda(\bigoplus\limits_{i=0}^{r-1} \fa^{(i)})\otimes S(\bigoplus\limits_{i=1}^{r} \fa^{(i)})$ as well but the map $H^{\bullet}(A(W_{n-1}(\bF_{p^r})),k)\to H^{\bullet}(A(W_{n}(\bF_{p^r})),k)$ is given by identity on the exterior algebra and by zero on the symmetric algebra. }
\end{enumerate}

Suppose that $F\supset \bQ$ is a finite Galois extension, and $\fp\subset\cO_F$ is an unramified over $\bZ$ prime ideal such that the residue field $\cO_F/\fp$ is identified with $\bF_q$. We have the following $T(\cO_F)$-equivariant identifications:

\begin{enumerate}
\setcounter{enumi}{2}
\item $H^{\bullet}(A(\cO_F),k)=\Lambda^{\bullet}(\bigoplus\limits_{\tau\in \Gal(F/\bQ)} \fa_k^{\tau}\cdot x_{\tau})$. Here $T(\cO_F)$ acts on $\fa_k$ through the chosen map $T(\cO_F)\to T(\cO_F/\fp)=T(\bF_q)$, and $\fa_k^{\tau}$ denotes the composition of this action with the automorphism $T(\cO_F)\xrightarrow{\tau}T(\cO_F)$. The map on cohomology induced by $A(\cO_F)\to A(\bF_q)$ annihilates $\beta(x_i)$ and sends $x_i$ to $x_{\tau_i}$ where $\tau_i\in \Gal(F/\bQ)$ is the element of the decomposition group of $\fp$ that induces the automorphism $\Fr_p^i$ on $\bF_q$.

\item $H^{\bullet}(A(\cO_F),W_2(k))=\Lambda^{\bullet}(\bigoplus\limits_{\tau\in \Gal(F/\bQ)} \fa_{W_2(k)}^{\tau}\cdot x_{\tau})$ such that the mod $p$ reduction of this isomorphism is the isomorphism in (3).
\end{enumerate}
\end{pr}

\begin{example}\label{group cohomology: ga cohomology example}
For cohomology in degree $1$ and $A\simeq\bG_{a,\bZ}$ statement (1) amounts to the fact that $H^1(\bG_{a,k},k)=\Hom_{\grp}(\bG_{a,k},\bG_{a,k})$ is spanned by homomorphisms of the form $t\mapsto t^{p^i}$ where $t$ is a coordinate on $\bG_{a,k}$. Statement (2) in this case is saying that $H^1(\bG_a(\bF_{p^r}),k)=\Hom(\bF_{p^r},k)$, and every homomorphism $\bF_{p^r}\to k$ of additive groups can be represented as $t\mapsto\sum\limits_{i=0}^{r-1}a_it^{p^i}$ for some $a_i\in k$, in a unique way.
\end{example}

\begin{proof}
By K\"unneth formula, it is enough to consider the case $A\simeq\bG_a$. Parts (1) and (2) are \cite[Theorem 4.1]{cpsk}, see also \cite[Proposition I.4.27]{jantzen}. To fix ideas, let us explicitly say that in part (1) the map $\fa_k^{(i)}\cdot x_i\to H^1_{\alg}(A,k)=\Hom_{\grp}(A_k,\bG_{a,k})$ sends an element $\alpha\cdot x_i$ with $0\neq \alpha\in \fa_k$ to the group scheme homomorphism $A_k\xrightarrow{L_{\alpha}}\bG_{a,k}\xrightarrow{t\mapsto t^{p^i}}\bG_{a,k}$ where $L_{\alpha}$ is the unique group scheme homomorphism that induces the functional $\alpha$ on Lie algebras.

We now turn to proving (4). Part (3) can either be proven by the same argument or deduced formally using that cohomology groups $H^{\bullet}(A(\cO_F),W_2(k))$ are flat $W_2(k)$-modules. Additionally to assuming that $A=\bG_a$ we may and do assume that $T=\bG_m$, acting on $A$ through some power of the standard character. Since $A(\cO_F)$ is isomorphic to $\bZ^{[F:\bQ]}$, the cohomology ring $H^{\bullet}(A(\cO_F),W_2(k))$ is $T(\cO_F)$-equivariantly isomorphic to $\Lambda^{\bullet}H^1(A(\cO_F),W_2(k))$ via the multiplication on cohomology. 

The $1$st cohomology module $H^1(A(\cO_F),W_2(k))$ is naturally identified with $\Hom(A(\cO_F),W_2(k))=\Hom_{W_2(k)}(\cO_F\otimes_{\bZ} W_2(k),W_2(k))$ where $T(\cO_F)=\cO_F^{\times }$ acts by multiplication on the source of the maps. The algebra $\cO_F\otimes_{\bZ}W_2(k)$ is isomorphic to $\bigoplus\limits_{\tau\in \Gal(F/\bQ)}W_2(k)$ via the isomorphism sending $a\otimes b\in \cO_F\otimes_{\bZ}W_2(k)$ to $\oplus \kappa(\tau(a))b$ where $\kappa:\cO_F\to{\cO_F/\fp^2}\simeq W_2(\bF_q)$ is the unique lift of the chosen identification $\cO_F/\fp\simeq\bF_q$. Therefore, $A(\cO_F)\otimes_{\bZ}W_2(k)$ is isomorphic to $\bigoplus\limits_{\tau\in\Gal(F/\bQ)} (\Lie A_{W_2(k)})^{\tau}$ as a $W_2(k)$-module with a $\cO_F^{\times}$-action, which implies the claim by dualizing.
\end{proof}

Since $k$ is assumed to contain $\bF_q$, every representation of the finite abelian group $T(\bF_q)$ on a finite-dimensional $k$-vector space decomposes as a direct sum of characters. We will sometimes refer to the characters of $T(\bF_q)$ appearing as direct summands of a representation as $T(\bF_q)$-{\it weights} of this representation. The image of the restriction map $X^*(T)=\Hom_{\grp}(T,\bG_m)\to \Hom(T(\bF_q),k^{\times})$ is identified with $X^*(T)/(q-1)$.

It follows from Proposition \ref{group cohomology: additive group cohomology} that 

\begin{cor}\label{group cohomology: additive group weights}In the setting or Proposition \ref{group cohomology: additive group cohomology}, the following holds.\begin{enumerate}
    \item If a character $\chi\in X^*(T)$ is a weight of $H^n_{\alg}(A,k)$ then $\chi$ can be expressed as a sum of $\leq n$ elements of $-p^{\bN}\cdot \Delta_A\subset X^*(T)$.
    \item If a character $\chi$ of $T(\bF_q)$ is a weight of $H^n(A(\bF_q),k)$ then $\chi$ extends to an algebraic character $\widetilde{\chi}$ of $T$ that is congruent modulo $q-1$ to an element of $X^*(T)$ expressible as a sum of $\leq n$ elements of $-p^{\bN}\cdot \Delta_A$.
\end{enumerate} 
\end{cor}

These observations imply the following injectivity criterion

\begin{cor}[{\hspace{1sp}\cite[5.4]{cpsk}}]\label{group cohomology: abelian restriction criteria} In the setting of Proposition \ref{group cohomology: additive group cohomology}, let $\chi$ be a character of $T$. As above, $\bF_q$ is the finite field of size $p^r$.
\begin{enumerate}
\item If every equality of the form $\chi=p^{r_1}\lambda_1+\ldots+p^{r_l}\lambda_l$ with $\lambda_1,\ldots,\lambda_l\in -\Delta_A$, $l\leq n$, and some $r_1,\ldots,r_l\geq 0$ satisfies $r_1,\ldots,r_l\leq r-1$, then the restriction map $\Hom_T(\chi, H^n_{\alg}(A,k))\to \Hom_{T(\bF_q)}(\chi, H^n(A(\bF_q),k))$ is injective.

\item If $\chi$ is not congruent modulo $(q-1)\cdot X^*(T)$ to a sum of the form $r_1+\ldots+r_l$ with $r_1,\ldots,r_l\in -p^{\bN}\cdot \Delta_A$ and $l\leq n$ then $\Hom_{T(\bF_q)}(\chi, H^n(A(\bF_q),k))=0$.
\end{enumerate}
\end{cor}

\begin{proof}
1) Proposition \ref{group cohomology: additive group cohomology} (1), (2) implies that there exists a $T(\bF_q)$-equivariant map $s_n:H^n(A(\bF_q),k)\to H^n_{\alg}(A,k)$ such that its composition with the restriction $H^n_{\alg}(A,k)\to H^n(A(\bF_q),k)$ is the identity map: we define $s_n$ by sending $x_i$ to $x_i$, and $\beta(x_i)$ to $\beta(x_i)$. The assumption on character $\chi$ implies that every appearance of $\chi$ as a $T$-equivariant direct summand of $H^n_{\alg}(A,k)$ is in the image of $s_n$, which implies the injectivity.

2) Immediate from Corollary \ref{group cohomology: additive group weights}(2).
\end{proof}

We can deduce from Proposition \ref{group cohomology: additive group cohomology} the following results about the action of $T$ on the cohomology of an arbitrary unipotent group.

\begin{lm}\label{group cohomology: algebraic action on unipotent cohomology}
Let $U$ be a unipotent algebraic group over $k$ equipped with an action of a split torus $T$. As before, $\Delta_U\subset X^*(T)$ is the set of characters of $T$ that appear as weights of the action of $T$ on the Lie algebra $\Lie U$.

\begin{enumerate}
\item Every weight of $T$ on $H^i_{\alg}(U,k)$ can be expressed as a sum of $\leq i$ elements of $-p^{\bN}\cdot \Delta_{U}$

\item Let $U'\subset U$ be a normal $T$-stable subgroup, and assume the subsets $p^{\bN}\Delta_{U'},p^{\bN}\Delta_{U/U'}\subset X^*(T)$ are disjoint. Fix an integer $i$. Suppose that for every expression $\chi=r_1+\dots+r_l$ with $l\leq i$ and $r_1,\ldots,r_l\in -p^{\bN}\cdot \Delta_U$, all $r_1,\ldots,r_l$ are contained in $-p^{\bN}\cdot\Delta_{U'}$. Then the map $\Hom_T(\chi,H^i_{\alg}(U,k))\to \Hom_T(\chi,H^i_{\alg}(U',k))$ is injective.
\end{enumerate}
\end{lm}

\begin{proof}
Since $U$ is unipotent, there exists a $T$-equivariant filtration $U_0=U\supset U_1\supset\dots\supset U_{n}=1$ by normal subgroups, such that all graded quotients $U_{m}/U_{m+1}$ are isomorphic to $\bG_a$. We have the Hochschild-Serre spectral sequence
\begin{equation}\label{group cohomology: algebraic action spectral sequence}
E_2^{r,s}=H^r_{\alg}(U/U_{n-1},H^s_{\alg}(U_{n-1},k))\Rightarrow H^{r+s}_{\alg}(U,k)
\end{equation}

Since $U_{n-1}$ is a central subgroup of $U$, the term $E_2^{r,s}$ is $T$-equivariantly isomorphic to $H^r_{\alg}(U/U_{n-1},k)\otimes H^s_{\alg}(U_{n-1},k)$. We can use this spectral sequence to prove (1) by induction on $n=\dim U$, with the base case given by Proposition \ref{group cohomology: additive group cohomology}(1). By the inductive assumption, the $T$-weights of $H^r_{\alg}(U/U_{n-1},k)$ are sums of $\leq r$ elements of $-p^{\bN}\cdot\Delta_{U/U_{n-1}}$, and the weights of $H^s_{\alg}(U_{n-1},k)$  are sums of $\leq s$ elements of $-p^{\bN}\cdot\Delta_{U_{n-1}}$ by Proposition \ref{group cohomology: additive group cohomology}(1). From the spectral sequence (\ref{group cohomology: algebraic action spectral sequence}) we have that $H^i_{\alg}(U,k)$ is a $T$-equivariant subquotient of $\bigoplus\limits_{r+s=i}H^r_{\alg}(U/U_{n-1},k)\otimes H^s_{\alg}(U_{n-1},k)$ which proves the inductive step, completing the proof of (1).

To prove (2), consider the spectral sequence
\begin{equation}
\tE_2^{r,s}=H^r_{\alg}(U/U',H^s_{\alg}(U',k))\Rightarrow H^{r+s}_{\alg}(U,k).
\end{equation}
To show injectivity of the restriction map \begin{equation}\Hom_T(\chi,H^i_{\alg}(U,k))\to \Hom_T(\chi,\tE_2^{0,i})\subset \Hom_T(\chi,H^i_{\alg}(U',k))\end{equation} it is enough to prove that the spaces $\Hom_T(\chi,H^{i-s}_{\alg}(U/U',H^s_{\alg}(U',k)))$ vanish for $s<i$. Since every representation of the solvable group scheme $T\ltimes(U/U')$ is a successive extension of characters, $H^s_{\alg}(U',k)$ has a filtration by $U/U'$-submodules with $1$-dimensional graded pieces which is also respected by $T$. Therefore $T$-weights of $H^{i-s}_{\alg}(U/U',H^s_{\alg}(U',k))$ form a subset of the set of $T$-weights of $H^{i-s}_{\alg}(U/U',k)\otimes H^s_{\alg}(U',k)$. 

For $s<i$ the character $\chi$ does not appear in the latter set because our assumption implies that $\chi$ cannot be written as $\chi'+\chi''$ where $\chi'$ is a sum of $\leq s$ elements of $-p^{\bN}\cdot \Delta_{U'}$, and $\chi''$ is a sum of $\leq i-s$ elements of $-p^{\bN}\cdot\Delta_{U/U'}$ .
\end{proof}

We will now compute $\Delta_{U_p}$ for the unipotent radical $U_p\subset B_p$ of a Borel subgroup of $SL_p$. Let $e_1,\ldots,e_p$ be a basis of $V$ so that $B_p$ is the subgroup preserving each of the subspaces $\langle e_1,\dots,e_{i}\rangle$, and $U_p\subset B_p$ is the subgroup of matrices that moreover act trivially on the quotients $\langle e_1,\dots,e_{i}\rangle/\langle e_1,\dots,e_{i-1}\rangle$. Let $\chi_i\in X^*(T_p)$ be the character of the torus $T_p=B_p/U_p$ through which it acts on $e_i$. Note that $\chi_p=-\chi_1-\ldots-\chi_{p-1}$, and $\chi_1,\ldots,\chi_{p-1}$ form a basis of $X^*(T_p)$. The set of positive roots $\Delta_{U_p}$ is equal to \begin{equation}\{\chi_i-\chi_j|1\leq i<j\leq p-1\}\cup \{\chi_i+(\chi_1+\ldots+\chi_{p-1})|1\leq i\leq p-1\}\end{equation} On the other hand, the weights of $T_p$ on the representation $V^{(1)}$ are given by $p\chi_1,\ldots,p\chi_{p-1},p\chi_p=-p(\chi_1+\ldots+\chi_{p-1})$.

Our goal is to prove injectivity of the map $H^{p-1}_{\alg}(U_p,V^{(1)})^T\to H^{p-1}(U_p(\bF_q),V^{(1)})$ for $q>p$. To do so we will analyze the $T$-action on $H^i_{\alg}(U_p,k)$ for $i\leq p-1$ and determine for which $i,j$ the cohomology group $H^i_{\alg}(U_p,\chi_j^p)$ might have non-zero $T$-invariants. This will be achieved through the following combinatorial computation:

\begin{lm}\label{group cohomology: weights of frobenius twist}
\begin{enumerate}
\item For $2\leq j\leq p$ the character $p\chi_{j}$ does not belong to the submonoid of $X^*(T_p)$ spanned by $\Delta_{U_p}$. 

\item The only (up to permutation) way to express $p\chi_1$ as a sum of $\leq p-1$ elements of $p^{\bN}\cdot\Delta_{U_p}$ is $(\chi_1-\chi_2)+\ldots+(\chi_1-\chi_{p-1})+(\chi_1+(\chi_1+\chi_2+\ldots+\chi_{p-1}))$.

\item For $q>p$, and any $2\leq j\leq p$, the character $p\chi_j$ is not congruent modulo $q-1$ to a sum of $\leq p-1$ elements of $p^{\bN}\cdot \Delta_{U_p}$.

\item For $q>p$, any congruence modulo $q-1$ between $p\chi_1$ and a sum of $\leq p-1$ elements of $p^{\leq r-1}\cdot \Delta_{U_p}$ is an equality. In particular, by (2) there are no such congruences with strictly less than $p-1$ summands. 
\end{enumerate}
\end{lm}

\begin{proof}
The first statement for $j=p$ is clear because the image of every element of $\Delta_{U_p}$ under the map $\sigma:X^*(T_p)\xrightarrow{b_1\chi_1+\ldots+b_{p-1}\chi_{p-1}\mapsto b_1+\ldots+b_{p-1}}\bZ$ is non-negative. Similarly, a linear combination with non-negative coefficients of elements of $\Delta_{U_p}$ that belongs to $\langle \chi_2,\ldots,\chi_{p-1}\rangle$ must be a combination of elements $\chi_i-\chi_j,2\leq i<j\leq p-1$ and is therefore killed by $\sigma$. This shows that $p\chi_j$ for $j=2,\ldots,p-1$ are not in the monoid generated by $\Delta_{U_p}$ either. 

For the second statement consider an arbitrary expression $p\chi_1=p^{r_1}\lambda_1+\ldots+p^{r_l}\lambda_l$ with $l\leq p-1$, $\lambda_1,\ldots,\lambda_l\in\Delta_{U_p}$. Since $\sigma(p\chi_1)=p$, there is exactly one $m$ such that $\lambda_m$ is from the set $\{\chi_1+(\chi_1+\ldots+\chi_{p-1}),\ldots,\chi_{p-1}+(\chi_1+\ldots+\chi_{p-1})\}$ and $r_m=0$ for this $m$. There are exactly $p-2$ other elements of $\Delta_{U_p}$ in which $\chi_1$ appears with a non-zero coefficient: $\chi_1-\chi_2,\ldots,\chi_1-\chi_{p-1}$. Therefore $l=p-1$, all $r_1,\ldots,r_{p-1}$ are equal to $0$, and all $\lambda_1,\ldots,\lambda_{p-1}$ are elements of the set $\{\chi_1-\chi_2,\ldots,\chi_1-\chi_{p-1},2\chi_1+\chi_2+\ldots+\chi_{p-1}\}$. Hence there must be no repetitions among $\lambda_1,\ldots,\lambda_{p-1}$ for them to sum up to $p\chi_1$, which proves assertion (2).

Suppose that, contrary to the assertion (3), there is a congruence \begin{equation}\label{group cohomology: weights congruence formula}p\chi_j\equiv p^{r_1}\lambda_1+\ldots+p^{r_l}\lambda_l\bmod q-1\end{equation} with $l\leq p-1$ and all $\lambda_m\in  \Delta_{U_p}$. The coefficient of $\chi_1$ in $p^{r_1}\lambda_1+\ldots+p^{r_l}\lambda_l$ is a sum of $\leq 2l$ powers of $p$. Hence it is a number whose sum of digits in base $p$ is at most $2l\leq 2(p-1)$. Moreover, its sum of digits is equal to $2(p-1)$ only if $l=p-1$, and all $\lambda_1,\ldots,\lambda_{p-1}$ are equal to $2\chi_1+\chi_2\ldots+\chi_{p-1}$. This would violate (\ref{group cohomology: weights congruence formula}), because the right hand side would have the shape $(p^{\lambda_1}+\ldots +p^{\lambda_{p-1}})\cdot (2\chi_1+\chi_2\ldots+\chi_{p-1})$. Therefore the sum of digits in base $p$ of the coefficient of $\chi_1$ in $p^{r_1}\lambda_1+\ldots+p^{r_l}\lambda_l$ is at most $2(p-1)-1$.

Note that a number with base $p$ expansion $\overline{a_n\ldots a_ra_{r-1}\ldots a_0}$ is congruent to $\overline{a_n\ldots a_r}+\overline{a_{r-1}\ldots a_0}$ modulo $p^r-1$. Applying this observation repeatedly, we see that for any non-zero number $a$ there exists a number $0<a'<p^r$ congruent to $a$ modulo $p^r-1$ and with sum of digits less or equal to that of $a$. For $2\leq j\leq p-1$ the coefficient of $\chi_1$ in $p\chi_j$ is zero, so by this discussion there would have to be an integer $0<a'<p^r$ divisible by $p^r-1$ with the sum of digits $\leq 2(p-1)-1$, but there is no such number. 

For $j=p$ the coefficient of $\chi_1$ in $p\chi_j$ is $-p$, but the only number $0<a'<p^r-1$ congruent to $-p$ modulo $p^r-1$ is $p^r-p-1=\overline{(p-1)(p-1)\ldots (p-1)(p-2)(p-1)}$ and its sum of digits is $r(p-1)-1$. This finishes the proof of (3) if $r>2$, but for $r=2$ we still need to rule out the possibility that the sum of digits in base $p$ of the coefficient of $\chi_1$ in $p^{r_1}\lambda_1+\ldots+p^{r_l}\lambda_l$ is $2(p-1)-1$. If this was the case, up to reordering the summands, this sum would have the form \begin{equation}(p^{r_1}+\ldots+p^{r_{p-2}})(2\chi_1+\chi_2+\ldots+\chi_{p-1})+p^{r_{p-1}}(\chi_i+(\chi_1+\ldots+\chi_{p-1}))\end{equation} or \begin{equation}(p^{r_1}+\ldots+p^{r_{p-2}})(2\chi_1+\chi_2+\ldots+\chi_{p-1})+p^{r_{p-1}}(\chi_1-\chi_i),\end{equation} for some $i=2,\ldots ,p-1$. Neither of these expressions can be congruent to $p\chi_p=-p(\chi_1+\ldots+\chi_{p-1})$ modulo $p^2-1$: if $p>3$ this is clear because for all $j\neq 1,i$ (which exists since $p>3$), the difference between the coefficients of $\chi_1$ and $\chi_j$ in these sums is a positive integer with the sum of digits $\leq p-1$, and in particular it cannot be zero modulo $p^2-1$. If $p=3$ then we have a congruence of the form $-3(\chi_1+\chi_2)\equiv 3^{r_1}(2\chi_1+\chi_2)+3^{r_2}(\chi_1+2\chi_2)$ or $-3(\chi_1+\chi_2)\equiv 3^{r_1}(2\chi_1+\chi_2)+3^{r_2}(\chi_1-\chi_2)$ modulo $8$. In the first case we would have $3^{r_1}\equiv 3^{r_2}$ and $3^{r_1+1}\equiv -3$ which is impossible because $(-1)$ is not a power of $3$ modulo $8$. In the second case comparing the coefficients of $\chi_2$ we arrive at the contradiction as well, finishing the proof of (3).

For the assertion (4), suppose that $p\chi_1\equiv p^{r_1}\lambda_1+\ldots+p^{r_l}\lambda_{l}$ is such a congruence. The value of $\sigma$ on its right hand side is a number less than or equal to $p^{r}(p-1)$ whose sum of digits is at most $p-1$, and which is congruent to $p$ modulo $p^r-1$. The only such number is $p$ itself (we use here that $r>1$), which implies that for exactly one value of $m$ we have $r_m=0$ and $\lambda_m=\chi_i+(\chi_1+\ldots+\chi_{p-1})$, while for all $m'\neq m$ the character $\lambda_{m'}$ is of the form $\chi_j-\chi_k$. 

Next, we consider the coefficient $c$ of $\chi_1$ in the right hand side of our congruence. Summands $p^{r_{m'}}\lambda_{m'}$ for $m'\neq m$ contribute at most $p^{r-1}$ to this coefficient, and $p^{r_m}\lambda_m=\lambda_m$ contributes $1$ or $2$. Hence $c$ is less than or equal to $p^{r-1}(p-2)+2$. As $c$ also has to be congruent to $p$ modulo $p^r-1$, it is forced to be equal to $p$. Given what we already know about the right hand side, this can only happen if $l=p-1$, and all $r_1,\ldots,r_l$ are equal to zero, hence the congruence is forced to be an equality. 
\end{proof}

Let $\bG_a^{p-1}\simeq A_p\subset U_p$ be the subgroup of matrices that act trivially on the quotient $V/\langle e_1\rangle$. Note that $A_p$ is preserved by the action of $T_p$ and its set of roots is \begin{equation}\Delta_{A_p}=\{\chi_1-\chi_i|2\leq i\leq p-1\}\cup \{2\chi_1+\chi_2+\ldots+\chi_{p-1}\}\subset \Delta_{U_p}.\end{equation} Lemma \ref{group cohomology: weights of frobenius twist} indicates that restriction to the subgroup $A_p\subset U_p$ should detect all cohomology classes of $V^{(1)}$ in degrees $\leq p-1$. We make this precise in Lemma \ref{group cohomology: A restriction injective} below. The deduction of Lemma \ref{group cohomology: A restriction injective} from Lemma \ref{group cohomology: weights of frobenius twist} is analogous to the discussion of injectivity conditions in \cite[\S 5]{cpsk}.

\begin{lm}\label{group cohomology: A restriction injective}
The following restriction maps are injective:
\begin{enumerate}
    \item $H^{p-1}_{\alg}(B_p,V^{(1)})=H^{p-1}_{\alg}(U_p,V^{(1)})^{T_p}\to H^{p-1}_{\alg}(A_p,V^{(1)})^{T_p}$.
    \item $H^{p-1}_{\alg}(A_p,V^{(1)})^{T_p}\to H^{p-1}(A_p(\bF_q),V^{(1)})^{T_p(\bF_q)}$, if $q$ is strictly larger than $p$.
\end{enumerate}
\end{lm}

\begin{proof}
As a representation of $B_p$, $V^{(1)}$ admits a filtration with graded pieces given by the characters $\chi_i^p$, for $i=1,\ldots,p$. Given Lemma \ref{group cohomology: weights of frobenius twist}(1) and (2), Lemma \ref{group cohomology: algebraic action on unipotent cohomology}(1) implies that $H^i_{\alg}(B_p,\chi_j^p)=H^i_{\alg}(U_p,\chi_j^p)^{T_p}=0$ for $2\leq j\leq p$ and all $i$, and for $j=1$ with $i<p-1$. Lemma \ref{group cohomology: algebraic action on unipotent cohomology}(2) shows that the restriction $H^j_{\alg}(U_p,\chi_j^p)^{T_p}\to H^j_{\alg}(A_p,\chi_j^p)^{T_p}$ is an isomorphism (both groups are in fact zero) for $j<p-1$ and is an injection for $j=p-1$. Therefore the restriction $H^{p-1}_{\alg}(B_p,V^{(1)})\to H^{p-1}_{\alg}(A_p,V^{(1)})^{T_p}$ is injective.

For the second statement, $H^{i}(A_p(\bF_q),\chi_j^p)^{T(\bF_q)}=0$ for $i<p-1$ and all $j$, by the combination of Corollary \ref{group cohomology: additive group weights}(2) and Lemma \ref{group cohomology: weights of frobenius twist}(3), (4). The restriction maps $H^{p-1}_{\alg}(A_p,\chi_j^p)^{T_p}\to H^{p-1}(A_p(\bF_q), \chi_j^p)^{T_p(\bF_q)}$ are injective by Lemma \ref{group cohomology: weights of frobenius twist}(3),(4) and Corollary \ref{group cohomology: abelian restriction criteria} (the source group is in fact zero for $j\neq 1$). This implies that the restriction $H^{p-1}_{\alg}(A_p,V^{(1)})^{T_p}\to H^{p-1}(A_p(\bF_q),V^{(1)})^{T_p(\bF_q)}$ is injective.
\end{proof}

\begin{proof}[Proof of Proposition \ref{group cohomology: from algebraic to discrete}]
Lemma \ref{group cohomology: A restriction injective} completes the proof of Proposition \ref{group cohomology: from algebraic to discrete}, because combined with Theorem \ref{group cohomology: restriction to Borel iso} it even shows that the composition $H^{p-1}_{\alg}(SL_p,V^{(1)})\simeq H^{p-1}_{\alg}(B_p,V^{(1)})\to H^{p-1}(A_p(\bF_q),V^{(1)})$ is injective.
\end{proof}

Moreover, the $1$-dimensional subrepresentation $\chi_1^p\subset V^{(1)}$ is responsible for all of cohomology of $V^{(1)}$ in degree $p-1$. Precisely, we have the following results that will be used in the next section, and in the proof of Lemma \ref{group cohomology: reducible formula}.

\begin{lm}\label{group cohomology: everything from chi1 algebraic}
\begin{enumerate}
    \item The map $H^{p-1}_{\alg}(A_p,\chi_1^p)^{T_p}\to H^{p-1}_{\alg}(A_p,V^{(1)})^{T_p}$ is an isomorphism of $1$-dimensional vector spaces.
    \item The map $H^{p-1}(A_p(\bF_q),\chi_1^p)^{T_p(\bF_q)}\to H^{p-1}(A_p(\bF_q),V^{(1)})^{T_p(\bF_q)}$ is an isomorphism when $q>p$.
\end{enumerate}
\end{lm}

\begin{proof}
1) The kernel and cokernel of this map are, respectively, a quotient and a subgroup of the groups  $H^{p-2}_{\alg}(A_p,\chi_2^p\oplus\ldots\oplus \chi_p^p)^{T_p}$ and $H^{p-1}_{\alg}(A_p,\chi_2^p\oplus\ldots \oplus \chi_p^p)^{T_p}$ and we saw in the proof of Lemma \ref{group cohomology: A restriction injective} that they both vanish. The fact that the $T_p$-invariant subspace of $H^{p-1}_{\alg}(A_p,\chi_1^p)$ is $1$-dimensional follows from Lemma \ref{group cohomology: weights of frobenius twist}(2).

2) Just like in the case of algebraic cohomology, the kernel and cokernel of this map are subquotients of the groups  $H^{p-2}(A_p(\bF_q),\chi_2^p\oplus\ldots\oplus \chi_p^p)^{T_p(\bF_q)}$ and $H^{p-1}(A_p(\bF_q),\chi_2^p\oplus\ldots \oplus \chi_p^p)^{T_p(\bF_q)}$ and they vanish by the proof of Lemma \ref{group cohomology: A restriction injective}.
\end{proof}

We can deduce from the results obtained so far the following expression for the class $\alpha(E)$ when $E$ is a vector bundle of rank $p$ that admits a line sub-bundle, which was used in Section \ref{nonsemisimp: section} (as Lemma \ref{nonsemisimp: reducible formula}) to relate the class $\alpha(\Omega^1_{X_0})$ to the Kodaira-Spencer map of a fibration.

\begin{lm}\label{group cohomology: reducible formula}
Let $X_0$ be arbitrary algebraic stack over $\bF_p$. Suppose that a vector bundle $E$ of rank $p$ on $X_0$ fits into an extension
\begin{equation}
0\to L\to E\to E'\to 0
\end{equation}
where $L$ is a line bundle, and $E'$ is a vector bundle of rank $p-1$. The class of this extension defines an element $v(E)\in \Ext^1_{X_0}(E',L)=H^1(X_0,L\otimes (E')^{\vee})$. Denote by $v(E)^{p-1}\in H^{p-1}(X_0,L^{\otimes p-1}\otimes (\det E')^{\vee})$ the image of $v(E)^{\otimes p-1}\in H^{p-1}(X_0, (L\otimes (E')^{\vee})^{\otimes p-1})$ under the map induced by $(L\otimes (E')^{\vee})^{\otimes p-1}\to \Lambda^{p-1}(L\otimes (E')^{\vee})=L^{\otimes p-1}\otimes (\det E')^{\vee}$.

The class $\alpha(E)\in \Ext^{p-1}_{X_0}(\Lambda^p E, F^*E)=H^{p-1}(X_0,F^*E\otimes L^{\vee}\otimes (\det E')^{\vee})$ is equal, up to multiplying by a scalar from $\bF_p^{\times}$, to the image of $v(E)^{p-1}$ under the map induced by $L^{\otimes p-1}\otimes(\det E')^{\vee}=F^*L\otimes L^{\vee}\otimes(\det E')^{\vee}\hookrightarrow F^*E\otimes L^{\vee}\otimes(\det E')^{\vee}$.
\end{lm}

\begin{proof}
The $GL_{p}$-torsor $\underline{\Isom}(E,\cO^{\oplus p})$ associated to the bundle $E$ naturally reduces to a maximal parabolic subgroup $P:=\left(\begin{matrix}* & * & \ldots & *\\
0 & * & \ldots & * \\
0 & * & \ddots & \vdots \\
0 & * & \ldots & *
\end{matrix}\right)\subset GL_{p}$. Therefore $E$ arises as the pullback of the tautological rank $p$ vector bundle $V$ along the classifying map $X_0\to BP$ to the classifying stack of the group $P$ over $\bF_p$. Hence it is enough to prove the corresponding expression for the class $\alpha(V)\in \Ext^{p-1}_{P,\alg}(\Lambda^p V,V^{(1)})=H^{p-1}_{\alg}(P,V^{(1)}\otimes (\Lambda^pV)^{\vee})$.

Denote by $B_p\subset P\cap SL_{p}$ the subgroup of upper triangular matrices. By \cite[Corollary II.4.7(c)]{jantzen} the restriction map $H^i(P\cap SL_p, W)\to H^i(B_p, W|_{B_p})$ is an isomorphism for any $P\cap SL_p$-module $W$. Since restriction along the inclusion $P\cap SL_p\subset P$ is an injection on cohomology,  it is enough to prove the desired expression for $\alpha(V)$ in $\Ext^{p-1}_{B_p,\alg}(\Lambda^p V,V^{(1)})=H^{p-1}_{\alg}(B_p,V^{(1)})$.

Recall that we denote by $T_p\subset B_p$ the maximal torus of diagonal matrices and by $\chi_i:T_p\to \bG_m$, for $i=1,\ldots,p$, the character sending $\diag(a_1,\ldots,a_p)$ to  $a_i$. We denote by the same symbol the composite character $B_p\to T_p\xrightarrow{\chi_i}\bG_m$. Since we are working inside $SL_p$ the character $\chi_p$ can be expressed as $(\chi_1\cdot\ldots\cdot \chi_{p-1})^{-1}$. Next, there is a subgroup $\bG_a^{p-1}\simeq A_p=\left(\begin{matrix}1 & * & \ldots & *\\
0 & 1 & \ldots & 0 \\
0 & 0 & \ddots & \vdots \\
0 & 0 & \ldots & 1
\end{matrix}\right)\subset B_p$. The Frobenius twist $V^{(1)}$ viewed as a representation of $B_p$ admits a filtration with quotients $\chi_1^p,\ldots,\chi_p^p$.

By Lemma \ref{group cohomology: A restriction injective}(1), we may further restrict to the subgroup $T_p\ltimes A_p\subset B_p$ to prove the desired equality of cohomology classes. The class $v(V|_{T_p\ltimes A_p})\in H^1_{\alg}(T_p\ltimes A_p, \chi_1\otimes (\chi_2\oplus\ldots\oplus \chi_p)^{\vee})=\Hom_{\grp}(A_p,\chi_1\otimes (\chi_2\oplus\ldots\oplus \chi_p)^{\vee})^{T_p}$ is an isomorphism between $A_p$ and the underlying vector space of the representation $\chi_1\otimes (\chi_2\oplus\ldots\oplus \chi_p)^{\vee}$, and its power $v(V|_{T_p\ltimes A_p})^{p-1}\in H^{p-1}_{\alg}(A_p,\chi_1^{p-1}\otimes \chi_2^{-1}\otimes\ldots\otimes \chi_p^{-1})^{T_p}=H^{p-1}(A_p,\chi_1^p)^{T_p}$ is therefore non-zero.

By Lemma \ref{group cohomology: everything from chi1 algebraic}(1) the class $\alpha(V|_{T_p\ltimes A_p})\in H^{p-1}(A_p,V^{(1)})^{T_p}$ is the image of some class in $H^{p-1}(A_p,\chi_1^p)^{T_p}$. Since $\alpha(V|_{T_p\ltimes A_p})$ is non-zero, and $H^{p-1}(A_p,\chi_1^p)^{T_p}$ is a one-dimensional vector space, the result follows.
\end{proof}

\begin{rem}
We do not expect Lemma \ref{group cohomology: A restriction injective}(2) to remain true for $q=p$. For instance, when $p=2$ the class $\alpha(V)$ vanishes in $H^1(SL_2(\bF_2),V^{(1)})$ because the surjection $S^2(\bF_2^{\oplus 2})\to \Lambda^2(\bF_{2}^{\oplus 2})$ has an $SL_2(\bF_2)$-equivariant section given by sending the only non-zero element of $\Lambda^2(\bF_2^{\oplus 2})$ to $x\cdot y+y\cdot z+z\cdot y$, where $x,y,z$ are the three non-zero elements of $\bF_2^{\oplus 2}$.

\comment{E.g. when $p=2$ in Remark \ref{group cohomology: p=2 remark} we computed explicitly a cocycle representing the class $\alpha(V)\in H^1_{\alg}(SL_2,V^{<(1)})$, and its restriction to $U_2\simeq\bG_a$ is the image of $\Id\in \Hom_{\grp}(U_2,\bG_a)=H^1_{\alg}(U_2,k)$ under the map in the long exact sequence \begin{equation}\dots\to H^0_{\alg}(U_2,k)\xrightarrow{\delta} H^1_{\alg}(U_2,k)\to H^1_{\alg}(U_2,V^{(1)})\to H^1_{\alg}(U_2,k)\to\ldots\end{equation} The connecting homomorphism $\delta$ sends $1\in k=H^0_{\alg}(U_2,k)$ to the Frobenius map $\Fr_2\in \Hom(U_2,\bG_a)=H^1_{\alg}(U_2,k)$. Since $\Id$ coincides with $\Fr_2$ when restricted to $U_2(\bF_2)$, the image of $\alpha(V)\in H^1_{\alg}(SL_2,V^{(1)})$ in $H^1(U_2(\bF_2),V^{(1)})$ (and consequently in $H^1(SL_2(\bF_2),V^{(1)})$) is zero.}
\end{rem}

\section{Cohomology of \texorpdfstring{$SL_p$}{} over rings of integers}
\label{borel: section}

We will enhance the results of the previous section by showing that for an appropriately chosen discrete group acting on the vector space $V$, the Bockstein homomorphism (associated to a lift of $V$ over $W_2(\bF_q)$) applied to the class $\alpha(V)$ is non-zero. We keep the notation of the previous section: we work over a finite field $k=\bF_q$, and $V$ is a $p$-dimensional $\bF_q$-vector space.

Let $\tV$ be a free $W_2(\bF_q)$-module such that $\tV/p\simeq V$, it is equipped with the tautological action of the discrete group $GL_p(W_2(\bF_q))$. Denote by $\tV^{(1)}:=\tV\otimes_{W_2(\bF_q),W_2(\Fr_p)}W_2(\bF_q)$ the twist of $\tV$ by the Frobenius automorphism of $W_2(\bF_q)$ induced by $\Fr_p:t\mapsto t^p$ on $\bF_q$. The module $\tV^{(1)}$ can be $GL_p(W_2(\bF_q))$-equivariantly identified with $\tV$ where the action on the latter is modified by precomposing with the Frobenius automorphism $GL_p(W_2(\bF_q))\xrightarrow{W_2(\Fr_p)}GL_p(W_2(\bF_q))$. We have a short exact sequence \begin{equation}0\to V^{(1)}\xrightarrow{\psi_1}\tV^{(1)}\xrightarrow{\psi_2} V^{(1)}\to 0\end{equation} of $GL_p(W_2(\bF_q))$-modules which gives rise to the connecting homomorphisms $\Bock^i: H^i(G, V^{(1)})\to H^{i+1}(G, V^{(1)})$ for any group $G$ mapping to $GL_p(W_2(\bF_q))$. 

Let now $F$ be an arbitrary number field in which $p$ is unramified and such that there exists a prime ideal $\fp\subset \cO_F$ with the residue field $\bF_q$. The ideal $\fp$ gives rise to a surjection $\kappa:\cO_F\twoheadrightarrow\bF_q$, and there is a unique homomorphism of rings $\cO_F\twoheadrightarrow W_2(\bF_q)$ lifting $\kappa$. This homomorphism gives rise to a map $SL_p(\cO_F)\to SL_p(W_2(\bF_q))$ which defines an action of $SL_p(\cO_F)$ on $V, V^{(1)}$, and $\tV^{(1)}$, and hence defines Bockstein homomorphisms $\Bock^i:H^i(SL_p(\cO_F),V^{(1)})\xrightarrow{} H^{i+1}(SL_p(\cO_F), V^{(1)})$.

\begin{pr}\label{group cohomology: ring of integers main}
For every prime $p$ there exists a quadratic extension $F/\bQ$ in which $p$ is not split, such that $\Bock^{p-1}(\alpha(V))\in H^p(SL_p(\cO_F),V^{(1)})$ is non-zero.
\end{pr}

We will prove this non-vanishing by a method similar to the one employed in previous section, using the technique of \cite{cpsk}. For the method to work we need the image of the reduction map $\cO^{\times}_F\to\bF_{p^2}^{\times}$ on groups of units to be large enough: the field $F$ will be chosen appropriately in Lemma \ref{group cohomology: sorry not sorry}.

As in the previous section, we denote by $B_p\subset SL_p$ the subgroup of upper triangular matrices with respect to a given basis, $T_p\subset B_p$ is the diagonal torus, and $\bG_a^{p-1}\simeq A_p\subset B_p$ is the subgroup of matrices that send the basis vector $e_i$ to a vector of the form $a_ie_1+e_i$, for all $i\geq 2$. By Lemmas \ref{group cohomology: A restriction injective} and \ref{group cohomology: everything from chi1 algebraic}(2), the image of the class $\alpha(V)\in H^{p-1}_{\alg}(SL(V),V^{(1)})$ in $H^{p-1}(A_p(\bF_q),V^{(1)})^{T_p(\bF_q)}$ is non-zero and moreover lies in the image of the homomorphism $H^{p-1}(A_p(\bF_q),\chi_1^p)^{T_p(\bF_q)}\to H^{p-1}(A_p(\bF_q),V^{(1)})^{T_p(\bF_q)}$. Therefore to prove that $\Bock(\alpha(V))$ is non-zero we may work with the cohomology of the group of $\cO_F$-points of the subgroup $T_p\ltimes A_p\subset SL_p$ which is a fairly explicit object thanks to Proposition \ref{group cohomology: additive group cohomology}. 

When restricted to $B_p(\cO_F)$, the representation $\tV$ admits a filtration with graded quotients $\tchi_1,\ldots,\tchi_p$ that are characters factoring through $B_p(\cO_F)\to T_p(\cO_F)$, and lifting the characters $\chi_1,\ldots,\chi_p$. We denote by $\tchi_i^{(1)}$ the character $B_p(\cO_F)\to T_p(\cO_F)\to W_2(\bF_q)^{\times}$ obtained by composing $\tchi_i$ with the Frobenius automorphism $\Fr_p:W_2(\bF_q)^{\times}\to W_2(\bF_q)^{\times}$. The representation $\tV^{(1)}$ of $B_p(\cO_F)$ is likewise filtered with graded quotients isomorphic to $\tchi_i^{(1)}$. Note that $\tchi_i^{(1)}/p\simeq \chi_i^{(1)}\simeq \chi_i^p$, but $\tchi_i^{(1)}$ is generally {\it not} isomorphic to $\tchi_i^p$. The discrepancy between these two characters will be key for proving that the class $\alpha(V)\in H^{p-1}(SL_p(\cO_F),V^{(1)})$ does not lift to a class in $H^{p-1}(SL_p(\cO_F),\tV^{(1)})$.

We treat separately cases $p=2$ and $p>2$ both because the case of $p=2$ allows for a significantly simpler proof, and because the general proof relies on the results of the previous section that were only shown away from the case $p=2$. In all of the cases, we choose an extension $F/\bQ$ as directed by Lemma \ref{group cohomology: sorry not sorry}.

\begin{lm}\label{group cohomology: sorry not sorry}
For every prime number $p$ there exists an integer $N>0$ such that $p$ is not split in the real quadratic field $F=\bQ(\sqrt{N})$, and the two conditions are satisfied:
\begin{enumerate}
    \item The group of units $\cO_F^{\times}$ surjects onto $\{x\in\bF_{p^2}^{\times}|N_{\bF_{p^2}/\bF_p}(x)=\pm 1\}$ under the reduction map $\cO_F\to \cO_F/p\simeq\bF_{p^2}$
    \item There exists a unit $u\in\cO_F^{\times}$ whose reduction $\ou$ in $W_2(\bF_{p^2})^{\times}$ satisfies $\Fr_p(\ou)\neq \ou^p$.
\end{enumerate}
\end{lm}

\begin{proof}
We first treat the case $p>2$. Let $u_0\in\bF_{p^2}^{\times}$ be a generator of the cyclic group $\{x\in\bF_{p^2}^{\times}|N_{\bF_{p^2}/\bF_p}(x)=\pm 1\}$. It satisfies the equation $u_0^2-2d_0u_0-1=0$ where $d_0=\frac{1}{2}\Tr_{\bF_{p^2}/\bF_p}(u_0)\in\bF_p$, hence we can write $u_0$ as $d_0+\sqrt{d_0^2+1}$. Let $d\in\bZ$ be an arbitrary integer reducing to $d_0$ modulo $p$, and take $F$ to be the field $\bQ(\sqrt{d^2+1})$. By construction, $d^2+1$ is not a square modulo $p$, hence $d^2+1$ is not a square in $\bQ$, and $p$ is non-split in $\cO_F$. The element $u=d+\sqrt{d^2+1}\in\cO_F$ is invertible and it (or its conjugate) reduces to $u_0$ in $\bF_{p^2}$, hence condition (1) is satisfied.

Next, let us check that we can choose the lift $d$ of $d_0$ to ensure that condition (2) is satisfied for $u=d+\sqrt{d^2+1}$. The element $\Fr_p(\ou)\in W_2(\bF_{p^2})$ is the mod $p^2$ reduction of $d-\sqrt{d^2+1}$, hence to prove that $\Fr_p(\ou)\neq \ou^p$ it is enough to ensure that the integers $\Tr_{F/\bQ}(d-\sqrt{d^2+1})$ and $\Tr_{F/\bQ}((d+\sqrt{d^2+1})^p)$ are not congruent modulo $p^2$. The first one is equal to $2d$, and we can expand their difference as
\begin{multline}\label{borel: wieferich formula}
\Tr_{F/\bQ}((d+\sqrt{d^2+1})^p)-\Tr_{F/\bQ}(d-\sqrt{d^2+1})=\\ 2(d^p+\binom{p}{2}d^{p-2}(d^2+1)+\ldots+\binom{p}{p-1}d(d^2+1)^{(p-1)/2})-2d
\end{multline}
This is a polynomial of degree $p$ in $d$ that reduces to $2(d^p-d)$ modulo $p$. In particular, by Hensel's lemma, this polynomial has exactly one root in $\bZ/p^2$ reducing to $d_0$, so we can choose the lift $d$ of $d_0$ such that the integer (\ref{borel: wieferich formula}) is not zero modulo $p^2$.

Finally, for $p=2$ take $F=\bQ(\sqrt{5})$. The unit $a=\frac{1+\sqrt{5}}{2}\in\cO_F=\bZ[\frac{1+\sqrt{5}}{2}]$ has minimal polynomial $a^2-a-1=0$, hence $2$ is not split in $\cO_F$, and $a$ reduces to an element of $\bF_4\setminus \bF_2$ that necessarily generates $\bF_4^{\times}$. Condition (2) is fulfilled simply by $u=-1$.
\end{proof}

\begin{proof}[Proof of Proposition \ref{group cohomology: ring of integers main} for $p=2$.] First, note that the class $\alpha(V)\in H^1(SL_2(\bF_q),V^{(1)})$ survives under the map to $H^1(SL_2(\cO_F),V^{(1)})$. This is because the reduction map $SL_2(\cO_F)\to SL_2(\bF_q)$ is surjective, since $SL_2(\bF_q)$ is generated by unipotent elements, and a surjection of groups induces an injection on cohomology in degree $1$.

Next, we will check that the Bockstein map $\Bock^1:H^1(SL_2(\cO_F),V^{(1)})\to H^2(SL_2(\cO_F),V^{(1)})$ induced by the $W_2(k)$-module $\tV^{(1)}$ is injective. The key input for this is that $H^1(SL_2(\cO_F),\tV^{(1)})$ is annihilated by multiplication by $2$. To prove this, observe that the central element $\diag(-1,-1)\in SL_2(\cO_F)$ acts in the representation $\tV^{(1)}$ via multiplication by $(-1)$, so multiplication by $(-1)$ on $H^i(SL_2(\cO_F),\tV^{(1)})$ for all $i$ is equal to identity, hence these cohomology groups are $2$-torsion.

Injectivity of $\Bock^1$ now follows by considering the long exact sequence
\begin{equation}
\ldots\to H^0(V^{(1)})\xrightarrow{\Bock^0} H^1(V^{(1)})\xrightarrow{\psi_1} H^1(\tV^{(1)})\xrightarrow{\psi_2} H^1(V^{(1)})\xrightarrow{\Bock^1} H^2(V^{(1)})
\end{equation}
where $H^i$ everywhere refers to the cohomology of $SL_2(\cO_F)$. The first visible term $H^0(V^{(1)})=(V^{(1)})^{SL_2(\cO_F)}$ is zero, because $SL_2(\cO_F)\to SL_2(\bF_q)$ is a surjection, hence $\psi_1$ is injective. The composition $\psi_1\circ\psi_2$ is the multiplication by $2$ map on $H^1(\tV^{(1)})$ which we know to be zero, so $\psi_2$ has to be zero, which is equivalent to injectivity of $\Bock^1$. This finishes the proof of Proposition \ref{group cohomology: ring of integers main} for $p=2$.
\end{proof}

\begin{proof}[Proof of Proposition \ref{group cohomology: ring of integers main} for $p>2$.]  We will show the following vanishing results, and Proposition \ref{group cohomology: ring of integers main} will be deduced as a formal consequence of these.

\begin{lm}\label{group cohomology: integral restriction facts}
Assume that $q=p^2$. There exists a real quadratic extension $F/\bQ$ with $\cO_F/p\simeq\bF_q=\bF_{p^2}$, such that 

\begin{enumerate}
\item The map $H^{p-1}((T_p\ltimes A_p)(\bF_q),\chi_1^p)\to H^{p-1}((T_p\ltimes A_p)(\cO_F),\chi_1^p)$ is injective.
\item The group $H^{p-1}((T_p\ltimes A_p)(\cO_F),\chi_2^p\oplus\ldots\oplus \chi_p^p)$ vanishes.
    \item The group $H^{p-2}((T_p\ltimes A_p)(\cO_F),\chi_1^p)$ vanishes.
    \item the $W_2(k)$-module $H^{p-1}((T_p\ltimes A_p)(\cO_F),\tchi_1^{(1)})$ is annihilated by $p$.
\end{enumerate}
\end{lm}

\begin{proof}
We choose $F$ satisfying the properties listed in Lemma \ref{group cohomology: sorry not sorry}. The reduction map $A_p(\cO_F)\to A_p(\bF_q)$ induces an isomorphism $A_p(\cO_F)/p\simeq A_p(\bF_q)$. Therefore the induced map $H^1(A_p(\bF_q),k)\to H^1(A_p(\cO_F),k)$ on cohomology in degree $1$ is an isomorphism, and induces a $T_p(\cO_F)$-equivariant surjection $H^i(A_p(\bF_q),k)\to H^i(A_p(\cO_F),k)\simeq \Lambda^i (\fa_k\cdot x_{e}\oplus \fa^{(1)}_k\cdot x_{\tau})$ for $i\geq 1$, where we use the notation of Proposition \ref{group cohomology: additive group cohomology}, and $e,\tau$ are the elements of the Galois group $\Gal(F/\bQ)\simeq \bZ/2$. The action of $T_p(\cO_F)$ on $H^1(A_p(\cO_F),k)\simeq \fa_k\cdot x_{e}\oplus \fa^{(1)}_k\cdot x_{\tau}$ factors through $T_p(\cO_F)\to T_p(\bF_q)$, and this module explicitly is given as the direct sum of inverses of the characters $\chi_1-\chi_2,\ldots,\chi_1-\chi_p,p(\chi_1-\chi_2),\ldots,p(\chi_1-\chi_p)$. 

The reduction map $(\cO_F^{\times})^{p-1}\simeq T_p(\cO_F)\to T_p(\bF_q)\simeq (\bF^{\times}_{p^2})^{p-1}$ is not surjective (as soon as $p>3$), but its image is equal to $(\{x|N_{\bF_{p^2}/\bF_p}(x)=\pm 1\})^{p-1}$, by our choice of the field $F$. We need the following partial refinement of Lemma \ref{group cohomology: weights of frobenius twist} to nevertheless be able to bound the invariant subspaces $H^j(A_p(\cO_F),\chi_i^p)^{T_p(\cO_F)}$.

\begin{lm}\label{borel: combinatorics}Denote by $S$ the set $\{\chi_1-\chi_2,\ldots,\chi_1-\chi_p,p(\chi_1-\chi_2),\ldots, p(\chi_1-\chi_p)\}\subset X^*(T_p)$.
\begin{enumerate}
    \item For $2\leq i\leq p$ the character $p\chi_i$ is not congruent modulo $p+1$ to a sum of $\leq p-1$ elements of $S$.
    \item The character $p\chi_1$ is not congruent modulo $p+1$ to a sum of $\leq p-2$ elements of $S$.
    \item The only, up to permutation, congruence modulo $p+1$ between $p\chi_1$ and a sum of $\leq p-1$ elements of $S$ is the equality $p\chi_1=(\chi_1-\chi_2)+\ldots+(\chi_1-\chi_p)$.
\end{enumerate}
\end{lm}

\begin{proof}[Proof of Lemma \ref{borel: combinatorics}]
We use $\chi_1,\ldots,\chi_{p-1}$ as a basis for $X^*(T_p)$, the character $\chi_p$ is expressed as $-(\chi_1+\ldots+\chi_{p-1})$. Denote by $\sigma:X^*(T_p)\to\bZ$ the map sending $a_1\chi_1+\ldots+a_{p-1}\chi_{p-1}$ to $a_1+\ldots+a_{p-1}$. We have $\sigma(\chi_1-\chi_i)=0$ for $i\leq p-1$ and $\sigma(\chi_1-\chi_p)=p$. In the rest of the proof, symbol $\equiv$ always refers to congruence modulo $p+1$.

1) The only elements of $S$ with a non-zero value of $\sigma$ are $\chi_1-\chi_p$ and $p(\chi_1-\chi_p)$, with the values $p\equiv -1$ and $p^2\equiv 1$, respectively. Suppose that we have a congruence $p\chi_i\equiv r_1+\ldots+r_{l}\bmod p+1$ with $l\leq p-1$, and all $r_1,\ldots,r_{l}$ from $S$. 

Suppose first that $i\neq p$. If $\chi_1-\chi_p$ appears in this sum $a$ times, and $p(\chi_1-\chi_p)\equiv \chi_p-\chi_1$ appears $b$ times, then $a-b\equiv 1\bmod p+1$ because $\sigma(p\chi_i)\equiv -1$. This forces $b$ to be equal to $a-1$. Therefore the difference $p\chi_i-(a(\chi_1-\chi_p)+(a-1)(\chi_p-\chi_1))\equiv -\chi_i-\chi_1+\chi_p=-2\chi_1-\chi_2-\ldots-2\chi_i-\ldots-\chi_{p-1}$ is congruent to a sum of $\leq p-2$ elements of the form $\pm(\chi_1-\chi_j)$ for $j=2,\ldots,p-1$. But such a congruence would have to use, for each $2\leq j\leq p-1,j\neq i$, an element of the form $\pm(\chi_1-\chi_j)$ at least once, and an element of the form $\pm(\chi_1-\chi_i)$ at least twice (because $p+1\geq 4$), so at least $p-1$ elements of $S$ would be needed.

Next, let us rule out the possibility of a congruence $p\chi_p\equiv r_1+\ldots+r_{l}$. We have $\sigma(p\chi_p)=-p(p-1)\equiv -2$. Hence if $\chi_1-\chi_p$ appears $a$ times in this congruence, then $p(\chi_1-\chi_p)$ appears $a-2$ times, so $p\chi_p-2(\chi_1-\chi_p)\equiv -3\chi_1-\chi_2-\ldots-\chi_{p-1}$ is congruent to a sum of $\leq p-3$ elements of the form $\pm(\chi_1-\chi_j),j\leq p-1$. But similarly to the previous case, such a sum would have to use at least $p-2$ such elements, and the original congruence cannot exist.

2) We have $\sigma(p\chi_1)=p\equiv -1$, hence a congruence $p\chi_1\equiv r_1+\ldots +r_{l}$ would have to use $a$ instances of $\chi_1-\chi_p$ and $a-1$ instances of $p(\chi_p-\chi_1)$, for some $a$. But this leaves us with $p\chi_1-(\chi_1-\chi_p)=(p-2)\chi_1-\chi_2-\ldots-\chi_{p-1}$ being congruent to a sum of $\leq p-3$ elements of the form $\pm(\chi_1-\chi_j),j\leq p-1$, which is impossible.

3) As in part (2), such a congruence would induce a congruence between $(p-2)\chi_1-\chi_2-\ldots-\chi_{p-1}$ and a sum of $\leq p-2$ elements of the form $\pm(\chi_1-\chi_j),j\leq p-1$. For each $j=2,\ldots,p-1$, we have to use an element of the form $\pm(\chi_1-\chi_j)$ at least once, hence exactly once, and this element must be $\chi_1-\chi_j$, forcing $a=1$ and implying the desired uniqueness.
\end{proof}

We can now proceed with the proof of Lemma \ref{group cohomology: integral restriction facts}. In (1), we will prove that even the composition of this map with further restriction $H^{p-1}((T_p\ltimes A_p)(\cO_F),\chi_1^p)\to H^{p-1}(A_p(\cO_F),\chi_1^p)$ is injective. By Lemma \ref{group cohomology: weights of frobenius twist}(4), in the notation of Proposition \ref{group cohomology: additive group cohomology}, the invariant subspace $H^{p-1}(A_p(\bF_q),\chi_1^p)^{T(\bF_q)}\subset H^{p-1}(A_p(\bF_q),\chi_1^p)$ is $1$-dimensional and is equal to $\Lambda^{p-1}(\fa_k\cdot x_0)$. This shows the injectivity asserted in part (1), because $\fa_k\cdot x_0\subset H^1(A_p(\bF_q),k)$ maps isomorphically onto $\fa_k\cdot x_e\subset H^1(A_p(\cO_F),k)$, and product on cohomology $A(\cO_F)$ induces isomorphisms $\Lambda^n H^1(A_p(\cO_F),k)\simeq H^n(A_p(\cO_F),k)$.

By Lemma \ref{group cohomology: nontrivial character cohomology} below, to prove part (2) it is enough to show that every character of $T_p(\cO_F)$ appearing as a subquotient of $H^i(A_p(\cO_F),\chi_j^p)$ for $j\geq 2, i\leq p-1$ is non-trivial. The module $H^{i}(A_p(\cO_F),k)$ is isomorphic to a direct sum of characters that are products of $i$ elements of $$\{-(\chi_1-\chi_2),\ldots,-(\chi_1-\chi_p),-p(\chi_1-\chi_2),\ldots,-p(\chi_1-\chi_p)\}\subset X^*(T_p)$$ Since the kernel of the restriction $X^*(T_p)\to \Hom(T_p(\cO_F),k^{\times})$ is contained in $(p+1)\cdot X^*(T_p)$ by our choice of the field $F$, the assertion follows from Lemma \ref{borel: combinatorics}(1). Analogously, part (3) follows from Lemma \ref{borel: combinatorics}(2).

We now turn to proving part (4). As we just established, $H^i(A_p(\cO_F),\chi_1^p)$ decomposes as a direct sum of non-trivial characters of $T_p(\cO_F)$ for $i<p-1$. Therefore $\RGamma(T_p(\cO_F),H^i(A_p(\cO_F),\chi_1^p))$ and $\RGamma(T_p(\cO_F),H^i(A_p(\cO_F),\tchi_1^{(1)}))$ are quasi-isomorphic to $0$. Hence $H^{p-1}((T_p\ltimes A_p)(\cO_F),\tchi_1^{(1)})$ injects into $H^{p-1}(A_p(\cO_F),\tchi_1^{(1)})^{T_p(\cO_F)}$, and it is enough to prove that the latter $W_2(k)$-module is annihilated by $p$. 

We have a $T_p(\cO_F)$-equivariant identification $H^{p-1}(A_p(\cO_F),\tchi_1^{(1)})\simeq \tchi_1^{(1)}\otimes\Lambda^{p-1}(\fa_{W_2(k)}\oplus\fa_{W_2(k)}^{(1)})$, and this module decomposes as a direct sum of characters of the form $\tchi_1^{(1)}\otimes \eta$ where $\eta$ is a product of $p-1$ characters of the form $-(\tchi_1-\tchi_j)$ or $-(\tchi^{(1)}_1-\tchi_j^{(1)})$ for $j=2,\ldots,p$. By Lemma \ref{borel: combinatorics}, even the mod $p$ reduction of $\tchi_1^{(1)}\otimes \eta$ is a non-trivial character of $T_p(\cO_F)$ unless $\eta=-(\tchi_1-\tchi_2)-\ldots-(\tchi_1-\tchi_{p-1})-(\tchi_1-\tchi_p)=-p\tchi_1$. Therefore $H^{p-1}(A_p(\cO_F),\tchi_1^{(1)})^{T_p(\cO_F)}=(\tchi_1^{(1)}\otimes\tchi_1^{-p})^{T_p(\cO_F)}$. 

Let now $u\in\cO_F^{\times}$ be a unit such that its reduction $\ou$ in $W_2(\bF_{p^2})^{\times}$ satisfies $\Fr_p(\ou)\neq\ou^p$, as provided by Lemma \ref{group cohomology: sorry not sorry}(2). Then the element $\diag(u,u^{-1},1,\ldots,1)\in T_p(\cO_F)$ acts in the character $\tchi_1^{(1)}\otimes\tchi_1^{-p}$ via multiplication by $\Fr_p(\ou)\ou^{-p}$, therefore the $W_2(k)$-module of invariants $(\tchi_1^{(1)}\otimes\tchi_1^{-p})^{T_p(\cO_F)}$ is isomorphic to $k$, which proves part (4).
\end{proof}

\begin{lm}\label{group cohomology: nontrivial character cohomology}
Suppose that $F$ is a number field such that the group of units of $\cO_F$ is infinite. For any split torus $T$ over $\cO_F$, if $\chi:T(\cO_F)\to k^{\times}$ is a non-trivial character then $\RGamma(T(\cO_F),\chi)=0$.
\end{lm}

\begin{proof}
Choose an isomorphism $T(\cO_F)\simeq\bZ^{\oplus N}\times T(\cO_F)^{\tors}$ such that the restriction of $\chi$ to $\bZ^{\oplus N}$ is still nontrivial: this is always possible by choosing a splitting of the short exact sequence $T(\cO_F)^{\tors}\to T(\cO_F)\to T(\cO_F)^{\mathrm{tf}}\simeq \bZ^{\oplus N}$ in such a way that the image of $T(\cO_F)^{\mathrm{tf}}$ under the splitting is not completely contained in $\ker\chi$. We have $\RGamma(T(\cO_F),\chi)=\RGamma(T(\cO_F)^{\tors},\RGamma(\bZ^{\oplus N},\chi))$ so it is enough to show that $\RGamma(\bZ^{\oplus N},\chi)=0$. By definition, $\RGamma(\bZ^{\oplus N},\chi)=\RHom_{k[x_1^{\pm 1},\dots,x_N^{\pm 1}]}(k,\chi)$ where we denote by $\chi$ the module over the group algebra $k[x_1^{\pm 1},\dots,x_N^{\pm 1}]$ corresponding to the character $\chi|_{\bZ^{\oplus N}}$, and $k$ is the module on which all $x_i$ act by $1$. By the assumption that $\chi|_{\bZ^{\oplus N}}$ is non-trivial, $k$ and $\chi$ have disjoint supports in $\Spec k[x_1^{\pm 1},\dots,x^{\pm 1}_N]$ which implies the vanishing.
\end{proof}

Having proven Lemma \ref{group cohomology: integral restriction facts}, we will now deduce Proposition \ref{group cohomology: ring of integers main}. Consider the long exact sequence induced by $0\to\chi_1^p\xrightarrow{\psi_1} \tchi_1^{(1)}\xrightarrow{\psi_2} \chi_1^p\to 0$:

\begin{equation}
\ldots\to H^{p-2}(\chi_1^p)\xrightarrow{\Bock^{p-2}}H^{p-1}(\chi_1^p)\xrightarrow{\psi_1} H^{p-1}(\tchi_1^{(1)})\xrightarrow{\psi_2} H^{p-1}(\chi_1^p)\xrightarrow{\Bock^{p-1}} H^p(\chi_1^p)\to \ldots
\end{equation}
where $H^i(M)$ is the abbreviation for $H^i((T_p\ltimes A_p)(\cO_F),M)$. The multiplication by $p$ map on $H^{p-1}(\tchi_1^{(1)})$ factors as the composition $H^{p-1}(\tchi_1^{(1)})\xrightarrow{\psi_2} H^{p-1}(\chi_1^p)\xrightarrow{\psi_1} H^{p-1}(\tchi_1^{(1)})$. By Lemma \ref{group cohomology: integral restriction facts}(3) the map $\psi_1$ is injective, but Lemma \ref{group cohomology: integral restriction facts}(4) says that the composition $\psi_1\circ\psi_2$ is zero, hence $\psi_2$ is zero itself and the Bockstein homomorphism $H^{p-1}((T_p\ltimes A_p)(\cO_F),\chi_1^p)\to H^p((T_p\ltimes A_p)(\cO_F),\chi_1^p)$ induced by the character $\tchi_1^{(1)}$ is injective.

Consider now the long exact sequence of cohomology associated with the sequence $0\to \chi_1^p\to V^{(1)}\to \chi_2^p\oplus\ldots\oplus \chi_p^p\to 0$. By Lemma \ref{group cohomology: integral restriction facts}(2) we get that the map $H^{p}((T_p\ltimes A_p)(\cO_F),\chi_1^p)\to H^{p}((T_p\ltimes A_p)(\cO_F),V^{(1)})$ is injective. Combined with Lemma \ref{group cohomology: integral restriction facts}(1) this implies that the composition $$H^{p-1}(A_p(\bF_q),V^{(1)})^{T_p(\bF_q)}\to H^{p-1}((T_p\ltimes A_p)(\cO_F),V^{(1)})\xrightarrow{\Bock^{p-1}_{\tV^{(1)}}}H^{p}((T_p\ltimes A_p)(\cO_F),V^{(1)})$$ is injective when restricted to the image of the map $H^{p-1}(A_p(\bF_q),\chi_1^p)^{T_p(\bF_q)}\to H^{p-1}(A_p(\bF_q),V^{(1)})^{T_p(\bF_q)}$. But that map is an isomorphism by Lemma \ref{group cohomology: everything from chi1 algebraic} (2) so Proposition \ref{group cohomology: ring of integers main} is proven.
\end{proof}

\bibliographystyle{alpha}
\bibliography{bibsteenrod}

\end{document}